\documentclass[11pt]{amsart}
\usepackage{amssymb}
\usepackage{amsmath}
\usepackage{graphicx,psfrag}
\usepackage{fancyhdr}

\numberwithin{equation}{section}

\title[Hamilton decompositions of regular expanders]{Hamilton decompositions of regular expanders: a proof of Kelly's conjecture for large tournaments}

\date{\today}
\author{Daniela K\"uhn and Deryk Osthus}
\thanks{The authors were supported by the EPSRC, grant no.~EP/J008087/1. D. K\"uhn was also supported by the ERC, grant no.~258345.} 

\newtheorem{firstthm}{Proposition}[section]
\newtheorem{theorem}[firstthm]{Theorem}
\newtheorem{prop}[firstthm]{Proposition}
\newtheorem{lemma}[firstthm]{Lemma}
\newtheorem{cor}[firstthm]{Corollary}

\newtheorem{obs}[firstthm]{Observation}

\def\noproof{{\unskip\nobreak\hfill\penalty50\hskip2em\hbox{}\nobreak\hfill%
       $\square$\parfillskip=0pt\finalhyphendemerits=0\par}\goodbreak}
\def\endproof{\noproof\bigskip}
\newdimen\margin   
\def\textno#1&#2\par{%
   \margin=\hsize
   \advance\margin by -4\parindent
          \setbox1=\hbox{\sl#1}%
   \ifdim\wd1 < \margin
      $$\box1\eqno#2$$%
   \else
      \bigbreak
      \hbox to \hsize{\indent$\vcenter{\advance\hsize by -3\parindent
      \it\noindent#1}\hfil#2$}%
      \bigbreak
   \fi}
\def\proof{\removelastskip\penalty55\medskip\noindent{\bf Proof. }}

\begin{document}

\def\COMMENT#1{}
\def\TASK#1{}

\def\eps{{\varepsilon}}
\newcommand{\ex}{\mathbb{E}}
\newcommand{\pr}{\mathbb{P}}
\newcommand{\cB}{\mathcal{B}}
\newcommand{\cS}{\mathcal{S}}
\newcommand{\cF}{\mathcal{F}}
\newcommand{\cC}{\mathcal{C}}
\newcommand{\cP}{\mathcal{P}}
\newcommand{\cQ}{\mathcal{Q}}
\newcommand{\cR}{\mathcal{R}}
\newcommand{\cK}{\mathcal{K}}
\newcommand{\cD}{\mathcal{D}}
\newcommand{\cI}{\mathcal{I}}
\newcommand{\cV}{\mathcal{V}}
\newcommand{\cT}{\mathcal{T}}
\newcommand{\eul}{{\rm e}}

\begin{abstract} \noindent
A long-standing conjecture of Kelly states that every regular tournament on $n$ vertices can be decomposed
into $(n-1)/2$ edge-disjoint Hamilton cycles. We prove this conjecture for large $n$.
In fact, we prove a far more general result, based on our recent concept of robust expansion
and a new method for decomposing graphs.
We show that every sufficiently large regular digraph $G$ on $n$ vertices whose degree is linear in $n$ and which is
a robust outexpander has a decomposition into edge-disjoint Hamilton cycles. 
This enables us to obtain numerous further results, e.g.~as a special case we confirm a conjecture of Erd\H{o}s
on packing Hamilton cycles in random tournaments. 
As corollaries to the main result, we also obtain several results on packing Hamilton cycles in undirected graphs,
giving e.g.~the best known result on a conjecture of Nash-Williams.
We also apply our result to solve a problem on the domination ratio of the Asymmetric Travelling Salesman problem,
which was raised e.g.~by Glover and Punnen as well as Alon, Gutin and Krivelevich.
\end{abstract}
\maketitle


\section{Introduction} \label{intro}

\subsection{Kelly's conjecture}
A graph or digraph $G$ has a Hamilton decomposition if it contains a set of edge-disjoint Hamilton cycles which together
cover all the edges of $G$.
The study of Hamilton decompositions is one of the oldest and most natural problems in Graph Theory.
For instance, in 1892 Walecki  showed that the complete graph $K_n$ on $n$ vertices has a
Hamilton decomposition if $n$ is odd (see e.g.~\cite{alspach,abs,lucas}). 
Tillson~\cite{till} solved the corresponding problem for complete digraphs.
Here every pair of vertices is joined by an edge in each direction, and there is a Hamilton decomposition unless the number
of vertices is 4 or 6.

However, though there are several deep conjectures in the area, little progress has been made so far in proving results 
on Hamilton decompositions for general classes of graphs. Possibly the most well known problem in this direction is Kelly's  conjecture from 1968
(see e.g.~the monographs and surveys~\cite{bang,bondy,KOsurvey,moon}),
which states that every regular tournament has a Hamilton decomposition. Here a tournament is an orientation of a complete (undirected) graph.
It is regular if the indegree of every vertex equals its outdegree. This condition is clearly necessary for a Hamilton decomposition.
Here, we prove this conjecture for all large tournaments. In fact, it turns out that we can prove a much stronger result -- 
we can obtain a Hamilton decomposition of any regular orientation of a sufficiently dense graph.
More precisely, an oriented graph $G$ is obtained by orienting the edges of an undirected graph.
So it contains no cycles of length two (whereas in a digraph this is permitted). 

\begin{theorem}\label{orientcor}
For every $\eps>0$ there exists $n_0$ such that every $r$-regular oriented graph $G$ on $n\ge n_0$
vertices with $r\ge 3n/8+\eps n$ has a Hamilton decomposition. In particular, there exists $n_0$ such that
every regular tournament on $n\ge n_0$ vertices has a Hamilton decomposition.
\end{theorem}

It is not clear whether the lower bound on $r$ in Theorem~\ref{orientcor} is best possible.
However, as discussed below, there are oriented graphs whose in- and outdegrees are all very close to $3n/8$ but which 
do not contain even a single Hamilton cycle.
Moreover, for $r < (3n-4)/8$, it is not even known whether an $r$-regular oriented graph contains a single 
Hamilton cycle (this is related to a conjecture of Jackson, see the survey~\cite{KOsurvey} for a more detailed discussion).
Both these facts indicate that any improvement in the lower bound on $r$ would be extremely difficult to obtain.

Regular tournaments obviously exist only if $n$ is odd, but we still obtain an interesting corollary in the even case.
Suppose that $G$ is a tournament on $n$ vertices where $n$ is even and which is as regular as possible, i.e.~the in- and outdegrees differ by~$1$.
Then Theorem~\ref{orientcor} implies that $G$ has a decomposition into edge-disjoint Hamilton paths.
Indeed, add an extra vertex to $G$ which sends an edge to all vertices of $G$ whose indegree is below $(n-1)/2$ and which receives an edge from all others.
The resulting tournament $G'$ is regular, and a Hamilton decomposition of $G'$ clearly corresponds to a decomposition of $G$ into Hamilton paths.

The difficulty of Kelly's conjecture is illustrated by the fact that even the existence of two edge-disjoint Hamilton cycles in a regular 
tournament is not obvious. The first result in this direction was proved by
Jackson~\cite{jackson}, who showed that
every regular tournament on at least 5 vertices contains a Hamilton cycle and a Hamilton path which are edge-disjoint.
Zhang~\cite{zhang} then demonstrated the existence of two edge-disjoint Hamilton cycles.
These results were improved by considering Hamilton cycles in oriented graphs of large in- and outdegree by
Thomassen~\cite{tom1}, H\"aggkvist~\cite{HaggkvistHamilton},
H\"aggkvist and Thomason~\cite{haggtom} as well as Kelly, K\"uhn and Osthus~\cite{kelly}.
Keevash, K\"uhn and Osthus~\cite{keevashko} then showed that every sufficiently large oriented graph~$G$ on $n$ vertices
whose in- and outdegrees are all at least $(3n-4)/8$ contains a Hamilton cycle. This bound on the degrees is best possible
and confirmed a conjecture of 
H\"aggkvist~\cite{HaggkvistHamilton} (as mentioned above, there are extremal constructions which are almost regular). 
Note that this result implies that every sufficiently large 
regular tournament on $n$ vertices contains at least $n/8$ edge-disjoint Hamilton cycles,
whereas Kelly's conjecture requires $(n-1)/2$ edge-disjoint Hamilton cycles. 
The conjecture has also been proved for small values of $n$ and for several special classes of tournaments
(see e.g.~\cite{abs,bt} for somewhat outdated surveys).

Recently, K\"uhn, Osthus and Treglown~\cite{KOTkelly}
proved an approximate version of Theorem~\ref{orientcor} by showing that every $r$-regular oriented graph $G$ on $n\ge n_0(\eps)$
vertices with $r\ge 3n/8+\eps n$ has an approximate Hamilton decomposition
(i.e.~a set of edge-disjoint Hamilton cycles covering almost all edges).


\subsection{Robust outexpanders}
In fact, we prove a theorem which is yet more general than Theorem~\ref{orientcor}.
Moreover, rather than being based on a degree condition, it uncovers an underlying structural property which 
guarantees a Hamilton decomposition. 
As discussed in the next subsection, this property is shared by several well known classes of digraphs.
Roughly speaking, this notion of `robust expansion' is defined as follows:
for any set $S$ of vertices, its robust outneighbourhood is the set of vertices which receive many edges from $S$.
A digraph is a robust outexpander if for every set $S$ which is not too small and not too large, its robust outneighbourhood is at least
a little larger than $S$. This notion was introduced explicitly in~\cite{KOTchvatal}, 
and was already used implicitly in the earlier papers~\cite{keevashko,kelly}. In these papers, we proved approximate and exact versions of
several conjectures on Hamilton cycles in digraphs.

More precisely, let $0<\nu\le  \tau<1$. Given any digraph~$G$ on $n$ vertices and
$S\subseteq V(G)$, the \emph{$\nu$-robust outneighbourhood~$RN^+_{\nu,G}(S)$ of~$S$}
is the set of all those vertices~$x$ of~$G$ which have at least $\nu n$ inneighbours
in~$S$. $G$ is called a \emph{robust $(\nu,\tau)$-outexpander}
if 
$$
|RN^+_{\nu,G}(S)|\ge |S|+\nu n \ \mbox{ for all } \
S\subseteq V(G) \ \mbox{ with } \ \tau n\le |S|\le  (1-\tau)n.
$$
Our main result states that every sufficiently large regular robust outexpander has a Hamilton decomposition.%
   \COMMENT{What we are actually proving is: For all $\alpha$ there exists $\tau_0$ such that for all $\tau\le \tau_0$ there exists
$\nu_0$ such that for all $\nu\le \nu_0$ there exists $n_0(\alpha,\tau,\nu)$ such that.... But this is equivalent to
the statement: it does not make sense to choose $\tau\le \tau_0$ since a robust $(\nu,\tau)$-outexpander is
also a robust $(\nu,\tau_0)$-outexpander.}

\begin{theorem} \label{decomp}
For every $ \alpha >0$ there exists $\tau>0$ such that for all $\nu>0$ there exists $n_0=n_0 (\alpha,\nu,\tau)$ 
for which the following holds. Suppose that
\begin{itemize}
\item[{\rm (i)}] $G$ is an $r$-regular digraph on $n \ge n_0$ vertices, where $r\ge \alpha n$;
\item[{\rm (ii)}] $G$ is a robust $(\nu,\tau)$-outexpander.
\end{itemize}
Then $G$ has a Hamilton decomposition.
Moreover, this decomposition can be found in time polynomial in $n$.
\end{theorem}
Since Lemma~\ref{38robust} states that every oriented graph $G$ on $n$ vertices with minimum in- and outdegree at least $3n/8+\eps n$
is a robust outexpander (provided that $n$ is sufficiently large compared to $\eps$), Theorem~\ref{decomp}
immediately implies Theorem~\ref{orientcor}.

Obviously, the condition that $G$ is regular is necessary. The robust expansion property can be viewed as a natural strengthening of this property:
Indeed, suppose that $\nu^{1/4}\le \tau\le \alpha$ and $|S|\ge \tau n$. Counting the edges from  $S$ to its $\nu$-robust outneighbourhood shows
that condition (i) already forces the $\nu$-robust outneighbourhood of $S$ to have size 
at least $(1-\sqrt{\nu})|S|$.%
    \COMMENT{$r(1-\sqrt{\nu})|S| \le r|S| -\nu n^2\le e(S,RN(S)) \le r|RN(S)|$ since $r\sqrt{\nu}|S|\ge \alpha  \sqrt{\nu}|S|n\ge \nu n^2$
(the latter inequality holds since $|S|\ge \tau n\ge \sqrt{\nu}n/\alpha$ as $\sqrt{\nu}\le \tau^2\le \tau\alpha$)}
Condition (ii) then ensures (amongst others) that $G$ is highly connected, which is obviously also necessary.

The following result of Osthus and Staden~\cite{OS}
gives an approximate version of Theorem~\ref{decomp} and will be used in its proof.
\begin{theorem}\label{approxdecomp}
For every $ \alpha >0$ there exists $\tau>0$ such that for all $\nu,\eta>0$ there exists $n_0=n_0 (\alpha,\nu,\tau,\eta)$ 
for which the following holds. Suppose that
\begin{itemize}
\item[{\rm (i)}] $G$ is an $r$-regular digraph on $n \ge n_0$ vertices, where $r\ge \alpha n$;
\item[{\rm (ii)}] $G$ is a robust $(\nu,\tau)$-outexpander.
\end{itemize}
Then $G$ contains at least $(1 -\eta)r$ edge-disjoint Hamilton cycles.
Moreover, this set of Hamilton cycles can be found in time polynomial in $n$.
\end{theorem} 
Note that Theorem~\ref{approxdecomp} is a generalization of the approximate version of 
Theorem~\ref{orientcor} proved in~\cite{KOTkelly} (which was already mentioned at the end of the previous subsection).
The approach in~\cite{KOTkelly} is not algorithmic though.
If we replace the use of Theorem~\ref{approxdecomp} in our main proof by the result in~\cite{KOTkelly}, 
this yields exactly Theorem~\ref{orientcor} (but not its algorithmic version).
However, this would not result in a shorter proof of Theorem~\ref{decomp}.

\subsection{Further applications}
In this section, we briefly discuss further applications of our main result.
\subsubsection{Regular digraphs and TSP tour domination}
As observed in Section~\ref{sec:proofs2}, it is very easy to check that regular digraphs of sufficiently large degree are robust outexpanders.
Together with Theorem~\ref{regdigraph} this implies the following result.

\begin{theorem}\label{regdigraph}
For every $\eps>0$ there exists $n_0$ such that every $r$-regular digraph $G$ on $n\ge n_0$ vertices with $r\ge (1/2+\eps)n$
has a Hamilton decomposition. Moreover, such a decomposition can be found in time polynomial in $n$.
\end{theorem}

Surprisingly, this has an immediate application to the area of TSP tour domination.
More precisely, the Asymmetric travelling salesman problem (ATSP) asks for a Hamilton cycle of least weight in an 
edge-weighted complete digraph (where opposite edges are allowed to have different weight).
An algorithm $A$ for the ATSP has \emph{domination ratio} $p(n)$ if it has the following property. For any problem instance $I$
let $w(I)$ be the weight of the solution produced by $A$. Then for all $n$ and for all instances $I$ on $n$ vertices, 
there are at least $p(n) (n-1)!$ solutions to instance $I$ whose weight is also at least $w(I)$.
(Note that the total number of possible solutions is $(n-1)!$.)
This notion is of particular interest for the ATSP as it is not known whether there is a polynomial time algorithm for the ATSP whose
approximation ratio is bounded by an absolute constant. Several well known TSP algorithms achieve a domination ratio of $\Omega(1/n)$ for the ATSP
but no better results are known. In particular,
a long-standing open problem (see e.g.~Glover and Punnen~\cite{GP}, Gutin and Yeo~\cite{GutinYeo} as well as Alon, Gutin and Krivelevich~\cite{AGK}) asks whether
there is a polynomial time algorithm which achieves a constant domination ratio for the ATSP.
Gutin and Yeo~\cite{GutinYeo} proved that the existence of a polynomial time algorithm with domination ratio $1/2-\eps$ would follow from an algorithmic
proof of Theorem~\ref{regdigraph}. So the result of~\cite{GutinYeo} together with Theorem~\ref{regdigraph} yields the following.
\begin{cor}
For any $\eps>0$, there is a polynomial time algorithm for the ATSP whose domination ratio is $1/2-\eps$.
\end{cor}
\subsubsection{Random tournaments}
Another application of Theorem~\ref{decomp} confirms a conjecture of Erd\H{o}s (see~\cite{thomassen79}) which can be regarded as a probabilistic version of Kelly's conjecture.
Given an oriented graph $G$, let $\delta^+(G)$ denote its minimum outdegree and $\delta^-(G)$ its minimum indegree.
Clearly, the minimum of these two quantities is an upper bound on the number of edge-disjoint Hamilton cycles that $G$ can have.
Erd\H{o}s conjectured that this bound is correct with high probability if $G$ is a random tournament and one can use Theorem~\ref{orientcor} to show
this is indeed the case. 
\begin{theorem} \label{erdosdecomp}
Almost all tournaments $G$ contain $\delta^0(G):=\min\{ \delta^+(G), \delta^-(G) \}$ edge-disjoint Hamilton cycles.
\end{theorem}
More precisely, the term `almost all' refers to the model where one considers the set $\cT_n$ of all tournaments on $n$ vertices and shows that the probability that 
a random tournament in $\cT_n$ has the required number of Hamilton cycles tends to $1$ as $n$ tends to infinity.
We prove Theorem~\ref{erdosdecomp} by showing that with high probability $G$ contains a $\delta^0(G)$-regular spanning subdigraph $G'$ and apply Theorem~\ref{orientcor} to $G'$
to find the required Hamilton cycles. As the first step requires some work, we defer this to a shorter companion paper~\cite{KellyII}.

The corresponding problem for the binomial random graph $G_{n,p}$ with edge probability $p$ has a long history, 
going back to a result of Bollob\'as and Frieze~\cite{BF85}, who showed that
a.a.s.~(asymptotically almost surely) $G_{n,p}$ contains $\lfloor \delta(G_{n,p})/2 \rfloor$ edge-disjoint Hamilton cycles in the range of $p$ where the minimum degree $\delta(G_{n,p})$
is a.a.s.~bounded.
A striking conjecture of Frieze and Krivelevich~\cite{FK08} asserts that this result extends to arbitrary edge probabilities $p$.
The range of $p$ was extended in several papers, in particular due to recent results of 
Knox, K\"uhn and Osthus~\cite{KnoxKO} as well as Krivelevich and Samotij~\cite{KrS}, the conjecture remains open only  in the (rather special) case when $p$ tends to  $1$ fairly quickly.
As we shall observe in~\cite{KellyII}, this case follows from Theorem~\ref{decomp} in a similar way as Theorem~\ref{erdosdecomp}.

\subsubsection{Undirected robust expanders}
In~\cite{KellyII}, we also derive an undirected version of Theorem~\ref{decomp},
where instead of the `robust outneighbourhood' we consider the `robust neighbourhood'.
As an immediate corollary of this undirected version, we obtain the following approximate version of the `Hamilton decomposition conjecture' of Nash-Williams~\cite{Nash-Williams71b}. 
\begin{theorem}\label{reggraph}
For every $\eps>0$ there exists $n_0$ such that every $r$-regular graph $G$ on $n\ge n_0$ vertices, where $r\ge (1/2+\eps)n$
is even, has a Hamilton decomposition.
\end{theorem} 
The conjecture of Nash-Williams asserts that the $\eps n$ error term can be removed. 
Theorem~\ref{reggraph} improves results by Christofides, K\"uhn and Osthus~\cite{CKO} as well as Perkovic and Reed~\cite{PRedges}.

Finally, the undirected version of Theorem~\ref{decomp} easily implies that every even-regular dense quasi-random graph has a Hamilton decomposition.
An approximate version of this result was proved earlier by Frieze and Krivelevich~\cite{fk}.
Our undirected decomposition result implies e.g.~a recent result of Alspach, Bryant and Dyer~\cite{Paley} that every Paley graph has a Hamilton decomposition
(for the case of large graphs).
Our undirected decomposition result also implies that with high probability, dense random regular graphs of even degree have a Hamilton decomposition. 
(Hamilton decompositions of random regular graphs of bounded degree have already  been studied intensively.)
These and other related results 
are discussed in more detail in~\cite{KellyII}.

\subsubsection{The robust decomposition lemma}
In a sequence of four papers~\cite{paper2,paper3,paper1,paper4} we build on the results of the current paper to prove the long-standing
`$1$-factorization conjecture': Suppose  that $n$ is even and sufficiently large, and that $D\geq 2\lceil n/4\rceil -1$. 
Then every $D$-regular graph $G$ on $n$ vertices has a decomposition into perfect matchings.
(Equivalently, $\chi'(G)=D$.)
Moreover, we improve Theorem~\ref{reggraph} to  completely solve the Hamilton decomposition conjecture of Nash-Williams from~\cite{Nash-Williams71b}.
Finally, we also solve another problem of Nash-Williams on optimal packings of edge-disjoint Hamilton cycles in graphs of large minimum degree
(the latter also uses results from~\cite{KLOmindeg}).
The proofs are based on (the undirected version of) Theorem~\ref{decomp} as well as the `robust decomposition lemma',
which can be viewed as a version of Theorem~\ref{decomp} which is more technical but has the advantage of being more widely applicable
(see Lemma~\ref{hcdec} or~\ref{rdeclemma}).

In the next section, we give a brief outline of our methods. 
The approach is a very general one and we are certain that it will have significant further applications.

\section{A brief outline of the argument}

\subsection{The general approach}
The basic idea behind the proof of Theorem~\ref{decomp} can be described as follows.
Let $G$ be a robustly expanding digraph as in Theorem~\ref{decomp}.
Suppose that inside $G$ we can find a sparse regular digraph $H^{\rm rob}$ which is robustly decomposable in the sense that it still has a 
Hamilton decomposition if we add a few edges to it. More precisely, $H^{\rm rob}$ is \emph{robustly decomposable} if $H^{\rm rob} \cup H_0$ has a Hamilton decomposition
whenever $H_0$ is a very sparse regular digraph which is edge-disjoint from $H^{\rm rob}$ and such that $V(H)=V(H^{\rm rob})$. Then 
Theorem~\ref{decomp} would be an immediate consequence of the existence of such an $H^{\rm rob}$.
Indeed, first we remove the edges of $H^{\rm rob}$ from $G$ to obtain $G'$. Then we apply Theorem~\ref{approxdecomp} to $G'$ to find edge-disjoint Hamilton cycles
covering almost all edges of $G'$. Let 
$H_0$ be the leftover -- i.e.~the set of edges of $G'$ which are not covered by one of these Hamilton cycles. 
Now apply the fact that $H^{\rm rob}$ is robustly decomposable to obtain a Hamilton decomposition of $H^{\rm rob} \cup H_0$ and thus of $G$.
Essentially, this is what the `robust decomposition lemma' (Lemma~\ref{hcdec}) achieves.

Unfortunately, we do not know how to construct such a digraph $H^{\rm rob}$ directly -- the main problem is that we have almost no control over what $H_0$ might look like.
However, one key idea is that we can define and find several digraphs $H_i^{\rm rob}$ which together play the role of $H^{\rm rob}$.
Indeed, suppose that we remove the edges of several $H_i^{\rm rob}$ at the start and let $H_0$ be the leftover of the approximate decomposition as above. 
Then we can show that $H_1^{\rm rob} \cup H_0$ contains a set of edge-disjoint Hamilton cycles so that the resulting leftover $H_1$
has more structure than $H_0$ (i.e.~it has some useful properties). 
This in turn means that we can improve on the previous step and now find a set of edge-disjoint Hamilton cycles in $H_2^{\rm rob} \cup H_1$ so that the resulting leftover $H_2$
has even more structure than $H_1$.  After $\ell-1$ steps, $H_{\ell-1}$ will be a sufficiently `nice' digraph so that $H_{\ell}^{\rm rob} \cup H_{\ell-1}$ does have a Hamilton decomposition.
This very general approach was first introduced in~\cite{KnoxKO}, where we used it to find optimal packings of edge-disjoint Hamilton cycles in random graphs.
In~\cite{KnoxKO}, the aim was that successive $H_i$ become sparser. 
In our setting, the density is less relevant -- our aim is to obtain successively stronger structural properties for the $H_i$.

When finding Hamilton cycles in $H_i^{\rm rob} \cup H_{i-1}$, we usually proceed as follows (except for the final step, i.e.~when $i=\ell$).  
First we decompose $H_{i-1}$ into edge-disjoint $1$-factors (where a $1$-factor is a spanning union of vertex-disjoint cycles). 
Each of these $1$-factors is then split into a set of paths in a suitable way
(we call this a path system).
In particular, the number of edges in each path system is small compared to $n$.
Then we extend each path system into a suitable $1$-factor $F$ using edges of $H_i^{\rm rob}$.
Then we transform $F$ into a Hamilton cycle~$C$, again using edges of $H_i^{\rm rob}$.
The resulting set of Hamilton cycles then covers all edges of $H_{i-1}$.
In other words, $H_{i-1}$ is `absorbed' into the set of Hamilton cycles that we have constructed so far and the leftover
$H_i$ is a subgraph of $H_i^{\rm rob}$. In particular, $H_i$ inherits any structural properties that $H_i^{\rm rob}$ has.
When constructing the Hamilton cycle $C$, we will use the result of~\cite{KOTchvatal} that every robustly expanding digraph contains a 
Hamilton cycle, or one of its corollaries (see Section~\ref{sec:hamtools}).


\subsection{The main steps} 
The construction of the graphs $H^{\rm rob}_i$ involves Szemer\'edi's regularity lemma.
We apply this to $G$ to obtain a partition of its vertices into clusters $V_1,\dots, V_k$ and a small exceptional set $V_0$
so that almost all ordered pairs of clusters induce a pseudorandom subdigraph of $G$.
As is well known, one can then define a `reduced digraph' $R$ whose vertex set consists of the clusters $V_i$, with an edge from 
$V_i$ to $V_j$, if the subdigraph of $G$ induced by the edges of $G$ from $V_i$ to $V_j$ is pseudorandom and dense.
$R$ inherits many of the properties of $G$. In particular, $R$ is also a robust outexpander.
So it contains a Hamilton cycle $C$ by the result mentioned above -- without loss of generality $C=V_1V_2 \dots V_k$.

We can now define three digraphs playing the role of $H_1^{\rm rob},H_2^{\rm rob},H_3^{\rm rob}$ above (so we have $\ell =3$ in our setting):
\begin{itemize}
\item the preprocessing graph $PG$;
\item the chord absorber $CA$;
\item the parity extended cycle absorber $PCA$.
\end{itemize}
We now describe the purpose of these three digraphs (see also Figure~\ref{fig:diagr}).
\begin{figure}
\centering\footnotesize
\includegraphics[scale=0.35]{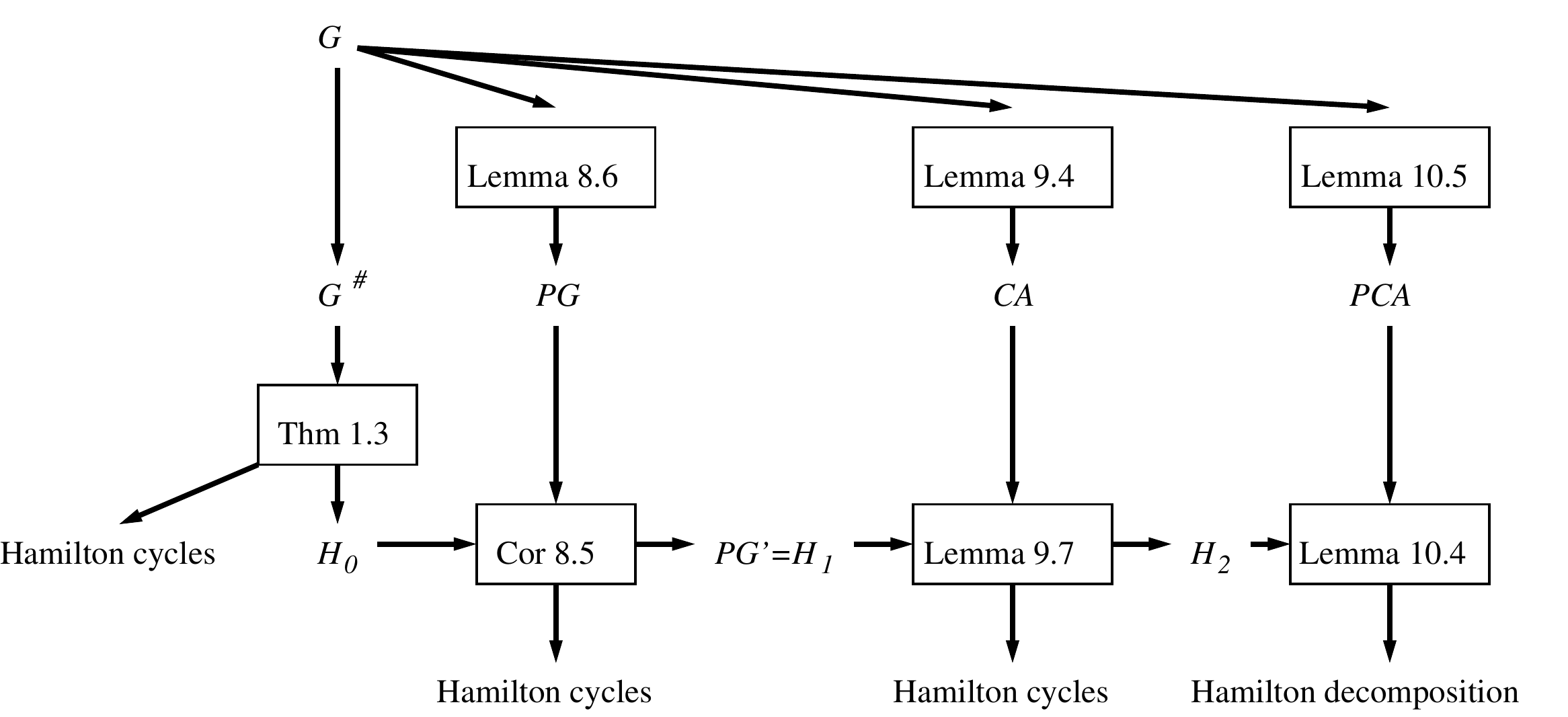}
\caption{An illustration of the structure of the proof of Theorem~\ref{decomp}. The purpose of Lemmas~\ref{findprepro},~\ref{findchordabs} and~\ref{find_pa}
is to find the preprocessing graph $PG$, the chord absorber $CA$ and the parity extended cycle absorber $PCA$ respectively.
The purpose of Corollary~\ref{preprocor}, Lemmas~\ref{absorballH} and~\ref{cycleparity} is to use these graphs to find suitable Hamilton cycles.
In Section~\ref{sec:rdeclemma}, we combine Lemmas ~\ref{findchordabs},~\ref{find_pa},~\ref{absorballH} and~\ref{cycleparity} into a single
`robust decomposition lemma' (see Lemma~\ref{hcdec} or~\ref{rdeclemma}).}
\label{fig:diagr}
\end{figure}

The \emph{preprocessing graph $PG$} has the following property: let $H_1$ be the leftover of $PG \cup H_0$ after removing suitably chosen edge-disjoint Hamilton cycles.
(So as discussed above, $H_1$ is a subdigraph of $PG$.) Then $H_1$ has no edges incident to $V_0$.
Thus $H_1$ is not a regular digraph (but it will be a regular subdigraph of $G- V_0$).
This degree irregularity will be compensated for by constructing the chord absorber $CA$ in a suitable way.

The \emph{chord absorber $CA$} has the following property: let $H_2$ be the leftover of $CA \cup H_1$ after removing suitably chosen edge-disjoint Hamilton cycles.
Then $H_2$ is a blow-up of $C$. In other words, for every edge $e$ of $H_2$, there is a $j$ so that the initial vertex of $e$ is in $V_j$ and the final one in $V_{j+1}$.
(So the edges of $H_2$ `wind around'~$C$.) $H_2$ will be a subdigraph of $CA$. 
But $CA$ itself will not only consist of a blow-up $\cB(C)$ of $C$, it will also contain a set of suitably chosen `chord edges' 
between clusters which are not adjacent on $C$. These chord edges lie in a digraph $\cB(U')$, which is a blow-up of a `universal walk' $U$ on the clusters
$V_i$.
To absorb $H_1$, we will split its edges into path systems $M$ as described above. We would like to extend each $M$ into a Hamilton cycle.
This may be impossible using the edges of $\cB(C)$ alone, e.g.~if $M$ contains an edge $e_1$ from $V_1$ to $V_3$ but no other edges.
So for each edge $e$ of such a path system $M$, we then choose a set of chord edges from $\cB(U')$
which `balance out' this edge $e$ to form a `locally balanced sequence'. 
We extend (and balance) $M$ in this way to obtain a path system $M'$.
We then further extend $M'$ to a Hamilton cycle using edges of $\cB(C)$.

As an example, suppose again that $M=\{e_1 \}$ with $e_1$ being an edge from $V_1$ to $V_3$.
It turns out that a simple way of balancing $e_1$ would be to add an edge $e_j$ from $V_j$ to $V_{j+2}$ for all $j$ with $1 < j  \le k$,
so that $e_1,e_2,\dots,e_k$ form a matching $M'$. 
It is easy to see that one can extend  $M'$ into a Hamilton cycle using edges from~$\cB(C)$, i.e.~edges which only wind around $C$.
Indeed, start by traversing $e_1$, then wind around $C$ to reach the initial vertex of $e_2$, then traverse $e_2$, then wind around $C$ again, 
and eventually traverse $e_k$. Before returning to the initial vertex of $e_1$,
wind around $C$ sufficiently many times to visit every vertex in 
every cluster, thus obtaining a Hamilton cycle which contains $M'$.
Unfortunately, we cannot guarantee the existence of such edges $e_2,\dots, e_k$ in an arbitrary robust outexpander.
So we will use sequences of balancing edges which involve more edges but can be found in any robust outexpander.
(They will be based on the concept of `shifted walks' which we introduced in~\cite{kelly}).

The crucial point is that we can carry out the balancing in such a way that we use up \emph{all} edges of $\cB(U')$ in the process
of balancing out the path systems of $H_1$. In particular, the surprising feature of the argument is that we can choose $\cB(U')$ in advance (i.e.~without knowing $H_1$)
so that it has this property.

In the above, we did not discuss edges incident to the exceptional set $V_0$. Obviously a Hamilton cycle has to contain these,
whereas neither of $H_1$, $\cB(C)$ and $\cB(U')$ have any edges incident to $V_0$.
To deal with this, we introduce the following trick:
suppose for example that $V_0$ contains a single exceptional vertex $x$.
We find an outneighbour $x^+$ of $x$ and an inneighbour $x^-$ in $V(G)\setminus V_0$. We can then define an `exceptional edge' $x^-x^+$ and add this edge 
$x^-x^+$ to the chord absorber $CA$. Then a Hamilton cycle of $(CA \cup H_1)-V_0$ containing this exceptional edge corresponds to a Hamilton cycle of $G$.
The systematic use of exceptional edges in this way allows us to ignore the exceptional set $V_0$ at many points of the argument.
It might seem natural to apply this trick directly to $H_0$ (i.e.~replace pairs of edges of $H_0$ incident to $V_0$ with exceptional edges)
with the aim of making the preprocessing step unnecessary.
However, this approach would run into considerable difficulties. It turns out that working inside $H_0$ (rather than the whole of $G$) when constructing the exceptional edges
might not give us enough choice to find suitable sets of exceptional edges.

Finally, the \emph{parity extended cycle absorber} $PCA$ has the following property:
let $H_2$ be the leftover of $CA \cup H_1$ after removing suitable Hamilton cycles.
Then $H_2 \cup PCA$ has a Hamilton decomposition.
We will find this Hamilton decomposition as follows:
first we decompose $PCA \cup H_2$ into carefully chosen $1$-factors (in particular, either half or all of the edges of each $1$-factor 
are contained in $PCA$). 
When finding this decomposition, we use the fact that $H_2$ is a blow-up of $C$. ($PCA$ will also be a blow-up of $C$.)
Our aim is then to turn this $1$-factorization of $PCA \cup H_2$ into 
a Hamilton decomposition of $PCA \cup H_2$ by successively switching edges between $1$-factors.
As an illustration, suppose that we are given two $1$-factors $F$ and $F'$ so that $F$ contains the edges $xx^+$ and $yy^+$.
Similarly, suppose that $F'$ contains the edges $xy^+$ and $yx^+$.
Note that these edges form an orientation of a four-cycle $C_4$. Now perform a `switch', which consists of
removing $xx^+$ and $yy^+$ from $F$, adding $xy^+$ and $yx^+$ to $F$ and proceeding similarly for $F'$.
This yields two new $1$-factors $F_{\rm new}$ and $F'_{\rm new}$. 
Suppose that $F$ consists of exactly two cycles and that $xx^+$ and $yy^+$ lie on different cycles. Then
$F_{\rm new}$ is a Hamilton cycle. The same holds for $F'$ (see Figure~\ref{fig:switch}).
\begin{figure}
\centering\footnotesize
\includegraphics[scale=0.45]{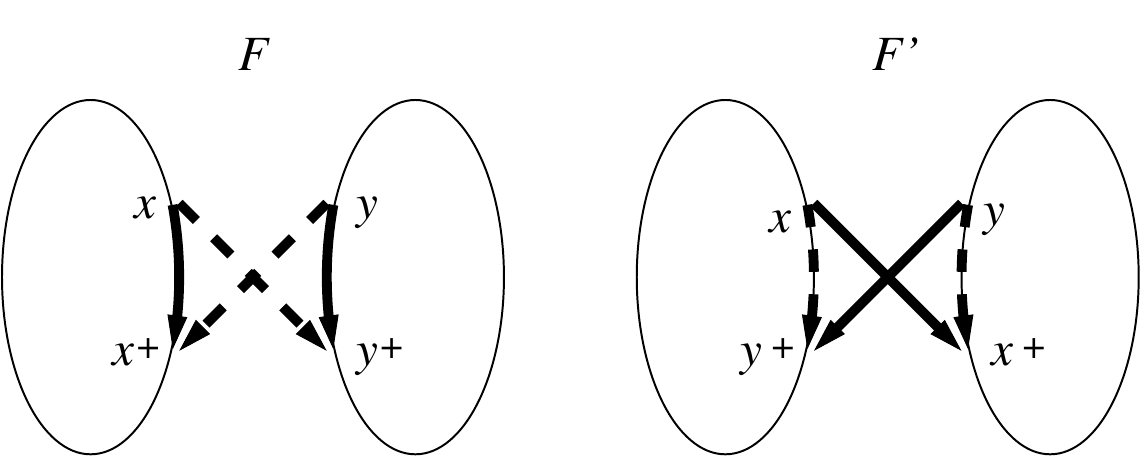}
\caption{Transforming the $1$-factors $F$ and $F'$ consisting of two cycles into Hamilton cycles by switching edges.}
\label{fig:switch}
\end{figure}
These switches will always involve edges from $PCA$ and not from $H_2$.
So if we ensure that $PCA$ has switches at the right places, we can eventually turn the $1$-factorization of $PCA\cup H_2$
into a Hamilton decomposition after several switches.

This paper is organized as follows: 
In the next section, we introduce some notation.
In Section~\ref{sec:regularity}, we collect tools which we will need in connection with Szemer\'edi's regularity lemma.
Similarly, in Section~\ref{sec:expand} we collect general properties of robustly expanding digraphs.
Section~\ref{sec:hamtools} is devoted to tools for finding Hamilton cycles (in robustly expanding digraphs).
In Section~\ref{sec:except}, we introduce a systematic and convenient way of dealing with exceptional vertices  which will be used
throughout the remainder of the paper. This will be based on the concept of exceptional edges and `balancing' edges via chord sequences.
Section~\ref{seccyclebreak} deals with the preprocessing step, which involves the preprocessing graph $PG$.
Then, in Section~\ref{sec:chordabsorb}, we define, find and use the chord-absorber $CA$.
Switches and the parity extended cycle absorber $PCA$ are then introduced in Section~\ref{sec:switches}.
In Section~\ref{sec:stopseln}, we put everything together to prove Theorem~\ref{decomp}.
In Section~\ref{sec:rdeclemma}, we state a standalone variant of the `robust decomposition lemma',
for use e.g.~in~\cite{paper2,paper1}.
Finally, in Section~\ref{sec:proofs2} we derive Theorems~\ref{orientcor} and~\ref{regdigraph} from Theorem~\ref{decomp}.

\section{Notation and probabilistic estimates} \label{notation}

\subsection{Notation}
Given a graph or digraph $G$, we write $V(G)$ for its vertex set, $E(G)$ for its edge set, $e(G):=|E(G)|$ for the number of its edges and $|G|$ for the number of its
vertices. Given $X\subseteq V(G)$, we write $G-X$ for the (di)graph obtained from $G$ by deleting all vertices in~$X$.
Given $F\subseteq E(G)$, we write $G\setminus F$ for (di)graph obtained from $G$ by deleting all edges in~$F$.
If $H$ is a sub(di)graph of $G$, we write $G\setminus H$ for $G\setminus E(H)$.

Suppose that $G$ is an undirected graph. We write $\delta(G)$ for the minimum degree of $G$ and $\Delta(G)$ for its maximum degree.
Whenever $X,Y\subseteq V(G)$, we write $e_G(X,Y)$ for the number of all those edges of $G$ which have one endvertex in $X$ and
the other endvertex in $Y$. If $X\cap Y=\emptyset$, we denote by $G[X,Y]$ the bipartite subgraph of $G$
with vertex sets $X$ and $Y$ whose edges are all the edges of $G$ between $X$ and $Y$. 
If $G$ is a bipartite graph with vertex classes $A$ and $B$, we often write $G=(A,B)$.

If $G$ is a digraph, we write $xy$ for an edge directed from $x$ to $y$. Unless stated otherwise, when we refer to paths and cycles in digraphs, we mean
directed paths and cycles, i.e.~the edges on these paths/cycles are oriented consistently. Given two vertices $x$ and $y$ on
a directed cycle~$C$, we write $xCy$ for the (directed) subpath of $C$ from~$x$ to~$y$.
If $x$ is a vertex of a digraph $G$, then $N^+_G(x)$ denotes the \emph{outneighbourhood} of $x$, i.e.~the
set of all those vertices $y$ for which $xy\in E(G)$. Similarly, $N^-_G(x)$ denotes the \emph{inneighbourhood} of $x$, i.e.~the
set of all those vertices $y$ for which $yx\in E(G)$. We write $d^+_G(x):=|N^+_G(x)|$ for the \emph{outdegree} of $x$ and
$d^-_G(x):=|N^-_G(x)|$ for its \emph{indegree}. We denote the \emph{minimum outdegree} of $G$ by $\delta^+(G):=\min_{x\in V(G)} d^+_G(x)$,
the \emph{minimum indegree} by $\delta^-(G):=\min_{x\in V(G)} d^-_G(x)$, the \emph{minimum degree} by
$\delta(G):=\min_{x\in V(G)} (d^+(x)+d^-(x))$ and the \emph{maximum degree} by $\Delta(G):=\max_{x\in V(G)} (d^+(x)+d^-(x))$.
The \emph{minimum semidegree} of $G$ is $\delta^0(G):=\min\{\delta^+(G),\delta^-(G)\}$.
Whenever $X,Y\subseteq V(G)$, we write $e_G(X,Y)$ for the number of all those edges of $G$ which have their initial vertex in $X$ and
their final vertex in $Y$. If $X\cap Y=\emptyset$, we denote by $G[X,Y]$ the bipartite subdigraph of $G$
with vertex sets $X$ and $Y$ whose edges are all the edges of $G$ directed from $X$ to $Y$.
In all these definitions we often omit the subscript $G$ if the graph or digraph $G$ is clear from the context.
A subdigraph $H$ of $G$ is an \emph{$r$-factor} of $G$ if the outdegree and the indegree of every vertex of $H$ is~$r$.
A \emph{path system} is the union of vertex-disjoint directed paths.

Given a digraph $R$ and a positive integer $r$, the {\em $r$-fold blow-up of
$R$} is the digraph $R \times E_r$ obtained from $R$ by replacing every vertex $x$ of $R$ by $r$ vertices
and replacing every edge $xy$ of $R$ by the oriented complete bipartite graph $K_{r,r}$ between the two
sets of $r$ vertices corresponding to $x$ and $y$ in which all the edges are oriented towards the $r$ vertices corresponding
to $y$. Now consider the case when $V_1,\dots,V_k$ is a partition of some set $V$ of vertices and
$R$ is a digraph whose vertices are $V_1,\dots,V_k$. Then a \emph{blow-up $\cB(R)$ of $R$} is obtained
from $R$ by replacing every vertex $V_i$ of $R$ by the vertices in $V_i$ and replacing every edge $V_iV_j$ of $R$ by
a certain bipartite graph with vertex classes $V_i$ and $V_j$ in which all the edges
are oriented towards the vertices in~$V_j$. 
Usually, these bipartite graphs will be $\eps$-regular or superregular
(as defined in Section~\ref{sec:regularity}).
If $R$ is a directed cycle, say $R=C=V_1\dots V_k$ and $G$ is a digraph with $V(G)\subseteq V=V_1\cup\dots\cup V_k$, we say that
\emph{(the edges of) $G$ wind(s) around~$C$} if for every edge $xy$ of $G$ there exists an index $j$ such that $x\in V_j$ and $y\in V_{j+1}$.
So if $V(G)=V$ then $G$ winds around $C$ if and only if $G$ is a blow-up of $C$.

In order to simplify the presentation, we omit floors and ceilings and treat large numbers as integers whenever this does
not affect the argument. The constants in the hierarchies used to state our results have to be chosen from right to left.
More precisely, if we claim that a result holds whenever $0<1/n\ll a\ll b\ll c\le 1$ (where $n$ is the order of the graph or digraph), then this means that
there are non-decreasing functions $f:(0,1]\to (0,1]$, $g:(0,1]\to (0,1]$ and $h:(0,1]\to (0,1]$ such that the result holds
for all $0<a,b,c\le 1$ and all $n\in \mathbb{N}$ with $b\le f(c)$, $a\le g(b)$ and $1/n\le h(a)$. 
We will not calculate these functions explicitly. Hierarchies with more constants are defined in a similar way.
Moreover, we will often assume that certain numbers involving these constants, e.g.~$b/a^2$ or $a n$ are integers. We will only make this
assumption if our hierarchy guarantees that these numbers are sufficiently large since then by adjusting the constants slightly one
can actually guarantee that these numbers are integers. However, all our results will hold if $n$ is sufficiently large, i.e.~we will make
no divisibility assumptions on $n$. (Note that if we assume that $an$ is an integer then this can be achieved by adjusting the constant $a$ slightly.)

\subsection{Probabilistic estimates, derandomization and algorithmic aspects}
We will use the following standard Chernoff type bound (see e.g.~Corollary 2.3 in~\cite{JLR} and Theorem~2.2 in~\cite{SrSt}).

\begin{prop} \label{chernoff} 
Suppose $X$ has binomial distribution and $0<a<1$. Then
$$
\mathbb{P}(X  \ge (1+a)\mathbb{E}X) \le \eul^{-\frac{a^2}{3}\mathbb{E}X}
\mbox{ and } \mathbb{P}(X  \le (1-a)\mathbb{E}X) \le \eul^{-\frac{a^2}{3}\mathbb{E}X}.
$$
\end{prop}
To obtain an algorithmic version of Theorem~\ref{decomp}, we need to `derandomize' our applications of Proposition~\ref{chernoff}.
This can be done via the well known `method of conditional probabilities', which is based on an idea of Erd\H{o}s and Selfridge,
and which was further developed e.g.~by Spencer as well as Raghavan.
The following result of Srivastav and Stangier (Theorem~2.10 in~\cite{SrSt}) 
is also based on this method.
Given a probabilistic existence proof of some structure based on polynomially many applications of Proposition~\ref{chernoff},
it guarantees an algorithm which finds this structure.

Suppose we are given $N$ independent $0/1$ random variables $X_1,\dots,X_N$ where $\pr (X_j=1)= p$ and $\pr (X_j =0)=1-p$
for some rational $0 \le p \le 1$. Suppose that $1 \le i \le m$. Let $w_{ij} \in \{0,1\}$.
Denote by $\phi_i$ the random variables $\phi_i:=\sum_{j=1}^N w_{ij}X_j$. Fix $\beta_i$ with $0 < \beta_i < 1$.
Now let $E_i^+$ denote the event that $\phi_i \ge (1+\beta_i)\ex [ \phi_i]$ and 
let $E_i^-$ denote the event that $\phi_i \le (1-\beta_i)\ex [ \phi_i]$.
Let $E_i$ be  either $E_i^+$ or $E_i^-$. 
Suppose that
\begin{equation} \label{eq:3.3}
\sum_{i=1}^m {\rm e}^{- \beta^2_i \ex(\phi_i)/3} \le 1/2.
\end{equation}
\begin{theorem}{\cite{SrSt}} \label{derandom}
Let $E_1,\dots,E_m$ be events such that~(\ref{eq:3.3}) holds. Then $$\pr \left(\bigcap_{i=1}^m E_i\right) \ge 1/2$$
and a vector $x \in  \bigcap_{i=1}^m E_i$ can be constructed in time $O (m N^2 \log (mN) )$.
\end{theorem}
In general, it will usually be clear that the proofs can be translated into polynomial time algorithms. 
Where this is not obvious, we will add a corresponding remark. We make no attempt to prove an explicit bound on the time needed to find the 
Hamilton decomposition (beyond the fact that it is polynomial in $n$). 


\section{Regularity}\label{sec:regularity}

\subsection{The Regularity Lemma}
We will use of a directed version of Szemer\'edi's regularity lemma. If
$G=(A,B)$ is an undirected bipartite graph with vertex classes $A$ and $B$, then the
\emph{density} of $G$ is defined as
$$ d(A, B) := \frac{e_G(A,B)}{|A||B|}.$$
For any $\eps >0$, we say that $G$ is \emph{$\eps$-regular} if for any $A'\subseteq A$
and $B' \subseteq B$ with $|A'| \geq \eps |A|$ and $|B'| \geq \eps |B|$
we have $|d(A',B') - d(A, B)| < \eps$. 
We say that $G$ is \emph{$(\eps, \ge d)$-regular} if
it is $\eps$-regular and has density $d'$ for some $d' \ge d-\eps$.

Given disjoint vertex sets $X$ and $Y$ in a digraph $G$, recall that $G[X,Y]$ denotes
the bipartite subdigraph of $G$ whose vertex classes are $X$ and $Y$ and whose
edges are all the edges of $G$ directed from $X$ to $Y$. We often view $G[X,Y]$ as
an undirected bipartite graph. In particular, we say $G[X,Y]$ is \emph{$\eps$-regular} or
\emph{$(\eps,\ge d)$-regular} if this holds when $G[X,Y]$ is viewed as an undirected graph.

Next we state
the degree form of the regularity lemma for digraphs. A regularity lemma for
digraphs was proved by Alon and Shapira~\cite{AS}. The degree form follows from
this in the same way as the undirected version (see~\cite{BCCsurvey}
for a sketch of the latter). An algorithmic version of the (undirected) regularity lemma was proved in~\cite{Aetal}.
An algorithmic version of the directed version can be proved in essentially the same way
(see~\cite{CKKO} for a sketch of the argument proving a similar statement).%
    \COMMENT{There the "moreover part" is replaced by "all but at most $\eps k^2$ pairs $i,j$ satisfy 
$G'[V_i,V_j]=G[V_i,V_j]$ or $d_G(V_i,V_j)<d$". The proof goes as follows: first find a regularity partition in
polytime as there (with $\eps'\ll \eps$). The problem is that we don't know which pairs are $\eps'$-regular.
But by checking the degrees and the codegrees we can decide (in polytime) whether a pair is not $\eps'$-regular
or whether it is $\eps''$-regular for some $\eps''$ with $\eps'\ll\eps''\ll \eps$. Now we do the tidying up, but
only for the pairs which are known to be $\eps''$-regular - in particular we make all the other pairs (as well as some
others) empty. This can be done in polytime as well. The difference to~\cite{CKKO} is that if $G'[V_i,V_j]$
is empty, then $G[V_i,V_j]$ might have density $>d$, but we don't care.}

\begin{lemma}[Regularity Lemma for digraphs] \label{regularity_lemma}
For all $\eps, M'>0$ there exist $M, n_0$ such that if~$G$ is a digraph on $n \geq
n_0$ vertices and $d \in [0, 1]$, then there exists a partition of $V(G)$ into
$V_0, \dots, V_k$ and a spanning subdigraph $G'$ of $G$ satisfying the following conditions:
\begin{itemize}
\item[{\rm (i)}] $M' \leq k \leq M$.
\item[{\rm (ii)}] $|V_0| \leq \eps n$.
\item[{\rm (iii)}] $|V_1| = \dots = |V_k| =: m$.
\item[{\rm (vi)}] $d^+_{G'}(x) > d^+_{G}(x) - (d+\eps) n$ for all vertices $x \in V(G)$.
\item[{\rm (v)}] $d^-_{G'}(x) > d^-_{G}(x) - (d+\eps) n$ for all vertices $x \in V(G)$.
\item[{\rm (vi)}] For all $i=1,\dots,k$ the digraph $G'[V_i]$ is empty.
\item[{\rm (vii)}] For all $1 \leq i, j \leq k$ with $i \neq j$ the pair $G'[V_i,V_j]$
is either empty or $\eps$-regular of density at least~$d$. Moreover, if $G'[V_i,V_j]$ is nonempty then
$G'[V_i,V_j]=G[V_i,V_j]$.
\end{itemize}
\end{lemma}

We refer to $V_0$ as the \emph{exceptional set} and to $V_1, \dots, V_k$ as \emph{clusters}. 
$V_0,V_1, \dots, V_k$ as above is also called a \emph{regularity partition} for $G$.
Given a digraph $G$ on $n$ vertices, we form the \emph{reduced digraph $R$ of $G$
with parameters $\eps, d$ and $M'$} by applying the regularity lemma to $G$ with these
parameters to obtain $V_0, \dots, V_k$. $R$ is then the digraph whose vertices are
the clusters $V_1,\dots,V_k$ and whose edges are those (ordered) pairs $V_iV_j$ of clusters
for which $G'[V_i,V_j]$ is non-empty.

Given $d\in [0,1]$ and a bipartite graph $G=(A,B)$, we say that~$G$ is $(\eps,d)$-\emph{superregular}
if it is $\eps$-regular and furthermore $d_G(a)\ge (d-\eps) |B|$ for every $a\in A$ and
$d_G(b)\ge (d-\eps)|A|$ for every $b\in B$. (This is a slight variation of the standard definition
of $(\eps,d)$-superregularity where one requires $d_G(a)\ge d |B|$ and
$d_G(b)\ge d|A|$.) 

We say that a bipartite graph $G=(A,B)$ is \emph{$[\eps,d]$-superregular} if it is $\eps$-regular and $d_G(a)=(d\pm \eps)|B|$
for every $a\in A$ and $d_G(b)=(d\pm \eps)|A|$ for every $b\in B$.
So if $G$ is $[\eps,d]$-\emph{superregular}, then it is $(\eps,d)$-\emph{superregular}.
We say that $G$ is \emph{$[\eps, \ge d]$-superregular} if
it is  \emph{$[\eps, d']$-superregular} for some $d' \ge d$.
As for $\eps$-regularity, these definitions extend naturally to bipartite graphs where all edges
are oriented towards the same vertex class.

The following well known observation states that in an $\eps$-regular bipartite graph
almost all vertices have the expected degree and almost all pairs of vertices have the expected codegree
(i.e.~the expected number of common neighbours). Its proof follows immediately from the definition of regularity.%
   \COMMENT{To see the statement about the co-degrees, consider any vertex $a\in A$ with degree $(d\pm \eps)m$.
Since $G$ is $\eps$-regular, it follows that all but at most $2\eps m$ vertices in $A$ have 
$(d\pm \eps)|N(a)|=(d^2\pm \eps)m$ neighbours in $N(a)$. So this means that all but $2\cdot 2\eps m^2$ pairs
of distinct vertices in $A$ have the correct codegree.}

\begin{prop}\label{degcodeg}
Suppose that $0<\eps\le d\le 1$.
Let $G$ be an $\eps$-regular bipartite graph of density $d$ with vertex classes $A$ and $B$ of size $m$. Then the following conditions
hold.
\begin{itemize}
\item All but at most $2\eps m$ vertices in $A$ have degree $(d\pm \eps)m$.
\item All but at most $4\eps m^2$ pairs $a\neq a'$ of distinct vertices in $A$
satisfy $|N(a)\cap N(a')|=(d^2\pm \eps)m$.
\item The vertices in $B$ satisfy the analogues of these statements.
\end{itemize}
\end{prop}

The following simple observation states that the removal of a small number of edges and vertices from 
a bipartite graph does not affect its $\eps$-regularity (and superregularity) too much.

\begin{prop} \label{superslice} 
Suppose that $0<1/m \ll \eps \le d' \le d \ll 1$. Let $G$ be a bipartite graph with vertex classes
$A$ and $B$ of size $m$. Suppose that $G'$ is obtained from $G$ by removing at most $d'm$
vertices from each vertex class and at most $d'm$ edges incident to each vertex from $G$.
\begin{itemize}
\item[(i)]If $G$ is $\eps$-regular of density at least $d$ then $G'$ is $2\sqrt{d'}$-regular
of density at least $d-2\sqrt{d'}$.
\item[(ii)] If $G$ is $(\eps,d)$-superregular then $G'$ is $(2\sqrt{d'},d)$-superregular.
\item[(iii)] If $G$ is $[\eps,d]$-superregular then $G'$ is $[2\sqrt{d'},d]$-superregular.
\end{itemize}
\end{prop}
\proof
Let us first prove~(i). Let $d^*$ denote the density of $G$.
Let $A'\subseteq A$ and $B'\subseteq B$ denote the vertex classes of $G'$.
Suppose that $S\subseteq A'$, $T\subseteq B'$ are such that $|S|\ge 2\sqrt{d'}|A'|$
and $|T|\ge 2\sqrt{d'}|B'|$. So $|S|,|T| \ge  \sqrt{d'} m \ge \eps m$
and thus
\begin{align*}
e_{G'}(S,T) & \ge (d^*-\eps)|S||T| - |S|d'm   \ge (d^*-\eps)|S||T|-|S|d'\cdot |T|/\sqrt{d'}\\ &\ge (d^*-2\sqrt{d'})|S||T|.
\end{align*}
Since clearly $e_{G'}(S,T)\le e_{G}(S,T)\le (d^*+\eps)|S||T|$, (i) follows.
To see (ii) and the lower bound on the vertex degrees for (iii), note that in $G'$ the degrees of the vertices in $A'$ are still at least
$(d-\eps)m-2d'm \ge (d-\eps-2d')|B'|\ge (d-\sqrt{d'})|B'|$.%
   \COMMENT{need to check at the end that (ii) is still needed. Also, note that (ii) and (iii) are only applied with $d' \ll d$
but this need not be the case when applying (i),  which is the reason for the more precise density bound in (i)}
Similarly, the degrees in $G'$ of the vertices in $B'$ are still at least $(d-\sqrt{d'})|A'|$.

To see (iii), note that in $G'$ the degrees of the vertices in $A'$ are still at most
$(d+\eps)m\le (d+\eps)|B'|/(1-d')\le (d+\sqrt{d'})|B'|$. Similarly, the degrees in $G'$ of the vertices in
$B'$ are still at most $(d+\sqrt{d'})|A'|$.
\endproof

The following lemma is also well known in several variations. 

\begin{lemma}\label{superreg}
Suppose that $0<1/n\ll 1/k\ll \eps \ll d,1/\Delta\le 1$. Let $G$ be a digraph on $n$ vertices. Let
$\cP_0$ be a partition of $V(G)$ into $k$ clusters $V'_1,\dots,V'_k$ and an exceptional set $V'_0$ such that $m':=|V'_1|=\dots =|V'_k|$.
Let $R$ be a digraph whose vertices are $V'_1,\dots,V'_k$ and such that whenever $V'_iV'_j\in E(R)$ the pair $G[V'_i,V'_j]$ is
$(\eps,\ge d)$-regular. Let $H$ be a subdigraph of $R$ of maximum degree $\Delta$. 
Then there is a partition of $V(G)$ into $V_0,\dots,V_k$ such that the following holds:
\begin{itemize}
\item[(i)] For each $i=1,\dots,k$, $V_i$ is obtained from $V'_i$ by moving exactly $\sqrt{\eps} m'$ vertices into $V'_0$.
$V_0$ is then the set consisting of $V'_0$ and these additional vertices.
\item[(ii)] Whenever $V'_iV'_j\in E(H)$, the pair $G[V_i,V_j]$ is $[2\eps^{1/4},\ge d]$-superregular.
\end{itemize}
\end{lemma}
\proof
For each edge $V'_iV'_j$ of $H$, let $d_{ij}$ denote the density of $G[V'_i,V'_j]$.
So $d_{ij} \ge d-\eps$. Call a vertex $x$ in $V'_i$ \emph{bad} if at least one of the following two conditions hold:
\begin{itemize}
\item There is an edge $V'_iV'_j$ of $H$ so that the degree of $x$ in $G[V'_i,V'_j]$ is not $(1 \pm 2\eps)d_{ij}m'$.
\item There is an edge $V'_jV'_i$ of $H$ so that the degree of $x$ in $G[V'_j,V'_i]$ is not $(1 \pm 2\eps)d_{ji}m'$.
\end{itemize}
Note that $V'_i$ contains at most $2\Delta \eps m' \le \sqrt{\eps} m'$ bad vertices.  
Let $V_i$ be obtained from $V'_i$ by removing all bad vertices (and some additional ones if there are fewer than $\sqrt{\eps} m'$ of these).
Now suppose that $V'_iV'_j$ is an edge of $H$.
Then Proposition~\ref{superslice}(i) implies that $G[V_i,V_j]$ is $2\eps^{1/4}$-regular.
Together with the choice of $V_i$ and $V_j$ this implies that $G[V_i,V_j]$ is $[2 \eps^{1/4},\ge d]$-superregular.
\endproof

The following result is proved in Section~3.1~of~\cite{fk} by a simple application of the Max-Flow-Min-Cut theorem
(similarly to that in Lemma~\ref{regrobust} below).%
\COMMENT{Actually, they remove a sparse random subgraph first. Also, note that their definition of regularity is not quite equivalent to ours, 
which is why it is stated explicitly below.}
In particular, the subgraph of $G$ guaranteed by this result can be found in time polynomial in~$m$.

\begin{lemma}\label{regularsub}
Let $0 < 1/m\ll\eps  \ll d \le  1$. Suppose that $G$ is a bipartite graph with vertex classes $U$ and $V$ of size $m$ and
with minimum degree at least $dm$. Also suppose that $|d(A,B)-d| \le \eps m$ for all $A \subseteq U$ and $B \subseteq V$ with $|A|,|B| \ge \eps m$.
Define $d':= (1- 4\eps)d$. Then $G$ contains a spanning $d' m$-regular subgraph.
\end{lemma}

\begin{lemma}\label{regularsub2}
Let $0 < 1/m\ll\eps \ll  d\ll 1$. Let $d':=(1-12\eps)(d-\eps)$. 
Suppose that $G$ is a $[\eps , d]$-superregular bipartite graph with vertex classes of size $m$.  Then $G$ contains a spanning $d'm$-regular
subgraph which is also $[4\sqrt{\eps},d']$-superregular.
\end{lemma}
\proof
Note that the assumptions imply that $G$ satisfies the conditions of Lemma~\ref{regularsub} with%
    \COMMENT{The density of $G$ is $d\pm \eps$ and we have $|d(A,B)-d(U,V)| \le \eps m$.
So $|d(A,B)-(d-\eps)| \le |d(A,B)-d(U,V)|+|d(U,V)-(d-\eps)|\le \eps m+ 2\eps m$.}
$\eps$ replaced by $3\eps$ and $d$ replaced by $d-\eps$.
So we can apply Lemma~\ref{regularsub} to obtain a $d'm$-regular subgraph $G'$ with $d'=(1-12\eps)(d-\eps)$.
Note that $G'$ is obtained from $G$ by removing at most $3\eps m$ edges at every vertex.
Thus Proposition~\ref{superslice}(i) (with $3\eps$ playing the role of $d'$) implies that
$G'$ is also $[4\sqrt{\eps},d']$-superregular.
\endproof 


\subsection{Uniform refinements}

Let $G$ be a digraph and let $\cP$ be a partition of $V(G)$ into an exceptional set $V_0$ and clusters of equal size.
Suppose that $\cP'$ is another partition of $V(G)$ into an exceptional set $V'_0$ and clusters of equal size. We say
that $\cP'$ is an \emph{$\ell$-refinement of $\cP$} if $V_0=V'_0$ and if the clusters in $\cP'$ are obtained by
partitioning each cluster in $\cP$ into $\ell$ subclusters of equal size. (So if $\cP$ contains $k$ clusters
then $\cP'$ contains $k\ell$ clusters.) $\cP'$ is an \emph{$\eps$-uniform $\ell$-refinement}
of $\cP$ if it is an $\ell$-refinement of $\cP$ which satisfies the following condition:
\begin{itemize}
\item[(URef)] Whenever $x$ is a vertex of $G$, $V$ is a cluster in $\cP$ and $|N^+_G(x)\cap V|\ge \eps |V|$
then $|N^+_G(x)\cap V'|=(1\pm \eps)|N^+_G(x)\cap V|/\ell$ for each cluster $V'\in \cP'$ with $V'\subseteq V$.
The inneighbourhoods of the vertices of $G$ satisfy an analogous condition.
\end{itemize}

\begin{lemma} \label{randompartition}
Suppose that $0<1/m \ll 1/k,\eps \ll \eps', d,1/\ell \le 1$ and that $m/\ell\in\mathbb{N}$.
Suppose that $G$ is a digraph on $n\le 2km$ vertices and that $\cP$ is a partition of $V(G)$ into
an exceptional set $V_0$ and $k$ clusters of size $m$. Then there exists an $\eps$-uniform $\ell$-refinement of $\cP$. Moreover, any
$\eps$-uniform $\ell$-refinement $\cP'$ of $\cP$ automatically satisfies
the following conditions:
\begin{itemize}
\item[(i)] Suppose that $V$, $W$ are clusters in $\cP$ and $V',W'$ are clusters in $\cP'$ with $V'\subseteq V$ and
$W'\subseteq W$. If $G[V,W]$ is $[\eps,d']$-superregular for some $d'\ge d$ then $G[V',W']$ is $[\eps',d']$-superregular.
\item[(ii)] Suppose that $V$, $W$ are clusters in $\cP$ and $V',W'$ are clusters in $\cP'$ with $V'\subseteq V$ and
$W'\subseteq W$. If $G[V,W]$ is $(\eps, \ge d)$-regular  then $G[V',W']$ is $(\eps',\ge d)$-regular.
\end{itemize}
\end{lemma}
\proof
To prove the existence of an $\eps$-uniform $\ell$-refinement of $\cP$, let $\cP^*$
be a partition obtained by splitting each cluster $V\in \cP$ uniformly at random into $\ell$ subclusters.
More precisely, the probability that a vertex $x \in V$ is assigned to the $i$th subcluster is $1/\ell$, 
independently of all other vertices.
Consider a fixed vertex $x$ of $G$ and a cluster $V\in \cP$ with $d^+:=|N^+_G(x)\cap V|\ge \eps m$.
Given a cluster $V'\in\cP^*$ with $V'\subseteq V$, we say that $x$ is
\emph{out-bad for $V'$} if the outdegree of $x$ into $V'$ is not $(1\pm \eps/2) d^+/\ell$. 
Then Proposition~\ref{chernoff} implies that the probability that $x$ is out-bad for $V'$ is at most 
$2\eul^{-\eps^2 d^+/3 \cdot 4 \ell} \le 2 \eul^{-\eps^4m}$. Since $\cP^*$ contains 
$k \ell \le n$ clusters, the probability that $G$ contains some vertex which is out-bad for at least one cluster
$V'\in \cP^*$ is at most $n^2\eul^{-\eps^4m}<1/8$. We argue analogously for the inneighbourhoods of the vertices in
$G$ (by considering `in-bad' vertices). 

We now say that a cluster $V'$ of $\cP^*$ is \emph{good} if $|V'|=(1\pm \eps^2/2)m/\ell$.
A similar argument as above shows that the probability that $\cP^*$ has a cluster which is not good is at most $1/4$.
So with probability at least $1/2$, all clusters of $\cP^*$ are good, and no vertices are out-bad or in-bad.

Now obtain $\cP'$ from $\cP^*$ as follows: for each cluster $V$ of $\cP$, equalize the sizes of the corresponding $\ell$
subclusters in $\cP$ by moving at most $\eps^2m/2\ell$ vertices from one subcluster to another.
So whenever $x$ is a vertex of $G$, $V$ is a cluster in $\cP$ and $|N^+_G(x)\cap V|\ge \eps |V|$,
it follows that we have $$
|N^+_G(x)\cap V'|=(1\pm \eps/2)|N^+_G(x)\cap V|/\ell\pm \eps^2m/2\ell
$$ 
for each cluster $V'\in \cP'$ with $V'\subseteq V$.
The inneighbourhoods of the vertices of $G$ satisfy an analogous condition.
So (URef) holds and so $\cP'$ is an $\eps$-uniform $\ell$-refinement of $\cP$.

To prove (i), suppose that $\cP'$ is any $\eps$-uniform $\ell$-refinement of $\cP$ and that $G[V,W]$ is $[\eps,d']$-superregular
for some $d'\ge d$ (where $V$ and $W$ are clusters in $\cP$). Let $V'$ and $W'$ be clusters in $\cP'$ with $V'\subseteq V$
and $W'\subseteq W$. Then $G[V',W']$ is $\eps\ell$-regular and thus $\eps'$-regular.
Consider any $x\in V'$ and let $d^+:=|N^+_G(x)\cap W|$. Thus $d^+=(d'\pm \eps)m$ since
$G[V,W]$ is $[\eps,d']$-superregular. Together with the $\eps$-uniformity of $\cP'$ this implies that
$|N^+_G(x)\cap W'| =(1\pm \eps)d^+/\ell=(d'\pm \eps')m/\ell$. The inneighbourhoods in $V'$ of the vertices
in $W'$ satisfy the analogous property. Thus $G[V',W']$ is $[\eps',d']$-superregular.

The proof of (ii) is almost the same.
\endproof
Note that Theorem~\ref{derandom} (with $p=1/\ell$) implies that the above applications of Proposition~\ref{chernoff} can be derandomized
to find $\cP'$ in polynomial time.%
\COMMENT{Actually, this requires $\ell-1$ successive applications of the Proposition where we successively split off a subclusters of size
$|V|/\ell$. But it's probably better to gloss over this.}


\subsection{A sparse notion of $\eps$-regularity}

We will also use a `sparse' version of $\eps$-(super)-regularity, which is defined below.
In particular, this definition allows for $d<\eps$.  
We will need this notion mainly in Section~\ref{seccyclebreak}, where we will have to work with graphs for which we cannot 
guarantee $(\eps,\ge d)$-regularity with $\eps \le d$. 
In general, one useful consequence of $(\eps,\ge d)$-regularity is that sets of
size between $\eps m$ and $(1-\eps)m$ expand robustly. With our sparse version, we will also be able to guarantee that even sets of size less than $\eps m$ expand robustly.
This will follow from condition (Reg2) below.%
       \COMMENT{Here is a more precise explanation why we need the notion:
We need that the density of the leftover of the preprocessing graph is much smaller than $\eps$ in order to be
able to absorb it. So the pairs in the preprocessing graph cannot be $\eps$-regular in the usual sense, but we need to find them
inside our original blow-up of $C$. Lemma 10 from the Kelly paper guarantees that an $\eps$-regular pair of density $d$
contains an $\{ \eps,\eps/K \}$-regular pair of density about $d/K$. This is not enough for us since we wish to apply the contraction trick
(ie Lemma~\ref{mergecycles0}). So we need to find a Hamilton cycle in the contracted graph. The fix is to replace $\{ \eps,\eps/K \}$-(super-)regularity
by $(\eps,d,d^*,c)$-superregularity. This guarantees that the contracted graph is a robust outexpander and thus has a Hamilton cycle.}

More precisely, let $G$ be a bipartite graph with vertex classes $U$ and $V$, both of size $m$.
Given $0<\eps,d,c<1$, we say that $G$ is \emph{$(\eps,d,c)$-regular} if the following conditions are satisfied:
\begin{itemize}
\item[(Reg1)] Whenever $A\subseteq U$ and $B\subseteq V$ are sets of size at least $\eps m$, then
$d(A,B)=(1\pm \eps)d$.
\item[(Reg2)] For all $u,u'\in U$ we have $|N(u)\cap N(u')|\le c^2m$. Similarly, for all $v,v'\in V$ we have $|N(v)\cap N(v')|\le c^2m$.
\item[(Reg3)] $\Delta(G)\le cm$.
\end{itemize}
We say that $G$ is \emph{$(\eps,d,d^*,c)$-superregular} if it is $(\eps,d,c)$-regular and
in addition the following condition holds:
\begin{itemize}
\item[(Reg4)] $\delta(G)\ge d^*m$.
\end{itemize}

The next result gives an analogue of Proposition~\ref{superslice} for the above notion of (super)-regularity.
\begin{prop} \label{superslice5} 
Suppose that $0<1/m \ll d^*,d, \eps, c\ll 1$. Let $G$ be a bipartite graph with vertex classes
$U$ and $V$ of size $m$. Suppose that $G'$ is obtained from $G$ by removing 
at most $\eps^2 dm$ edges incident to each vertex from $G$.
\begin{itemize}
\item[(i)] If $G$ is $(\eps,d,c)$-regular then $G'$ is $(2\eps,d,c)$-regular.
\item[(ii)] If $G$ is $(\eps,d,d^*,c)$-superregular then $G'$ is $(2\eps,d,d^*-\eps^2 d,c)$-superregular.
\end{itemize}
\end{prop}
\proof
Let us first prove~(i). Clearly $G'$ still satisfies (Reg2) and (Reg3). So we only need to check
that it also satisfies~(Reg1). Suppose that $S\subseteq U$, $T\subseteq V$ are such that $|S|,|T|\ge \eps m$.
Then
$$e_{G'}(S,T)\ge (1-\eps)d|S||T| - |S|\eps^2d m   \ge (1-\eps)d|S||T|-|S|\eps^2 d\cdot |T|/\eps = (1-2\eps)d |S||T|.
$$
Since clearly $e_{G'}(S,T)\le e_{G}(S,T)\le (1+\eps)d|S||T|$, (i) follows.
To see (ii) note that the degrees in $G'$ are still at least $d^*m-\eps^2d m$.
\endproof

We will construct sparse $(\eps,d,c)$-regular graphs in the proof of Lemma~\ref{randomregslice}.
To verify (Reg1) in the proof of Lemma~\ref{randomregslice}, we will use a variant of the well known characterization in terms of
codegrees of pairs of vertices which was proved as Lemma~3.2 in~\cite{Aetal} (the version in~\cite{Aetal} gives more precise bounds but the statement is not
suitable for sparse regularity).

Suppose that $G=(U,V)$ is a bipartite graph with vertex classes $U$ and $V$ of size~$m$. We say that $G$ is 
\emph{$\{\eps,d \}$-regular} if  for all $A\subseteq U$
and $B \subseteq V$ with $|A|, |B| \geq \eps m$
we have $d(A,B) = (1 \pm \eps)dm$. (So (Reg1) is equivalent to saying that $G$ is $\{\eps,d \}$-regular.)
Note that this notion allows for $d<\eps$. Moreover, it is stronger than $\eps$-regularity in the sense that
if $\eps<d\ll 1$ then every $\{\eps,d \}$-regular pair is also $\eps$-regular of density $(1 \pm \eps)d$.
 
Call a pair of distinct vertices in $V$ \emph{bad} if the number of common neighbours in $U$ is at least $(1+\eps)d^2m$.
\begin{lemma} \label{codegrees}
Suppose that $1/m \ll \eps,d \le 1/C \le 1$ and $\eps \ll 1/C$.
Let $G=(U,V)$ be a bipartite graph with vertex classes $U$ and $V$ of size $m$. Suppose that all but at most $\eps m$ vertices
in $V$ have degree at least  $(1 - \eps)dm$ and for all pairs of distinct vertices in $V$ the number of common neighbours is at most $Cd^2m$.
Suppose also that the number of bad pairs of distinct vertices in $V$ is at most $\eps m^2$. 
Then  $G$ is $\{\eps^{1/6},d\}$-regular.
\end{lemma} 
\proof
The proof follows the argument in~\cite{Aetal}. Let $\eps_0:=\eps^{1/6}$.
It is easy to see that it suffices to check that $d(X,Y)=(1 \pm \eps_0)d$ for all pairs $X\subseteq U$ and $Y\subseteq V$ with $|X|=|Y|=\eps_0 m$.
For any pair of vertices $y_1,y_2\in Y$, let $\sigma (y_1,y_2):=|N(y_1) \cap N(y_2)|-d^2 m$.
Then we always have $\sigma (y_1,y_2) \le Cd^2m$ and can improve this to $\sigma (y_1,y_2) \le \eps d^2m$ if the pair $y_1,y_2$ is not bad.
Also define 
$$
\sigma(Y):=\frac{1}{|Y|^2}\sum_{y_1,y_2 \in Y, \ y_1\neq y_2} \sigma(y_1,y_2).
$$
Then our assumption on $|Y|$ and on the number of bad pairs implies that 
\begin{equation} \label{sigma}
\sigma(Y) \le \frac{1}{(\eps_0 m)^2}\left( (\eps m^2)Cd^2m  + (\eps_0m)^2\eps d^2m \right) = C\eps_0^4 d^2m+ \eps d^2m \le \eps_0^3 d^2m/3.
\end{equation}
We claim that 
\begin{equation} \label{claim}
\sum_{x \in X}\left( |N(x) \cap Y| - d|Y|\right)^2 \le \eps_0^3 d^2m |Y|^2.
\end{equation}
To prove the claim, we use that the left hand side (as shown in~\cite{Aetal}) is at most 
$$
e(U,Y)+ \sigma(Y)|Y|^2+2d^2|Y|^2m -2e(U,Y)d|Y|.
$$
But our assumption on the vertex degrees implies that
$$e(U,Y) \ge (|Y|-\eps m)(1-\eps)dm =|Y|(1-\eps/\eps_0)(1-\eps)dm \ge |Y|(1-\eps_0^3/6)dm
$$ and so 
$$
2d^2|Y|^2 m -2e(U,Y)d|Y|
\le \eps_0^3 d^2|Y|^2m/3.
$$
So together with (\ref{sigma}) this implies that the left hand side of (\ref{claim}) is at most
$$
e(U,Y) + \frac{2}{3}\eps_0^3d^2m |Y|^2  \le |Y|m+\frac{2}{3}\eps_0^3d^2m |Y|^2 \le \eps_0^3d^2m |Y|^2.
$$
This proves the claim. On the other hand, the Cauchy-Schwarz inequality implies that
$$
\sum_{x \in X}\left( |N(x) \cap Y| - d|Y|\right)^2 
\ge \frac{1}{|X|} \left( \left( \sum_{x \in X} |N(x) \cap Y| \right) -d|X||Y| \right)^2.
$$
So together with~(\ref{claim}), this implies that
$$
\left( \left( \sum_{x \in X} |N(x) \cap Y| \right) -d|X||Y| \right)^2
\le |X| \left( \eps_0^3 d^2m |Y|^2  \right).
$$
Thus dividing both sides by $|X|^2|Y|^2$ yields
$$
|d(X,Y) -d|^2 \le \frac{1}{|X|} \left( \eps_0^3 d^2m \right) =\eps_0^2 d^2,
$$ 
as required.
\endproof

The first part of the following lemma implies that inside an $\eps$-regular pair we can find sparse subgraphs
which satisfy (Reg1)--(Reg3) with good bounds on the parameters. 
Assertion~(iii) will only be used in~\cite{KellyII,OS}.
\begin{lemma}\label{randomregslice}
Suppose that $0<1/m\ll \eps, d'\le d\le 1$ and $\eps \ll d$.
\begin{itemize}
\item[{\rm (i)}] If $G$ is an $\eps$-regular bipartite graph of density $d$ with vertex classes of size $m$, then
it contains an $(\eps^{1/12},d',3d'/2d)$-regular spanning subgraph.
\item[{\rm (ii)}] If $G$ is an $(\eps,d)$-superregular bipartite graph with vertex classes of size $m$, then it
contains an $(\eps^{1/12},d',d'/2,3d'/2d)$-superregular spanning subgraph.
\item[{\rm (iii)}] If $\eps\ll d'$ and $G$ is an $\eps$-regular bipartite graph of density $d$ with vertex classes of size $m$, then
it contains an $\{ \eps^{1/12},d' \}$-regular spanning subgraph $J$.
Moreover, if $x \in V(G)$ satisfies $d_G(x) =(d \pm \eps)m$, then $d_{J}(x) =(d' \pm \sqrt{\eps})m$.
\item[{\rm (iv)}] If $\eps\ll d'$ and $G$ is an $[\eps,d]$-superregular bipartite graph with vertex classes of size $m$, then it
contains an $[\eps^{1/12},d']$-superregular spanning subgraph.
\end{itemize}
\end{lemma}
\proof
We only prove~(i). Since a $(\eps^{1/12},d',3d'/2d)$-regular pair is $\{\eps^{1/12},d'\}$-regular, (iii) follows from~(i)
(and the `moreover' part follows from the proof of (i)).
The argument for~(ii) and~(iv) is similar to the proof of~(i).
So suppose that $G$ is $\eps$-regular of density $d$ with vertex classes $U$ and $V$ of size $m$.
Let $G'$ be the spanning subgraph obtained from $G$ by picking every edge of $G$ with probability $p:=d'/d$, independently from all other edges.
Consider any vertex $v\in V$ with $d_G(v)=(d\pm \eps)m$. Then the expected degree of $v$ in $G'$ is $p(d\pm \eps)m=(1\pm \sqrt{\eps}/2)d'm$.
So Proposition~\ref{chernoff} implies that
\begin{align*}
\pr\left(d_{G'}(v)\neq (1\pm \sqrt{\eps})d'm\right) 
& \le \pr\left(|d_{G'}(v)-\ex(d_{G'}(v))|\ge \frac{\sqrt{\eps}}{3}\ex(d_{G'}(v))\right)\\
& \le 2\eul^{-\eps\ex(d_{G'}(v))/27} \le 2\eul^{-\eps d'm/28}.
\end{align*}
Similarly, consider any vertex $x\in U\cup V$ (with no restriction on its degree $d_G(x)$).
If $d_{G}(x)\le 3d'm/2d$, then clearly $d_{G'}(x)\le 3d'm/2d$. So suppose that $d_{G}(x)\ge 3d'm/2d$.
Then $3(d')^2m/2d^2= p  \cdot 3d'm/2d \le \ex(d_{G'}(x))\le pm=d'm/d$ and so
\begin{align*}
\pr\left(d_{G'}(x)\ge \frac{3d'm}{2d}\right) & \le \pr\left( d_{G'}(x)\ge \frac{3}{2}\ex(d_{G'}(x))\right)\\
& \le \eul^{-\ex(d_{G'}(x))/12}\le \eul^{-(d')^2m/8d^2}.
\end{align*}
For the remainder of the proof, we let the  \emph{codegree} $d_G(x,x')$ of a pair $x,x'$ of vertices in $G$ be the number of common neighbours of $x$ and $x'$.
Consider any pair $v,v'\in V$ of distinct vertices with codegree $d_G(v,v')=(d^2\pm \eps)m$.
Then the expected codegree of $v,v'$ in $G'$ is $\ex(d_{G'}(v,v'))=p^2(d^2\pm \eps)m=(1\pm \sqrt{\eps}/2)(d')^2m$.
So Proposition~\ref{chernoff} implies that
\begin{align*}
\pr\left(d_{G'}(v,v')\ge (1\pm \sqrt{\eps})(d')^2m\right) 
& \le \pr\left( d_{G'}(v,v')\ge \left( 1+\sqrt{\eps}/3 \right)\ex(d_{G'}(v,v'))\right)\\
& \le \eul^{-\eps \ex(d_{G'}(v,v'))/27} \le \eul^{-\eps (d')^2m/28}.
\end{align*}
Similarly, consider any pair $x\neq x'$ of vertices in $G$ (with no restriction on the codegree $d_G(x,x')$).
If $d_{G}(x,x')\le 3(d')^2m/2d^2$, then clearly $d_{G'}(x,x')\le 3(d')^2m/2d^2$. So suppose that $d_{G}(x,x')\ge 3(d')^2m/2d^2$.
Then $3(d')^4m/2d^4\le \ex(d_{G'}(x,x'))\le p^2m=(d')^2m/d^2$ and so
\begin{align*}
\pr\left(d_{G'}(x,x')\ge \frac{3(d')^2m}{2d^2}\right) 
& \le \pr\left(d_{G'}(x,x')\ge \frac{3}{2}\ex (d_{G'}(x,x'))\right)\\
& \le \eul^{-\ex (d_{G'}(x,x'))/12}\le \eul^{-(d')^4m/8d^4}.
\end{align*}
Proposition~\ref{degcodeg} implies that $V$ contains at most $2\eps m$ vertices whose degree in $G$ is not $(d\pm \eps)m$
as well as at most $4\eps m^2$ pairs of distinct vertices whose codegree in $G$ is not $(d^2\pm \eps)m$.
Thus a union bound implies that with probability at least
$$
1-4m\eul^{-\eps d'm/28}+2m\eul^{-(d')^2m/8d^2}+m^2\eul^{-\eps (d')^2m/28}+m^2\eul^{-(d')^4m/8d^4} \ge 1/2
$$
all of the following properties are satisfied:
\begin{itemize}
\item All but at most $2\eps m$ vertices $v\in V$ satisfy $d_{G'}(v)=(1\pm \sqrt{\eps})d'm$.
\item All but at most $4\eps m^2$ pairs $v\neq v'$ of vertices in $V$
satisfy $d_{G'}(v,v')=|N_{G'}(v)\cap N_{G'}(v')| \le (1+ \sqrt{\eps})(d')^2m$.
\item All pairs $v\neq v'$ of vertices in $V$
satisfy $d_{G'}(v,v')\le 3(d')^2m/2d^2$.
\item $\Delta(G')\le 3d'm/2d$.
\end{itemize}
Thus we can choose  $G'$ to satisfy (Reg2) and~(Reg3). Moreover, Lemma~\ref{codegrees} (applied with $\sqrt{\eps}$, $d'$, $3/2d^2$ playing the roles
of $\eps$, $d$, $C$) together with the first two properties
implies that $G'$ also satisfies~(Reg1) with $\eps$ replaced by $\eps^{1/12}$.
\endproof

Theorem~\ref{derandom} implies that the applications of Proposition~\ref{chernoff} in the proof of Lemma~\ref{randomregslice} can
be derandomized to find the spanning subgraphs guaranteed by the lemma in polynomial time. Note that instead of using Lemma~\ref{codegrees}
to check~(Reg1) in our proof of Lemma~\ref{randomregslice}, one could have checked (Reg1) directly. But this would have involved a
union bound over exponentially many sets. So we would not have been able to apply Theorem~\ref{derandom}.

Let $G$ be a $(\eps,d,d^*,d/\mu)$-superregular bipartite graph with vertex classes $U=\{u_1,\dots,u_m\}$ and $V=\{v_1,\dots,v_m\}$.
The following lemma implies that the digraph obtained from $G$
by orienting all the edges from $U$ to $V$ and identifying $u_i$ and $v_i$ for all $i=1,\dots,m$ is a robust $(\nu,\tau)$-outexpander of minimum
semidegree at least $d^*m$. Together with Theorem~\ref{expanderthm} below this will imply that this `contracted' digraph has a Hamilton cycle,
which will in turn be used in the proof of Lemma~\ref{mergecycles}.

\begin{lemma}\label{regtoexpander}
Let $0<1/m\ll \nu\ll \tau\ll d\le \eps\ll \mu,\zeta \le 1/2$ and let $G$ be a
$(\eps,d,\zeta d,d/\mu)$-superregular bipartite graph with vertex classes $U$ and $V$ of size $m$. Let $A\subseteq U$ be such that
$\tau m\le |A|\le (1-\tau) m$. Let $B\subseteq V$ be the set of all those vertices in $V$ which have at least $\nu m$ neighbours in $A$.
Then $|B|\ge |A|+\nu m$.
\end{lemma}
\proof
Let us first prove the following claim.

\medskip

\noindent
\textbf{Claim.} \emph{Let $U'\subseteq U$ be such that $|U'|\ge \tau m/2$ and let $RN(U')$ be the set of all those vertices
in $V$ which have at least $\nu m$ neighbours in $U'$. Then $|RN(U')|\ge \min\{10\eps m, |U'|/\sqrt{\eps}\}$. Similarly,
let $V'\subseteq V$ be such that $|V'|\ge \tau m/2$ and let $RN(V')$ be the set of all those vertices
in $U$ which have at least $\nu m$ neighbours in $V'$. Then $|RN(V')|\ge \min\{10\eps m, |V'|/\sqrt{\eps}\}$.}

\smallskip

\noindent
We only prove the first part of the claim. The argument for the second part is identical.
Suppose that $|RN(U')|\le 10\eps m$. Given $V',V''\subseteq V$, let $f(V',V'')$ denote the number of paths of length two which
have their midpoint in $U'$, one endpoint in $V'$ and the other endpoint in $V''$. Since $\delta(G)\ge \zeta dm$ by (Reg4)
we have that
$$f(V,V)\ge |U'|\binom{\zeta d m}{2}\ge |U'|\frac{\zeta^2d^2m^2}{3}.
$$
On the other hand, $f(V\setminus RN(U'), V)\le \nu |V\setminus RN(U')|m^2$ since every vertex in $V\setminus RN(U')$ has at
most $\nu m$ neighbours in $U'$.%
    \COMMENT{To choose a 2-path with endpoint in $x\in V\setminus RN(U')$ one first chooses an edge from $x$ to some $y\in U'$ and
then there are at most $m$ choices for the other endpoint.} 
Thus
\begin{eqnarray*}
f(V,V)& = & f(RN(U'),RN(U'))+f(V\setminus RN(U'), V)\\
& \le & \sum_{v,v'\in RN(U'),\ v\neq v'} |N(v)\cap N(v')|+ \nu |V\setminus RN(U')| m^2\\
& \stackrel{({\rm Reg2})}{\le} & |RN(U')|^2\frac{d^2m}{\mu^2}+\nu m^3\le |RN(U')|\frac{10\eps d^2m^2}{\mu^2}+\nu m^3.
\end{eqnarray*}
(To obtain the final inequality we use that $|RN(U')|\le 10\eps m$.)
Altogether this implies that
\begin{align*}
|RN(U')| & \ge \frac{\mu^2}{10\eps d^2m^2}\left(|U'|\frac{\zeta^2d^2m^2}{3}-\nu m^3\right)\\
& \ge \frac{\mu^2}{10\eps d^2m^2}\cdot |U'|\frac{\zeta^2d^2m^2}{4}
=\frac{\mu^2\zeta^2}{40\eps}|U'|\ge \frac{|U'|}{\sqrt{\eps}},
\end{align*}
which proves the claim.

\medskip

\noindent In order to prove the lemma, we distinguish several cases according to the size of~$A$.
Suppose first that $\tau m\le |A|\le \eps m$. Then the claim implies that 
$B=RN(A)$ has size at least $\min\{10\eps m, |A|/\sqrt{\eps}\}\ge |A|+\nu m$, as required.%
     \COMMENT{Ok since $|A|/\sqrt{\eps}-|A|\ge |A|/2\sqrt{\eps}\ge \nu m$ as $|A|\ge \tau m\ge 2\sqrt{\eps}\nu m$.}

Suppose next that $\eps m\le |A|\le (1-2\eps)m$. Since
$$e(A,V\setminus B)\le \nu m |V\setminus B| <(1-\eps)d|A||V\setminus B|,$$
together with (Reg1) this
implies that $|V\setminus B|<\eps m$. Thus $|B|\ge (1-\eps)m \ge |A|+\nu m$.

Finally, consider the case when $(1-2\eps)m\le |A|\le (1-\tau)m$. Suppose for a contradiction that $|B|<|A|+\nu m$.
From the previous case it follows that $|B|\ge (1-2\eps)m+\nu m$. Let $V':=V\setminus B$.
Then
\begin{equation}\label{eq:sizeV'}
\tau m/2\le m-|A|-\nu m<m-|B|=|V'|\le 2\eps m.
\end{equation}
Let $A':=A\cap RN(V')$.
Since every vertex in $A'$ has at least $\nu m$ neighbours in $V'$, but
every vertex in $V'$ has less than $\nu m$ neighbours in $A\supseteq A'$, it follows that
$|A'|\nu m\le e(A',V')\le |V'|\nu m$. Thus $|A'|\le |V'|$ and so
$$|RN(V')|\le |U\setminus A|+|A'|\le m-|A|+|V'|\le 2|V'|+\nu m< 3|V'|\le 6\eps m.
$$
(Here we used~(\ref{eq:sizeV'}) in the final three inequalities.)
On the other hand, the claim implies that
$|RN(V')|\ge \min\{10\eps m, |V'|/\sqrt{\eps}\}$, a contradiction.
\endproof

We will also use the above lemma to show that bipartite graphs satisfying (Reg1)--(Reg3) have large matchings.
\begin{lemma}\label{PEmatching}
Suppose that $0<1/m\ll d'\ll 1/k\ll \eps\ll d\ll \zeta\le 1/2$ and let $G$ be a bipartite graph with vertex
classes $U$ and $V$of size $m$.
\begin{itemize}
\item[(i)] If $G$ is $(\eps,d'/k,d'/dk)$-regular, then 
it contains a matching of size at least $(1-\eps) m$.
\item[(ii)] If $G$ is $(\eps,d',\zeta d',d'/d)$-superregular, then 
it has a perfect matching.
\end{itemize}
\end{lemma}
Since maximum matchings can be found in polynomial time, this is also the case for the matchings guaranteed by the lemma.
\proof
To prove (i), note that (Reg1) implies that $|N(A)|\ge (1-\eps)m$
for every set $A\subseteq U$ with $|A|\ge \eps m$. Together with the defect version of Hall's theorem this implies
that $G$ has a matching of size at least $(1-\eps)m$.

To prove (ii), we use Hall's theorem. Lemma~\ref{regtoexpander} implies that Hall's condition holds for all sets
$A \subseteq U$ with $\tau m \le |A| \le (1-\tau)m$ (where we apply the lemma with 
$d'$, $d$ playing the roles of $d$, $\mu$ and with some $\nu$, $\tau$ satisfying $1/m \ll \nu \ll \tau\ll d'$).
(Reg4) implies that $G$ has minimum degree at least $\zeta d'm\ge \tau m$, so Hall's condition also holds
for sets $A \subseteq U$ with $|A| \le \tau m$ or $|A| \ge (1-\tau) m$.
\endproof

\subsection{The Blow-up Lemma}

We will use the blow-up lemma of Koml\'os, S\'ark\"ozy and Szemer\'edi~\cite{kss2} (see~\cite{algoblowup} for their proof of
an algorithmic version). 
Roughly speaking, it states that superregular pairs behave like complete bipartite graphs with respect to embedding
subgraphs of bounded degree. 

\begin{lemma}[Blow-up Lemma]\label{Blow-up}
Suppose that $0<1/m\ll \eps\ll 1/k,d,1/\Delta\le 1$. Let $R$ be a graph with vertex
set $\{1,\dots,k\}$. Let $V_1,\ldots,V_k$ be
pairwise disjoint sets of vertices, each of size $m$.
Let $R(m)$ be the graph on $V_1\cup\dots\cup V_k$ which is obtained from $R$ by replacing each edge $ij$ of $R$
with a complete bipartite graph $K_{m,m}$ between $V_i$ and $V_j$.
Let $G$ be a graph on $V_1\cup\dots\cup V_k$ which is obtained from $R$ by replacing each edge $ij$ of $R$
with some $(\eps,d)$-superregular pair $G[V_i,V_j]$. If a graph $H$ with maximum degree
$\Delta(H) \leq \Delta$ is embeddable into $R(m)$ then it is already
embeddable into $G$. 
\end{lemma}

The following proposition is a very special case of the blow-up lemma. 
It is also easy to prove it in the same way as Lemma~\ref{PEmatching}(ii).
\begin{prop} \label{perfmatch}
Suppose that $0<1/m \ll \eps \ll d\le 1$.
Suppose that $G$ is a $(\eps,d)$-superregular bipartite graph with vertex classes of size $m$.
Then $G$ contains a perfect matching.
\end{prop}

The following consequence of the blow-up lemma states that we can link up arbitrary sets of vertices which 
are joined by a `blown-up' path.
\begin{cor} \label{linking}
Suppose that $0<1/m \ll \eps \ll d\le 1$ and that $q\ge 4$.
Let $V_1,\dots,V_q$ be pairwise disjoint sets of vertices, each of size $m$.
Let $G$ be a graph on $V_1\cup \dots\cup V_q$ such that $G[V_i,V_{i+1}]$ is $(\eps,d)$-superregular
for each $i=1,\dots,q-1$. Let $a_1,\dots,a_m$ be an arbitrary enumeration of the vertices in $V_1$ and let 
$b_1,\dots,b_m$ be an arbitrary enumeration of the vertices in $V_q$.
Then $G$ contains a set of $m$ vertex-disjoint paths connecting $a_i$ to $b_i$ for every $i$. 
\end{cor}

\proof
First suppose that $q=4$.
Consider the graph $G'$ which is obtained from $\bigcup_{i=1}^{q-1} G[V_i,V_{i+1}]$ by identifying the vertices $a_i$ and $b_i$.
Thus $G'$ can be viewed as the blow-up of a triangle. The blow-up lemma implies that $G'$
contains a set of disjoint triangles covering all vertices of $G'$ (and where each triangle has a vertex in each $V_i$).
This corresponds to the desired set of paths connecting the $a_i$ and $b_i$ for all $i=1,\dots,m$.

If $q>4$, we simply apply Proposition~\ref{perfmatch} $q-4$ times 
to get $m$ vertex-disjoint paths joining all vertices in $V_4$ to $V_q$.
Then we proceed as in the case when $q=4$.
\endproof


\section{Robust Outexpanders} \label{sec:expand}

The next result (Lemma~14 from~\cite{KOTchvatal}) implies that the property of a digraph~$G$ being a robust outexpander is `inherited'
by the reduced digraph of~$G$. For this, we need that~$G$
is a robust outexpander, rather than just an outexpander.
 
\begin{lemma}\label{robustR} 
Suppose that $0<1/n\ll  \varepsilon \ll d\ll \alpha,\nu,\tau < 1$ and $M'/n\ll 1$. Let~$G$ be a
digraph on $n$ vertices with $\delta ^0 (G) \ge \alpha n$ and such that~$G$
is a robust $(\nu,\tau)$-outexpander. Let~$R$ be the
reduced digraph of $G$ with parameters $\varepsilon$, $d$ and~$M'$. Then 
$\delta^0(R)\ge \alpha |R|/2$ and~$R$ is a robust $(\nu/2,2\tau)$-outexpander.
\end{lemma}

The following result shows that in a robust outexpander, we can guarantee a spanning subdigraph with a given degree sequence
(as long as the required degrees are not too large and do not deviate too much from each other). We will state a more general version
of this lemma for multidigraphs, which will be used in~\cite{OS}. In the current paper, we will only use Lemma~\ref{regrobust} in the
case when $Q=G$. If $x$ is a vertex of a multidigraph $Q$, we write $d^+_Q(x)$ for the number of edges in $Q$ whose initial vertex is $x$
and $d^-_Q(x)$ for the number of edges in $Q$ whose final vertex is $x$.

\begin{lemma}\label{regrobust} Let $q\in\mathbb{N}$.
Suppose that $0<1/n\ll\eps\ll \nu \le \tau \ll \alpha<1$ and that $1/n \ll \xi\le q\nu^2/3$. Let~$G$ be a digraph on~$n$ vertices with
$\delta^0(G)\ge \alpha n$ which is a robust $(\nu,\tau)$-outexpander. Suppose that $Q$ is a multidigraph on $V(G)$ such that
whenever $xy\in E(G)$ then $Q$ contains at least $q$ edges from $x$ to $y$. For every vertex $x$ of $G$, let $n^+_x,n^-_x\in\mathbb{N}$
be such that $(1-\eps)\xi n\le n^+_x, n^-_x\le (1+\eps)\xi n$ and such that $\sum_{x\in V(G)} n^+_x=\sum_{x\in V(G)} n^-_x$.
Then $Q$ contains a spanning submultidigraph $Q'$ such that $d^+_{Q'}(x)=n^+_x$ and $d^-_{Q'}(x)=n^-_x$ for every $x\in V(G)=V(Q)$.
\end{lemma}
\proof
Our aim is to apply the Max-Flow-Min-Cut theorem.
So let $H$ be the (unoriented) bipartite multigraph whose vertex classes $A$ and $B$ are both copies of $V(G)$ and in which
$a\in A$ is joined to $b\in B$ once for every (directed) edge $ab$ of $Q$. Give every edge of $H$ capacity~1. Add a source $s^*$ which is
joined to every vertex $a\in A$ with an edge of capacity $n^+_a$. Add a sink $t^*$ which is
joined to every vertex $b\in B$ with an edge of capacity $n^-_b$. Let $r:=\sum_{x\in V(G)} n^+_x/n=\sum_{x\in V(G)} n^-_x/n$.
Note that an integer-valued $rn$-flow corresponds to the desired spanning submultigraph $Q'$ of $Q$. Thus by the
Max-Flow-Min-Cut theorem it suffices to show that every cut has capacity at least $rn$.

So consider a minimal cut $\mathcal{C}$. Let $S$ be the set of all those vertices $a\in A$ for which $s^*a\notin \mathcal{C}$.
Similarly, let $T$ be the set of all those vertices $b\in B$ for which $bt^*\notin \mathcal{C}$.
Let $S':=A\setminus S$ and $T':=B\setminus T$.
Thus the capacity of $\mathcal{C}$ is
$$ c:=\sum_{s\in S'} n^+_s+e_B(S,T)+ \sum_{t\in T'} n^-_t.
$$
Suppose first that $|T'|\ge |S|+\nu n/2$. But then
$$c\ge \sum_{s\in S'} n^+_s+\sum_{t\in T'} n^-_t
\ge (1-\eps)\xi n (|S'|+|T'|)\ge (1-\eps)\xi n (n+\nu n/2)\ge (1+\eps)\xi n^2\ge rn,$$
as required.

So suppose next that $|T'|\le |S|+\nu n/2$ and $\tau n<|S|<(1-\tau)n$. Then at least $\nu n/2$ vertices
from $T$ lie in $RN^+_{\nu,G}(S)$ and each such vertex receives at least $\nu n$ edges from $S$ (in the digraph~$G$).
Thus $$c\ge e_B(S,T) \ge q\cdot e_G(S,T)\ge q\nu^2 n^2/2\ge (1+\eps)\xi n^2\ge rn,$$ as required.

Now suppose that $|T'|\le |S|+\nu n/2$ and $|S|\le \tau n$. Then in $G$ every vertex in $S$ sends at least
$\alpha n/2$ edges to $T$ as $\delta^0(G)\ge \alpha n$. Thus in $Q$ every vertex in $S$ sends at least
$q\alpha n/2$ edges to $T$ and so
$$c\ge \sum_{s\in S'} n^+_s+|S|\cdot q\alpha n/2\ge \sum_{s\in S'} n^+_s +|S|(1+\eps)\xi n\ge rn,
$$
as required. 

Finally, suppose that $|S|\ge (1-\tau )n$. Then in $G$ every vertex in $T$ receives at least $\alpha n/2$ edges from $S$.
Thus in $Q$ every vertex in $T$ receives at least $q\alpha n/2$ edges from $S$ and so
$$c\ge \sum_{t\in T'} n^-_t+|T|\cdot q\alpha n/2\ge \sum_{t\in T'} n^-_t +|T|(1+\eps)\xi n\ge rn,
$$
as required. 
\endproof

Recall from Section~\ref{notation} that the $r$-fold blow-up of a digraph $G$ is obtained from $G$ by replacing every vertex $x$ of $G$ by $r$ vertices
and replacing every edge $xy$ of $G$ by the complete bipartite graph $K_{r,r}$ between the two
sets of $r$ vertices corresponding to $x$ and $y$ such that all the edges of $K_{r,r}$ are oriented towards the $r$ vertices corresponding
to $y$. 

\begin{lemma}\label{expanderblowup}
Let $r\ge 3$ and let $G$ be a robust $(\nu,\tau)$-outexpander with $0<3\nu\le \tau<1$. Let $G'$ be the $r$-fold blow-up of
$G$. Then $G'$ is a robust $(\nu^3,2\tau)$-outexpander.
\end{lemma}
\begin{proof}
Let $n:=|G|$. So $|G'|=rn$. Call two vertices in $G'$ \emph{friends} if they correspond
to the same vertex of $G$. (In particular, every vertex is a friend of itself.) Consider any $S' \subseteq V(G')$ with $2\tau r n
\leq |S'| \leq (1 - 2\tau)rn$. Call a vertex $x\in S'$ \emph{bad} if $S'$ contains
at most $\nu^2 r$ friends of $x$. So if $\nu^2 r<1$ then no vertex in $S'$ is bad. Let $b$ denote the number of bad vertices in $S'$.
Then $S'$ contains a set $S^*$ of at least $b/\nu^2 r$ bad vertices corresponding to different vertices of $G$.
(So no two vertices in $S^*$ are friends.) But every $x\in S^*$ has at least $r-1-\nu^2 r\ge r/2$ friends%
    \COMMENT{$\nu<1/3$ and so $r-1-\nu^2 r\ge 8r/9-1\ge r/2$ since $r\ge 3$}
outside $S'$.  
Thus $$\frac{b}{\nu^2 r}\cdot \frac{r}{2}\le |S^*|\cdot \frac{r}{2}\le |G'|=rn$$
and so $b\le 2\nu^2 rn$. Let $S''\subseteq S'$ be the set of all those vertices in $S'$ which are not bad
and let $S$ be the set of all those vertices $x$ in $G$ for which $S''$ contains a copy of $x$.
Thus
\begin{equation}\label{eqS}
|S|\ge |S''|/r=(|S'|-b)/r \ge |S'|/2r\ge \tau n.
\end{equation}
Since $G$ is a robust $(\nu,\tau)$-outexpander, it follows that:
\begin{itemize}
\item [(i)] Either $|RN^+_{\nu,G} (S)| \geq  |S| + \nu n$;
\item [(ii)] or $|S| \geq (1 - \tau)n$, in which case (considering
a subset of $S$ of size $(1 - \tau)n$) we have
$|RN^+_{\nu,G}(S)| \geq (1 - \tau + \nu)n$.
\end{itemize}
Note that if a vertex $x$ of $G$ belongs to $RN^+_{\nu,G}(S)$, then any copy
$x'$ of $x$ in $G'$ has at least $\nu^2 r\cdot \nu n=\nu^3 |G'|$ inneighbours in $S''$ (since no vertex in $S''$ is bad) and
so $x'\in RN^+_{\nu^3,G'}(S')$. It follows that $|RN^+_{\nu^3,G'}(S')|
\geq r |RN^+_{\nu,G}(S)|$. Thus, in case (i) we have
\[ |RN^+_{\nu^3,G'}(S')| \geq r |RN^+_{\nu,G}(S)| \geq r|S| + r \nu n \stackrel{(\ref{eqS})}{\ge} |S'|-b+ r \nu n \geq |S'| + \nu^3 rn,\]
while in case (ii) we have
\[ |RN^+_{\nu^3,G'}(S')| \geq r |RN^+_{\nu,G}(S)| \geq (1-\tau)rn + \nu r n \geq (1-\tau)rn \geq |S'| + \nu^3 rn,\]
as required.
\end{proof}


\section{Tools for finding Hamilton cycles} \label{sec:hamtools}

The following well known observation states that every regular multidigraph $G$ has a 1-factorization, i.e.~the edges of
$G$ can be decomposed into edge-disjoint 1-factors. (In a multidigraph $G$ we allow multiple edges between any two vertices.
$G$ is $r$-regular if every vertex sends out $r$ edges and receives $r$ edges.)

\begin{prop}\label{1factor}
Let $G$ be a regular multidigraph. Then $G$ has a 1-factorization.
\end{prop}
\proof
Consider an auxiliary bipartite multigraph $G'$ whose vertex classes $A$ and $B$ are copies of the vertex set $V(G)$ of $G$
and in which the number of (undirected) edges between $a\in A$ and $b\in B$ equals the number of (directed) edges from $a$ to $b$ in $G$.
Then $G'$ is regular and so Hall's theorem (which holds for bipartite multigraphs as well) implies that $G'$ has a decomposition into
edge-disjoint perfect matchings. But every perfect matching in $G'$ corresponds to a $1$-factor of $G$.
\endproof
As mentioned earlier, there are several well known polynomial time algorithms for finding maximum matchings -- these can be applied repeatedly to find the above 
factorization.

The following result (Theorem~16 from~\cite{KOTchvatal}) guarantees a Hamilton cycle in a robust outexpander $G$, as long as
the minimum semidegree of $G$ is not too small. As shown in~\cite{CKKO}, this Hamilton cycle can be found in polynomial time. 
\begin{theorem}\label{expanderthm} 
Suppose that $0<1/n\ll\nu\le \tau\ll\alpha<1$. Let~$G$ be a digraph on~$n$ vertices with
$\delta^0(G)\ge \alpha n$ which is a robust $(\nu,\tau)$-outexpander. Then~$G$ contains a Hamilton cycle.
\end{theorem}

The \emph{$k$th power} of a cycle $C$ is obtained from $C$ by adding an edge between every pair of
vertices whose distance on $C$ is at most~$k$. We will also use the following result of K\'omlos,
S\'ark\"ozy and Szemer\'edi~\cite{KKSpower} (where they proved the so-called P\'osa-Seymour conjecture for sufficiently
large graphs).
We do not use the full strength of the result:
any bound of the form $(1-\eps)n$ instead of $10n/11$
would be sufficient for our proof. Moreover, we will only apply the result to a graph of bounded size
(in the proof of Lemma~\ref{splitinitcleanH}), so an algorithmic version is not necessary.

\begin{theorem}\label{10thpower}
There exists an $n_0\in\mathbb{N}$ such that every graph $G$ on $n\ge n_0$ vertices with minimum degree
at least $10n/11$ contains the $10$th power of a Hamilton cycle.
\end{theorem}

The next two lemmas will be used to turn a 1-regular digraph $F$ into a cycle on the same vertex set.
More precisely, suppose that $G$ is a blow-up of a cycle $C=V_1 \dots V_k$ so that for any successive clusters $V_i$, $V_{i+1}$
the subdigraph $G[V_i,V_{i+1}]$ of $G$ is superregular. 
Suppose also that $F$ is a $1$-regular digraph on $V_1\cup\dots\cup V_k$ which `winds around' $C$,
i.e.~each edge goes from $V_i$ to $V_{i+1}$ for some~$i$. Then we can replace the edges of $F$ from e.g.~$V_1$ to $V_2$
with edges from $G$ between these clusters to turn $F$ into a Hamilton cycle. 
(Actually, the lemma is more general and does not require all edges of $F$ to wind around $C$.) 

To prove the lemmas, we will use an idea from~\cite{CKKOsemi}.

\begin{lemma}\label{mergecycles0}
Suppose that $0<1/m\ll d'\ll \eps\ll d\ll \zeta ,1/t\le 1/2$.
Let $V_1,\dots,V_k$ be pairwise disjoint clusters, each of size $m$ and let $C=V_1\dots V_k$ be a directed cycle on these clusters.
Let $G$ be a digraph on $V_1\cup \dots\cup V_k$ and let $J\subseteq E(C)$. For each edge $V_iV_{i+1}\in J$, let $V^1_i\subseteq V_i$
and $V^2_{i+1}\subseteq V_{i+1}$ be such that $|V^1_i|=|V^2_{i+1}|\ge m/100$ and%
   \COMMENT{the condition that $\ge m/100$ is necessary as Theorem~\ref{expanderthm} requires $1/n \ll \nu$}
such that $G[V^1_i,V^2_{i+1}]$ is $(\eps,d',\zeta d',td'/d)$-superregular. 
Suppose that $F$ is a $1$-regular digraph with $V_1\cup \dots \cup V_k\subseteq V(F)$ such that the following properties hold:
\begin{itemize}
\item[\rm{(i)}]For each edge $V_iV_{i+1}\in J$ the digraph $F[V^1_i,V^2_{i+1}]$ is a perfect matching.
\item[\rm{(ii)}] For each cycle $D$ in $F$ there is some edge $V_iV_{i+1}\in J$ such that $D$ contains a vertex
in $V^1_i$.
\item[\rm{(iii)}] Whenever $V_iV_{i+1}, V_jV_{j+1}\in J$ are such that $J$ avoids all edges in the segment $V_{i+1}CV_j$ of
$C$ from $V_{i+1}$ to $V_j$, then $F$ contains a path $P_{ij}$ joining some vertex $u_{i+1}\in V^2_{i+1}$ to some
vertex $u'_j\in V^1_j$ such that $P_{ij}$ winds around~$C$.
\end{itemize}
Then we can obtain a cycle on $V(F)$ from $F$ by replacing $F[V^1_i,V^2_{i+1}]$ with a suitable perfect matching
in $G[V^1_i,V^2_{i+1}]$ for each edge $V_iV_{i+1}\in J$.
Moreover, if $J=E(C)$ then (iii) can be replaced by
\begin{itemize}
\item[\rm{(iii$'$)}] $V^1_i\cap V^2_i\neq \emptyset$ for all $i=1,\dots,k$.
\end{itemize}
\end{lemma}
\proof
For any edge $V_iV_{i+1}\in J$, let  Old$_i$ be the perfect matching $F[V^1_i,V^2_{i+1}]$.
We will first prove the following:

 \textno
For any edge $V_iV_{i+1}\in J$, we can find a perfect matching ${\rm New}_i$ in $G[V^1_i,V^2_{i+1}]=:G_i$
so that if we replace ${\rm Old}_i$ in $F$ with ${\rm New}_i$, then all
vertices of $G_i$ will lie on a common cycle in the new
$1$-factor~$F'$ thus obtained from $F$. Moreover any pair of
vertices of $F$ that were formerly on a common cycle in $F$ are still on a
common cycle in $F'$. &(\dagger)

To prove ($\dagger$), we proceed as follows. Pick $\nu$ and $\tau$ such that $1/m\ll \nu \ll \tau\ll d'$.
For every $u \in V^2_{i+1}$, we move along the cycle $C_u$ of $F$ containing $u$
(starting at $u$) and let $f(u)$ be the first vertex on $C_u$ in
$V^1_i$ (note that $f(u)$ exists by (i)). Define an auxiliary digraph $A$ on $V^2_{i+1}$ such that $N^+_A(u):=N^+_{G_i}(f(u))$.
So $A$ is obtained by identifying each pair $(u,f(u))$ into one vertex with an edge from $(u,f(u))$ to
$(v,f(v))$ if $G_i$ has an edge from $f(u)$ to $v$. So Lemma~\ref{regtoexpander} applied
with $d'$, $d/t$ playing the roles of $d$, $\mu$ implies
that $A$ is a robust $(\nu,\tau)$-outexpander. Moreover, $\delta^0(A)= \delta^0(G_i)\ge\zeta d'|V^2_{i+1}|=\zeta d'|A|$ by (Reg4).%
    \COMMENT{Actually $A$ might have loops. But after deleting them we still have a robust $(\nu,\tau)$-outexpander
with $\delta^0(A)\ge \zeta d'|A|-1$. So it's maybe better to gloss over it...}
Thus Theorem~\ref{expanderthm} implies
that $A$ has a Hamilton cycle, which clearly corresponds to a perfect matching New$_i$ in~$G_i$ with the desired property.

Now we apply ($\dagger$) to every edge $V_iV_{i+1}\in J$ sequentially. 
We claim that after repeating this for every such edge, the resulting $1$-regular digraph $F''$ is a cycle.
To see this, note that (ii) and the last part of ($\dagger$) together imply that every cycle of $F''$ contains a vertex in
$V^1_i$ for some edge $V_iV_{i+1}\in J$. Moreover, by the first part of ($\dagger$), all the vertices
in $V^1_i\cup V^2_{i+1}$ lie on a common cycle of $F''$, $C_i$ say. So all the $C_i$ together form~$F''$.
Consider any two edges $V_iV_{i+1}, V_jV_{j+1}\in J$ such that $J$ avoids all edges in $V_{i+1}CV_j$ and let $P_{ij}=u_{i+1}\dots u'_j$ be the
path guaranteed by~(iii). Since $P_{ij}$ winds around~$C$, it follows that%
    \COMMENT{Also use that $J$ avoids all edges in $V_{i+1}CV_j$ here. So we didn't change any edges of $F$ on this segement.}
$P_{ij}\subseteq F''$. Thus $u_{i+1}$ and $u'_j$ lie on a common cycle in $F''$. But 
$u_{i+1}\in C_i$ and $u'_j\in C_j$. Thus $C_i=C_j$ and so all the $C_i$ are identical. Thus $F''$ is a cycle.

It remains to check that if $J=E(C)$ then (iii) can be replaced by~(iii$'$). But this is clear since in this case we have
$i+1=j$ in~(iii) and so we can take $P_{ij}$ to be any vertex in $V^2_j\cap V^1_j$.
\endproof

We will also use the following slightly different version of Lemma~\ref{mergecycles0}, which involves the usual 
notion of $\eps$-regularity. In this paper, we will only use Lemmas~\ref{mergecycles0} and~\ref{mergecycles} in the special case when $J=E(C)$.
(So in Lemma~\ref{mergecycles} condition~(iii) can be omitted.) 
The more general version of Lemma~\ref{mergecycles} will be used in~\cite{OS}.

\begin{lemma}\label{mergecycles}
Let $0<1/m\ll \eps\ll d< 1$.
Let $V_1,\dots,V_k$ be pairwise disjoint clusters, each of size $m$ and let $C=V_1\dots V_k$ be a directed cycle on these clusters.
Let $J\subseteq E(C)$.
Let $G$ be a digraph on $V_1\cup \dots\cup V_k$ such that $G[V_i,V_{i+1}]$ is $(\eps, d)$-superregular
for every $V_iV_{i+1} \in J$. For each edge $V_iV_{i+1}\in J$ let $V^1_i\subseteq V_i$
and $V^2_{i+1}\subseteq V_{i+1}$ be sets of size at least $(1-d/2)m$ such that $|V^1_i|=|V^2_{i+1}|$.
Suppose that $F$ is a $1$-regular digraph with $V_1\cup \dots \cup V_k\subseteq V(F)$ such that the following properties hold:
\begin{itemize}
\item[\rm{(i)}]For each edge $V_iV_{i+1}\in J$ the digraph $F[V^1_i,V^2_{i+1}]$ is a perfect matching.
\item[\rm{(ii)}] For each cycle $D$ in $F$ there is some edge $V_iV_{i+1}\in J$ such that $D$ contains a vertex
in $V^1_i$.
\item[\rm{(iii)}] Whenever $V_iV_{i+1}, V_jV_{j+1}\in J$ are such that $J$ avoids all edges in the segment $V_{i+1}CV_j$ of
$C$ from $V_{i+1}$ to $V_j$, then $F$ contains a path $P_{ij}$ joining some vertex $u_{i+1}\in V^2_{i+1}$ to some
vertex $u'_j\in V^1_j$ such that $P_{ij}$ winds around~$C$.
\end{itemize}
Then we can obtain a cycle on $V(F)$ from $F$ by replacing $F[V^1_i,V^2_{i+1}]$ with a suitable perfect matching
in $G[V^1_i,V^2_{i+1}]$ for each edge $V_iV_{i+1}\in J$. Moreover, if $J=E(C)$ then (iii) can be omitted.
\end{lemma}
\proof
The proof is very similar to that of Lemma~\ref{mergecycles0}. As before, for any edge $V_iV_{i+1}\in J$ let $G_i:=G[V^1_i,V^2_{i+1}]$.
Since $|V^1_i|,|V^2_{i+1}|\ge (1-d/2)m$, the $(\eps,d)$-superregularity of $G[V_i,V_{i+1}]$
implies that $G_i$ is still $(2\eps,d/2)$-superregular. Now we proceed as before to define $A$.
Using the superregularity of $G_i$, it is easy to show that
$A$ is a robust $(\eps d,3\eps)$-outexpander with $\delta^0(A) \ge d|A|/3$. 
(Indeed, for the outexpansion it suffices to observe that for all $U \subseteq V_i^1$ with 
$|U| \ge 3\eps |A|$, the number of vertices in $V_{i+1}^2$
which receive at least $(d/2-2\eps) |U|\ge \eps d |A|$ edges from $U$ in $G_i$ is at least $(1-2\eps) |A|$.)%
    \COMMENT{Don't get $\delta^0(A) \ge d|A|/2$ since $\delta(G_i)\ge d/2-2\eps m$.}
So as before we can apply
Theorem~\ref{expanderthm} to find a Hamilton cycle in $A$, which corresponds to a matching as required in~$(\dagger)$.
The remainder of the argument is now identical.
\endproof


\section{Schemes, consistent systems, chord sequences, and exceptional factors} \label{sec:except}

\subsection{Schemes and consistent systems}
In order to simplify (and shorten) the statements of our lemmas throughout the paper, we will introduce the notions of a $(k,m,\eps,d)$-scheme
and of a consistent $(\ell^*,k,m,\eps,d,\nu,\tau,\alpha,\theta)$-system. 
$(G,\cP,R,C)$ is called a \emph{$(k,m,\eps,d)$-scheme} if the following properties
are satisfied:
\begin{itemize}
\item [(Sch1)] $G$ and $R$ are digraphs. $\mathcal{P}$ is a partition of $V(G)$ into an exceptional set $V_0$ of size at most $\eps |G|$
and into $k$ clusters of size $m$. The vertex set of $R$ consists of these clusters.
\item[(Sch2)] For every edge $UW$ of $R$ the corresponding pair $G[U,W]$ is $(\eps,\ge d)$-regular.
\item[(Sch3)] $C$ is a Hamilton cycle in $R$ and for every edge $UW$ of $C$ the corresponding pair $G[U,W]$ is $[\eps,\ge d]$-superregular.
\item[(Sch4)] $V_0$ forms an independent set in $G$.
\end{itemize}
So roughly speaking, a scheme consists of a digraph $G$ with a regularity partition $\cP$ where the corresponding reduced digraph $R$ contains a Hamilton cycle $C$.
A consistent system has several additional features: mainly, the digraph $G$ needs to be a robust outexpander and the definition involves an additional partition $\cP_0$
which is coarser than $\cP$. More precisely, 
$(G,\cP_0, R_0,C_0,\cP,R,C)$ is called a \emph{consistent
$(\ell^*,k,m,\eps,d,\nu,\tau,\alpha,\theta)$-system} if the following properties
are satisfied (for (CSys3), recall that an $\ell^*$-refinement was defined before Lemma~\ref{randompartition}):
\begin{itemize}
\item[(CSys1)] Each of the digraphs $G$, $R_0$ and $R$ is a robust $(\nu,\tau)$-outexpander. Moreover $\delta^0(G)\ge \alpha n$, where $n:=|G|$,
$\delta^0(R_0)\ge \alpha |R_0|$ and $\delta^0(R)\ge \alpha |R|$. 
\item[(CSys2)] $\mathcal{P}_0$ is a partition of $V(G)$ into an exceptional set $V_0^0$ of size at most $\eps n$
and into $k/\ell^*$ equal sized clusters. The vertex set of $R_0$ consists of these clusters. So $|R_0|=k/\ell^*$.
Similarly, $\mathcal{P}$ is a partition of $V(G)$ into an exceptional set $V_0$ of size at most $\eps n$ and into $k$ clusters of size $m$.
The vertex set of $R$ consists of these clusters. So $|R|=k$.
\item[(CSys3)] $\mathcal{P}$ is obtained from an $\ell^*$-refinement of $\mathcal{P}_0$ by removing some vertices from each cluster
of this refinement and adding them to the exceptional set $V_0^0$. So $V_0$ is the union of $V_0^0$ with the set of these
vertices.
\item[(CSys4)] For every edge $UW$ of $R$ the corresponding pair $G[U,W]$ is $(\eps,\ge d)$-regular.
\item[(CSys5)] $C_0$ is a Hamilton cycle in $R_0$. Similarly, $C$ is a Hamilton cycle in $R$ and for every edge $UW$ of $C$ the corresponding
pair $G[U,W]$ is $[\eps,\ge d]$-superregular.
\item[(CSys6)] Suppose that $W,W'$ are clusters in $\cP_0$ and $V,V'$ are clusters in $\cP$ with $V\subseteq W$ and $V'\subseteq W'$.
Then $WW'\in E(R_0)$ if and only if $VV'\in E(R)$.
\item[(CSys7)] $C$ can be viewed as obtained from $C_0$ by winding $\ell^*$ times around $C_0$, i.e.~for every edge $WW'$ of $C_0$
there are precisely $\ell^*$ edges $VV'$ of $C$ such that $V\subseteq W$ and $V'\subseteq W'$. 
\item[(CSys8)] Whenever $W$ is a cluster in $\mathcal{P}_0$ and $x\in V(G)$ is a vertex with $n^+\ge \tau |W|$ outneighbours%
     \COMMENT{Will apply this in Lemma~\ref{prelimfactor} and there we apply it with $\alpha/4$ instead of $\tau$.
But we cannot write $\alpha/4$ here since in Lemma~\ref{deletesystem} the $\alpha$ becomes smaller. For the same reason, we cannot
write $d$ or $\nu$. But $\tau$ doesn't become smaller in Lemma~\ref{deletesystem}, so we can take $\tau$ here.}
in $W$, then $x$ has at least $\theta n^+/\ell^*$ outneighbours in each cluster $V$ in $\mathcal{P}$ with $V\subseteq W$.
A similar condition holds for inneighbours of the vertices of $G$.
\item[(CSys9)] $V_0$ forms an independent set in $G$.
\end{itemize}
Note that (CSys3) and (CSys6) together imply that $R$ is an $\ell^*$-fold blow-up of $R_0$. Moreover,
if $(G,\cP_0, R_0,C_0,\cP,R,C)$ is a consistent $(\ell^*,k,m,\eps,d,\nu,\tau,\alpha,\theta)$-system
then $(G,\cP,R,C)$ is a $(k,m,\eps,d)$-scheme. We will usually denote the clusters in $\cP$ by $V_1,\dots,V_k$ and assume that
they are labelled in such a way that $C=V_1\dots V_k$. 

The next result states that if $(G,\cP_0, R_0,C_0,\cP,R,C)$ is a consistent system and one deletes only a few edges at every vertex of $G$,
then one still has a consistent system with slightly worse parameters. The analogue also holds for schemes.

\begin{lemma}\label{deletesystem}
Suppose that $0<1/n\ll 1/k\ll \eps\le \eps'\ll d\ll \nu\ll \tau\ll \alpha,\theta\le 1$.
\begin{itemize}
\item[{\rm (i)}] Let $(G,\cP_0, R_0,C_0,\cP,R,C)$ be a consistent $(\ell^*,k,m,\eps,d,\nu,\tau,\alpha,\theta)$-system
with $|G|=n$. Let $G'$ be a digraph obtained from $G$ by deleting at most $\eps' m$ outedges and at most $\eps' m$ inedges at every vertex of $G$.
Then $$(G',\cP_0, R_0,C_0,\cP,R,C)$$ is still a consistent $(\ell^*,k,m,3\sqrt{\eps'},d,\nu/2,\tau,\alpha/2,\theta/2)$-system.
\item[{\rm (ii)}] Let $(G,\cP,R,C)$ be a $(k,m,\eps,d)$-scheme
with $|G|=n$. Let $G'$ be a digraph obtained from $G$ by deleting at most $\eps' m$ outedges and at most $\eps' m$ inedges at every vertex of $G$.
Then $(G',\cP,R,C)$ is still a $(k,m,3\sqrt{\eps'},d)$-scheme.
\end{itemize}
\end{lemma}
\proof
We only prove~(i). The argument for~(ii) is similar.
It is easy to see that $(G',\cP_0, R_0,C_0,\cP,R,C)$ still satisfies (CSys2), (CSys3), (CSys6), (CSys7) and (CSys9).
Proposition~\ref{superslice} (applied with $d':=\eps'$) implies that every edge of $R$ still corresponds to an
$2\sqrt{\eps'}$-regular pair in $G'$ of density at least $d-\eps -2\sqrt{\eps'}\ge d-3\sqrt{\eps'}$. Similarly, every edge of $C$ still corresponds to an
$[2\sqrt{\eps'},d]$-superregular pair in $G'$. Thus (CSys4) and (CSys5) hold with $\eps$ replaced by $3\sqrt{\eps'}$.
Moreover, it is easy to check that $G'$ is still a robust $(\nu/2,\tau)$-outexpander with $\delta^0(G)\ge \alpha n/2$.
So (CSys1) holds with $\nu$ and $\alpha$ replaced by $\nu/2$ and $\alpha/2$.
Finally, suppose that $x$, $n^+$, $W$ and $V$ are as in (CSys8). Note that $|W|\ge \ell^* m$ and so
$$\eps ' m\le \theta \tau m/2\le \theta \tau |W|/2\ell^*\le \theta n^+/2\ell^*.$$ Thus
in the digraph $G'$ the number of outneighbours of $x$ in $V$ is at least $\theta n^+/\ell^*-\eps' m\ge \theta n^+/2\ell^*$.
So (CSys8) holds with $\theta$ replaced by $\theta/2$.
\endproof


\subsection{Shifted walks and chord sequences} \label{sec:shiftwalks}

Roughly speaking, the Hamilton cycles we will find usually wind around a blown-up cycle 
$C=V_1 \dots V_k$. Here the $V_i$ are clusters. However, we also need to incorporate the vertices of an exceptional set $V_0$
into the cycle. For each $x \in V_0$, Lemma~\ref{prelimfactor}(i) below will give suitable in- and outneighbours $x^-$ and $x^+$
which attach $x$ to the blown-up cycle. However, to build a Hamilton cycle, we need additional edges:
Suppose for example that $V_0=\{x\}$ and $x^+$ is not in the cluster succeeding the cluster containing $x^-$.
Then it is impossible to extend the path $x^-xx^+$ into a Hamilton cycle in which all other edges wind around $C$.
So we need additional edges which will `balance out' the edges $x^-x$ and $xx^+$. 
These additional edges are found via so-called `shifted walks' and their associated chord sequences, which we define next.
Shifted walks were first introduced in~\cite{kelly}, also in order to find Hamilton cycles in directed graphs.

Let $R$ be a digraph and let $C$ be a Hamilton cycle in $R$.
Given a vertex $V$ of $R$, let $V^+$ denote the vertex succeeding $V$ on $C$ and let 
$V^-$ denote the vertex preceding $V$. (Later on, the vertices of $R$ will be clusters, so
we use capital letters to denote them.) A {\em shifted walk} from a vertex $A$ to a vertex $B$ in $R$
is a walk $SW(A,B)$ of the form
\[ SW(A,B) = V_1 C V^-_1 V_2 C V^-_2 \dots V_t C V^-_t V_{t+1},\]
where $V_1=A$, $V_{t+1} = B$ and the edge $V^-_i V_{i+1}$ belongs to $R$ for each $i=1,\dots,t$.
(Here we write $V_iCV_i^-$ for the path obtained from $C$ by deleting the edge $V_i^-V_i$.) We say that $SW(A,B)$ \emph{traverses $C$ $t$ times}.
We call the edges $V^-_i V_{i+1}$ the \emph{chord edges} of $SW(A,B)$. 
If $A=B$ then $A$ is also a shifted walk from $A$ to $B$.
Without loss of generality, we may assume that an edge of $C$ is not a chord edge
according to the above definition. (Indeed, suppose that  $V^-_i V_{i+1}$ is an edge of $C$.
Then $V_{i+1}=V_i$ and so we can obtain a shorter shifted walk from $A$ to $B$.) 

For our purposes, it turns out that shifted walks contain too many edges. So we will only use their chord edges. So
given a shifted walk $$SW(A,B)=V_1 C V^-_1 V_2 C V^-_2 \dots V_t C V^-_t V_{t+1},$$
the corresponding \emph{chord sequence $CS(A,B)$} from $A$ to $B$ consists of all chord edges in $SW(A,B)$
in the same order as they appear in $SW(A,B)$.
(In~\cite{KOTkelly}, this was called a skeleton walk, but we prefer not to use this name here as the chord edges do not actually form a walk.)
We say that $V$ lies in the \emph{interior of $CS(A,B)$} if
$V\in \{V_2,V^-_2,\dots,V_t,V^-_t\}$.

The next result guarantees a short chord sequence between any two vertices in a robust outexpander. Moreover, this chord sequence
can be chosen so that its interior avoids a given small set. The proof does not require the outexpansion property to be robust.

\begin{lemma} \label{buildchord}
Let $R$ be a robust $(\nu,\tau)$-outexpander with $\delta^0(R) \geq 2\tau |R|$ and $\nu \le \tau\le 1/3$. Let
$C$ be Hamilton cycle in $R$. Given vertices $A,B\in V(R)$ and a set of vertices $\mathcal{V}'\subseteq V(R)$ with $|\mathcal{V}'| \le \nu |R|/4$,
there is a  chord sequence $CS(A,B)$ in $R$ containing at most $3/\nu$ edges
whose interior avoids $\mathcal{V}'$.
\end{lemma}
\proof
Let $\mathcal{V}''$ be the union of $\mathcal{V}'$ and the set of all those clusters $V$ for which $V^-\in \mathcal{V}'$
(where $V^-$ is the predecessor of $V$ on $C$). Thus $|\mathcal{V}''|\le \nu |R|/2$.
Let $\mathcal{V^*}:=V(R) \setminus \mathcal{V}''$. Pick any outneighbour $A_0$ of $A^-$ in~$\mathcal{V^*}$.
Let $S_1:=N^+_R(A_0^-)\cap \mathcal{V^*}$. So $|S_1|\ge 2\tau |R|-|\mathcal{V}''|\ge \nu |R|/2$. For each $i\ge 2$ let $S_i:=N^+_R(N^-_C(S_{i-1}))\cap \mathcal{V^*}$.
Thus $S_{i-1}\subseteq S_i$ and each cluster in $S_i$ can be reached by a shifted walk from $A_0$ that traverses
$C$ at most $i$ times and avoids $\mathcal{V}'$. Moreover for each $i\ge 2$ either $|S_{i-1}|< (1-\tau)|R|$ and
$$
|S_i|\ge |RN^+_{\nu,R}(N^-_C(S_{i-1}))|-|\mathcal{V}''|
\ge |N^-_C(S_{i-1})| + \nu |R|-|\mathcal{V}''|
\ge |S_{i-1}|+\nu |R|/2
$$ or
$|S_{i-1}|\ge (1-\tau)|R|$ and $|S_i|\ge |S_{i-1}|\ge (1-\tau)|R|$. But this implies that $|S_{\lceil 2/\nu\rceil}|\ge (1-\tau)|R|$.
Thus $B$ has a neighbour $B_0$ such that its successor $B^+_0$ on $C$ lies in $S_{\lceil 2/\nu\rceil}$.
Since $B^+_0\in S_{\lceil 2/\nu\rceil}$ there is a shifted walk $SW(A_0,B^+_0)$ which traverses $C$ at most $\lceil 2/\nu\rceil$ times
and which avoids $\mathcal{V}'$. Thus $ACA^-A_0\cup SW(A_0,B^+_0)\cup B^+_0CB_0B$ is a shifted walk from $A$ to $B$ which traverses $C$ at most
$\lceil 2/\nu\rceil+2\le 3/\nu$ times and which meets $\mathcal{V}'$ at most in the clusters $A^-$ and $B$. So the chord sequence corresponding to this
shifted walk is as required.
\endproof

The following proposition records the crucial property of chord sequences for later use.
As indicated earlier, it means we can use these sequences to `balance out' an arbitrary edge $x^-x^+$
(or a path $x^-xx^+$ with $x\in V_0$) to obtain a `locally balanced' set of edges.

\begin{prop} \label{balanced}
Let $(G,\cP,R,C)$ be a $(k,m,\eps,d)$-scheme. Given vertices $x^-,x^+$ which are contained in clusters $U(x^-)$ and $U(x^+)$
in $\cP$ respectively, consider any chord sequence $CS(U(x^+),U(x^-)^+)=:CS$ in $R$. 
Let $CS'$ be obtained from $CS$ by replacing each edge $UW$ of $CS$ with an edge of $G[U,W]$.
Suppose that $CS'$ is a matching which avoids both $x^-$ and $x^+$. Let $CS^*$ be obtained from $CS'$ by adding the edge $x^-x^+$.
For each cluster $U$ in $\cP$, let $U^1$ be the set of vertices of $U$ which are not an initial vertex of an edge in $CS^*$
and let $U^2$ be the set of vertices of $U$ which are not a final vertex of an edge in $CS^*$.
Then for each edge $UW$ on $C$, we have $|U^1|=|W^2|$.
\end{prop}
\proof
This follows immediately from the fact that for every edge $W^*W$ of $CS$ (apart from the final edge of $CS$), the next edge of $CS$
will be of the form $UU^*$, where $U$ is the predecessor of $W$ on $C$. If $W^*W$ is the final edge of $CS$ then $W=U(x^-)^+$ and
so the edge $x^-x^+\in CS^*$ has its initial vertex $x^-$ in the predecessor of $W$ on $C$ (see Figure~\ref{figchordseq}).
\endproof
\begin{figure}
\centering\footnotesize
\includegraphics[scale=0.35]{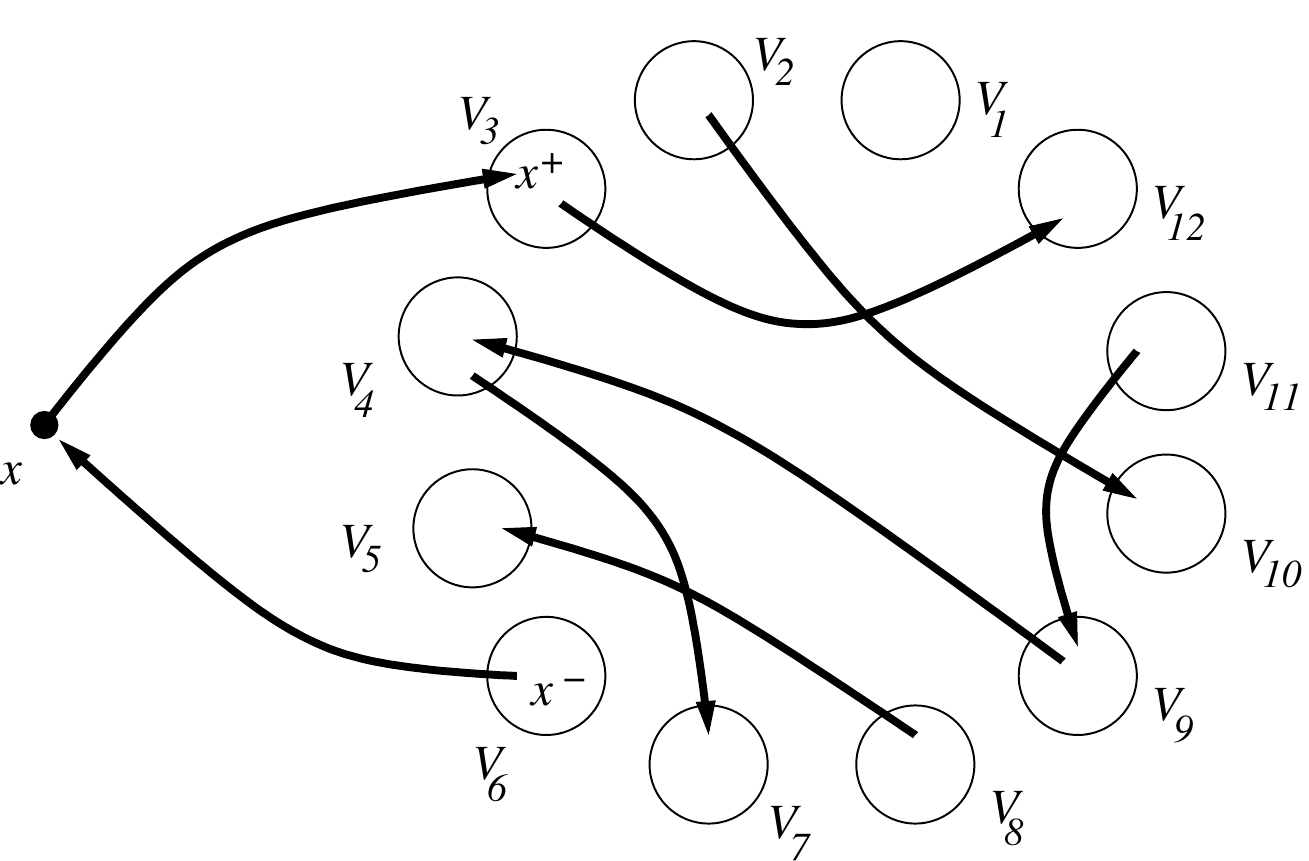}
\caption{Balancing out a path $x^-xx^+$ using a sequence $CS'$ obtained from
the chord sequence $CS(U(x^+),U(x^-)^+)=CS(V_3,V_7)=(V_2V_{10}, V_9V_4, V_3V_{12}, V_{11}V_9, V_8V_5, V_4V_7)$.
Here $C=V_1\dots V_{12}$.}\label{figchordseq}
\end{figure}

In a typical application of this observation, the assertion that $|U^1|=|W^2|$ means that it will be possible to choose a perfect matching in $G[U^1,W^2]$.
If we do this for each pair of consecutive clusters on $C$, then the union of all these matchings and the edges in $CS^*$ forms a $1$-regular digraph
$F$ covering all vertices in all clusters. We will then be able to transform $F$ into a Hamilton cycle, e.g.~using Lemma~\ref{mergecycles0} or~\ref{mergecycles}.


\subsection{Complete exceptional sequences and exceptional factors}\label{sec:ces}
Suppose that $(G,\cP,R,C)$ is a $(k,m,\eps,d)$-scheme. Let $V_1,\dots,V_k$ denote the clusters in $\cP$ such that $C=V_1\dots V_k$.
Recall that $V_0$ denotes the exceptional set. An \emph{original exceptional edge} in $G$ is an edge with one endvertex in $V_0$.
(So by (Sch4) the other endvertex lies in $V(G)\setminus V_0$.) An \emph{exceptional cover} $EC$ consists of precisely one outedge
and precisely one inedge incident to every vertex in $V_0$. Thus $|EC|=2|V_0|$.

Let $G^{\rm basic}$ be obtained from $G$ by adding all edges from $N^-(x)$ to $N^+(x)$ for every $x\in V_0$ and
by deleting $V_0$. We call each edge $yz$ from $N^-(x)$ to $N^+(x)$
that was added to $G$ in order to obtain $G^{\rm basic}$ an \emph{exceptional edge} and $x$ the \emph{vertex
associated with~$yz$}. Note that we might have $y=z$ in which case we add a loop.
Moreover, we allow $G^{\rm basic}$ to have multiple edges: if e.g.~$yz\in E(G)$ and $yz$ is an edge from $N^-(x)$ to $N^+(x)$
for precisely two exceptional vertices then $yz$ has multiplicity~3 in $G^{\rm basic}$ and precisely two of the edges from $y$ to $z$ in $G^{\rm basic}$
are exceptional edges. We sometimes write $G^{\rm orig}$ for $G$.

Given a spanning subdigraph $H$ of $G$, we define $H^{\rm basic}$ in a similar way. Conversely, if $H$ is a subdigraph of $G^{\rm basic}$, then
$H^{\rm orig}$ is the subdigraph of $G$ obtained from $H$ by replacing each exceptional edge $yz$ of $H$ with the path $yxz$, where $x\in V_0$ is the exceptional vertex
associated with~$yz$. Note that if $F$ is a $1$-factor of $G$, 
then $F^{\rm basic}$ is a 1-factor of $G^{\rm basic}$. Conversely, a $1$-factor $F$ of $G^{\rm basic}$ which contains exactly one
exceptional edge associated with every exceptional vertex corresponds to a $1$-factor $F^{\rm orig}$ in $G^{\rm orig}$.
Moreover, in this case $F$ is a Hamilton cycle if and only if $F^{\rm  orig}$ is a Hamilton cycle. 

A \emph{complete exceptional sequence} $CES$ is a matching in $G^{\rm basic}$ which consists of precisely one exceptional edge associated
with every exceptional vertex. So in particular $CES$ does not contain loops. Note that the original version $CES^{\rm orig}$ of $CES$
forms an exceptional cover. However, if $EC$ is an exceptional cover, then $EC^{\rm basic}$ might contain paths, cycles (and loops) and so it might
not form a complete exceptional sequence.

When constructing Hamilton cycles, we will usually do this by constructing a Hamilton cycle in $G^{\rm basic}$
which contains exactly one complete exceptional sequence and no other exceptional edges.
In other words, we will use the following observation, which is an immediate consequence of the above discussion.
\begin{obs} \label{basicobs}
Suppose that $D$ is a Hamilton cycle in $G^{\rm basic}$ which contains exactly one complete exceptional sequence $CES$ and no other exceptional edges.
Then $D^{\rm orig}$ is a Hamilton cycle in $G$. 
\end{obs}
For our arguments it is convenient to be able to define and use a digraph  which is regular (when viewed as a subdigraph of $G^{\rm basic}$)
and contains many different complete exceptional sequences
-- each of these will be part of a different Hamilton cycle in our decomposition. The exceptional factors $EF$ defined below have the required properties.
 
Given $L\in\mathbb{N}$ which divides $k$, the \emph{canonical interval partition
of $C$ into $L$ intervals} consists of the intervals $V_{(i-1)L+1}\dots V_{iL+1}$ for all $i=1,\dots,k/L$ (where $V_{k+1}:=V_1$).
Given an interval $I$ on~$C$, we write $I^\circ$ for the interior of $I$ and $I^{\circ\circ}$
for the interior of $I^\circ$. Moreover, we write $U\in I$ if $U$ is a cluster on $I$.

Suppose that $k/L, m/K\in\mathbb{N}$ and let $\cI$ be the canonical interval partition of $C$ into $L$
intervals of equal length. A \emph{complete exceptional path system} $CEPS$ (with respect to $C$)
with parameters $(K,L)$ spanning an interval $I=U_jU_{j+1}\dots U_{j'}$ with $I \in \cI$ consists of $m/K$
vertex-disjoint paths $P_1,\dots,P_{m/K}$ in $G^{\rm basic}$ such that the following conditions hold.
\begin{itemize}
\item[(CEPS1)] Every $P_s$ has its initial vertex in $U_j$ and its final vertex in $U_{j'}$.
\item[(CEPS2)] $CEPS$ contains a complete exceptional sequence $CES$ (but no other exceptional edges) and all the edges in $CES$ avoid the
endclusters $U_j$ and $U_{j'}$ of $I$.%
    \COMMENT{use the latter condition in the proof of Lemmas~\ref{exceptseq} and~\ref{cyclebreak}}
\item[(CEPS3)] $CEPS$ contains precisely $m/K$ vertices from every cluster in $I$ and no other vertices.
\end{itemize}
Note that the above implies that $CEPS^{\rm orig}$ consists of $m/K$ vertex-disjoint paths which cover all vertices in $V_0$ as well as $m/K$ vertices in each cluster
in $I$. Moreover, every path of $CEPS^{\rm orig}$ has its initial vertex in $U_j$ and its final vertex in $U_{j'}$.
If $K=1$ (and thus the vertex set of $CEPS$ is the union of all the clusters in~$I$), then we say that $CEPS$ \emph{completely spans~$I$}.

Suppose that $\cP'$ is a $K$-refinement of $\mathcal{P}$.
For each cluster $U\in \mathcal{P}$, let $U(1),\dots,U(K)$ denote the subclusters of $U$ in $\mathcal{P}'$. 
Consider a complete exceptional path system $CEPS$ as above. We say that $CEPS$ has \emph{style $b$} if 
its vertex set is $U_{j}(b) \cup \dots \cup U_{j'}(b)$.
An \emph{exceptional factor $EF$ with parameters $(K,L)$} (with respect to $C$, $\cP'$) is a $1$-factor of $G^{\rm basic}$ which
satisfies the following properties:
\begin{itemize}
\item[(EF1)] On each of the $L$ intervals $I\in\cI$, $EF$ induces the vertex-disjoint union of $K$ complete exceptional path systems.
\item[(EF2)] Moreover, for each $I \in \cI$ and each $b=1,\dots, K$, exactly one of the exceptional path systems in $EF$ spanning $I$ has style $b$.
\end{itemize}
Note that $EF$ consists of $KL$ edge-disjoint complete exceptional path systems. Moreover, the second part of (CEPS2)
implies that the union of all the $KL$ complete exceptional sequences contained in these complete exceptional path systems
forms a matching. This will be used in the proof of Lemma~\ref{cyclebreak}.

The reason that the definition of a consistent system involves not only the reduced digraph $R_0$ but also its
refinement $R$ is that this enables us to find complete exceptional path systems within an interval $I$ of $C$
and thus we will be able to find exceptional factors (see the proof of Lemma~\ref{prelimfactor} for more details).

\subsection{Finding exceptional factors in a consistent system}\label{sec:findexcfactor}

The following lemma will be used to construct the exceptional factors defined in the previous subsection.
For (a), recall that an $\eps$-uniform $K$-refinement of a partition $\cP$ was defined before Lemma~\ref{randompartition}.
Assertion (i) guarantees a `localized' exceptional cover for $V_0$, assertion (ii) finds  chord sequences in the reduced graph which `balance out'
this exceptional cover and (iii) finds edges in $G$ which correspond to these chord sequences.

\begin{lemma}\label{prelimfactor}
Suppose that $0<1/n\ll 1/k\ll \eps\ll \eps'\ll d\ll \nu\ll \tau\ll \alpha,\theta\le 1$, that $\ell^*/L, m/K\in\mathbb{N}$,%
    \COMMENT{This implies that $k$ is a multiple of $L$.}
that $L/\ell^*\ll 1$ and $\eps\ll 1/K,1/L$.
Let $(G,\cP_0, R_0,C_0,\cP,R,C)$ be a consistent $(\ell^*,k,m,\eps,d,\nu,\tau,\alpha,\theta)$-system
with $|G|=n$. Suppose that $G'$ is a spanning subdigraph of $G$ and that $\cP'$ is a partition of $V(G)$ such that the following
conditions are satisfied:
\begin{itemize}
\item[{\rm (a)}] $\cP'$ is an $\eps$-uniform $K$-refinement of $\cP$. (So in particular $V_0$ is
the exceptional set in $\cP'$.)
\item[{\rm (b)}] Every vertex $x\in V_0$ satisfies $d^\pm_{G}(x)- d^\pm_{G'}(x)\le \eps n$.
\item[{\rm (c)}] Every vertex $x\in V(G)\setminus V_0$ satisfies $d^\pm_{G}(x)- d^\pm_{G'}(x)\le (\eps')^3m/K$.
\end{itemize}
For every cluster $U$ in $\cP$ let $U(1),\dots,U(K)$ denote all those clusters in $\cP'$ which are contained in $U$.
Let $\mathcal{I}$ be the canonical interval partition of $C$ into $L$ intervals of equal length. Consider any $I\in \mathcal{I}$ and any $j$ with $1\le j\le K$. 
Then the following properties hold:
\begin{itemize}
\item[{\rm (i)}] For every exceptional vertex $x\in V_0$ there is a pair $x^-\neq x^+$ such that $x^-\in N^-_{G'}(x)\cap \bigcup_{U\in I^{\circ\circ}}U(j)$,
$x^+\in N^+_{G'}(x)\cap \bigcup_{U\in I^{\circ\circ}}U(j)$, such that the vertices in all these pairs $x^-,x^+$ are distinct for different exceptional vertices
and where each cluster in $\cP$ contains at most $\eps^{1/4} m$ vertices in these pairs.
\item [{\rm (ii)}] Given a vertex $z\in V(G)\setminus V_0$, let $U(z)$ denote the cluster in $\cP$ which contains $z$
and let $U(z)^+$ be the successor of that cluster on $C$.
Then for each of the pairs $x^-,x^+$ guaranteed by (i) there is a chord sequence $CS_x$
from $U(x^+)$ to $U(x^-)^+$ in $R$ which contains at most $3/\nu^3$ edges such that each of these edges has both endvertices in~$I^\circ$
and such that in total no cluster in $\cP$ is used more than $4\eps^{1/4}m$ times by all these $CS_x$.
\item [{\rm (iii)}] For each $x\in V_0$ there is a sequence $CS'_x$ of edges in $G'$ obtained by replacing every edge $UU'$ in $CS_x$ by an edge
in $G'[U(j),U'(j)]$ such that $CS'_x$ forms a matching which avoids all the pairs $y^-, y^+$ guaranteed by (i) (for all $y\in V_0$)
and such that the $CS'_x$ are pairwise vertex-disjoint. (Thus $\bigcup_{x\in V_0} CS'_x$ is a matching too.) 
\item [{\rm (iv)}] For every edge $UW$ of $C$ the pair $G'[U(j),W(j)]$ is $[\eps',\ge d]$-superregular.
\end{itemize}
\end{lemma}
\proof Recall that $k=|C|$ denotes the number of clusters in $\cP$. 
We will first prove (i). We will choose the pairs $x^-,x^+$ for every exceptional vertex $x$ in turn.
So suppose that for some exceptional vertices we have already chosen such pairs and that we now wish to choose such a pair for $x\in V_0$.
Let $S$ denote the set of all vertices lying in the pairs that we have chosen already. So $|S|< 2|V_0|\le 2\eps n$.
We say that a cluster of $\cP$ is \emph{full} if it contains at least $\eps^{1/4} m$ vertices from $S$.
Then the number of full clusters is at most 
$$
\frac{|S|}{\eps^{1/4}m} \le \frac{2\eps n}{\eps^{1/4}m}\le \sqrt{\eps} k.
$$
So for any vertex $x \in V_0$, the number of outneighbours of $x$ in the full clusters is 
at most $\sqrt{\eps} k m\le \sqrt{\eps} n$.

Set $k_0:=k/\ell^*$ and recall that $k_0=|R_0|=|C_0|$.
Note that the length of $I$ is $k/L =k_0\ell^*/L$ and thus the length of $I$ is a multiple of $k_0$
(as $\ell^*$ is a multiple of $L$ by assumption). But by (CSys7) $C$ was obtained from the Hamilton cycle $C_0$
on the clusters in $\cP_0$ by winding $\ell^*$ times around $C_0$. Thus for each cluster $V$ in $\cP_0$, $I$ contains at least
\begin{equation}\label{eqcountsub}
\frac{{\rm length}(I)}{k_0}=\frac{k/L}{k/\ell^*}= \frac{\ell^*}{L}
\end{equation}
clusters of the current partition $\cP$ which are contained in $V$.

We say that a cluster $V\in \cP_0$ is \emph{friendly} if $|N^+_{G}(x)\cap V|\ge \alpha |V|/3$.
(CSys1) implies that $\delta^0(G) \ge \alpha n$ and so at least $\alpha k_0/3$ clusters in $\cP_0$ are friendly. 
But (CSys8) now implies that for each such cluster~$V$, $x$ has at least
$\theta |N^+_{G}(x)\cap V|/\ell^*$ outneighbours in each cluster $U$ of $\cP$ with $U \subseteq V$ (in the digraph $G$).
Then by (a) and the condition (URef) in the definition of an $\eps$-uniform refinement, we have that $x$ has at least
$$\frac{\theta |N^+_{G}(x)\cap V|}{2\ell^*K}\ge \frac{\theta \alpha |V|}{6\ell^*K}\ge \frac{\theta \alpha n}{7\ell^*Kk_0}$$ outneighbours%
      \COMMENT{as $\theta |N^+_{G}(x)\cap V|/\ell^*\ge \theta \alpha |V|/3\ell^*\ge\theta \alpha |U|/3\ge \eps |U|$.}
in each subcluster $U(j)$ of $\cP'$ with $U(j)\subseteq V$ (again in the digraph $G$).
Together with (\ref{eqcountsub}) this implies that $x$ has at least
$$
\frac{\alpha k_0}{3}\cdot \frac{\theta \alpha n}{7\ell^*Kk_0}\cdot \frac{\ell^*}{2L}\ge \eps^{1/8} n
$$
outneighbours in $G$ which lie in $\bigcup_{U\in I^{\circ\circ}} U(j)$. (Here we multiply with $\ell^*/2L$ instead of $\ell^*/L$ since
counting outneighbours in $\bigcup_{U\in I^{\circ\circ}} U(j)$ means that we lose 4 clusters when considering $I^{\circ\circ}$ instead of $I$.)
Together with (b) and the fact that at most $\sqrt{\eps} n$ outneighbours of $x$ lie in full clusters of $\cP$ this shows that we can always find an
outneighbour $x^+\in \bigcup_{U\in I^{\circ\circ}} U(j)\setminus S$ of $x$ in $G'$ such that $x^+$ lies in a cluster of $\cP$ which is not full.
We can argue similarly to find $x^-$. This proves~(i).

\medskip

Our next aim is to prove~(ii). We will choose the sequence $CS_x$ for every exceptional vertex $x$ in turn.
So suppose that for some exceptional vertices we have already chosen such sequences and that we now wish to choose $CS_x$ for $x\in V_0$.
Call a cluster in $\cP$ \emph{crowded} if it is used at least $\eps^{1/4}m$ times by the sequences $CS_{x'}$
found so far. Thus the total number of crowded clusters in $\cP$ is at most $|V_0|\cdot \frac{6/\nu^3}{\eps^{1/4}m}\le \eps^{2/3}k\le \sqrt{\eps}|I|$.

Let $U_{\rm first}$ and $U_{\rm final}$ denote the first and the final cluster of $I$.
Let $I^*$ denote the interval obtained from~$I$ by deleting $U_{\rm final}$. As already observed in our proof of (i), the
length of $I$ is a multiple of $k_0$. So for every cluster $V$ in $\cP_0$ there are precisely ${\rm length}(I)/k_0=:D$ clusters in $I^*$ which are contained in~$V$.
Moreover, $U_{\rm first}$ and $U_{\rm final}$ are contained in the same cluster of $\cP_0$.
Consider the subgraph $R^*$ of $R$ induced by all the clusters in $I^*$. Note that $R^*$ contains an
edge $E^*$ from the final cluster $U^*_{\rm final}$ of $I^*$ to $U_{\rm first}$. (This follows since (CSys6)
implies that $R_0$ has an edge
from the cluster in $\cP_0$ containing $U^*_{\rm final}$ to the cluster in $\cP_0$ containing $U_{\rm final}$. But
$U_{\rm first}$ and $U_{\rm final}$ are contained in the same cluster of $\cP_0$. So (CSys6) now implies that $R$
(and thus also $R^*$) has an edge from $U^*_{\rm final}$ to $U_{\rm first}$.)
Let $C^*$ be the Hamilton cycle of $R^*$ which consists of $I^*$ together
with this edge~$E^*$. (CSys6) also implies that $R^*$ can be viewed as being obtained from $R_0$
by replacing each vertex of $R_0$ by $D$ vertices
and replacing each edge $WW'$ of $R_0$ by a complete bipartite graph between the two corresponding sets of $D$ vertices (where all the edges
are directed from the $D$ vertices corresponding to $W$ to the $D$ vertices corresponding to $W'$). In other words, $R^*$ is
a $D$-fold blow-up of $R_0$. Together with (CSys1)
and Lemma~\ref{expanderblowup} this in turn implies that $R^*$ is a robust $(\nu^3,2\tau)$-outexpander.
Moreover, it is easy to check that $\delta^0(R^*)\ge \alpha |R^*|$.

Apply Lemma~\ref{buildchord} to $R^*$ (with $\mathcal{V}'$ consisting of $U_{\rm first}$ and all the crowded clusters in $I^*$
and with $C^*$ playing the role of $C$) to find a chord sequence $CS_x$ from $U(x^+)$ and $U(x^-)^+$ in $R^*$ which contains most $3/\nu^3$ edges and
whose interior avoids $U_{\rm first}$ as well as all the crowded clusters. Since $x^-,x^+\in \bigcup_{U\in I^{\circ\circ}} U(j)$ by (i),
this implies that $CS_x$ uses only edges whose endclusters both lie in $I^\circ$. But $R^*\subseteq R$. 
Thus $CS_x$ is actually a chord sequence in $R$. 

We proceed in this way to choose $CS_x$ for all $x\in V_0$. We claim that in total no cluster in $\cP$
is used more than $4\eps^{1/4}m$ times by the sequences $CS_x$. To see this, note that no cluster in $\cP$
is used more than $\eps^{1/4}m+3/\nu^3\le 2\eps ^{1/4}m$ times by the interiors of the all $CS_x$ (as all these interiors avoid the crowded clusters).
Moreover, since by (i) no cluster in $\cP$ contains more than $\eps^{1/4}m$ vertices belonging to pairs $x^-,x^+$ guaranteed by (i), it follows that
no cluster in $\cP$ is used more than $2\eps^{1/4}m$ times as one of the two remaining clusters $U(x^+)^-$, $U(x^-)^+$ of $CS_x$ (for all $x\in V_0$
together). This completes the proof of~(ii).

\medskip

It remains to check (iii) and~(iv). To do this, we will first prove the following claim.

\medskip

\noindent \textbf{Claim 1.} \emph{For every edge $UU'$ of $R$ the pair $G'[U(j),U'(j)]$ is $\eps'$-regular of density at least $d/2$.
Moreover, for every edge $UW$ of $C$ the pair $G'[U(j),W(j)]$ is $[\eps',\ge d]$-superregular.}

\smallskip
To prove the first part of Claim~1, consider any edge $UU'$ of $R$. (CSys4) implies that $G[U(j),U'(j)]$ is $K\eps$-regular of density at least $d-2\eps$.
Together with (c) and Proposition~\ref{superslice}(i) this implies that $G'[U(j),U'(j)]$ is still $\eps'$-regular of density at least~$d-2\eps'\ge d/2$.

Now consider any edge $UW$ of $C$. Then (CSys5), (a) and Lemma~\ref{randompartition}(i) together imply that
$G[U(j),W(j)]$ is $[(\eps')^3,\ge d]$-superregular. 
As before, together with (c) and Proposition~\ref{superslice}(iii) this implies that $G'[U(j),W(j)]$
is still $[\eps',\ge d]$-superregular. This proves Claim~1 and thus in particular~(iv).

\medskip

In order to replace the edges of $CS_x$ to obtain $CS'_x$,
we again consider each exceptional vertex $x\in V_0$ in turn. By making $CS_x$ shorter if necessary, we may assume that every edge of $R$ occurs
at most once in each $CS_x$. Suppose that we have already chosen $CS'_{x'}$ for
some vertices $x'\in V_0$ and that we now wish to replace the edges of $CS_x$ in order to choose $CS'_x$. Moreover, suppose that we have already replaced
some edges $U''U'''$ of $CS_x$ and we now wish to replace the edge $UU'$ of $CS_x$.

For this, we let $U_*(j)$ denote the set of all those vertices $u\in U(j)$ that satisfy the following three conditions:
\begin{itemize}
\item There is no $y\in V_0$ such that $u$ lies in the pair $y^-,y^+$ guaranteed by~(i).
\item There exists no $x'\in V_0$ for which we have already defined $CS'_{x'}$ and for which $u$ is an endvertex of some edge in $CS'_{x'}$.
\item $u$ is not an endvertex of an edge of $G'[U''(j),U'''(j)]$ which was used to replace some edge $U''U'''$ of $CS_x$.
\end{itemize}
(i) implies that $U$ contains at most $\eps^{1/4}m$ vertices violating the first condition
and (ii) implies that $U$ contains at most $4\eps^{1/4}m$ vertices violating the second or third condition.
Thus
\begin{equation}\label{eq:U*j}
|U_*(j)|\ge |U(j)|-5\eps^{1/4}m=\frac{m}{K}-5\eps^{1/4}m\ge \frac{|U(j)|}{2}.
\end{equation}
Our aim is to replace $UU'$ in $CS_x$
by an edge in $G'[U_*(j),U'_*(j)]$. But by Claim~1, $G'[U(j),U'(j)]$ is $\eps'$-regular of density at least $d/2$.
Together with (\ref{eq:U*j}) this implies that $G'[U_*(j),U'_*(j)]$ contains an edge. We do this for every edge of
$CS_x$ in turn and let $CS'_x$ denote the sequence obtained in this way. Then these sequences $CS'_x$ are as required in (iii). 
\endproof

Based on the previous lemma, it is now straightforward to construct many edge-disjoint exceptional factors.

\begin{lemma}\label{exceptseq}
Suppose that
$$0<1/n\ll 1/k \ll \eps\ll d\ll \nu\ll \tau\ll \alpha,\theta\le 1; \ \ \eps \ll 1/K,1/L; \ \ Kr_0/m\ll d,$$
that $\ell^*/L, m/K\in\mathbb{N}$ and $L/\ell^*\ll 1$. Let $(G,\cP_0, R_0,C_0,\cP,R,C)$ be a consistent
$(\ell^*,k,m,\eps,d,\nu,\tau,\alpha,\theta)$-system with $|G|=n$. Let $\cP'$ be an $\eps$-uniform $K$-refinement of $\cP$.
Then there is a set $\mathcal{EF}$ of $r_0$ exceptional factors with parameters $(K,L)$
(with respect to $C$, $\cP'$) such that the original versions of all these $r_0$ exceptional factors
are pairwise edge-disjoint subdigraphs of~$G$.
\end{lemma}
\proof  Choose a new constant $\eps'$ with $\eps, Kr_0/m\ll \eps'\ll d$.
We will choose the $r_0$ exceptional factors for $\mathcal{EF}$ with respect to $C,\cP'$ one by one. So suppose that for some $0\le s<r_0$
we have already chosen exceptional factors $EF_1,\dots,EF_s$ with parameters $(K,L)$ such that the original
versions of these factors are pairwise edge-disjoint from each other. So our aim is to show that we can choose
$EF_{s+1}$ such that all these properties still hold.

Let $G'$ be the digraph obtained from $G$ by deleting all the edges contained in the original versions
of all the exceptional factors chosen so far. Note that
$$
d^\pm_G(x)-d^\pm_{G'}(x)\le KL r_0\le Ld m\le Ldn/k\le \eps n.
$$
for every vertex $x\in V_0$ and
$$
d^\pm_G(x)-d^\pm_{G'}(x)\le r_0\le (\eps')^3 m/K
$$
for every vertex $x\in V(G)\setminus V_0$. 
Thus $G'$ satisfies conditions (b) and (c) of Lemma~\ref{prelimfactor}.

Let $\cI$ denote the canonical interval partition of $C$ into $L$ intervals of equal length.
For each cluster $U$ of $\cP$ let $U(1),\dots,U(K)$
denote the subclusters in $\cP'$ which are contained in $U$. Consider any interval $I\in \cI$
and any $j$ with $1\le j\le K$. In order to construct $EF_{s+1}$, it suffices to show that there is a
complete exceptional path system $CEPS$ spanning the interval $I$ whose vertex set is $\bigcup_{U\in I} U(j)$
and such that $CEPS^{\rm orig}\subseteq G'$.
To do this, we apply Lemma~\ref{prelimfactor} to find pairs $x^-$, $x^+$, chord sequences $CS_x$ in $R$ and sequences $CS'_x$
(for every vertex $x\in V_0$) satisfying (i)--(iii).

Now let $CES$ be the union of all the edges $x^-x^+$ over all exceptional vertices $x$. 
(Note that $x^-x^+$ is an exceptional edge in $(G')^{\rm basic}$ and it is irrelevant whether it lies in $G'$ or not.)
It remains to enlarge $CES$ into a complete exceptional path system $CEPS$.
For this, we first add all edges in the sequences $CS'_x$ guaranteed by (iii) (for all vertices $x \in V_0$).
Together with $CES$, this gives a matching $M$ which meets
every cluster in $\cP$ in at most $5\eps^{1/4}m$ vertices.

Suppose that $U,W\in \cP$ are consecutive clusters on $I$.
Let $U^1(j)$ be the set of all those vertices in $U(j)$ which are not the initial vertex of an edge in $M$
and let $W^2(j)$ be the set of all those vertices in $W(j)$ which are not the final vertex of an edge in $M$.
Then Proposition~\ref{balanced} implies that $|U^1(j)|=|W^2(j)|$. Also note that $|U(j)\setminus U^1(j)| \le 5\eps^{1/4}m\le \eps' |U(j)|$.
Together with Lemma~\ref{prelimfactor}(iv) and Proposition~\ref{superslice}(iii) this implies that $G'[U^1(j),W^2(j)]$ is still
$[2\sqrt{\eps'},\ge d]$-superregular and thus it contains a perfect matching $M_{UW}$ by Proposition~\ref{perfmatch}.

The union $CEPS'$ of $M$ with these matchings $M_{UW}$ for all pairs $U,W$ of consecutive clusters on $I$
contains paths satisfying (CEPS1) and (CEPS2), but in addition $CEPS'$ might contain cycles. (Note each such cycle will contain at least
one edge from $M$.) So our aim is to apply Lemma~\ref{mergecycles} in order to transform  
$CEPS'$ into a path system. To do this, we let $U_{\rm first}\in \cP$ denote the first cluster in $I$ and let $U_{\rm final}\in \cP$
denote the last cluster in $I$. Let $C_I$ be the cycle obtained from $I$ by identifying $U_{\rm first}$ and $U_{\rm final}$.
Let $G'_{Ij}$ be the digraph obtained from $\bigcup_{UW\in E(I)} G'[U(j),W(j)]$
by identifying each vertex in $U_{\rm first}(j)$ with a different vertex of $U_{\rm final}(j)$. Let $CEPS''$ be obtained from
$CEPS'$ by identifying these vertices. Note that $G'_{Ij}$ can be viewed
as a blow-up of $C_I$ in which every edge
of $C_I$ corresponds to an $[\eps',\ge d]$-superregular pair (the latter holds by Lemma~\ref{prelimfactor}(iv)). Moreover, $CEPS''$
is a $1$-regular digraph on $V(G'_{Ij})$ which has the property that every cycle $D$ in $CEPS''$
contains at least one edge in some matching $M_{UW}$ for some
pair $U,W$ of consecutive clusters on $C_I$. (To see the latter,
recall that $M$ was a matching which avoids both $U_{\rm first}$ and $U_{\rm final}$. So every vertex in $V(G'_{Ij})$
is incident to at most one edge in $M$.) So $D$ contains a vertex in $U^1(j)$.  Thus we can apply Lemma~\ref{mergecycles}
with $G'_{Ij}$, $C_I$, $U^1(j)$, $U^2(j)$ and $E(C_I)$ playing the roles of $G$, $C$, $V^1_i$, $V^2_i$ and $J$.
This shows that we can replace each matching $M_{UW}$ by a different matching in $G'[U^1(j),W^2(j)]$ in order to transform $CEPS''$
into a Hamilton cycle of $G'_{Ij}$. But this Hamilton cycle corresponds to a complete exceptional path system $CEPS$
spanning the interval $I$ whose vertex set is $\bigcup_{U\in I} U(j)$ and such that $CEPS^{\rm orig}\subseteq G'$. (Note that $CEPS$ still
contains $M$ and thus $CES$.)

We do this for each $j=1,\dots,K$ and each interval $I$ in the canonical interval partition $\mathcal{I}$ of $C$ in
turn. Then the union $EF_{s+1}$ of all these complete exceptional path systems is an exceptional factor with the
desired properties. This completes the proof of the existence of $\mathcal{EF}$.
\endproof


\section{The preprocessing step}\label{seccyclebreak}

Let $G'$ be the leftover of $G$ obtained from an application of Theorem~\ref{approxdecomp}. So $G'$ is regular and very sparse.
Roughly speaking, the aim of this section is to find a sparse `preprocessing graph' $PG$ so that $G' \cup PG$
has a set of edge-disjoint Hamilton cycles covering all edges of $G'$.
We need to do this because the edges at the exceptional vertices in the leftover $G'$ might be distributed very badly.
So the idea is to choose $PG$ at the beginning of the proof of Theorem~\ref{decomp}, to apply Theorem~\ref{approxdecomp}
to $G\setminus E(PG)$ and then to cover the leftover $G'$ by edge-disjoint Hamilton cycles in $G' \cup PG$.
In this way we replace $G'$ by the resulting leftover $PG'$ of $PG$. Moreover, the goal is that $PG'$
will have no edges incident to $V_0$. So $PG'$ will not be regular. But the chord absorber, which we will use in Section~\ref{sec:chordabsorb} to absorb the edges
of $PG'$, will contain additional edges at the exceptional vertices to compensate for this, and these edges will be nicely distributed.
(More precisely, this chord absorber will contain some additional exceptional factors.)

In order to find the Hamilton cycles covering $G'$, we decompose $G'$ into $1$-factors $H$, split each $1$-factor $H$ into small path systems $H_i$ and extend
$H_i$ into a Hamilton cycle using the edges of $PG$. As we shall see, when finding these Hamilton cycles, the most difficult part is to 
find suitable edges joining $V_0$ to the (non-exceptional) clusters. For this step we use the complete exceptional sequences defined in
Section~\ref{sec:ces}.

As we shall explain in Section~\ref{sec:prepro}, the preprocessing graph $PG$ will consist (mainly) of the edge-disjoint union of
exceptional factors and a `path system extender' $PE$, 
which we define and use in Section~\ref{sec:seqremove}. Once we have found suitable edges joining $V_0$ to the clusters, the
path system extender will be used to extend these into Hamilton cycles.


\subsection{Cycle breaking}
Suppose that $H$ is a $1$-factor of the leftover $G'$. 
As discussed above, our aim is to split $H$ into small path systems $H_i$ which we then extend into Hamilton cycles $C_i$ of $G$.
Note that $H$ contains exactly one exceptional cover $H^{\rm exc}$ and each $C_i$ must contain exactly one as well.
The most obvious way to achieve the latter might be to let $H_1$ consist of $H^{\rm exc}$, and add a new exceptional cover
(from the preprocessing graph $PG$) to each of the other 
$H_i$ when forming the $C_i$. However, $H^{\rm exc}$ might contain cycles, in which case we cannot extend $H_1=H^{\rm exc}$ to a Hamilton cycle.
So when splitting $H$ into the $H_i$, we also need to split $H^{\rm exc}$ to ensure that such cycles are `broken'.
As an intermediate step towards extending the $H_i$ into Hamilton cycles, we then add edges from (the original version of) an exceptional factor $CB$
to extend each $H_i$ into a path system $Q_i$ which contains exactly one exceptional cover.
In other words, one can split $H \cup CB^{\rm orig}$  into small path systems $Q_i$ such that each of them contains precisely one
exceptional cover. 

Let $(G,\cP,R,C)$ be a $(k,m,\eps,d)$-scheme and assume that $m/50\in\mathbb{N}$.
Consider an $\eps$-uniform $50$-refinement $\cP'$ of $\cP$ (recall again that these were defined before Lemma~\ref{randompartition}).
For each cluster $V_i$ of $\cP$, let $V_i(1),\dots,V_i(50)$ denote all those clusters in $\cP'$ which
are contained in $V_i$. Suppose $J \subseteq \{1,\dots,50\}$. Generalizing the notion of styles defined at the end of Section~\ref{sec:ces}, 
we say that a vertex $x\in V(G)\setminus V_0$ has \emph{style $J$} if $x\in \bigcup_{i=1}^{k}\bigcup_{j \in J} V_i(j)$.
We then say that a set $E$ of edges of $G^{\rm basic}$ or of $G$ has \emph{style $J$} if every endvertex $x$ of an edge in $E$ with $x\notin V_0$
has style $J$. If $E$ has style $J$, we say that $E$ has \emph{style size $|J|$} (with respect to $\cP'$).
Note that in this definition $J$ need not have minimum size, i.e.~if $E$ has style size $t$, then it also has style size $t+1$.

\begin{lemma}\label{cyclebreak}
Suppose that $0<1/n\ll 1/k\ll \eps\ll d\ll 1/s \ll 1$ and that $m/50, 50k/(s-1)\in\mathbb{N}$.
Let $(G,\cP,R,C)$ be a $(k,m,\eps,d)$-scheme with $|G|=n$.
Let $\cP'$ be an $\eps$-uniform $50$-refinement of $\cP$.
Let $H$ be a $1$-factor of $G$. Let $CB$ be an exceptional factor with parameters $(50,(s-1)/50)$ with respect to $C,\cP'$. Suppose that $H$
and $CB^{\rm orig}$ are edge-disjoint. Then the edges of $H \cup CB^{\rm orig}$ can be decomposed into edge-disjoint path systems $Q_1,\dots,Q_{s}$ so that
\begin{itemize}
\item[(i)] each $Q_i$ contains precisely one exceptional cover (and no other original exceptional edges);
\item[(ii)] each $Q_i$ has style size 5 (with respect to $\cP'$);
\item[(iii)] $|E(Q_i)| \le 1230n/s$.
\end{itemize}
\end{lemma}
\proof
Let $H^{\rm exc}$ denote the set of original exceptional edges of $H$. Note that each edge in $H^{\rm exc}$
has precisely one of $50$ styles. So we can partition $H^{\rm exc}$ into $H_1,\dots,H_{50}$ such that each $H_i$ has
style~$i$. Now we split each $H_i$ into $H_i^-$ and $H_i^+$ by placing one edge from 
each cycle of $H_i$ into $H_i^-$ and all the others into $H_i^+$.
Enumerate the path systems obtained in this way as $H^{\rm exc}_1,\dots,H^{\rm exc}_{100}$.
Thus $H^{\rm exc}$ is the union of all these path systems.
Denote the complete exceptional path systems of $CB$ by $CEPS_i$, for $i=1,\dots,s-1$, and
let $CES_i$ denote the complete exceptional sequence contained in $CEPS_i$.
Without loss of generality, we may assume that $CEPS_1,\dots,CEPS_{100}$ are complete exceptional path systems
of $CB$ so that exactly 20 of them have style $j$ for $j=1,\dots,5$
(as there are many more with these styles in $CB$). Moreover, by relabeling the $CEPS_i$ if necessary, we may assume
that $H^{\rm exc}_i$ and $CEPS_i$ have different styles for all $i=1,\dots,100$.

For each $i=1,\dots,100$, let $V^+_{0,i}\subseteq V_0$ be the set of exceptional vertices $x$ so that $x$ is the initial vertex of an edge in $H_i^{\rm exc}$. 
Similarly, let $V^-_{0,i}\subseteq V_0$ be the set of exceptional vertices $x$ so that $x$ is the final vertex of an edge in $H_i^{\rm exc}$. 
Let $CES'_i$ be the set of all those (original exceptional) edges $xy$ in $CES^{\rm orig}_i$ for which
$x\in V_0\setminus V^+_{0,i}$ as well as all those (original exceptional) edges $yx$ in $CES^{\rm orig}_i$ for which
$x\in V_0\setminus V^-_{0,i}$. Let $Q_i:=H^{\rm exc}_i \cup CES'_i\cup (CEPS_i\setminus CES_i)$. Note that $Q_i$ contains the exceptional cover
$H^{\rm exc}_i \cup CES'_i$. Moreover, since  $H^{\rm exc}_i$ and $CEPS_i$ have different styles it follows that $Q_i$ is a path system,
i.e. it does not contain any cycles. Furthermore, $Q_i$ has style size $2$
and $|E(Q_i)|\le |V_0|+n/(s-1)\le 2n/s$.

Now let $Q_{101}$ consist of those edges of $CES^{\rm orig}_1,\dots,CES^{\rm orig}_{100}$ which are not contained in
any of $Q_1,\dots, Q_{100}$. 
Using the fact that $H^{\rm exc}=H^{\rm exc}_1 \cup \dots \cup H^{\rm exc}_{100}$ is an exceptional cover, 
it is easy to see that $Q_{101}$ also forms an exceptional 
cover. Moreover, as remarked after the definition of an exceptional factor,  
$CES_1 \cup \dots \cup CES_{100}$ is a matching. 
In particular, every vertex in $V(G)\setminus V_0$ is incident to at most one edge in $CES^{\rm orig}_1\cup \dots\cup CES^{\rm orig}_{100}$.
Together with the fact that $Q_{101}$ is an exceptional cover, it follows that $Q_{101}$ is a path system.
Also, $Q_{101}$ has style size $5$ and $|E(Q_{101})|= 2|V_0|\le 2\eps n\le n/s$.

Note that each edge of $H\setminus H^{\rm exc}$ has%
   \COMMENT{By definition, $H^{\rm exc}$ is already a set of edges. So can't write $H\setminus E(H^{\rm exc})$}
at least one of $\binom{50}{2}=1225$ styles $ij$ with $i\neq j$.
(Recall that if an edge has style $i$ then it also has style $ij$ for any $j$.) Partition $H\setminus H^{\rm exc}$ into
$1225$ sets $H_{i,j}$ such that all edges in $H_{i,j}$ have style $ij$.
Greedily split each set $H_{i,j}$ further into $(s-1)/1226$ sets such that each of them has size
at most $\lceil 1226n/(s-1)\rceil$ and none of them contains a cycle.%
    \COMMENT{This can be done as follows. First enumerate the edges in $H_{i,j}$ in such a way that the
edges of every cycle in $H_{i,j}$ form a segment in this enumeration. Now add the $r$th edge
into the set whose index is congruent to $r$ modulo $(s-1)/1226$ (for each $r=1,\dots, e(H_{i,j})$).}
Let $s':=1225(s-1)/1226$ and
let $H^*_1,\dots,H^*_{s'}$ denote the resulting sets of edges.

\medskip

\noindent
\textbf{Claim.} \emph{We may assume that  $CEPS_{101},\dots,CEPS_{s-1}$ are enumerated in such a way that for all
$t=1,\dots,s'$ the following property holds:
if $H^*_t$ has style $ij$ and $CEPS_{100+t}$ has style $j'$ then $j' \notin \{i,j\}$.}

\smallskip

\noindent 
To prove the claim, we consider the auxiliary bipartite graph $Y$ with vertex classes $A$ and $B$,
where $A$ consists of all the $H^*_t$ and $B$ consists of $CEPS_{101},\dots,CEPS_{s-1}$. (So $|A|=s'$ and $|B|=(s-1)-100>s'$.)
Suppose that $H^*_t$ has style $ij$ and $CEPS_\ell$ has style $j'$, where
$j' \notin \{i,j\}$. Then $H^*_t$ and $CEPS_\ell$ are connected by an edge in $Y$.
Note that every $CEPS_\ell$ has degree at least $|A|-49 \cdot (s-1)/1226 \ge |A|/2$ (since if $j'=i$ then there are
only $49$ possibilities for $j$ and since only $(s-1)/1226$ of the $H^*_t$ have the same style).
Similarly every $H^*_t$ has degree at least $(s-1-100)-2\cdot (s-1)/50 \ge |B|/2$. So
$Y$ has a matching covering $A$. (Indeed, this follows from Hall's theorem: Consider any $A'\subseteq A$. If $|A'|>|A|/2$ then
$N_Y(A')=B$ and if $|A'|\le |A|/2$ then $|N_Y(A')|\ge |B|/2\ge |A|/2\ge |A'|$.) This completes the proof of the claim.

\medskip

For $i=102,\dots,101+s'$, let $Q_i:=H^*_{i-101}\cup CEPS^{\rm orig}_{i-1}$. Then the claim implies that $Q_i$ is a path system.
Moreover, $Q_i$ has style size~$3$ and%
    \COMMENT{as each exceptional edge in $CEPS$ turns into 2 original exceptional edges in $CEPS^{\rm orig}$ we need
the $\eps n$ term}
$|E(Q_i)|\le \lceil 1226n/(s-1)\rceil +n/(s-1)+\eps n\le 1230n/s$. Finally, for
$i=102+s',\dots,s$ we let $Q_i:=CEPS^{\rm orig}_{i-1}$. Then $Q_1,\dots,Q_s$ have the desired properties.
\endproof


\subsection{Extending path systems into Hamilton cycles}\label{sec:seqremove}

Suppose that $(G,\cP,R,C)$ is a $(k,m,\eps,d)$-scheme. Recall that $V_0$ denotes the exceptional set in $\cP$.
The purpose of this subsection is to define and use the `path system extender' $PE$.  
To motivate its purpose, let $G'$ be the leftover of $G$ obtained by an application of Theorem~\ref{approxdecomp}
and let $H$ be a $1$-factor in a $1$-factorization of $G'$.
Recall that Lemma~\ref{cyclebreak} assigns edges of $H$ to edge-disjoint path systems $Q_i$.
$PE$ will be used to extend each path system $Q_i$ into a Hamilton cycle. 
Altogether this means that we will find edge-disjoint Hamilton cycles
in the union of $G'\cup PE$ with some exceptional factors $CB_i$ which cover both the edges of $G'$ and the edges of $CB_i$. Since
$PE$ will be a spanning subdigraph of $G-V_0$, this in turn implies that we have replaced $G'$ by a digraph (namely the digraph obtained from
the leftover of $PE$ by adding $V_0$)
in which every exceptional vertex is isolated. 
 
A \emph{path system extender $PE$ (for $C,R$) with parameters $(\eps,d,d',\zeta)$} is a spanning subgraph of $G-V_0$
consisting of an edge-disjoint union of two graphs $\cB(C)_{PE}$ and $\cB(R)_{PE}$ on $V(G)\setminus V_0$
which are defined as follows:
\begin{itemize}
\item[(PE1)] $\cB(C)_{PE}$ is a blow-up of $C$ in which every edge $UW$ of $C$
corresponds to an $(\eps,d',\zeta d',2d'/d)$-superregular pair $\cB(C)_{PE}[U,W]$.
\item[(PE2)] $\cB(R)_{PE}$ is a blow-up of $R$ in which every edge $UW$
of $R$ corresponds to an $(\eps,d'/k,2d'/dk)$-regular pair $\cB(R)_{PE}[U,W]$.
   \COMMENT{Need to divide by $k$ since otherwise the degrees in $PE$ would be too large.}
\end{itemize}
Note that the path system extender $PE$ is not necessarily a regular digraph.
(Reg3) and the fact that $\Delta(R) \le 2k$ imply that
\begin{equation} \label{DeltaPE}
\Delta(PE) \le 2\cdot \frac{2d'm}{d}+ \Delta(R)\cdot \frac{2d'm}{dk} \le \frac{8d'm}{d}.
\end{equation}
In order to cover the edges of the leftover of the path system extender $PE$ with Hamilton cycles in the chord-absorbing step
(Section~\ref{sec:chordabsorb}), we will need that this maximum degree is small 
compared to $\eps m$ (so $PE$ is a rather sparse graph). This is what forces us to use the above more technical notion of regularity 
when defining $PE$, rather than the usual $\eps$-regularity.

\begin{lemma}\label{findPE}
Suppose that $0<1/n\ll d'\ll 1/k\ll \eps\ll d\ll \zeta\le 1/2$.
Let $(G,\cP,R,C)$ be a $(k,m,\eps,d)$-scheme with $|G|=n$. Then $G$ contains a path system extender for $C,R$ having
parameters $(\eps^{1/25},d,d',\zeta)$.
\end{lemma}
\proof
Recall that by (Sch3) every edge $UW$ of $C$ corresponds to an $[\eps,d_{UW}]$-superregular pair $G[U,W]$
for some $d_{UW}\ge d$. Thus we can apply Lemma~\ref{randomregslice}(ii) to each such $G[U,W]$ to obtain a
blow-up $\cB(C)_{PE}$ of $C$ in which every edge of $C$ corresponds to an
$(\eps^{1/12},d',d'/2,3d'/2d_{UW})$-superregular pair (which is therefore also $(\eps^{1/25},d',\zeta d',2d'/d)$-superregular).
Let $G^*$ be the digraph obtained from $G$ by deleting all the edges in $\cB(C)_{PE}$.

Now consider any edge $UW$ of $R$ and recall from (Sch2) that $G[U,W]$ is $\eps$-regular of density $d_{UW}\ge d-\eps$.
Since $G^*[U,W]$ is obtained from $G[U,W]$ by deleting at most $2d'm/d\le \eps m$ edges at every vertex,
Proposition~\ref{superslice}(i) with $d':=\eps$ implies that $G^*[U,W]$ is still $2\sqrt{\eps}$-regular of density at least $3d_{UW}/4$.
Thus we can apply Lemma~\ref{randomregslice}(i) to find a
$((4\eps)^{1/24},d'/k,2d'/d_{UW}k)$-regular spanning subgraph $G'[U,W]$ of $G^*[U,W]$
(which is therefore also $(\eps^{1/25},d'/k,2d'/dk)$-regular).
Let $\cB(R)_{PE}$ be the union of all the $G'[U,W]$ over all edges $UW$ of $R$. Then $\cB(C)_{PE}\cup \cB(R)_{PE}$
is a path system extender for $C,R$ with parameters $(\eps^{1/25},d,d',\zeta)$.
\endproof

The following lemma implies that we can use the edges of a path system extender to extend the path systems $Q_i$ obtained from Lemma~\ref{cyclebreak} 
into Hamilton cycles.

\begin{lemma}\label{findham}
Suppose that $0<1/n\ll d'\ll 1/k\ll \eps\ll 1/\ell^*\ll d\ll \nu\ll \tau\ll \alpha, \theta\le 1$,
that $d\ll \zeta\le 1/2$ and that $m/50\in \mathbb{N}$.
Let $(G,\cP_0, R_0,C_0,\cP,R,C)$ be a consistent $(\ell^*,k,m,\eps,d,\nu,\tau,\alpha,\theta)$-system
with $|G|=n$. Let $\cP'$ be a $(d')^2$-uniform $50$-refinement of $\cP$.
Let $PE$ be a path system extender with parameters $(\eps,d,d',\zeta)$ for $C$, $R$.
Let $s:=10^7/\nu^2$ and suppose that $Q$ is a path system in $G$ such that
\begin{itemize}
\item $Q$ and $PE$ are edge-disjoint;
\item $Q$ contains precisely one exceptional cover and no other original exceptional edges;
\item $Q$ has style size 5 (with respect to $\cP'$);
\item $|E(Q)| \le 1230n/s$.
\end{itemize}
Then $G$ contains a Hamilton cycle $D$ such that $Q\subseteq D\subseteq PE \cup Q$. 
\end{lemma}

For the proof of Lemma~\ref{findham}, we add additional edges to $Q$ to obtain a `locally balanced' path system.
(To ensure the existence of these additional edges we need to work with a consistent system rather than the simpler notion of a scheme.)
This path system can then be extended into a $1$-factor using matchings between consecutive clusters of $C$.
Finally, we apply Lemma~\ref{mergecycles0} to transform this $1$-factor into a Hamilton cycle.

\proof
Recall that $Q^{\rm basic}$ is obtained from $Q$ by replacing each path of the form $x^-xx^+$ (where $x\in V_0$) by the exceptional
edge $x^-x^+\in E(G^{\rm basic})$.  
For each edge $e=ab$ of $Q^{\rm basic}$ in turn, our aim is to apply Lemma~\ref{buildchord} to find a chord sequence $CS(U(b),U(a)^+)$
in $R$ which consists of at most $3/\nu$ edges. (Here $U(b)$ denotes the cluster in $\cP$ containing $b$ and
$U(a)^+$ denotes the successor on $C$ of the cluster in $\cP$ containing $a$.)
We say that a cluster in $\cP$ is \emph{full} if it is visited at least $m/110$ times by the interiors of
the chord sequences chosen so far. (Recall that we disregard the first cluster $U(b)^-$ and the final cluster
$U(a)^+$ of $CS(U(b),U(a)^+)$ when considering its interior.)
Then the number of full clusters is at most
$$\frac{2|Q^{\rm basic}|(3/\nu)}{m/110}\le \frac{811800 n}{\nu s m} \le \frac{\nu k}{4}.
$$
After each application of Lemma~\ref{buildchord}, let  $\mathcal{V}'\subseteq V(R)$ be the set of all those clusters
which are now full. So we can apply Lemma~\ref{buildchord} to $R$ and $\mathcal{V}'$ as above to find a chord sequence
$CS(U(b),U(a)^+)$ whose interior avoids the full clusters. Thus altogether the interiors of all the
$CS(U(b),U(a)^+)$ (for all edges $ab$ of $Q^{\rm basic}$) visit each cluster in $\cP$ at most $m/110+3/\nu\le m/100$ times.

Note that since $Q$ has style size 5, every cluster in $\cP$ meets $Q^{\rm basic}$ in at most $5m/50=m/10$ vertices.
Thus every cluster plays the role of $U(b)^-$ or $U(a)^+$ for at most $2m/10$ edges $ab$ of $Q^{\rm basic}$. This implies that
the union $S(Q)$ of all the chord sequences $CS(U(b),U(a)^+)$ (over all edges $ab$ of $Q^{\rm basic}$) visits each cluster
at most 
\begin{equation} \label{visits}
m/100+2m/10=21m/100 
\end{equation}
times (where we count the edges in $S(Q)$ with multiplicities).

Now for each edge $E$ of $R$, let $s_E$ be the number of times that $E$ occurs in $S(Q)$. Note that if $E$ occurs in $CS(U(b),U(a)^+)$
then at least one of the endclusters of $E$ lies in the interior of $CS(U(b),U(a)^+)$. 
In particular, suppose that $E=UW$  occurs as an initial edge in $CS(U(b),U(a)^+)$. Then $W$ lies in the interior of the sequence.
Since altogether the interiors of all the $CS(U(b),U(a)^+)$ visit $W$ at most $m/100$ times,
this implies that altogether $E$ can occur at most $m/100$ times as an initial edge of some $CS(U(b),U(a)^+)$.
Similarly, suppose that $E=UW$  occurs in $CS(U(b),U(a)^+)$, but not as an initial edge.
Then $U$ lies in the interior of the sequence and altogether $E$ can occur at most $m/100$ times in this way.
It follows that
$s_E \le 2\cdot m/100=m/50$. 

Let $J\subseteq \{1,\dots, 50\}$ be a set of size $5$ so that $Q$ has style $J$.
Without loss of generality, we may assume that 
$J=\{1,\dots,5\}$. Let $B_E$ be the bipartite subgraph of $\cB(R)_{PE}$ which corresponds to $E$.
Let $B'_E$ be the induced bipartite subgraph of $B_E$ of style $\{6,\dots,29\}$. (So if $E=UW$
then $B'_E$ is the subgraph of $B_E$ induced by $U(6)\cup\dots\cup U(29)$ and $W(6)\cup\dots\cup W(29)$, where
$U(1),\dots,U(50)$ are the clusters in $\cP'$ contained in $U$.)

For each edge $E$ of $R$ in turn, we choose a matching $M_E$ of size $s_E$
in $B'_E$ such that $M_E$ avoids all matchings chosen previously.
To see that this can be done, recall from (PE2) that $B_E$ is $(\eps,d'/k,2d'/dk)$-regular.
Since the size of the vertex classes of $B'_E$ is precisely $24/50$ times the size of the vertex classes of $B_E$,
this implies that $B'_E$ is still $(3\eps,d'/k,6d'/dk)$-regular.%
    \COMMENT{Need to replace $2d'/dk$ by $6d'/dk$ since in the def it asks for max degree at most $cn$
where $n$ is the size of the vertex classes. So originally we had max degree at most $2d'm/dk$.
But now our vertex classes are smaller...}
Thus Lemma~\ref{PEmatching}(i) implies that $B'_E$ contains
a matching $M'_E$ of size $(1-3\eps)\cdot 24m/50\ge 23m/50$. 
But (\ref{visits}) implies that at most $2\cdot 21m/100=21m/50$ edges
of $M'_E$ are incident to an edge of a previously chosen matching
(the extra factor $2$ comes from the fact that there is a contribution from both endclusters of $E$).
So we can indeed find the required matchings $M_E$.
 
Let $Q_*$ consist of the edges of $Q^{\rm basic}$ together with the edges in the matchings $M_E$
(for all edges $E$ of $R$).
So $Q_*$ is a path system whose style is $\{1,\dots,29\}$.

Let $Q^{in}$ denote the set of final vertices of the edges in $Q_*$ and let 
$Q^{out}$ denote the set of their initial vertices.
For each cluster $U$ in $\cP$, let $U^2:=U \setminus Q^{in}$ and $U^1:=U \setminus Q^{out}$.
Consider any edge $UW$ on $C$. Then Proposition~\ref{balanced} implies that $|U^1|=|W^2|$.
Let $B^*_{UW}$ be the bipartite subgraph of $\cB(C)_{PE}[U,W]$ induced by $U^1$
and $W^2$. 
Recall from (PE1) that $\cB(C)_{PE}[U,W]$ is $(\eps,d',\zeta d',2d'/d)$-superregular.
Together with the facts that $B^*_{UW}$ contains all vertices of $U\cup W$ of style $30,\dots,50$ and
that $\cP'$ is an $(d')^2$-uniform refinement%
    \COMMENT{Since $d'\ll \eps$ an $\eps$-uniform refinement would not be enough here.}
of $\cP$ this implies that $B^*_{UW}$
is still $(3\eps,d',\zeta d'/3,6d'/d)$-superregular. 
So Lemma~\ref{PEmatching}(ii) implies that $B^*_{UW}$ contains a perfect matching~$M_{UW}$.
The union of all the $M_{UW}$ (over all edges $UW$ of $C$) together with all the edges of $Q_*$ forms a $1$-factor $F$ of $Q^{\rm basic}\cup PE$.

Since $Q_*$ is a path system, each cycle of $F$ has a vertex in $U^1$
for some cluster $U$ in $\cP$.
Now apply Lemma~\ref{mergecycles0} with $\cB(C)_{PE}$, $F$, $C$, $E(C)$, $U^1$
and $U^2$ playing the roles of $G$, $F$, $C$, $J$, $V^1_i$ and $V^2_i$
in order to modify $F$ into a Hamilton cycle $D'$ of $G^{\rm basic}$.
(Note that $U^1 \cap U^2 \neq \emptyset$ as they both contain all vertices of style $\{ 30, \dots, 50 \}$.)
Then Observation~\ref{basicobs} implies that $D:=(D')^{\rm orig}$ is a Hamilton cycle of $G=G^{\rm orig}$, as required.
\endproof

We obtain the following lemma by repeated applications of Lemma~\ref{cyclebreak} and~\ref{findham}.

\begin{lemma}\label{replacesequence}
Suppose that $0<1/n\ll r/m \ll d'\ll 1/k\ll \eps\ll 1/\ell^*\ll d\ll \nu\ll \tau\ll \alpha, \theta\le 1$
and that $d\ll \zeta\le 1/2$. Let $s:=10^7/\nu^2$. Suppose that $m/50, 50k/(s-1)\in\mathbb{N}$.
Let $(G,\cP_0, R_0,C_0,\cP,R,C)$ be a consistent $(\ell^*,k,m,\eps,d,\nu,\tau,\alpha,\theta)$-system
with $|G|=n$. Let $\cP'$ be a $(d')^2$-uniform $50$-refinement of $\cP$.
Let $H$ be an $r$-factor of $G$.  Let $CB_1,\dots,CB_{r}$ be $r$ exceptional factors with
parameters $(50,(s-1)/50)$ with respect to $C, \cP'$. Let $PE$ be a path system extender with parameters $(\eps,d,d',\zeta)$
for $C,R$. Suppose that $H$, $PE$ and the original versions of $CB_1,\dots,CB_r$ are
pairwise edge-disjoint.%
   \COMMENT{Need to say that the original versions are pairwise edge-disjoint as eg $H^{\rm basic}$ and $CB_1$ could be edge-disjoint but the original
versions of them could use the same edge at some exceptional vertex.}
Then~$G$ contains edge-disjoint Hamilton cycles $C_1,\dots,C_{sr}$ such that
\begin{itemize}
\item[(i)] altogether $C_1,\dots,C_{sr}$ contain all edges of $H  \cup CB^{\rm orig}_1 \cup \dots \cup CB^{\rm orig}_{r}$;
\item[(ii)] each $C_i$ lies in $PE \cup H \cup CB^{\rm orig}_1 \cup \dots \cup CB^{\rm orig}_{r}$.
\end{itemize}
\end{lemma}
\proof 
Consider a $1$-factorization $F_1,\dots,F_{r}$ of $H$ (which exists by Proposition~\ref{1factor}).
So each $F_j$ contains precisely one exceptional cover (and no other original exceptional edges).
For each $j=1,\dots, r$ apply Lemma~\ref{cyclebreak} with%
    \COMMENT{In Lemma~\ref{cyclebreak} $\cP'$ is an $\eps$-uniform refinement. But that's not a problem since any $(d')^2$-uniform
refinement is an $\eps$-uniform refinement as well (since $d'\ll \eps$).}
$F_j$ and $CB_j$ playing the roles of $H$ and $CB$ to obtain path systems
$Q_{j,1},\dots,Q_{j,s}$ as described there. Relabel all these path systems as 
$Q_1,\dots,Q_{sr}$. 
Note that the $Q_i$ form a decomposition of the edges of $H \cup CB^{\rm orig}_1 \cup \dots \cup CB^{\rm orig}_{r}$.

Apply Lemma~\ref{findham} to obtain a Hamilton cycle $C_1$ in $G$ with $Q_1\subseteq C_1\subseteq PE \cup Q_1$. 
Repeat this for each of $Q_2,\dots,Q_{sr}$ in turn,
with the subdigraph $PE'$ of $PE$ obtained by deleting all the edges in the Hamilton cycles found so far playing the
role of $PE$. Note that at each stage we have removed at most $sr= 10^7 r /\nu^2 \le \eps^2 d'm/k$ outedges and at most
$sr\le \eps^2 d'm/k$ inedges at each vertex of $PE$.
So using Proposition~\ref{superslice5} it is easy to check that $PE'$ is still a path system extender with
parameters $(2\eps,d,d',\zeta/2)$. Thus we can indeed apply
Lemma~\ref{findham} in each step.
\endproof


\subsection{The preprocessing graph} \label{sec:prepro}
Let $G$ be the digraph given in Theorem~\ref{decomp}.
As described at the beginning of Section~\ref{sec:seqremove}, the purpose of Lemma~\ref{replacesequence} was to
replace the leftover $H$ of an almost decomposition of $G$ by a digraph $H'$
on $V(G)\setminus V_0$, i.e.~by a digraph in which every exceptional vertex can be thought of as being isolated. This digraph $H'$ is obtained from $PE$ by deleting all the edges in
the Hamilton cycles guaranteed by Lemma~\ref{replacesequence}. But this means that $H'$ will not be regular since $PE$ was not regular.
However, for later purposes we want to assume that this leftover $H'$ is regular.
So we will need to add another digraph $PG^\diamond$ to $H'$ whose degree sequence complements that of $PE$. Then the union of
$H'$ and $PG^\diamond$ will be regular. So the preprocessing graph defined below contains all the digraphs which we need in order
to replace $H$ by a regular digraph on $V(G) \setminus V_0$ not containing any exceptional edges.

Suppose that $(G,\cP,R,C)$ is a $(k,m,\eps,d)$-scheme and that $m/50, 50k/(s-1)\in\mathbb{N}$. Let $\cP'$ be an $\eps$-uniform 50-refinement of $\cP$.
A \emph{preprocessing graph} $PG$ with parameters $(s,\eps,d,r',r'',r,\zeta)$ (with respect to $C,R,\cP'$)
is the edge-disjoint union of two graphs $PG^*$ and $PG^\diamond$ satisfying the following conditions:
\begin{itemize}
\item[(PG1)] $PG^*$ is the edge-disjoint union of a path system extender $PG_{PE}$ with parameters
$(\eps,d,r'/m,\zeta)$ (for $C, R$) and of $r$ exceptional factors $CB_1,\dots,CB_r$ with parameters $(50, (s-1)/50)$
(with respect to $C,\cP'$).
Moreover, the original versions of these $r$ exceptional factors are pairwise edge-disjoint.
\item[(PG2)] $PG^\diamond$ is a spanning subgraph of $G-V_0$ which satisfies
$d^+_{PG^\diamond}(x)=r''-d^+_{PG^*}(x)$ and $d^-_{PG^\diamond}(x)=r''-d^-_{PG^*}(x)$.
\end{itemize}
Note that $PG$ is a spanning $r''$-regular subgraph of $G^{\rm basic}$. 
Moreover, in $PG^{\rm orig}$ we have 
\begin{equation} \label{predegrees}
d^\pm(x)=r(s-1) \ \ \forall x\in V_0\ \ \quad \mbox{and} \quad \ \ d^\pm (y) =r''\ \ \forall y\in V(G)\setminus V_0.
\end{equation}
 Note also that (\ref{DeltaPE}) implies that 
\begin{equation} \label{DeltaPG}
\Delta(PG^*) \le \frac{8r'}{d} +2r.
\end{equation}

The following corollary is an immediate consequence of Lemma~\ref{replacesequence}.
\begin{cor}\label{preprocor}
Suppose that $0<1/n\ll r/m \ll r'/m\ll r''/m\ll 1/k\ll \eps\ll 1/\ell^*\ll d\ll \nu\ll \tau\ll \alpha, \theta\le 1$
and that $d\ll \zeta\le 1/2$. Let $s:=10^7/\nu^2$. Suppose that $m/50, 50k/(s-1)\in\mathbb{N}$.
Let $(G,\cP_0, R_0,C_0,\cP,R,C)$ be a consistent $(\ell^*,k,m,\eps,d,\nu,\tau,\alpha,\theta)$-system
with $|G|=n$. Let $\cP'$ be a $(r'/m)^2$-uniform $50$-refinement of $\cP$.
Let $H$ be an $r$-factor of $G$.
Let $PG$ be a preprocessing graph with respect to $C,R,\cP'$ with parameters $(s,\eps,d,r',r'',r,\zeta)$.
Suppose that $H$ and $PG^{\rm orig}$ are edge-disjoint. 
Then $G$ contains edge-disjoint Hamilton cycles $C_1,\dots,C_{rs}$ such that the following conditions hold:
\begin{itemize}
\item[{\rm (i)}] Altogether $C_1,\dots,C_{rs}$ cover all edges of $H$ and $C_i\subseteq H\cup PG^{\rm orig}$
for each $i=1,\dots,rs$.
\item[{\rm (ii)}] Every vertex $x\in V_0$ is isolated in $PG':=PG^{\rm orig}\setminus E(C_1 \cup \dots \cup C_{rs})$
and every vertex $x\in V(G)\setminus V_0$ has in- and outdegree $r''-(s-1)r$ in $PG'$.
\end{itemize}
\end{cor}
\proof Apply Lemma~\ref{replacesequence} with $PG_{PE}$ playing the role of $PE$. This gives us edge-disjoint Hamilton
cycles $C_1,\dots,C_{rs}$ of $G$ as described there. So in particular condition~(i) of Corollary~\ref{preprocor} holds. 
By~(i) of Lemma~\ref{replacesequence}, all the $C_1,\dots,C_{rs}$ together cover
all the original exceptional edges of $PG^{\rm orig}$ (as they cover $CB^{\rm orig}_1,\dots,CB^{\rm orig}_r$
and each original exceptional edge in $PG^{\rm orig}$ is contained in one of $CB^{\rm orig}_1,\dots,CB^{\rm orig}_r$).
Since $V_0$ is an independent set in $G$ by (CSys9), it follows that every vertex $x\in V_0$ is isolated in $PG'$.
Moreover, every vertex $x\in V(G)\setminus V_0$ has indegree $r+r''$ in $H\cup PG^{\rm orig}$ and indegree $rs$ in
$C_1\cup \dots\cup C_{rs}$. So the indegree of $x$ in $PG'$ is $r+r''-rs=r''-(s-1)r$. Similarly, every vertex in
$V(G)\setminus V_0$ has outdegree $r''-(s-1)r$ in $PG'$.
Altogether, this implies condition~(ii) of Corollary~\ref{preprocor}.
\endproof

The next lemma shows that one can find a preprocessing graph within a consistent system.

\begin{lemma}\label{findprepro}
Suppose that $0<1/n\ll r/m\ll d'\ll d''\ll 1/k\ll \eps\ll \eps'\ll 1/\ell^*\ll d\ll \nu\ll \tau\ll \alpha, \theta\le 1$
and that $d\ll\zeta\le 1/2$. Let $s:=10^7/\nu^2$. Suppose that%
    \COMMENT{Need that $50\ell^*/(s-1)\in\mathbb{N}$ since we need that $\ell^*/L\in\mathbb{N}$ in Lemma~\ref{exceptseq}.
Since $k/\ell^*\in\mathbb{N}$ we automatically have that $50k/(s-1)\in\mathbb{N}$.}
$m/50, 50\ell^*/(s-1)\in\mathbb{N}$.
Let $(G,\cP_0, R_0,C_0,\cP,R,C)$ be a consistent $(\ell^*,k,m,\eps,d,\nu,\tau,\alpha,\theta)$-system
with $|G|=n$. Let $\cP'$ be an $\eps$-uniform 50-refinement of $\cP$.
Then $G^{\rm basic}$ contains a preprocessing graph with respect to $C,R,\cP'$ with parameters
$(s,\eps',d,d'm,d''m,r,\zeta)$.
\end{lemma}
\proof
Apply Lemma~\ref{exceptseq} to find $r$ exceptional factors $CB_1,\dots,CB_r$ with parameters $(50,(s-1)/50)$
(with respect to $C,\cP'$) whose
original versions $CB_1^{\rm orig},\dots,CB_r^{\rm orig}$ are pairwise edge-disjoint.
(So we apply the lemma with $K:=50$, $L:=(s-1)/50$ and $r_0:=r$.) Let $G_1$ be obtained
from $G$ by deleting all the edges in $CB_1^{\rm orig},\dots,CB_r^{\rm orig}$. Thus $G_1$ is obtained
from $G$ by deleting at most $(s-1)r$ outedges and at most $(s-1)r$ inedges at each vertex.%
    \COMMENT{Delete exactly $r$ out/inedges at the vertices in $V(G)\setminus V_0$.}
Since $(s-1)r\le \eps m$ we can apply
Lemma~\ref{deletesystem} with $\eps$ playing the role of $\eps'$ in Lemma~\ref{deletesystem} to see that
$(G_1,\cP_0, R_0,C_0,\cP,R,C)$ is still a consistent $(\ell^*,k,m,3\sqrt{\eps},d,\nu/2,\tau,\alpha/2,\theta/2)$-system.  
Apply Lemma~\ref{findPE} to find a path system extender $PG_{PE}$ in $G_1$ with parameters $(\eps',d,d',\zeta)$
(with respect to $C,R$). Let $PG^*$ be the union of $PG_{PE}$ and  $CB_1,\dots,CB_r$.

To find $PG^\diamond$, for every vertex $x\in V(G)\setminus V_0$ let $n^+_x:=d''m-d^+_{PG^*}(x)$ and
$n^-_x:=d''m-d^-_{PG^*}(x)$. Note that (\ref{DeltaPG}) implies that 
\begin{equation} \label{DeltaPG*}
\Delta(PG^*)\le 8d'm/d+2r\le 9d'm/d. 
\end{equation}
So for every vertex $x\in V(G)\setminus V_0$ we have that
\begin{equation} \label{DeltaPG2}
\left( 1- \frac{9d'}{dd''} \right) d'' m =d''m-\frac{9d'm}{d} \le n^+_x,n^-_x\le d''m.
\end{equation}
Let $G_2$ be the digraph
obtained from $G-V_0$ by deleting all the edges in $PG^*$.
Then (\ref{DeltaPG*}) and the fact that $G$ is a robust $(\nu,\tau)$-outexpander with $\delta^0(G)\ge \alpha n$
imply that $G_2$ is still a robust $(\nu/2,2\tau)$-outexpander
with $\delta^0(G_2)\ge \alpha n/2$. 
Together with~(\ref{DeltaPG2}) this shows that we can apply Lemma~\ref{regrobust} with $q=1$, with $G_2$ playing the role of $G=Q$,
and with $d''m/|G_2|$, $9d'/dd''$ playing the roles of $\xi$, $\eps$
to obtain a spanning subgraph $PG^\diamond$ of $G_2$
satisfying $d_{PG^\diamond}^\pm(x)=n^\pm_x$ for every vertex $x$.
Then $PG^*\cup PG^\diamond$ is a preprocessing graph as required.
\endproof

The following lemma decomposes the leftover $PG'$ of the preprocessing graph obtained from Corollary~\ref{preprocor} into path systems $H_i$.
In Section~\ref{sec:chordabsorb}, these path systems will be extended into Hamilton cycles using the `chord absorber'.

Suppose that $k/g\in\mathbb{N}$ and that $C=V_1 \dots V_{k}$.
Consider the canonical interval partition of $C$ into $g$ edge-disjoint intervals of equal length and for each $i=1,\dots,g$
let $X_i$ denote the union of all clusters in the $i$th interval. So $X_i=V_{(i-1)k/g+1}\cup \dots\cup V_{ik/g+1}$.
We say that an edge of $G-V_0$ has \emph{double-type $ij$} if its endvertices are contained in $X_i\cup X_j$.
So the number of double-types is $\binom{g}{2}$. A digraph has \emph{double-type $ij$} if all its edges have double-type $ij$.

\begin{lemma} \label{splitinitcleanH}
Suppose that $0<1/n\ll 1/k,\eps,d,1/q^*,1/g\ll 1$ and that $2q^*/3g(g-1), k/g\in\mathbb{N}$.
Let $(G,\cP,R,C)$ be a $(k,m,\eps,d)$-scheme with $|G|=n$ and $C=V_1\dots V_k$.
Suppose that $H$ is a $1$-regular digraph on $V_1\cup\dots\cup V_k$.
Then we can decompose $E(H)$ into $q^*$ (possibly empty) matchings $H_1,\dots,H_{q^*}$ such that the following conditions hold.
\begin{itemize}
\item[(i)] For all $i=1,\dots,q^*$, $H_i$ consists of at most $2g^2km/q^*$ edges.
\item[(ii)] If $|i-j| \le 10$, then $H_i$ and $H_j$ are vertex-disjoint, with 
the indices considered modulo $q^*$.
\item[(iii)] Each $H_i$ consists entirely of edges of the same double-type
and for each $t\in\binom{g}{2}$ the number of $H_i$ of double-type $t$ is $q^*/\binom{g}{2}$.%
    \COMMENT{Clearly, getting the same type (i.e. being concentrated on an interval) is impossible here.
Using double-types is strong enough to find disjoint exceptional sequences.}
\end{itemize}
\end{lemma}

\proof
First split the edges of $H$ into $2q^*/3g(g-1)$ sets whose sizes are as equal as possible.
Now we arbitrarily split each of these sets into three matchings. Finally, we split each of these matchings further into
submatchings consisting of edges of the same double-type.
Since there are $\binom{g}{2}$ different double-types this gives $\binom{g}{2} \cdot 3\cdot 2q^*/3g(g-1)=q^*$
matchings, which we denote by $H_1,\dots,H_{q^*}$. Moreover, each $H_i$ consists of at most $2g^2km/q^*$ edges.
Thus the $H_i$ satisfy~(i) and~(iii).

Given a double-type~$ij$, let $cl(ij)$ denote the set of all those numbers $s$ for which at least one
of $s-1,s,s+1$ belongs to $\{i,j\}$. (So $cl(23)=\{1,2,3,4\}$ and $cl(27)=\{1,2,3,6,7,8\}$.)

To obtain (ii), we first show the following claim:
there is a cyclic ordering of the double-types so that if two double-types $ab$ and $cd$ have distance at most 10 in the ordering, then
$cl(ab)\cap cl(cd)=\emptyset$.

To prove the claim, consider the following auxiliary graph $A$:
the vertices are the double-types (i.e.~the unordered pairs of numbers in $\{1,\dots,g\}$). We connect two double-types $ab$ and $cd$ by an edge if 
$cl(ab)\cap cl(cd)=\emptyset$. Then $A$ has minimum degree at least $\binom{g}{2}-6(g-1) \ge \frac{10}{11} \binom{g}{2}$.
So $A$ contains the $10$th power of a Hamilton cycle by Theorem~\ref{10thpower}. The ordering of the double-types on the Hamilton cycle
gives the required ordering of the double-types.

We now relabel the $H_i$ as follows: first we take one $H_i$ of each double-type in this ordering of the double-types, 
and then repeat with another $H_i$ of each double-types and so on.
\endproof



\section{The chord absorbing step} \label{sec:chordabsorb}

Recall from the previous section that we have a `leftover' digraph $G'$ which is a regular subdigraph of $G-V_0$.
($G'$ is a subdigraph of the preprocessing graph.)
The aim of this section is to define and use a `chord absorber' $CA$ in $G^{\rm basic}$ so that
\begin{itemize}
\item[(a)] $G' \cup CA$ contains a collection of Hamilton cycles $C_i$ which together cover all edges of $G'$;
\item[(b)] removing the $C_i$ from $CA$ leaves a digraph which is a blow-up of the cycle~$C=V_1 \dots V_k$;
\item[(c)] each $C_i$ contains exactly one complete exceptional sequence and thus $C_i^{\rm orig}$ corresponds to a Hamilton cycle in $G$.
\end{itemize}
Similarly to the previous section, we first split $G'$ into $1$-factors.
Then we split each such $1$-factor $H$ into small matchings $H_i$ (as described in Lemma~\ref{splitinitcleanH}).
Then each $H_i$ is extended into a Hamilton cycle $C_i$ using edges of the chord absorber $CA$.
This is achieved mainly by Lemma~\ref{absorbH}. The main difficulty compared to the argument in the previous section is that we need to achieve (b).

The chord absorber consists of a blow-up of the cycle $C=V_1\dots V_k$, some exceptional factors as well as some additional edges which are
constructed via a `universal walk'. These additional edges will be used to `balance out' the edges in the $H_i$.
The universal walk $U$ will be constructed in the next subsection using the chord sequences defined in Section~\ref{sec:shiftwalks}.

It turns out that a natural way to construct the universal walk $U$ and $CA$ would be the following
(where we ignore requirement (c) for the moment):
for each pair $V_i,V_{i+1}$ of consecutive clusters on the cycle $C$, we fix a shifted walk $SW_i$ from $V_i$ to $V_{i+1}$
(recall these were also defined in Section~\ref{sec:shiftwalks}). We then let $U$ be the concatenation of all the $SW_i$ with $1 \le i \le k$.
Then it is not hard to check that $U$ is a closed walk which visits each cluster the same number of times.
$CA$ is then defined to be the union of a regular blow-up $\cB(C)$ of $C$ (which is also $\eps$-regular) together with a regular blow-up $\cB(U)$ of $U$.
As a step towards extending each $H_i$ into a Hamilton cycle, we balance out each edge $e=x_jx_{j'}$ of $H_i$ by adding a suitable shifted walk $SW_{jj'}$
(see also the proof of Lemma~\ref{findham} for a similar argument).
We do this as follows: for any $j$, let $s(j)$ be such that $x_{j}\in V_{s(j)}$.
Then we add the shifted walk $SW_{jj'}$ consisting of the concatenation $SW_{s(j')}SW_{s(j')+1}\dots SW_{s(j)-1}SW_{s(j)}$.
If we replace each edge $E$ of each $SW_{jj'}$ used for $H_i$ with an edge from the bipartite subgraph of $\cB(U)$
corresponding to~$E$, then the union of these edges together with $H_i$ 
satisfy a `local balance property' (as described in Proposition~\ref{balanced})
and can thus be extended into a Hamilton cycle using edges of $\cB(C)$.
The crucial fact now is that since $H$ is a $1$-factor, it turns out that altogether (i.e.~when considering the union of the $H_i$),
each $SW_j$ is used the same number of times in this process. 
So for each edge $E$ of $U$, overall we use the same number $t$ of edges from the bipartite subgraph of $\cB(U)$ corresponding to $E$.
Thus $t$ is independent of $E$. This means that we can indeed choose $\cB(U)$ to be regular, which will enable us to satisfy (b).

Unfortunately, the above construction of $U$ makes $U$ too long -- $U$ would visit each cluster more than $k$ times, which would create major
technical difficulties in the proof of Lemma~\ref{absorbH}. So in Lemma~\ref{buildU} we present a `compressed' construction (based on chord sequences) which has the same set of
chord edges as the one given above, but which visits each cluster only $\ell'$ times, where $1/k \ll \eps \ll 1/\ell'$.

\subsection{Universal walks, setups and chord absorbers}\label{sec:defabsorb}
Suppose that $R$ is a digraph whose vertices are $k$ clusters $V_1,\dots,V_k$ and that $C:=V_1\dots V_k$ is
a Hamilton cycle in $R$.%
   \COMMENT{Can't talk about $(k,m,\eps,d)$-schemes here since when we apply Lemma~\ref{buildU} in the proof
of Theorem~\ref{decomp} we don't have a $(k,m,\eps,d)$-scheme yet.}
A closed walk $U$ in $R$ is a \emph{universal walk for $C$
with parameter $\ell'$} if the following conditions hold:
\begin{itemize}
\item[(U1)] For every $i=1,\dots, k$ there is a chord sequence $ECS(V_i,V_{i+1})$
from $V_i$ to $V_{i+1}$ such that $U$ contains all edges of all these chord sequences (counted with multiplicities) and all
remaining edges of $U$ lie on $C$.
\item[(U2)] Each $ECS(V_i,V_{i+1})$ consists of at most $\sqrt{\ell'}/2$ edges.
\item[(U3)] $U$ enters every cluster
$V_i$ exactly $\ell'$ times and it leaves every cluster $V_i$ exactly $\ell'$ times.
\end{itemize}
We will often view $U$ as a multidigraph.
Whenever $U$ is a universal walk for $C$ with
parameter~$\ell'$, then $ECS(V_i,V_{i+1})$ will always refer to the chord sequence from $V_i$ to $V_{i+1}$ which
is contained in $U$. We will call $ECS(V_i,V_{i+1})$ an \emph{elementary chord sequence from $V_i$ to $V_{i+1}$}
and the edges in $ECS(V_1,V_2)\cup\dots\cup ECS(V_{k},V_1)$ the \emph{chord edges of $U$}.

Note that condition~(U1) means that if an edge $V_iV_j\in E(R)\setminus E(C)$ occurs in total 5 times (say) in
$ECS(V_1,V_2),\dots,ECS(V_{k},V_1)$ then it occurs precisely 5 times in $U$. We will identify each occurrence of $V_iV_j$ in
$ECS(V_1,V_2),\dots,ECS(V_{k},V_1)$ with a (different) occurrence of $V_iV_j$ in $U$. 
Note that the edges of $ECS(V_i,V_{i+1})$ are allowed to appear in a different order within $ECS(V_i,V_{i+1})$ and within $U$.

Suppose that $F$ is a chord edge of $U$ and that $F'$ is the next chord edge of $U$.
We let $P(F)$ denote the subwalk of $U$ from $F$ to $F'$ (without the edges $F$ and $F'$). So $P(F)$ contains no chord edges and if $F_1$ and $F_2$ are
two occurrences of the same chord edge on $U$ then $P(F_1)$ and $P(F_2)$ might be different from each other.
We say that the edges in $P(F)$ are the \emph{cyclic edges associated with $F$}.
The \emph{augmented elementary  chord sequence} $AECS(V_i,V_{i+1})$ consists of all edges $F$ in 
the elementary chord sequence $ECS(V_i,V_{i+1})$ together with all the edges in the corresponding subwalks
$P(F)$. We order the edges of $AECS(V_i,V_{i+1})$ by taking the ordered sequence $ECS(V_i,V_{i+1})$ and by inserting
the edges of $P(F)$ (in their order on $U$) after each edge $F\in ECS(V_i,V_{i+1})$.
Thus the collection of all augmented elementary  chord sequences $AECS(V_i,V_{i+1})$ forms a partition of the edges of $U$
into $k$ parts. But $AECS(V_i,V_{i+1})$ might not necessarily be connected (i.e.~it might not form a walk in~$R$).

\begin{lemma} \label{buildU}
Suppose that $0<1/k \ll \nu\ll \tau\ll \alpha< 1$. Suppose that $R$ is a robust $(\nu,\tau)$-outexpander
whose vertices are $k$ clusters $V_1,\dots,V_k$ and $\delta^0(R)\ge \alpha k$. Let $C:=V_1\dots V_k$ be
a Hamilton cycle in $R$. Then there exists a universal walk $U$ for $C$ with parameter $\ell':=36/\nu^2$.
\end{lemma}
\proof
Let $\cP:=\{V_1,\dots,V_k\}$. Our first aim is to choose the elementary chord sequences
$ECS(V_j,V_{j+1})$ greedily in such a way that they satisfy~(U2). 
Suppose that we have already chosen $ECS(V_1,V_{2}),\dots,ECS(V_{j-1},V_{j})$ such that 
each of them contains at most $3/\nu=\sqrt{\ell'}/2$ edges and together the interiors of these elementary chord sequences
visit every cluster in $\cP$ at most $2\ell'/3+3/\nu$ times. (Here the number of visits for a cluster $V_i$ is the sum of the number of
entries into $V_i$ and the number of exits from $V_i$.) We say that a cluster in $\cP$ is \emph{full}
if it is visited at least $2\ell'/3$ times by the interiors of the previously chosen chord sequences.
Since each elementary chord sequence contains at most $3/\nu$ edges, it follows that the 
number of full clusters is at most $2\cdot (3j/\nu)/(2\ell'/3) \le 9k/\nu \ell' = \nu k/4$. 
Let $\cV'$ denote the set of clusters in $\cP$ which are full.
Then we can apply Lemma~\ref{buildchord} to obtain an (elementary) chord sequence $ECS(V_j,V_{j+1})$ 
which has at most $3/\nu$ edges and whose interior avoids $\cV'$. 
We continue in this way to choose $ECS(V_1,V_{2}),\dots,ECS(V_{k},V_1)$ such that 
together their interiors visit every cluster in $\cP$ at most $2\ell'/3+3/\nu$ times and let $U^*$ be the union
of these chord sequences. Then $U^*$ visits every cluster in $\cP$ at most $2\ell'/3+3/\nu+2\le 3\ell'/4$ times.%
   \COMMENT{Get precisely two additional visits when considering $U^*$ instead of the union of the interiors of
the elementary chord sequences.} Thus $U^*$ satisfies (U2) and the first part of (U1). 

Our next aim is to add further edges to $U^*$ so that we will be able to satisfy (U3) and the second part of (U1).
For each $j=1,\dots,k$, let $n_j^{\rm out}$ denote the number of edges of $U^*$ which leave the cluster $V_j$
and let $n_j^{\rm in}$ denote the number of edges of $U^*$ which enter the cluster $V_j$.
We claim that for each $j$, we have $n_j^{\rm out}=n_{j+1}^{\rm in}$. 
To see this, suppose that $VW$ is an edge of an elementary chord sequence which is not its final edge.
Then the  next edge of this elementary chord sequence leaves the cluster $W^-$ preceding $W$ on $C$. If $VW$ is the final edge of the
elementary chord sequence, then the first edge of the next elementary chord sequence will leave $W^-$.
This proves the claim.

Now let $\ell_j:=\ell'-1-n_j^{\rm in}$ for each $j=1,\dots,k$. Note that
$\ell_j>0$. Let $U^\diamond$ be obtained from $U^*$ by adding exactly $\ell_j$ copies of the edge $V_{j-1}V_{j}$ for all $j$.
The above claim implies that $U^\diamond$ is $(\ell' -1)$-regular.

Finally we add another copy of each edge of $C$ to $U^\diamond$ and denote the resulting multidigraph by $U$. So now $U$ satisfies (U1)--(U3).
It remains to show that the edges in $U$ can be ordered so that the resulting sequence forms a (connected) closed walk in $R$.
To see this, note that since $U^\diamond$ is an $(\ell'-1)$-regular multidigraph, it has a decomposition into 1-factors
by Proposition~\ref{1factor}. We order the edges of $U$ as follows: 
We first traverse all cycles of the 1-factor decomposition of $U^\diamond$ which contain the cluster $V_1$.
Next, we traverse the edge $V_1V_2$ of $C$. Next we traverse all those cycles of the 1-factor decomposition which contain
$V_2$ and which have not been traversed so far. Next we traverse the edge $V_2V_3$ of $C$ and so on
until we reach $V_1$ again. This completes the construction of $U$.
\endproof

$(G,\cP,\cP',R,C,U,U')$ is called an \emph{$(\ell',k,m,\eps,d)$-setup} if $(G,\cP,R,C)$ is a $(k,m,\eps,d)$-scheme and
the following conditions hold:
\begin{itemize}
\item[(ST1)] $U$ is a universal walk for $C=V_1\dots V_{k}$ with parameter~$\ell'$ and $\cP'$ is an $\eps$-uniform $\ell'$-refinement of $\cP$.
\item[(ST2)] Let $V_j^1,\dots,V_j^{\ell'}$ denote the clusters in $\cP'$ which are contained
in $V_j$ (for each $j=1,\dots,k$). Then $U'$ is a closed walk on the clusters in $\cP'$ which is obtained from $U$ as follows:
When $U$ visits $V_j$ for the $a$th time, we let $U'$ visit the subcluster $V_j^a$ (for all $a=1,\dots,\ell'$).
\item[(ST3)] Each edge of $U'$ corresponds to an $[\eps,\ge d]$-superregular pair in $G$.
\end{itemize}
Since $U$ visits every cluster in $\cP$ precisely $\ell'$ times, it follows that 
$U'$ visits every cluster in $\cP'$ exactly once. So $U'$ can be viewed as a Hamilton cycle on the clusters in $\cP'$.
We call $U'$ the \emph{universal subcluster walk} (with respect to $C$, $U$ and $\cP'$).

Given a digraph $T$ whose vertices are clusters,
an \emph{$(\eps,r)$-blow-up $\cB(T)$ of $T$} is obtained by replacing each vertex $V$ of $T$ with the vertices in the cluster $V$ and replacing
each edge $VW$ of $T$ with a bipartite graph
$\cB(T)[V,W]$ with vertex classes $V$ and $W$ which satisfies the following three properties:
\begin{itemize}
\item $\cB(T)[V,W]$ is $\eps$-regular.
\item All the edges in $\cB(T)[V,W]$ are oriented towards the vertices in $W$.
\item The underlying undirected graph of $\cB(T)[V,W]$ is $r$-regular.
\end{itemize}
An \emph{$r$-blow-up of $T$} is defined similarly: we do not require the bipartite graphs to be $\eps$-regular.

We say that a digraph $CA$ on $V(G)\setminus V_0=V_1\cup\dots\cup V_k$  is a \emph{chord absorber for $C$, $U'$
with parameters $(\eps,r,r',r'',q,f)$} if 
$CA$ is the union of two digraphs $\cB(C)$ and $\cB(U')$ on $V(G)\setminus V_0$ satisfying conditions (CA1)--(CA3) below.
In (CA1),  $\cP^*$ is a $(q/f)$-refinement
of $\cP$. For (CA2), recall from (ST2) that $\cP'$ denotes the partition whose clusters correspond to the vertices of $U'$.
\begin{itemize}
\item[(CA1)] $\cB(C)$ is the union of $\cB(C)^*$ and $CA^{\rm exc}$.
$\cB(C)^*$ is an $(\eps,r)$-blow-up of $C$.
$CA^{\rm exc}$ consists of $r''$ exceptional factors with parameters $(q/f,f)$ (with
respect to $C$, $\cP^*$) whose original versions are pairwise edge-disjoint.
\item[(CA2)] $\cB(U')$ is an $r'$-blow-up of $U'$. 
Moreover, $\cB(U')$ has the following stronger property: 
for every cluster $A$ in $\cP'$ there is a partition $A_1,\dots,A_4$ of $A$ into sets of equal size
such that for every edge $AB$ of $U'$ and each $j=1,\dots,4$ there are $r'$
edge-disjoint perfect matchings
between $A_j$ and $B_j$ such that $\cB(U')[A,B]$ is the union of all these $4r'$ matchings.
\item[(CA3)] $\cB(C)^*$, $\cB(U')$ and the original version of $CA^{\rm exc}$ are pairwise edge-disjoint subdigraphs
of $G$.
\end{itemize}
Thus $CA$ is a $(r+r'+r'')$-regular subdigraph of $G^{\rm basic}$. However, in the original version
$CA^{\rm orig}=\cB(C)^*\cup \cB(U')\cup (CA^{\rm exc})^{\rm orig}$ of $CA$ we have
\begin{equation} \label{chorddegrees}
d^\pm(x) = r'' q  \ \ \forall x\in V_0\ \ \quad \mbox{and} \quad  \ \ d^\pm(y)= r+r'+r'' \ \ \forall y\in V(G)\setminus V_0.
\end{equation}

\subsection{Bi-universal walks, bi-setups and chord absorbers}\label{sec:defbiabsorb}

Suppose that $R$ is a digraph whose vertices are $k$ clusters $V_1,\dots,V_k$, where $k$ is even,
and that $C:=V_1\dots V_k$ is a Hamilton cycle in $R$.
Let $\cV_{\rm even}$ denote the set of all those clusters $V_i$ for which $i$ is even and define
$\cV_{\rm odd}$ similarly.
We will now define a bi-universal walk, which is an analogue of a universal walk for a bipartite setting.
The difference to Section~\ref{sec:defabsorb} is that now we only assume the existence of
a chord sequence from $V$ to $V'$ whenever $V,V'\in \cV_{\rm even}$ or $V,V'\in \cV_{\rm odd}$.
Roughly speaking, if $H$ is a bipartite graph whose vertex classes are $\bigcup \cV_{\rm even}$ and $\bigcup \cV_{\rm odd}$,
$U$ is a bi-universal walk and $U'$ is a bi-universal subcluster walk, then a chord absorber for $C$, $U'$ can still absorb all
edges of $H$. (Note that $C$ is also bipartite with vertex classes $\cV_{\rm even}$ and $\cV_{\rm odd}$.)
This will only be used in~\cite{paper2} and not in this paper.

A closed walk $U$ in $R$ is a \emph{bi-universal walk for $C$
with parameter $\ell'$} if the following conditions hold:
\begin{itemize}
\item[(BU1)] The edge set of $U$ has a partition into $U_{\rm odd}$ and $U_{\rm even}$.
For every $i=1,\dots, k$ there is a chord sequence $ECS^{bi}(V_i,V_{i+2})$
from $V_i$ to $V_{i+2}$ such that $U_{\rm even}$ contains all edges of all these chord sequences for even $i$ (counted with multiplicities)
and $U_{\rm odd}$ contains all edges of these chord sequences for odd $i$.
All remaining edges of $U$ lie on $C$. 
\item[(BU2)] Each $ECS^{bi}(V_i,V_{i+2})$ consists of at most $\sqrt{\ell'}/2$ edges.
\item[(BU3)] $U_{\rm even}$ enters every cluster
$V_i$ exactly $\ell'/2$ times and it leaves every cluster $V_i$ exactly $\ell'/2$ times.
The same assertion holds for $U_{\rm odd}$.%
\COMMENT{So need to make sure $\ell'$ is even when we apply this.}
\end{itemize}

Whenever $U$ is a bi-universal walk for $C$ with
parameter~$\ell'$, then $ECS^{bi}(V_i,V_{i+2})$ will always refer to the chord sequence from $V_i$ to $V_{i+2}$ which
is contained in $U$. As before, we will call $ECS^{bi}(V_i,V_{i+2})$ an \emph{elementary chord sequence from $V_i$ to $V_{i+2}$}
and the edges in $ECS^{bi}(V_1,V_3)\cup ECS^{bi}(V_2,V_4)\cup \dots\cup ECS^{bi}(V_{k},V_2)$ the \emph{chord edges of $U$}.

If $F$ is a chord edge of $U$ then $P(F)$ is defined as in Section~\ref{sec:defabsorb} and we again call
the edges in $P(F)$ are the \emph{cyclic edges associated with $F$}. The \emph{augmented elementary chord sequence}
$AECS^{bi}(V_i,V_{i+2})$ consists of all edges $F$ in 
the elementary chord sequence $ECS^{bi}(V_i,V_{i+2})$ together with all the edges in the corresponding subwalks
$P(F)$. So similarly as in Section~\ref{sec:defabsorb}, the collection of all augmented elementary  chord sequences
$AECS^{bi}(V_i,V_{i+2})$ forms a partition of the edges of $U$ into $k$ parts.

We define an \emph{$(\ell',k,m,\eps,d)$-bi-setup} $(G,\cP,\cP',R,C,U,U')$ similarly as an \emph{$(\ell',k,m,\eps,d)$-setup},
the only difference is that $U$ is a bi-universal walk for $C$ (rather than a universal walk).%
   \COMMENT{Since we assume that $C:=V_1\dots V_k$ when defining the universal walk, it is implicit that $C$ is bipartite with vertex
classes $\cV_{\rm even}$ and $\cV_{\rm odd}$.}
We call $U'$ the \emph{bi-universal subcluster walk} (with respect to $C$, $U$ and $\cP'$).
A \emph{chord absorber for $C$, $U'$ with parameters $(\eps,r,r',r'',q,f)$} is defined analogously as before.


\subsection{Finding chord absorbers} 
The first lemma of this subsection states that if one is given a setup and
one deletes a few edges at every vertex, then one still has a setup with slightly worse parameters.%
   \COMMENT{Previously had this for universal systems}

\begin{lemma}\label{deleteunivsystem}
Suppose that $0<1/n\ll 1/k\ll \eps\le \eps'\ll d\ll 1/\ell'\ll 1$.
Let $(G,\cP,\cP',R,C,U,U')$ be an $(\ell',k,m,\eps,d)$-setup
with $|G|=n$. Let $G'$ be a digraph obtained from $G$ by deleting at most $\eps' m$ outedges and at most $\eps' m$ inedges at every vertex of $G$.
Then $(G',\cP,\cP',R,C,U,U')$ is still a $(\ell',k,m,(\eps')^{1/3},d)$-setup.
The analogue holds if $(G,\cP,\cP',R,C,U,U')$ is an $(\ell',k,m,\eps,d)$-bi-setup.
\end{lemma}
\proof
We only consider the case when $(G,\cP,\cP',R,C,U,U')$ is an $(\ell',k,m,\eps,d)$-setup. The argument for bi-setups is
identical.
By Lemma~\ref{deletesystem}(ii) $(G',\cP,R,C)$ is still a
$(k,m,3\sqrt{\eps'},d)$-scheme. Moreover, (ST1) and (ST2) clearly still hold.
So we only need to check that (ST3) still holds with $\eps$ replaced by $(\eps')^{1/3}$. But 
since the clusters in $\cP'$ have size $m/\ell'$, Proposition~\ref{superslice}(iii) implies that each edge of $U'$
still corresponds to a $[2\sqrt{\eps'\ell'},\ge d]$-superregular pair in $G'$ (and thus to an $[(\eps')^{1/3},\ge d]$-superregular pair).
\endproof

The next lemma asserts that we can find a blow-up of the cycle $C$ and of the universal subcluster walk $U'$ within a setup,%
    \COMMENT{Previously the lemma was stated for universal systems rather than setups. But we can use it in the 1-factorization if
we have it for setups (and we don't need universal systems here).}
so that each edge of $C$ corresponds to a graph which is both regular and superregular and each edge of $U'$ corresponds to a regular graph.

\begin{lemma}\label{findcycleblowups}
Suppose that
$0<1/n\ll 1/k \ll \eps \ll \eps'\ll d\ll 1/\ell'\ll 1$, that $ r_0/m ,r'_0/m\ll d$
and that $m/4\ell'\in \mathbb{N}$. If $(G,\cP,\cP',R,C,U,U')$ is
an $(\ell',k,m,\eps,d)$-setup  with $|G|=n$,
then $G-V_0$ contains edge-disjoint spanning subdigraphs $\cB(C)$ and $\cB(U')$
such that 
\begin{itemize}
\item[(i)] $\cB(C)$ is an $(\eps',r_0)$-blow-up of $C$;
\item[(ii)] $\cB(U')$ is an $r'_0$-blow-up of $U'$ which satisfies (CA2) with  $r'_0$ playing the role of $r'$.
\end{itemize}
The analogue holds if $(G,\cP,\cP',R,C,U,U')$ is
an $(\ell',k,m,\eps,d)$-bi-setup.
\end{lemma}
\proof
Choose new constants $\eps^*$ and $\eps_1$ with $\eps\ll \eps^* \ll \eps'$
and $\eps^*,r_0/m,r'_0/m\ll \eps_1\ll d$.
For each edge $VW$ of $C$ let $d_{VW}$ denote the density of $G[V,W]$.
(So $d_{VW}\ge d-\eps$.) Apply Lemma~\ref{randomregslice}(iv) with 
$$
d':=\frac{r_0/m}{1-12\cdot \eps^*}+\eps^*
$$
to each edge $VW$ of $C$ to obtain a
spanning subgraph $G'[V,W]$ of $G[V,W]$ which is $[\eps^*,d']$-superregular. 
Now apply Lemma~\ref{regularsub2} to obtain a
spanning $r_0$-regular subgraph $G''[V,W]$ of $G'[V,W]$ which is also $[\eps',r_0/m]$-superregular.
Let $\cB(C)$ be the union of all the $G''[V,W]$ over all edges $VW$ of $C$. Then $\cB(C)$ is an $(\eps',r_0)$-blow-up of~$C$.

To construct $\cB(U')$ we now proceed as follows. First we apply Lemma~\ref{randompartition}(i) with $\ell=4$ to
obtain a partition of each cluster $A$ in $\cP'$ into subclusters
$A_1\dots A_4$ such that, for every edge $AB$ of $U'$, the pair $G[A_j,B_j]$ is $[\eps^*,\ge d]$-superregular.
Let $G_1$ be the digraph obtained from $G$ by deleting every edge in $\cB(C)$.
Thus $G_1$ is obtained from $G$ by deleting $r_0$ outedges and $r_0$ inedges at every vertex in $V(G)\setminus V_0$
(and no edges at the vertices in $V_0$). Since the subclusters $A_j$ have size
$m/4\ell'$ and since $r_0 \le \eps_1^3 m/4\ell'$, this means that $G_1[A_j,B_j]$ is still $[\eps_1,\ge d]$-superregular
by Proposition~\ref{superslice}(iii).
For every edge $AB$ of $U'$ and for all $j=1,\dots,4$, we choose $r'_0$ edge-disjoint perfect matchings in $G_1[A_j,B_j]$
(recall that these are guaranteed by Proposition~\ref{perfmatch}). 
At each stage we delete the edges in all the matchings chosen so far before
we choose the next matching.  Since $r'_0 \le \eps_1 m/4\ell'$, this means that the leftover of $G_1[A_j,B_j]$ will always be
$[2\sqrt{\eps_1},\ge d]$-superregular by Proposition~\ref{superslice}(iii). 
So we can choose the next matching. The union $\cB(U')$ of all these perfect matchings
(over all edges of $U'$) is an $r'_0$-blow-up of $U'$ which satisfies (CA2).

The proof for bi-setups is identical.
\endproof

The next lemma is an immediate consequence of Lemma~\ref{findcycleblowups}.
Given suitable exceptional factors, it guarantees a chord absorber within a setup which contains these exceptional factors.
Since Lemma~\ref{exceptseq} implies the existence of such exceptional factors in a consistent system, this will
enable us to find a chord absorber.

\begin{lemma}\label{findchordabs}
Suppose that
$0<1/n\ll 1/k, qr''_0/m \ll \eps \ll \eps'\ll d\ll 1/\ell'\ll 1$, that $ r_0/m ,r'_0/m\ll d$
and that $q/f,fm/q, m/4\ell'\in \mathbb{N}$.
Let $(G,\cP,\cP',R,C,U,U')$ be an $(\ell',k,m,\eps,d)$-setup  with $|G|=n$.
Suppose that $\cP^*$ is a $(q/f)$-refinement
of $\cP$ and $EF_1,\dots,EF_{r''_0}$ are exceptional factors with parameters $(q/f,f)$ with respect to~$C$, $\cP^*$
whose original versions are pairwise edge-disjoint. Then there is a chord absorber $CA$ for $C$, $U'$ in~$G$
having parameters $(\eps',r_0,r'_0,r''_0,q,f)$ such that $CA^{\rm exc}= EF_1\cup \dots\cup EF_{r''_0}$.
The analogue holds if $(G,\cP,\cP',R,C,U,U')$ is an $(\ell',k,m,\eps,d)$-bi-setup.
\end{lemma}
\proof
We only consider the case when $(G,\cP,\cP',R,C,U,U')$ is an $(\ell',k,m,\eps,d)$-setup.
The proof for bi-setups is identical.
Let $G_1$ be the digraph obtained from $G$ by deleting every edge in $(CA^{\rm exc})^{\rm orig}$.
Thus $G_1$ is obtained from $G$ by deleting $r''_0$ outedges and $r''_0$ inedges at every vertex in $V(G)\setminus V_0$
and by deleting $qr''_0$ outedges and $qr''_0$ inedges at every vertex in $V_0$. Since $qr''_0/m\le \eps$, 
Lemma~\ref{deleteunivsystem} implies that
$(G_1,\cP,\cP',R,C,U,U')$ is still a $(\ell',k,m,\eps^{1/3},d)$-setup. So we can apply Lemma~\ref{findcycleblowups} 
to find edge-disjoint subgraphs $\cB(C)^*$ and $\cB(U')$ in $G_1$ as guaranteed by (i) and (ii).
Let $CA^{\rm exc}:= EF_1\cup \dots\cup EF_{r''_0}$. We then take $CA:=\cB(C)^*\cup \cB(U')\cup CA^{\rm exc}$. 
\endproof

\subsection{Absorbing chords into a blown-up Hamilton cycle}

The following lemma contains the key statement of this section (and of the entire proof of Theorem~\ref{decomp}).
Let $(G,\cP,\cP',R,C,U,U')$ be a $(\ell',k,m,\eps,d)$-setup.
Suppose we are given a $1$-factor $H$ of $G-V_0$ which is split into small matchings $H_i$ and that we
are given complete exceptional path systems $H''_i$ which avoid the $H_i$ (see conditions (b)--(d)).
The lemma states that we can 
extend each path system $H_i':=H_i \cup H''_i$ into a Hamilton cycle $C_i$ using edges from the chord absorber $CA$.
The crucial point is that the set of edges we are allowed to use for $C_i$ from $\cB(U')$ is predetermined,
i.e.~this set does not depend on $H$. More precisely, for each $1$-factor $H$ we will split off a regular digraph $\cB'(U')$
from $\cB(U')$. The Hamilton cycles guaranteed by Lemma~\ref{absorbH} cover $H\cup \cB'(U')$, but no other edges of $\cB(U')$.
In Lemma~\ref{absorballH} we will use this property to ensure that when we cover the leftover of the preprocessing graph from the previous section
with edge-disjoint Hamilton cycles, we use all edges of $\cB(U')$ in the process. 
So, as mentioned earlier, the leftover from the chord absorbing step is a subdigraph of $\cB(C)$.

The basic strategy is similar to that of Lemma~\ref{findham}: first we balance out the edges of each $H_i$ by adding edges corresponding to (augmented) chord sequences
(see Claims~1 and~4) in order to obtain path systems $W''_i$ containing $H_i$.
As mentioned above, we will use all edges of $\cB'(U')$ in this process. 
Claim~2 is a step towards this -- it implies that the set of edges we added to the $H_i$ in the above step are themselves `globally balanced'
in the sense that we use the same number from each bipartite graph corresponding to an edge of $U'$. 
(Note however that we achieve this `global balance' property only when considering all the $H_i$ together -- it need not hold when we consider $H_1$ on its own say.)
Then we extend the path system $W''_i$ into a $1$-factor $F_i$.
$F_i$ is then transformed into a Hamilton cycle $C_i$ using Lemma~\ref{mergecycles} (see Claim~5).

Moreover,  Lemma~\ref{absorbH} also works for $(\ell',k,m,\eps,d)$-bi-setups,
as long as $H$ is bipartite with vertex classes $\bigcup \cV_{\rm even}$ and $\bigcup \cV_{\rm odd}$.

\begin{lemma} \label{absorbH}
Suppose that
$0<1/n\ll 1/k,1/q\ll \eps\ll \phi,\eps' \ll  r_1/m\ll d\ll 1/\ell'\ll 1$ and that $m/4\ell'\in\mathbb{N}$.
Let 
\begin{equation} \label{defr2}
r_2:=12\phi \ell' q.
\end{equation}
Suppose that $(G,\cP,\cP',R,C,U,U')$ is a $(\ell',k,m,\eps,d)$-setup with $|G|=n$ and $C=V_1\dots V_k$.
Let $\cB(C)^*$ be a blow-up of~$C$ such that every edge of $C$ corresponds to
an $(\eps',r_1/m)$-superregular pair in $\cB(C)^*$. Let $\cB'(U')$ be an $r_2$-blow-up of $U'$ which satisfies the following condition:
\begin{itemize}
\item[(a)] For every cluster $A$ in $\cP'$ there is a partition $A_1,\dots,A_4$ of $A$ into sets of equal size
such that for every edge $AB$ of $U'$ and each $j=1,\dots,4$ there are $r_2$ edge-disjoint
perfect matchings $S^j_1(AB),\dots,S^j_{r_2}(AB)$
between $A_j$ and $B_j$ such that $\cB'(U')[A,B]$ is the union of all these $4r_2$ matchings.
\end{itemize} 
Suppose that $H$ is a $1$-factor of $G-V_0$ and that
$H_1,\dots,H_q$ is a partition of $H$ into matchings which satisfy the
following properties:
\begin{itemize}
\item[(b)] For each $i=1,\dots,q$ there is a complete exceptional path system%
   \COMMENT{It seems that we need not to specify the parameters - some restrictions are implicit by the condition
$|H'_i\cap V_j|\le \phi m$ in (c). Moreover, it follows from (b) and (c) that each $H_i$ consists of less than $\phi m$ edges.}
$H''_i$ (with respect to $C$) which is vertex-disjoint from $H_i$. 
\item[(c)] Write $H'_i:=H_i\cup H''_i$ for $i=1,\dots,q$. Then for all $i= 1,\dots, q$ the original
versions $(H'_i)^{\rm orig}=H_i\cup (H''_i)^{\rm orig}$ of
$H'_i$ are pairwise edge-disjoint, each $H'_i$ consists of at most $\phi m$ paths and $|H'_i\cap V_j|\le \phi m$
for every cluster $V_j$ in $\cP$.
Moreover, $H'_i$ and $H'_j$ are pairwise vertex-disjoint whenever $|i-j|\le 10$. 
\item[(d)] $\cB(C)^*$, $\cB'(U')$ and the original version of $H':=H'_1\cup\dots\cup H'_q$ are pairwise edge-disjoint
subdigraphs of $G$.
\end{itemize}
Then there are edge-disjoint Hamilton cycles $C_1,\dots,C_{q}$ in $G$ such that the following properties hold:%
   \COMMENT{The first condition is new.}
\begin{itemize}
\item $H'_i\subseteq C^{\rm basic}_i$ for all $i=1,\dots,q$.
\item All the $C_1,\dots,C_q$ together cover all
the edges of $(H')^{\rm orig}\cup \cB'(U')$ and all remaining edges
in $C_1,\dots,C_q$ are contained in $\cB(C)^*$.
\end{itemize}
The analogue holds for an $(\ell',k,m,\eps,d)$-bi-setup $(G,\cP,\cP',R,C,U,U')$ if we assume in addition that $H$ is
bipartite with vertex classes $\bigcup \cV_{\rm even}$ and $\bigcup \cV_{\rm odd}$ (where $\cV_{\rm even}$ is
the set of all those $V_i$ such that $i$ is even and $\cV_{\rm odd}$ is defined analogously).
\end{lemma}
\proof
Recall that $C=V_1\dots V_k$. So $V_1,\dots,V_k$ are the clusters in $\cP$, $|V_j|=m$ for each $j=1,\dots,k$ and each cluster in $\cP'$ has size
\begin{equation}\label{eq:m'}
m':=m/\ell'.
\end{equation}
Given a vertex $x$ of $G-V_0$, we will write $V(x)$ for the cluster in $\cP$ containing $x$ and $V(x)^+$
for the successor of $V(x)$ on $C$. For each $i=1,\dots,q$ in turn, we will find a Hamilton cycle $C_i$ in $G$ which contains the original
version of $H'_i$. So consider any $i$. We will first add suitable edges of $R$ to $H_i$ to form a sequence $W'_i$  of edges 
which is `locally balanced'. So $W'_i$ will consist both of edges of $H$ and edges of $R$. When constructing the Hamilton cycle $C_i$,
we will replace each occurrence of an edge from $R$ in $W'_i$ by an edge in the corresponding bipartite subgraph of $G$.

We will first consider the case when $(G,\cP,\cP',R,C,U,U')$ is a $(\ell',k,m,\eps,d)$-setup.
Recall that the augmented elementary chord sequence $AECS(V_a,V_{a+1})$
was defined in Section~\ref{sec:defabsorb}.
Given clusters $V_j$ and $V_{j'}$, the \emph{augmented  chord sequence $ACS(V_j,V_{j'})$ in $U$}
from $V_j$ to $V_{j'}$ is the (ordered) sequence defined as
$$
ACS(V_j,V_{j'}):=AECS(V_j,V_{j+1}) \cup AECS(V_{j+1},V_{j+2}) \cup \dots \cup AECS(V_{j'-1},V_{j'});
$$ 
where the indices are modulo $k$. If $j=j'$ we take $ACS(V_j,V_{j'}):=\emptyset$.%
    \COMMENT{Previously in this case $ACS(V_j,V_{j'})$ had to be the concatenation of the $k$ augmented elementary chord sequences
$AECS(V_a,V_{a+1})$ for all $a=j,\dots,k+j-1$, where the indices are modulo $k$. But this does not seem to be necessary anymore since
we are now proving Claim~2 differently.}
Let $W'_i$ be obtained from $H_i$ by including the augmented
chord sequence $ACS(V(y),V(x)^+)$
after each edge $xy$ in $H_i$. Ordering the edges in $H_i$ gives us an ordering of the edges of $W'_i$.
Suppose for example that $xy$ is an edge of $H_i$ with $x,y$ both contained in the cluster $V$. Then we include
exactly $AECS(V,V^+)$. If $x \in V$ and $y \in V^+$, where $V^+$ is the successor of $V$ on $C$, then we do not include any edge
(apart from $xy$ itself).

\smallskip

\noindent
{\bf Claim 1. `Sequences are locally balanced.'} \emph{For each cluster $V_j$ in $\cP$ and each $i=1,\dots,q$,
the number of edges of $W'_i$ leaving $V_j$ equals the number of edges of $W'_i$ entering $V_{j+1}$.}

\smallskip

\noindent
Here (and below) multiple occurrences of edges in $W'_i$ are considered separately, i.e.~if the edge $AB$ appears
$u$ times in $W'_i$, then we count it $u$ times in Claim~1. Moreover, if $xy$ is an edge of $H_i$ whose endvertices
lie in the same cluster~$V$, then (here and below) we count $xy$ both as an edge leaving~$V$ and an edge entering~$V$.

\smallskip

\noindent
To prove Claim~1, consider any edge $xy$ of $H_i$. Recall that each augmented elementary chord sequence
$AECS(V_j,V_{j+1})$ consists of chord edges (namely those in the elementary chord sequence $ECS(V_j,V_{j+1})$
corresponding to $AECS(V_j,V_{j+1})$) and of cyclic edges (namely those in the subwalks $P(F)$ which were
added to $ECS(V_j,V_{j+1})$ in order to obtain $AECS(V_j,V_{j+1})$). But every cyclic edge,
joining $V_j$ to $V_{j+1}$ say, contributes both to the number of edges leaving $V_j$ and to the 
number of edges entering $V_{j+1}$. On the other hand, in the (cyclic) sequence
$$
xy\cup ECS(V(y),V(y)^+) \cup ECS(V(y)^+,V(y)^{++}) \cup \dots \cup ECS(V(x),V(x)^+)
$$
obtained from $xy\cup ACS(V(y),V(x)^+)$ by deleting all cyclic edges in the augmented elementary chord sequences, every edge entering
some cluster $V_{j+1}$ is followed by an edge leaving $V_j$. 
(This is essentially the same observation as Proposition~\ref{balanced}.)
Altogether this shows that
for each cluster $V_j$ in $\cP$ the number of edges of $xy\cup ACS(V(y),V(x)^+)$ leaving $V_j$ equals the
number of edges of $xy\cup ACS(V(y),V(x)^+)$ entering $V_{j+1}$. Thus this is also true for the union
$W'_i$ of the $xy\cup ACS(V(y),V(x)^+)$ over all edges $xy\in H_i$. So the claim follows.

\medskip

In what follows, the order of
the edges in $W'_i$ does not matter anymore. So we will view $W'_i$ as a multiset consisting of edges in $E(U)\cup E(H)$.

Recall that each occurrence of an edge in $U$ corresponds to an edge in $U'$ (and that these edges in $U'$
are different for different occurrences of the same edge from $R$ in $U$).
So we might also view each $W'_i$ as a multiset consisting of edges in $E(U')\cup E(H)$.

\smallskip

\noindent
{\bf Claim~2. `Unions of sequences are globally balanced.'} \emph{Let $W$ denote the union of $W'_1,\dots,W'_q$.  
Then there is an integer $t$ so that%
   \COMMENT{Previously also had that $t\ge m$, which we don't get anymore. We this is not needed.}
$W$ contains each edge of $U$ exactly $t$ times.
Thus if $W$ is viewed as a multiset consisting of edges in $E(U')\cup E(H)$, then $W$ also contains
each edge of $U'$ exactly $t$ times.}

\smallskip

\noindent
As before, here (and below) multiple occurrences of an edge in $U$ are considered separately, i.e.~if the edge $AB$ appears
$u$ times in $U$, then it altogether appears $ut$ times in $W$.

\smallskip
To prove Claim~2, consider the auxiliary multidigraph $D$ whose vertices are $V_1,\dots,V_k$ and which
contains an edge from $V_i$ to $V_j$ for every $ACS(V_i,V_j)$ included into~$W$.
So the multiplicity of the edge $V_iV_j$ in $D$ is the number of edges $xy$ of $H$ with $V(y)=V_i$ and
$V(x)^+=V_j$. Let $H^c$ be obtained from $H$ by first reversing the orientation of every edge and
then contracting all the vertices lying in each cluster $V_i$ into a new vertex $v_i$.
So $H^c$ is an $m$-regular multigraph (which might contain loops). Moreover, $D$ can be obtained from $H^c$ by replacing each edge $v_iv_j$ with the edge $v_iv_{j+1}$.
Thus $D$ is $m$-regular too and so it can be decomposed into edge-disjoint $1$-factors.
Consider any cycle $D'=V_{i_1}\dots V_{i_r}$ in one of these $1$-factors.
Then $W$ contains all edges in the multiset
$$S(D'):=ACS(V(x_{i_1}),V(x_{i_2}))\cup ACS(V(x_{i_2}),V(x_{i_3}))\cup \dots\cup ACS(V(x_{i_r}),V(x_{i_1})).$$
But
\begin{align*}
ACS(V_{i_j},V_{i_{j+1}}) =AECS(V_{i_j},V_{i_j+1})\cup \dots \cup AECS(V_{i_{j+1}-1},V_{i_{j+1}}).
\end{align*}
So it follows that $S(D')$ contains every $AECS(V_i,V_{i+1})$ the same number of times and thus $S(D')$ is a multiple of $E(U)$. 
(As an example, if $D'=V_9V_4V_8$, then $S(D')$ will contain each edge of $E(U)$ once.
If however $D'=V_9V_8V_4$, then $S(D')$ will contain each edge of $E(U)$ twice.)
Since $W$ is the union of the $S(D')$ over all
cycles $D'$ in the $1$-factor decomposition of $D$, this implies that $W$ is a multiple of $E(U)$, i.e.~
there exists $t$ such that $W$ contains every edge in $U$ exactly $t$ times. 

Since we consider multiple occurrences of an edge in $U$ separately and each such occurrence corresponds to an edge of $U'$ it 
follows immediately that $W$ (viewed as a multiset consisting of edges in $E(U')\cup E(H)$)
also contains each edge of $U'$ exactly $t$ times. This proves Claim~2.

\smallskip

Let
\begin{equation} \label{defs'}
s':=\phi m.
\end{equation}
Note that (\ref{eq:m'}) implies that
\begin{equation} \label{uppers'}
s'=\phi \ell' m' \le m'/10^4.
\end{equation}
We will need the following claim when replacing the edges of $U$ by edges in $G$.

\smallskip

\noindent
{\bf Claim~3. `Sequences are well spread out.'} 
\emph{For all $i=1,\dots,q$, any edge of $U'$ (and of $U$) occurs in  $W'_i$ at most $s'$ times.}

\smallskip
To prove Claim 3, note that for each $j=1,\dots,k$ each edge in $H_i$ contributes at most one occurrence of
$AECS(V_j,V_{j+1})$ in $W'_i$. But $H''_i$ and $H_i$ are vertex-disjoint by (b) and $H'_i=H_i\cup H''_i$ consists of at most $\phi m$
paths by~(c). So $H_i$ consists of at most $\phi m$ edges and thus the total number of occurrences of $AECS(V_j,V_{j+1})$ in $W'_i$ is
at most $\phi m=s'$. This proves Claim~3 since $AECS(V_1,V_2),\dots,AECS(V_k,V_1)$ forms a partition of $E(U)$.

\smallskip

\noindent
By summing over all $i$ with $1 \le  i \le q$, it immediately follows that the constant $t$ defined in Claim~2 satisfies
\begin{equation} \label{boundont}
t \le s'q.
\end{equation}
We will now add some further edges to $W$ to obtain $W'$
which contains all edges of $H$ and which uses every edge $AB$ of $U'$ (and thus of $U$) \emph{exactly} $t'$ times, where
\begin{equation} \label{equt'}
t':=r_2 m' \stackrel{(\ref{defr2})}{=} 12\phi \ell' q m'\stackrel{(\ref{eq:m'})}{=}12\phi q m \stackrel{(\ref{defs'})}{=} 12 s'q.
\end{equation}

\noindent
{\bf Claim 4.} \emph{By adding some some further edges of $U$ to each $W'_i$ we can obtain multisets $W''_i$ which satisfy the
following properties (as before, we also view $W''_i$ as a multiset consisting of edges in $E(U')\cup E(H)$):
\begin{itemize}
\item Each $W''_i$ is still locally balanced. That is, for every cluster $V_j$ in $\cP$ and each $i=1,\dots,q$, 
the number of edges of $W''_i$ leaving $V_j$ equals the number of 
edges of $W''_i$ entering $V_{j+1}$.
\item For each $i=1,\dots,q$ and each edge $AB$ of $U'$, let $s_i(AB)$ be the number of times that $W''_i$ uses $AB$. 
Then
\begin{equation} \label{sumsi}
\sum_{i=1}^q s_i(AB)= t' 
\end{equation}
and
\begin{equation} \label{sibound}
11s' \le  s_i(AB) \le 13s'
\end{equation}
for each $i=1,\dots, q$.
\end{itemize}
}

\smallskip

\noindent To prove Claim~4, first note that
\begin{equation*}\label{boundt'}
t+11 s'q \stackrel{(\ref{boundont})}{\le} 12s'q \stackrel{(\ref{equt'})}{=} t'.
\end{equation*} 
Choose integers $s'_{1},\dots,s'_{q}$ so that 
$$
\sum_{i=1}^{q} s'_i = t' -t\qquad \mbox{and} \qquad 11s' \le s'_i \le 12s' \quad \mbox{for all $i=1,\dots, q$}.
$$
(The $s'_i$ exist since $11s'q \le t'-t \le 12s'q$.)
For each $i=1,\dots, q$, let $W''_i$ be obtained from $W'_i$ by adding $s'_i$
copies of $U$. Then clearly~(\ref{sumsi}) holds. Also, Claim~3 and the bounds on $s'_i$ imply that
$$
11s' \le s_i' \le s_i(AB) \le s_i'+s' \le 13s'
$$
for each $i=1,\dots, q$. So (\ref{sibound}) holds too.

The fact that each $W''_i$ is still locally balanced immediately follows from Claim~1 and the fact that
$W''_i$ was obtained from $W'_i$ by adding copies of $U$. (Note that (U3) implies that each copy of $U$ contributes
exactly $\ell'$ to the number of edges entering a cluster and $\ell'$ to the number of edges
leaving a cluster.)
This completes the proof of Claim~4.

\medskip

Let $W'$ be the union of the multisets $W''_i$ over all $i=1,\dots,q$. Thus Claim~4 implies that in total $W'$ contains each edge
$AB$ of $U'$ exactly $t'$ times. Since $\cB'(U')$ is an $r_2$-blow-up of $U'$, it follows that
for any edge $AB$ of $U'$, $t'$ equals the number of edges in the subgraph 
\begin{equation} \label{SAB}
S(AB):=\cB'(U')[A,B]
\end{equation}
of $\cB'(U')$ spanned by the clusters $A$ and $B$. (So here $A$ and $B$ are clusters in $\cP'$.)

Next we will replace each occurrence of an edge $AB$ of $U'$ in $W''_i\setminus E(H_i)$ by an edge of $S(AB)$
to obtain a digraph $W'''_i$ with $V(W'''_i)=V(H)=V(G)\setminus V_0$ which has the following properties:
\begin{itemize}
\item[($\alpha_1$)] $W'''_i$ contains all edges in%
    \COMMENT{We really want $H_i$ here and not $H'_i$. But we do need $H'_i$ in ($\alpha_3$).}
$H_i$.
\item[($\alpha_2$)] For each edge $AB$ of $U'$, the bipartite subgraph $W'''_i[A,B]$ consisting of all edges in
$W'''_i$ from $A$ to $B$ contains exactly $s_i(AB)$ edges of $S(AB)$.
\item[($\alpha_3$)] $W'''_i \cup H'_i$ is a path system.%
\COMMENT{earlier version also had that  and $E(H'_i)$ and $E(W'''_i)\setminus E(H'_i)$ have no common endpoints -- similarly in $(\gamma_2)$.}
\item[($\alpha_4$)] For every pair $V,V^+$ of consecutive clusters on $C$ and every $i=1,\dots,q$, there is an
integer $w_i(V) \le \sqrt{\phi} m/2$ so that 
\begin{equation}\label{eqwi}
w_i(V)=\sum_{v \in V} d^+_{W'''_i}(v) = \sum_{v \in V^+} d^-_{W'''_i}(v).
\end{equation}
\item[($\alpha_5$)] $W'''_1,\dots,W'''_{q}$ are pairwise edge-disjoint.%
   \COMMENT{Don't need to say that the original versions of $W'''_1,\dots,W'''_{q}$ are pairwise edge-disjoint since
no $W'''_i$ contains exceptional edges.}
\end{itemize}
Note that $(\alpha_2)$ is equivalent to stating that 
each occurrence of $AB$ in $W''_i$ is replaced by an edge of $S(AB)$.
Moreover $(\alpha_1)$, $(\alpha_2)$, $(\alpha_5)$, (\ref{equt'}) and~(\ref{sumsi}) together imply that 
the edge sets of the $W'''_i$ form a partition $E(H \cup \cB'(U'))$.

Before describing the construction of $W'''_i$, first note that ($\alpha_4$) is an immediate consequence of ($\alpha_2$) and 
Claim~4: any edge $AB$ of $W''_i$ (where $AB\in E(U')$ and so $A$ and $B$ are clusters in $\cP'$) corresponds to an
edge in $W'''_i$ which goes from $A$ to $B$. So~(\ref{eqwi}) holds.
To check that $w_i(V) \le \sqrt{\phi} m /2$ recall from (U3) and~(ST2) that for each cluster $V$ on $C$
there are exactly $\ell'$ edges $AB$ of $U'$ leaving $V$. (\ref{sibound}) implies
that each such edge $AB$ contributes at most $13s' =13\phi m$ to $w_i(V)$.
Together with (c) this implies that $$w_i(V)  \le |H_i\cap V|+13\phi \ell' m \le \phi m+13\phi \ell' m\le \sqrt{\phi}m/2.$$ 

To construct $W'''_i$ for each $i$, we proceed as follows.
Let $u$ be the length of $U'$ and label the edges of $U'$ as $E_1,\dots,E_u$.
Consider any edge $E_a=AB$ of $U'$. For each $j=1,\dots,4$ let $S^j(AB):=S^j_1(AB)\cup \dots\cup S^j_{r_2}(AB)$.
(Recall the $S^j_i(AB)$ were defined in condition~(a) of the lemma.) Then
$S(AB)= S^1(AB) \cup \dots \cup S^4(AB)$
(where $S(AB)$ is as defined in (\ref{SAB})) and (a) implies that
\begin{equation} \label{sjab}
|S^j(AB)| = |S(AB)|/4=r_2 m'/4.
\end{equation}
Order the edges of $S(AB)$ in such a way that the following conditions hold:
\begin{itemize}
\item[($\beta_1$)] Every set of at most $m'/20$ consecutive edges in $S(AB)$ forms a matching.
\item[($\beta_2$)] If $AB\neq E_u$ then for each $j=1,2,3$ all the edges in $S^j(AB)$ precede all those in $S^{j+1}(AB)$. 
\item[($\beta_3$)] If $AB= E_u$ then all the edges in $S^3(AB)$ precede all those in $S^{4}(AB)$, which in turn precede
those in $S^1(AB)$ and all the edges in $S^2(AB)$ are at the end of the ordering.
\end{itemize}
$(\beta_3)$ will be used to ensure that $(\alpha_3)$ holds in the construction of the $W'''_i$.

To see that the above properties can be guaranteed, we use the properties of $S(AB)$ described in the assumption (a) of the lemma: 
for each edge $AB \neq E_u$ of $U'$, order the edges in $S(AB)$ so that ($\beta_2$)
is satisfied and so that within some $S^j(AB)$
all the edges of the matching $S^j_i(AB)$ come before all edges of the matching $S^j_{i+1}(AB)$ (for all $i=1,\dots,r_2-1$).
Order the edges of $S_1^j(AB)$ arbitrarily.
Given an ordering of the edges in $S^j_{i}(AB)$, order the edges of $S^j_{i+1}(AB)$ in such a way
that the first $m'/20$ edges of $S^j_{i+1}(AB)$ avoid the $m'/10$ endvertices of the final $m'/20$ edges of $S^j_i(AB)$.
This ensures that ($\beta_1$) will be satisfied.
If $AB=E_u$ then the argument is similar, but we start with an ordering of the edges in $S(AB)$ so that ($\beta_3$)
is satisfied.

We now carry out the actual construction of the $W'''_i$, where we consider the $W_i'''$ in batches of 10.
For each $a=1,\dots,q/10$ and each edge $E_j$ of $U'$ we let 
$$
u_a(E_j):=s_{10(a-1)+1}(E_j)+\dots+s_{10a}(E_j).
$$
Thus (\ref{sibound}) implies that for all $a=1,\dots,q/10$, 
\begin{equation}\label{eq:ua}
110s'\le u_a(E_j)\le 130s'.
\end{equation}
We let $S^*_a(E_j)$ denote the set of all those
edges whose position in the ordering of the edges of $S(E_j)$ lies between $1+\sum_{a'=1}^{a-1} u_{a'}(E_j)$ and $\sum_{a'=1}^{a} u_{a'}(E_j)$. 
So $u_a(E_j)=|S^*_a(E_j)|$. Together with ($\beta_1$) and the fact that $u_a(E_j)\le 130s'\le m'/20$ by (\ref{uppers'}), this implies that $S^*_a(E_j)$ forms a matching.
Note that $u_a(E_j)$ is the total number of edges in $S(E_j)$ that we need to choose for $W'''_{10(a-1)+1},\dots,W'''_{10a}$.
We will choose all these edges from $S^*_a(E_j)$. 

To choose these edges, we consider an auxiliary bipartite graph $B^*$ which is
defined as follows. The first vertex class $B_1$ of $B^*$ consists of $u_a(E_j)$ placeholders for the edges
in $S(E_j)$ that we need to choose for $W'''_{10(a-1)+1},\dots,W'''_{10a}$, so for each $i=10(a-1)+1,\dots,10a$
there will be precisely $s_{i}(E_j)$ of these placeholders for (the edges to be chosen for) $W'''_{i}$.
The second vertex class of $B^*$ is $S^*_a(E_j)$.
So
$$
110 s' \le |S^*_a(E_j)| = |B_1| \le 130 s'
$$
by~(\ref{eq:ua}).
We join an edge $e\in S^*_a(E_j)$ to a placeholder for $W'''_{i}$ if $e$ is vertex-disjoint from $H'_{i}$.
Since by condition~(c) of the lemma $H'_{10(a-1)+1},\dots,H'_{10a}$ are pairwise vertex-disjoint,
each edge $e\in S^*_a(E_j)$ can meet at most two $H'_i$ with $10(a-1)+1\le i\le 10a$ and so $e$ will be joined
to all placeholders apart from those for the corresponding two $W'''_i$. Since there are $s_{i}(E_j)\le 13 s'$ placeholders
for each $W'''_i$, this means that $e$ is joined in $B^*$ to all but at most $2\cdot 13s'\le |B_1|/2$ placeholders in $B_1$.
Similarly, since $S^*_a(E_j)$ forms a matching and since by (c) every $H'_i$ meets each of the two endclusters of $E_j$
in at most $\phi m$ vertices, each placeholder in $B_1$ is joined to all but at most $2\phi m=2s'\le |S^*_a(E_j)|/2$
edges in $S^*_a(E_j)$. Thus $B^*$ has a perfect matching. For each $i=10(a-1)+1,\dots,10a$ we add the $s_{i}(E_j)$
edges to $W'''_{i}$ which are matched to the placeholders for $W'''_{i}$. 

We carry out this procedure for every edge $E_j$ of $U'$ in turn. This completes the construction of the $W'''_i$.
Clearly, ($\alpha_1$) and ($\alpha_2$) are satisfied. To check that ($\alpha_5$) holds, note that the $W'''_i\setminus E(H_i)$ are pairwise edge-disjoint
by construction and the $H_i$ are pairwise edge-disjoint by definition
(as they form a partition of the edges of $H$ into matchings).
Also $W'''_i\setminus E(H_i)$ is edge-disjoint from any $H_j$ by (d).

So let us now check that ($\alpha_3$) is satisfied. To do this, let $W'''_i[E_j]$ denote the bipartite subdigraph of $W'''_i$ which
consists of all edges from the first endcluster of $E_j$ to the final endcluster of $E_j$. Note that by definition of $B^*$,
the edges of $W'''_i\setminus E(H_i)$ and those of $H'_i$ have no endvertices in common.
Moreover, (b) implies that $H'_i$ is a path system. So for ($\alpha_3$), it suffices to show that $W'''_i\setminus E(H_i)$ is a path system.
Now recall that the definition of $B^*$ implies that $W'''_i[E_j]\setminus E(H_i)$ forms a matching for each edge $E_j$ of $U'$.
Thus the only possibility for a cycle $C'$ in $W'''_i\setminus E(H_i)$ would be for $C'$ to `wind around' $U'$.

So in order to show that $W'''_i\setminus E(H_i)$ is a path system, it suffices 
to show that no vertex is incident to both an edge in $W'''_{i}[E_1] \setminus E(H_i)$ and an edge in%
    \COMMENT{Cannot just write that $W'''_{i}[E_1] \setminus E(H_i)$ and $W'''_{i}[E_u] \setminus E(H_i)$ are vertex-disjoint
since they share the vertices in the cluster belonging to both $E_1$ and $E_u$.}
$W'''_{i}[E_u] \setminus E(H_i)$.
But this follows from ($\beta_2$) and ($\beta_3$). Indeed, recall that when choosing the edges in $W'''_{i}[E_1] \setminus E(H_i)$
we considered all the $W'''_i$ in batches of~10. Let $a:=\lceil i/10\rceil$. So the edges in $W'''_{i}[E_1] \setminus E(H_i)$
were chosen in the $a$th batch. Let $p^{\rm first}_1$ and $p^{\rm final}_1$ denote the first and the final position of an edge from
$W'''_{i}[E_1] \setminus E(H_i)$ in the ordering of all edges of $S(E_1)$. Define $p^{\rm first}_u$ and $p^{\rm final}_u$
similarly. Note that 
\begin{equation}\label{eq:pos1}
110s'(a-1)\stackrel{(\ref{eq:ua})}{\le} p^{\rm first}_1, p^{\rm final}_1\stackrel{(\ref{eq:ua})}{\le} 130s'a.
\end{equation}
But
$$130s'a-110s'(a-1)=20s'a+110s'\le 20s'\frac{q}{10}+110s'< 3s'q \stackrel{(\ref{equt'})}{=} \frac{r_2m'}{4}
\stackrel{(\ref{sjab})}{=} |S^{j}(E_1)|.
$$
Together with (\ref{eq:pos1}) this implies that
\begin{equation}\label{eq:pos2}
110s'(a-1)\le p^{\rm first}_1, p^{\rm final}_1< 110s'(a-1)+ |S^{j}(E_1)|.
\end{equation}
Then (\ref{eq:pos2}), its analogue for $p^{\rm first}_u$ and $p^{\rm final}_u$, ($\beta_2$) and ($\beta_3$) together imply that
there is some $j\le 3$ such that 
$$
W'''_{i}[E_1] \setminus E(H_i) \subseteq  S^{j}(E_1)\cup S^{j+1}(E_1)
\mbox{ and } W'''_{i}[E_u] \setminus E(H_i) \subseteq  S^{j+2}(E_u)\cup S^{j+3}(E_u)
$$ 
(where $S^5(E_u):=S^1(E_u)$ and
$S^6(E_u):=S^2(E_u)$, see Figure~\ref{fig:Wmatch}). 
\begin{figure}
\centering\footnotesize
\includegraphics[scale=0.4]{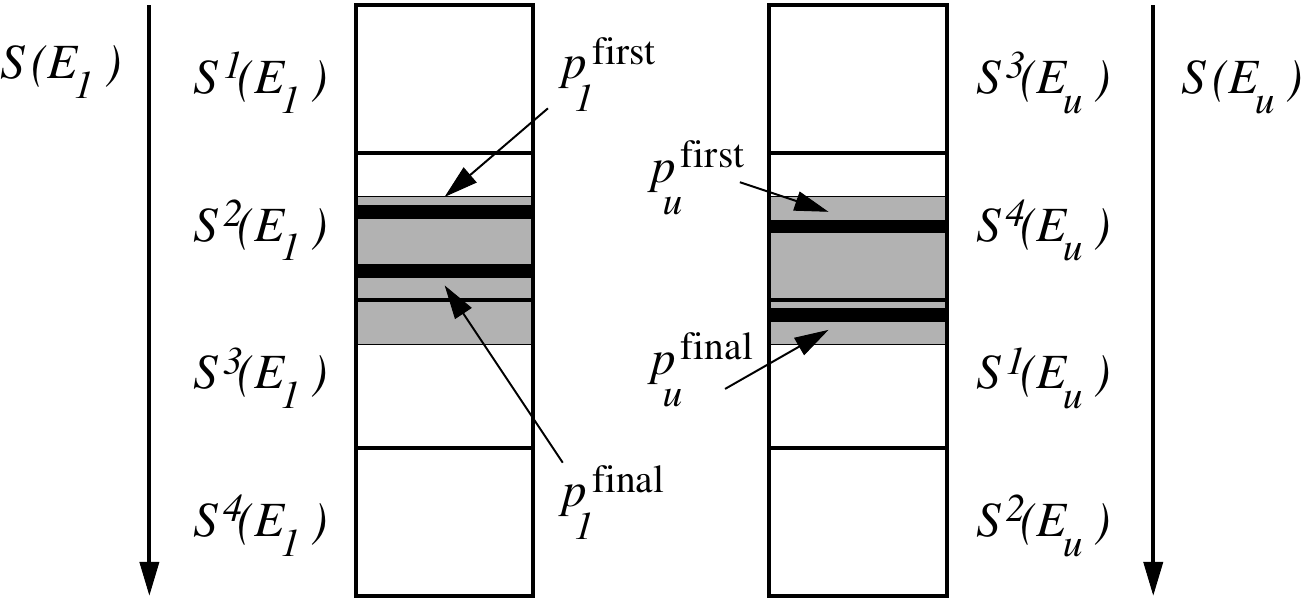}
\caption{The shaded area corresponds to the possible positions of the edges in $W'''_{i}[E_1] \setminus E(H_i)$
within the set $S(E_1)$ and the edges in $W'''_{i}[E_u] \setminus E(H_i)$
within the set $S(E_u)$.}
\label{fig:Wmatch}
\end{figure}
So no vertex is incident to both an edge in $W'''_{i}[E_1] \setminus E(H_i)$ and an edge in $W'''_{i}[E_u] \setminus E(H_i)$.
Altogether this shows that~($\alpha_3$) holds.
So we have shown that ($\alpha_1$)--($\alpha_5$) hold.

\smallskip

We now add all the edges in $H''_i=H'_i\setminus E(H_i)$ to $W'''_i$ and let $W^*_i$ denote the graph
on $V(H)=V(W'''_i)$ obtained in this way. Recall that  by (b), $H''_i$ consists of a complete
exceptional path system with respect to $C$. Thus (CEPS1) and (CEPS3) together imply that
$H''_i$ is `locally balanced', in the sense that for every cluster~$V$ in $\cP$
the number of edges in $H''_i$ leaving $V$ equals the number of edges in $H''_i$ entering $V^+$.
(c) implies that the number of edges leaving $V$ is at most $\phi m$.
Together with ($\alpha_4$) this implies that condition ($\gamma_3$) below holds. ($\gamma_1$) and ($\gamma_2$)
follow from ($\alpha_1$) and ($\alpha_3$) respectively. ($\gamma_4$) follows from ($\alpha_5$), the fact that the original versions
of the $H'_i$ (and thus of the $H''_i$) are pairwise edge-disjoint by~(c) and the fact that $W^*_i\setminus E((H''_i)^{\rm orig})$
is edge-disjoint from any $(H''_j)^{\rm orig}$ by~(d).
\begin{itemize}
\item[($\gamma_1$)] $W^*_i$ contains all edges in $H'_i$.
\item[($\gamma_2$)] $W^*_i$ is a path system.
\item[($\gamma_3$)] For every pair $V,V^+$ of consecutive clusters on $C$ and every $i$, there is an
integer $w_i(V) \le \sqrt{\phi} m$ so that 
$$
w_i(V)=\sum_{v \in V} d^+_{W^*_i}(v) = \sum_{v \in V^+} d^-_{W^*_i}(v).
$$ 
\item[($\gamma_4$)] The original versions of $W^*_1,\dots,W^*_{q}$ are pairwise edge-disjoint.
\end{itemize}
Note that for each $i=1,\dots,q$ every exceptional edge in $W^*_i$ lies in $H''_i$. 
Our next aim is to turn the original versions of $W^*_1,\dots,W^*_{q}$ into edge-disjoint Hamilton cycles by adding suitable
edges from $\cB(C)^*$.

\medskip

\noindent
{\bf Claim~5.} \emph{For all $i=1,\dots,q$, there is a Hamilton cycle $C_i$ in $G$ which contains all edges in the original version
$(W^*_i)^{\rm orig}$ of $W^*_i$ and such that all the edges in $E(C_i)\setminus (W^*_i)^{\rm orig}$ lie in $\cB(C)^*$.
Moreover, all these Hamilton cycles $C_i$ are pairwise edge-disjoint.}

\smallskip Choose a new constant $\eps''$ such that
$$\phi,\eps'\ll \eps''\ll r_1/m.
$$
Suppose that we have already transformed the original versions of $W^*_1,\dots,W^*_{i-1}$
into Hamilton cycles $C_1,\dots,C_{i-1}$. Let $\cB(C)^*_i$ denote the subdigraph 
of $\cB(C)^*$ obtained by removing all edges in $C_1,\dots,C_{i-1}$.
For every cluster $V_j\in \cP$ let $V^1_j$ be the set of all those vertices $v\in V_j$ for which
$d^+_{W^*_j}(v)=0$ and let $V^2_j$ be the set of all those vertices $v\in V_j$ for which
$d^-_{W^*_j}(v)=0$. Thus ($\gamma_3$) implies that
$$
|V^1_j|\ge m-w_i(V_j)\ge (1-\sqrt{\phi})m\ge (1-(\eps''/2)^2)m,
$$
and similarly $|V^2_j|\ge (1-(\eps''/2)^2)m$.
Proposition~\ref{superslice}(ii) applied with%
    \COMMENT{Can't take $d'=\sqrt{\phi}$ since we need $\eps\le d'$ in Proposition~\ref{superslice}(ii) (which would mean $\eps'\le \sqrt{\phi}$).}
$d'=(\eps''/2)^2$ now implies that $\cB(C)^*_i[V^1_j,V^2_{j+1}]$ is still
$(\eps'', r_1/m)$-superregular. (To see that Proposition~\ref{superslice}(ii) can be applied we use that the removal
of each $C_j$ decreases the minimum out- and indegree of every vertex of $\cB(C)^*[V_j,V_{j+1}]$
by at most $1$. Thus $\cB(C)^*_i[V^1_j,V^2_{j+1}]$ is obtained from $\cB(C)^*[V_j,V_{j+1}]$ by deleting at most
$q\le (\eps''/2)^2 m$ edges at every vertex and by removing at most $(\eps''/2)^2m$ vertices from each vertex class.)

On the other hand, ($\gamma_2$) and ($\gamma_3$) imply that $|V^1_j|=|V^2_{j+1}|$.
So we can apply Proposition~\ref{perfmatch} to find a perfect matching $M_j$ in $\cB(C)^*_i[V^1_j,V^2_{j+1}]$.
Then the union $F_i$ of the $M_j$ (for all $j=1,\dots,k$) and of $W^*_i$ is a 1-regular digraph on $V(G)\setminus V_0$.
We can now apply Lemma~\ref{mergecycles} with $F_i$, $\cB(C)^*_i$, $E(C)$, $\eps''$, $r_1/m$ playing the roles of
$F$, $G$, $J$, $\eps$, $d$ to replace each $M_j$ with a suitable other perfect matching in $\cB(C)^*_i[V^1_j,V^2_{j+1}]$
to make $F_i$ into a Hamilton cycle $C'_i$ on $V(G)\setminus V_0$. To see that (ii) of Lemma~\ref{mergecycles}
is satisfied, consider any cycle $D$ in $F_i$. If $D$ does not contain any edges from $W^*_i$, then it
meets $V^1_j$ for every $j=1,\dots,k$. So suppose that $D$ contains some edges from $W^*_i$ and let $v$
be a final vertex on a subpath in $W^*_i\cap D$ (such a vertex exists by ($\gamma_2$)).
Then $v\in V^1_j$, where $V_j$ is the cluster containing $v$. 

Let $C_i$ be the original version $(C'_i)^{\rm orig}$ of $C'_i$. Then Observation~\ref{basicobs} implies that
 $C_i$ is a Hamilton cycle of $G$.
By ($\gamma_4$) all the Hamilton cycles $C_1,\dots,C_{q}$ will be pairwise edge-disjoint.
This completes the proof of Claim~5.

\medskip

Let us now consider the case when $(G,\cP,\cP',R,C,U,U')$ is a $(\ell',k,m,\eps,d)$-bi-setup.
The argument for this case is similar, so we only highlight the places where it differs.
Given clusters $V_j$ and $V_{j'}$ such that $|j'-j|$ is even, we now define
$$
ACS(V_j,V_{j'}):=AECS^{bi}(V_j,V_{j+2}) \cup AECS^{bi}(V_{j+2},V_{j+4}) \cup \dots \cup AECS^{bi}(V_{j'-2},V_{j'}).
$$ 
Since $H$ is bipartite with vertex classes $\bigcup \cV_{\rm even}$ and $\bigcup \cV_{\rm odd}$,
every edge $xy$ of $H$ either satisfies $V(x)^+,V(y)\in \cV_{\rm even}$ or $V(x)^+,V(y)\in \cV_{\rm odd}$.
Thus we can define the $W'_i$ as before and Claim~1 still holds. 

Recall that $U_{\rm even}$ and $U_{\rm odd}$ form a partition of the edges of $U$.
Since each occurrence of an edge in $U$ corresponds to an edge in $U'$, this also defines sets $U'_{\rm even}$ and $U'_{\rm odd}$
corresponding to $U_{\rm even}$ and $U_{\rm odd}$. Moreover, $U'_{\rm even}$ and $U'_{\rm odd}$ form a partition of the edges of $U'$.
Instead of Claim~2 we now have the following claim.

\smallskip

\noindent
{\bf Claim~2$'$.} \emph{Let $W$ denote the union of $W'_1,\dots,W'_q$.  
Then there are integers $t_{\rm even}$ and $t_{\rm odd}$ so that $W$
contains each edge of $U_{\rm even}$ exactly $t_{\rm even}$ times and every edge of $U_{\rm odd}$ exactly $t_{\rm odd}$ times.
Thus if $W$ is viewed as a multiset consisting of edges in $E(U')\cup E(H)$, then $W$ also contains
each edge of $U'_{\rm even}$ exactly $t_{\rm even}$ times and every edge of $U'_{\rm odd}$ exactly $t_{\rm odd}$ times.}

\smallskip

\noindent
To prove Claim~2$'$, define an auxiliary digraph $D$ as before. Note that this time, if $V_iV_j\in E(D)$
then either both $i$ and $j$ are even or both $i$ and $j$ are odd. As before, $D$ is regular and thus there exists a decomposition of~$D$
into edge-disjoint $1$-factors. Consider any cycle $D'=V_{i_1}\dots V_{i_r}$ in one of these $1$-factors.
Then either all of $i_1,\dots,i_r$ are even or all of them odd.
Suppose first that the former holds.
Then $W$ contains all edges in the multiset
$$S(D'):=ACS(V(x_{i_1}),V(x_{i_2}))\cup ACS(V(x_{i_2}),V(x_{i_3}))\cup \dots\cup ACS(V(x_{i_r}),V(x_{i_1})).$$
But
\begin{align*}
ACS(V_{i_j},V_{i_{j+1}}) =AECS(V_{i_j},V_{i_j+2})\cup \dots \cup AECS(V_{i_{j+1}-2},V_{i_{j+1}}).
\end{align*}
So it follows that $S(D')$ contains every $AECS(V_i,V_{i+2})$ for which $i$ is even the same number of times and
thus $S(D')$ is a multiple of $E(U_{\rm even})$. 

If all of $i_1,\dots,i_r$ are odd then it follows that $S(D')$ is a multiple of $E(U_{\rm odd})$.
Since $W$ is the union of the $S(D')$ over all
cycles $D'$ in the $1$-factor decomposition of $D$, this implies Claim~2$'$.

\medskip

As before, one can show that Claim~3 holds and so instead of~(\ref{boundont}) we now have
$t_{\rm even}, t_{\rm odd} \le s'q$ and so
$$
t_{\rm even}+11 s'q, \ t_{\rm odd}+11 s'q \le 12s'q \stackrel{(\ref{equt'})}{=} t'.
$$
Choose integers $s^{\rm even}_{1},\dots,s^{\rm even}_{q}$ so that 
$$
\sum_{i=1}^{q} s^{\rm even}_i = t' -t_{\rm even}  \qquad \mbox{and} \qquad 11s' \le s^{\rm even}_i \le 12s' \quad
\mbox{for all $i=1,\dots, q$}.
$$
Define $s^{\rm odd}_{1},\dots,s^{\rm odd}_{q}$ similarly.
For each $i=1,\dots, q$, let $W''_i$ be obtained from $W'_i$ by adding $s^{\rm even}_i$
copies of $U_{\rm even}$ and $s^{\rm odd}_i$ copies of $U_{\rm odd}$. Then the $W''_i$ are as desired in Claim~4.
The remainder of the proof is now identical.
\endproof

Suppose we are given  a $1$-factor $H$ of $G-V_0$ which is split into suitable matchings $H_i$. 
We will apply the following lemma in the proof of Lemma~\ref{absorballH}
to assign a complete exceptional path system $CEPS_i$ to each $H_i$ so that $CEPS_i$ can play the role of $H''_i$ in
Lemma~\ref{absorbH}. We then use Lemma~\ref{absorbH} to extend $H_i$ into a Hamilton cycle.

\begin{lemma}\label{auxbipart}
Suppose that $0<1/n\ll 1/k, 1/q,\eps,1/f\ll 1$, that $t, k/f,fm/q\in \mathbb{N}$
and that $(G,\cP,R,C)$ is a $(k,m,\eps,d)$-scheme on $n$ vertices.
Suppose that $\cP^*$ is a $(q/f)$-refinement
of $\cP$ and that $EF_1,\dots,EF_t$ are exceptional factors with parameters $(q/f,f)$ with respect to $C$, $\cP^*$.
Let $\cI$ denote the canonical interval partition of $C$ into $f$ intervals of equal length.
Suppose that $H_1,\dots,H_{tq}$ are subdigraphs of $G$ satisfying the following properties:
\begin{itemize}
\item[{\rm (a)}] For each $i=1,\dots,tq$ there are at most $f/100$ intervals $I\in \cI$ such that
$H_i$ contains a vertex lying in a cluster on~$I$.
\item[{\rm (b)}] For each interval $I\in \cI$ there are at most $tq/100$ indices $i$ with $1\le i\le tq$
and such that $H_i$ contains a vertex lying in a cluster on~$I$.
\item[{\rm (c)}] If $|i-j| \le  10$, then $H_i$ and $H_j$ are vertex-disjoint, with the indices considered modulo~$tq$.
\end{itemize}
Then the $tq$ complete exceptional path systems contained in $EF_1,\dots,EF_t$ can be labelled
$CEPS_1,\dots,CEPS_{tq}$ such that the following conditions hold:
\begin{itemize}
\item $H_i\cup CEPS_i$ and $H_j\cup CEPS_j$ are
pairwise vertex-disjoint whenever $|i-j|\le 10$.
\item $H_i$ and $CEPS_i$ are vertex-disjoint.
\end{itemize}
\end{lemma}
\proof
Let $\mathcal{CEPS}$ denote the set of the $tq$ complete exceptional path systems contained in $EF_1,\dots,EF_t$.
In order to label them, we consider an auxiliary bipartite graph $B$ defined as follows.
The first vertex class $B_1$ of $B$ consists of all the $H_i$. 
The second vertex class $B_2$ is $\mathcal{CEPS}$. So $|B_1|=|B_2|=tq$. We join $H_i\in B_1$ to $CEPS\in B_2$ by an edge in $B$
if $CEPS$ is vertex-disjoint from each of $H_{i-10},H_{i-9},\dots, H_{i+10}$.
Our aim is to find a perfect matching in $B$ which has the following additional property:

\textno For all $1\le i<j\le tq$ with $|i-j|\le 10$ the two complete exceptional path systems which are matched to $H_i$ and
$H_j$ are vertex-disjoint from each other. &(\heartsuit)

To show that such a perfect matching exists, let $M$ be a matching of maximum size satisfying $(\heartsuit)$
and suppose that $M$ is not perfect. Pick $H_i\in B_1$ and $CEPS^*\in B_2$ such that they are not covered by $M$.
We say that an interval $I\in \cI$ is \emph{bad for $H_j$} if $H_j$ contains a vertex lying in a cluster on~$I$.
Let $\cI'$ denote the set of all those intervals $I\in \cI$ which are bad for at least one of
$H_{i-10},H_{i-9},\dots, H_{i+10}$. Thus $|\cI'|\le 21f/100$ by~(a). But every complete exceptional
path system which spans an interval $I\in \cI\setminus \cI'$ is vertex-disjoint from each of $H_{i-10},H_{i-9},\dots, H_{i+10}$.
Since for each such $I$ the set $\mathcal{CEPS}$ contains precisely $q t/f$ complete exceptional path systems spanning $I$,
it follows that the degree of $H_i$ in $B$ is at least $(1-21/100)tq$. On the other hand, for each $CEPS\in \mathcal{CEPS}$
there are precisely $3t-1$ other complete exceptional path systems in $\mathcal{CEPS}$ which are not vertex-disjoint from $CEPS$.
This implies that at most $20\cdot 3t$ neighbours $CEPS$ of $H_i$ in $B$ are not vertex-disjoint from each of the
at most $20$ complete exceptional path systems matched to $H_{i-10},\dots,H_{i-1},H_{i+1},\dots,H_{i+10}$ in $M$.
Call a neighbour $CEPS$ of $H_i$ in $B$ \emph{nice} if $CEPS$ is vertex-disjoint from each of these
at most $20$ complete exceptional path systems.
So $H_i$ has at least 
\begin{equation}\label{eq:nbdHi}
(1-21/100)tq-60t\ge (1-22/100)tq=(1-22/100)|B_2|
\end{equation}
nice neighbours in $B$. Note that each nice neighbour $CEPS$ of $H_i$ has to be covered by $M$ 
(otherwise we could enlarge $M$ into a bigger matching satisfying $(\heartsuit)$ by adding the edge between $H_i$ and $CEPS$).
Thus in particular,
\begin{equation}\label{eq:sizeM}
|M|\ge (1-22/100)|B_2|.
\end{equation}

Let $I^*\in \cI$ be the interval which $CEPS^*$ spans. Then (b) implies that $I^*$ is bad for at most $tq/100$
of the $H_j$. But this implies that there are at most $21tq/100$ indices $j$ with $1\le j\le tq$ and such that
$I^*$ is bad for at least one of $H_{j-10},H_{j-9},\dots, H_{j+10}$. Thus the degree of $CEPS^*$ in $B$ is at least $(1-21/100)tq$.
Together with (\ref{eq:sizeM}) this implies that $CEPS^*$ has at least $(1-43/100)|B_1|$ neighbours in $B$ which
are covered by $M$. We call such a neighbour $H_j$ \emph{useful} if $CEPS^*$ is vertex-disjoint from
each of the (at most) $21$ complete exceptional path systems matched to $H_{j-10},H_{j-9},\dots, H_{j+10}$ in $M$.
Recall $\mathcal{CEPS}$ contains precisely $3t-1$ other complete exceptional path systems
which are not vertex-disjoint from $CEPS^*$. But each of these can force at most $21$ neighbours $H_j$ of
$CEPS^*$ to become useless (by being matched to one of $H_{j-10},H_{j-9},\dots, H_{j+10}$). So $CEPS^*$ has at least
$$
(1-43/100)|B_1|-63t=(1-43/100)|B_1|-63|B_1|/q\ge (1-44/100)|B_1|
$$
useful neighbours which are covered by $M$. Together with (\ref{eq:nbdHi}) this implies that there is
a matching edge $e\in M$ such that its endpoint in $B_1$ is a useful neighbour of $CEPS^*$ while its endpoint in $B_2$
is a nice neighbour of $H_i$. Let $H_j$ and $CEPS$ be the endpoints of $e$. Let $M'$ be the matching
obtained from $M$ by deleting $e$ and adding the edge between $H_i$ and $CEPS$ and the edge between $H_j$
and $CEPS^*$. Then $M'$ is a larger matching which still satisfies $(\heartsuit)$, a contradiction.

This shows that $B$ has a perfect matching satisfying $(\heartsuit)$. For each $i=1,\dots,tq$
we take $CEPS_i$ to be the complete exceptional path system which is matched to~$H_i$.
Then~(c), the definition of our auxiliary graph $B$ and~$(\heartsuit)$ together imply that
the $CEPS_i$ are as desired. 
\endproof
To obtain an algorithmic version of the above proof, we simply start with an empty matching in the auxiliary graph $B$ and use the above argument
to gradually extend the matching into a perfect one.

For the final lemma of this section, we are given an $r$-factor $H$ of $G-V_0$.
$H$ is then split into $1$-factors $F_i$ and these $1$-factors are split further into small matchings $H_j^i$.
We use Lemma~\ref{auxbipart} to assign a suitable complete exceptional path system $CEPS_j^i$ to each $H_j^i$.
We then  apply Lemma~\ref{absorbH} to extend each $CEPS_j^i \cup H_j^i$ into a Hamilton cycle using edges of $CA$.
Since Lemma~\ref{absorbH} allows us to prescribe a regular subgraph $\cB'(U')$ of $\cB(U')$ whose edges will all be used for the Hamilton cycles,
this means we can use up all edges of $\cB(U')$ in the process, so the leftover of the entire process is a blow-up of $C$, as required.

\begin{lemma} \label{absorballH}
Suppose that $0<1/n\ll 1/k\ll \eps \ll 1/q \ll 1/f \ll r_1/m\ll d\ll 1/\ell',1/g\ll 1$ and
that $rk\le m/f^2$. Let
$$s:=rfk, \ \ \ r_2:=96\ell'g^2kr, \ \ \ r_3:=s/q
$$
and suppose that $k/f, k/g,q/f, m/4\ell', fm/q, 2fk/3g(g-1) \in \mathbb{N}$.
Suppose that $$(G,\cP,\cP',R,C,U,U')$$ is an $(\ell',k,m,\eps,d)$-setup with $|G|=n$ and $C=V_1\dots V_k$.
Suppose that $H$ is an $r$-factor of $G-V_0$ and that 
$CA=\cB(C)^* \cup \cB(U')\cup CA^{\rm exc}$ is a chord absorber for $C$, $U'$ with
parameters $(\eps,r_1,r_2,r_3,q,f)$ whose original version is edge-disjoint from~$H$. 
Then there are edge-disjoint Hamilton cycles $C_1,\dots,C_s$ in $G$ 
which satisfy the following conditions:%
   \COMMENT{Condition~(iii) is new.}
\begin{itemize}
\item[(i)] Altogether $C_1,\dots,C_{s}$ contain all the edges of $H\cup \cB(U')\cup (CA^{\rm exc})^{\rm orig}$.
Moreover, all remaining edges in $C_1,\dots,C_{s}$ are contained in $\cB(C)^*$.
\item[(ii)] $CA^{\rm orig}\setminus \bigcup_i E(C_i)=\cB(C)^*\setminus \bigcup_i E(C_i)$ is an $(r_1+r_2+r-(q-1)s/q)$-blow-up of $C$.
\item[(iii)] Each $C_i^{\rm basic}$ contains one of the $s$ complete exceptional path systems contained in $CA^{\rm exc}$.
\end{itemize}
The analogue holds for an $(\ell',k,m,\eps,d)$-bi-setup $(G,\cP,\cP',R,C,U,U')$ if we assume in addition that $H$ is
bipartite with vertex classes $\bigcup \cV_{\rm even}$ and $\bigcup \cV_{\rm odd}$ (where $\cV_{\rm even}$ is
the set of all those $V_i$ such that $i$ is even and $\cV_{\rm odd}$ is defined analogously).
\end{lemma}
\proof
Define new constants by
$$q':=fk \ \ \text{   and  } \ \ \phi:=8g^2 k/q'=8g^2/f.
$$
Thus
$$s=rq' \ \ \text{  and  } \ \ r_2=12\ell'\phi q'r.$$
Recall that $V_1,\dots,V_k$ denote the clusters in $\cP$.
Apply Proposition~\ref{1factor} to decompose the edges of $H$ into $r$ edge-disjoint $1$-factors $F_1,\dots,F_{r}$ of $G-V_0$.
We now consider each $F_{i^*}$ with $i^*=1,\dots,r$. Apply Lemma~\ref{splitinitcleanH} with $F_{i^*}$ and $q'$ playing the
roles of $H$ and $q^*$ to decompose $F_{i^*}$ into $q'$ matchings
$H^{i^*}_1,\dots,H^{i^*}_{q'}$ which satisfy the following properties:
\begin{itemize}
\item[(a$_1$)] For all $i=1,\dots, q'$, $H^{i^*}_i$ consists of at most $2g^2 km/q'= \phi m/4$ edges.
Moreover $|H^{i^*}_i \cap V_{j}| \le 4g^2 km/q'\le \phi m/2$ for all $j=1,\dots, k$.
\item[(a$_2$)] If $|i-j| \le  10$, then $H^{i^*}_i$ and $H^{i^*}_j$ are vertex-disjoint, with the indices considered modulo~$q'$.
\item[(a$_3$)] Each $H^{i^*}_i$ consists entirely of edges of the same double-type and for each $t\in\binom{g}{2}$
the number of $H^{i^*}_i$ of double-type $t$ is $q'/\binom{g}{2}$.
\end{itemize}
Note that the `moreover part' of (a$_1$) follows immediately from the first part of (a$_1$). 
For (a$_3$), recall that we considered a canonical interval partition $\mathcal{I}_g$ of $C$ into $g$ edge-disjoint intervals
of equal length and for each $j=1,\dots,g$ we denote the union of the clusters in the $j$th interval by $X_j$. Then
$H^{i^*}_i$  has double-type $ab$ (where $a,b\le g$) if all its vertices are contained in $X_a\cup X_b$.

For each $i^*=1,\dots,r$ and each $i=1,\dots,q'$
we now assign a suitable complete exceptional path system $CEPS^{i^*}_i$ from $CA^{\rm exc}$ to $H^{i^*}_i$.
We do this in such a way that all these complete exceptional path systems are distinct from each other and
the following properties hold:
\begin{itemize}
\item[(b$_1$)] For all $i=1,\dots, q'$ and all $j=1,\dots, k$  we have $|(H^{i^*}_i\cup CEPS^{i^*}_i) \cap V_{j}| \le \phi m$.
Moreover, each $H^{i^*}_i\cup CEPS^{i^*}_i$ consists of at most $\phi m$ paths.
\item[(b$_2$)] $H^{i^*}_i\cup CEPS^{i^*}_i$ and $H^{i^*}_j\cup CEPS^{i^*}_j$ are
pairwise vertex-disjoint whenever $|i-j|\le 10$ (for each $i^*=1,\dots,r$).
\item[(b$_3$)] $H^{i^*}_i$ and $CEPS^{i^*}_i$ are vertex-disjoint.
\item[(b$_4$)] $H^{i^*}_i\cup (CEPS^{i^*}_i)^{\rm orig}$ and $H^{i^*}_j\cup (CEPS^{i^*}_j)^{\rm orig}$ are
pairwise edge-disjoint whenever $i \neq j$ (for each $i^*=1,\dots,r$).
\end{itemize}
Note that since $CA^{\rm exc}$ consists of exceptional factors with parameters $(q/f,f)$,
the $CEPS^{i^*}_i$ will always satisfy $|CEPS^{i^*}_i \cap V_{j}| \le fm/ q\le m/f\le \phi m/2$
and each $CEPS^{i^*}_i$ will always consist of $fm/ q\le \phi m/2$ paths. So (b$_1$)
will follow from (a$_1$) and~(b$_3$). (b$_4$) follows immediately from the fact that $H$ and $(CA^{\rm exc})^{\rm orig}$ are edge-disjoint.

In order to choose the $CEPS^{i^*}_i$ we proceed as follows. Let $t:=q'/q=r_3/r$.
Let $\mathcal{I}$ be the canonical interval partition of $C$ into $f$ intervals of equal length.
So $CA^{\rm exc}$ consists of $r_3$ exceptional factors $EF_1,\dots,EF_{r_3}$, where each $EF_j$ induces the disjoint union of $q/f$
complete exceptional path systems on each interval $I\in\mathcal{I}$. For each $i^*=1,\dots,r$ let $\mathcal{CEPS}^{i^*}$
denote the set of all complete exceptional path systems contained in $EF_{(i^*-1)t+1},\dots, EF_{i^*t}$.
Each of these exceptional factors contains $q$ complete exceptional path systems, so altogether we have $tq=q'$ of them in $\mathcal{CEPS}^{i^*}$.
We will take $CEPS^{i^*}_1,\dots,CEPS^{i^*}_{q'}$ to be the complete exceptional path systems in $\mathcal{CEPS}^{i^*}$.

To choose a suitable labeling of the $CEPS^{i^*}_i$, we aim to apply Lemma~\ref{auxbipart} with $H^{i^*}_1,\dots,H^{i^*}_{q'}$
playing the roles of $H_1,\dots,H_{tq}$ and $EF_{(i^*-1)t+1},\dots, EF_{i^*t}$ playing the roles of $EF_1,\dots,EF_t$.
So we need to check that conditions (a)--(c) of Lemma~\ref{auxbipart} hold. Condition~(c) follows from (a$_2$).
To check~(a), consider any $H^{i^*}_i$ and let $ab$ denote its double-type. Note that $\mathcal{I}$ can be obtained from $\mathcal{I}_g$
by splitting each interval in $\mathcal{I}_g$ into $f/ g$ intervals of equal length. Thus at most $4+2f/ g\le f/100$
intervals%
    \COMMENT{Each of $X_a$ and $X_b$ is the union of (the clusters in) $f/g$ intervals. In addition, there is one interval $I\in \mathcal{I}$ with
$I\subseteq X_{a-1}$ which also contains a (single) cluster in $X_a$ and there is another such interval in $X_{a+1}$.
Similarly, we get 2 additional intervals for $X_b$. So we have to add 4 to $2f/ g$.}
$I\in \cI$ have the property that $X_a\cup X_b\supseteq V(H^{i^*}_i)$ contains a vertex lying in a cluster on~$I$
(the extra $4$ accounts for those intervals sharing exactly one cluster with $X_a$ or $X_b$).
To check~(b), consider any interval $I\in \cI$.
Then there are at most $2g$ double-types $ab$ such that%
   \COMMENT{Let $1\le a\le g$ be such that $\bigcup_{V\in I}V\subseteq X_a$. Then there are $g-1$ choices for $b$.
But if $I$ lies on the `boundary' of $X_a$, then either $X_{a-1}$ or $X_{a+1}$ contains a cluster on $I$.
So we get a factor of 2.}
the set $X_a\cup X_b$ does not avoid all the clusters on $I$. Since by (a$_3$) for each double-type
precisely $q'/\binom{g}{2}$ of $H^{i^*}_1,\dots,H^{i^*}_{q'}$ have that double-type,
this implies at most $2gq'/ \binom{g}{2}\le q'/100$ of $H^{i^*}_1,\dots,H^{i^*}_{q'}$ contain a vertex lying in a
cluster on~$I$. Thus we can indeed apply Lemma~\ref{auxbipart} to find a labeling $CEPS^{i^*}_1,\dots,CEPS^{i^*}_{q'}$
of the complete exceptional path systems in $\mathcal{CEPS}^{i^*}$ as described there. 
Then the $CEPS^{i^*}_i$ also satisfy (b$_2$) and (b$_3$). 

Our aim now is to apply Lemma~\ref{absorbH}.
Let $r^*:=12 \ell' \phi q'$. Note that $r_2=r^* r$.
Recall that (CA2) implies that $\cB(U')$ is an $r_2$-blow-up of $U'$ so that for each edge $AB$ of $U'$, there is a partition of 
both $A$ and $B$ into four subclusters $A_1,\dots,A_4$ and $B_1,\dots,B_4$ of equal size so that $\cB(U')[A_j,B_j]$ consists of
exactly $r_2$ edge-disjoint perfect matchings between each pair $A_j,B_j$ (for all $j=1,\dots,4$). 
So we can decompose the edges of $\cB(U')$ into edge-disjoint graphs $S_1,\dots,S_{r}$ 
so that each of these contains exactly  $r^*$ of  these perfect matchings for each pair $A_j,B_j$ of subclusters of each edge $AB$.
(So $S_{i^*}$ is an $r^*$-blow-up of $U'$ for each $i^*=1,\dots,r$.)
Thus we can satisfy condition (a) of Lemma~\ref{absorbH} if we let $S_{i^*}$ and $r^*$ play the roles of $\cB'(U')$ and~$r_2$.

In particular, we can now apply Lemma~\ref{absorbH} with $S_1$ playing the role of $\cB'(U')$, $F_1$
playing the role of $H$, and $\phi$, $q'$, $2/f^{1/2}$, $r^*$ playing the roles of $\phi$, $q$, $\eps'$, $r_2$ to obtain a collection
$\mathcal{C}_1$ of $q'$ edge-disjoint Hamilton cycles in $G^{\rm orig}$.
We next apply Lemma~\ref{absorbH} for each of $F_2,\dots,F_{r}$ in turn to find collections $\mathcal{C}_2,\dots,\mathcal{C}_{r}$,
each consisting of $q'$ edge-disjoint Hamilton cycles in $G^{\rm orig}$.
For each $F_{i^*}$ we use only $S_{i^*}$, the unused part of $\cB(C)^*$ and the complete exceptional path systems $CEPS_i^{i^*}$
guaranteed by (b$_1$)--(b$_4$). Note that in each of the applications of Lemma~\ref{absorbH} the in- and outdegrees of a vertex in $\cB(C)^*$
decrease by at most $q'$. So in total the in- and outdegrees will decrease by at most $rq'= rfk \le m/f$.
Thus Proposition~\ref{superslice}(ii) applied with $d':=1/f$ implies that in each step the remainder of $\cB(C)^*$
will still be $(2/f^{1/2},r_1/m)$-superregular. So this means we can indeed apply Lemma~\ref{absorbH}.

We take $C_1,\dots,C_s$ to be the Hamilton cycles in $\mathcal{C}_1\cup\dots\cup\mathcal{C}_{r}$.
Then clearly $C_1,\dots,C_s$ are pairwise edge-disjoint and they satisfy~(i) and~(iii). To check (ii), consider any vertex
$x\in V(G)\setminus V_0$. Then $x$ has outdegree $s$ in $C_1 \cup \dots \cup C_s$ and all the $r_2+r_3+r$ outedges at $x$ in
$\cB(U') \cup (CA^{\rm exc})^{\rm orig}\cup H$ are covered by $C_1 \cup \dots \cup C_s$.
Moreover, the outdegree in $\cB(C)^*$ is $r_1$.
Thus the outdegree of $x$ in $\cB(C)^*\setminus \bigcup_i E(C_i)$ is
$$
r_1+r_2+r_3+r-s =r_1+r_2+r-\frac{(q-1)s}{q}.
$$
Since the analogue also holds for the indegree of $x$, this proves~(ii).
The proof of the bipartite analogue goes through unchanged.
\endproof


\section{Absorbing a blown-up cycle via switches} \label{sec:switches}

Our main aim in this section is to define (and find) a `cycle absorber' $CyA$ which will be removed from the original digraph $G$ at the start of the proof
of Theorem~\ref{decomp}.
We would like to find a Hamilton decomposition of the union of several cycle absorbers $CyA$ and the `leftover' $G'$ of the chord absorber obtained by an
application of Lemma~\ref{absorballH}. Recall that this leftover $G'$ is a blow-up of $C$. 
Consider a $1$-factor $H$ in a $1$-factorization of the leftover $G'$ -- so the edges of $H$ wind around $C$.
$CyA$ will also be a blow-up of $C$ (if one ignores the edges in the complete exceptional path systems which will be contained in $CyA$).
We will first find a special $1$-factorization of $H \cup CyA$ which makes use of this property. In particular, either half or all the edges of each $1$-factor will come from $CyA$. 
We will then successively switch pairs of edges between pairs of these $1$-factors of $H \cup CyA$ 
with the goal of turning each of them into a Hamilton cycle after a certain number of these switches (see Figure~\ref{fig:switch}).
These switches will always involve edges from $CyA$ and not from $H$.

However, it will turn out that if these switches only involve pairs of $1$-factors, then the parity of the total number of cycles in a $1$-factorization is preserved.
In particular, this will imply that we cannot find a Hamilton decomposition of $H \cup CyA$ if we start of with a $1$-factorization into an odd number of cycles.
So in Section~\ref{sec:parity}, we also define a `parity switcher' which involves switches between triples of $1$-factors to overcome this problem.
We then extend the cycle absorber $CyA$ into a `parity extended cycle absorber' $PCA$ and find a Hamilton decomposition of $H \cup PCA$.
We proceed in the same way for each $1$-factor $H$ in the above  $1$-factorization of $G'$.

\subsection{Definition of the cycle absorber}

Let $C_4$ denote the orientation of a $4$-cycle in which two vertices have outdegree $2$ (and thus the two other vertices have indegree~$2$).
Given digraphs $H$ and $H'$, we say they form a \emph{switchable pair} if there are vertices $x,x^+,y,y^+$ 
so that $xx^+,yy^+$ are edges of $H$ and $xy^+,yx^+$ are edges of $H'$.
So the union $C^*_4$ of these four edges forms a copy of $C_4$. 
We say that $C^*_4$ is a \emph{$HH'$-switch}.
More generally, we also say that a copy of $C_4$ in a digraph $G$ (again with the above orientation) is a \emph{potential switch}.
A \emph{$C^*_4$-exchange} consists of moving the edges $xx^+, yy^+$ from $H$ to $H'$ and moving the edges $xy^+, yx^+$ from $H'$ to $H$
(see Figure~\ref{fig:switch}). The following proposition (whose proof follows immediately from the definition of a $C^*_4$-exchange)
states the crucial property of switches.

\begin{prop} \label{switch}
Given $1$-regular digraphs $H$ and $H'$, suppose there is a $HH'$-switch $C^*_4$ and let  
$H_{\rm new}$ and $H'_{\rm new}$ be obtained from $H$ and $H'$ via a $C^*_4$-exchange.
\begin{itemize}
\item[(i)] If the two edges of $C^*_4 \cap H$ lie on the same cycle of $H$, then $H_{\rm new}$ has one more cycle than $H$.
\item[(ii)] If the two edges of $C^*_4 \cap H$ lie on different cycles $D_1$ and $D_2$ of $H$, then $H_{\rm new}$ has one less cycle than $H$.
More precisely, the set of cycles of $H_{\rm new}$ is the same as that of $H$ except that the vertices of $D_1$ and $D_2$ now lie on a common cycle.
\end{itemize}
Moreover, the analogous assertions hold for $H'$.
\end{prop}

Consider a $(k,m,\eps,d)$-scheme $(G,\mathcal{P},R,C)$. As usual, let $C=V_1,\dots,V_k$ and recall that $V_0$ denotes the
exceptional set in $\cP$. Throughout this section, when referring to `clusters', we will mean the clusters in~$\cP$, i.e.~$V_1,\dots,V_k$.
We assume that $k$ is a multiple of 14.%
    \COMMENT{$k$ even ensures top and bottom are symmetric, multiple of 7 ensures canonical intervals have equal size.
In the construction of $F$ later on we will also assume that $m$ is even.}

Given a subdigraph $F$ of $G-V_0$, we  
say the \emph{top half of $F$} is the subdigraph $F_{\rm top}$ of $F$ induced by all the vertices in 
$V_1\cup V_2\cup \dots\cup V_{k/2+1}$. The \emph{lower half of $F$} is the subdigraph $F_{\rm low}$ of $F$ induced by
all the vertices in $V_{k/2+1}\cup \dots\cup V_{k-1}\cup V_k\cup V_1$. 

Roughly speaking, a cycle absorber consists of three edge-disjoint $1$-factors $F$, $S$ and $S'$ whose edges wind around $C$
(we ignore exceptional vertices and edges in this explanatory paragraph). There will be switches $C_{4,j}$ between $F$ and $S$ and $C'_{4,j}$ between $F$ and $S'$.
Suppose we are given a $1$-factor $H$ which also winds around $C$.
In the proof of Lemma~\ref{cycleabsorb}, we will construct two $1$-factors $T:= H_{\rm low} \cup F_{\rm top}$ and 
$T':= H_{\rm top} \cup F_{\rm low}$. The switches $C_{4,j}$ between $F$ and $S$ will then correspond to switches between $T$ and $S$.
We will use these to turn $T$ into a Hamilton cycle in Lemma~\ref{cycleabsorb}. 
Moreover, after these switches, the resulting $1$-factor obtained from $S$ will be either a Hamilton cycle or will consist of two cycles.
We will proceed similarly for $T'$ and $S'$. If necessary, $S$ and $S'$ will then be transformed into Hamilton cycles using the parity switcher
in Section~\ref{sec:parity}.

A \emph{bicycle} $B$ on $V$ is a digraph with $V(B)=V$ which consists of exactly two vertex-disjoint (directed) cycles.
A \emph{spanning bicycle} $B$ in a digraph $G$ is a $1$-factor of $G$ which consists of exactly two vertex-disjoint cycles.

Let $I_1,\dots,I_7$ be a canonical interval partition of $C$ into $7$ intervals of equal length. Recall that a
complete exceptional path system $CEPS$ completely spans $I_i$ if $CEPS$ spans $I_i$ and the vertex set of $CEPS$ is
the union of all the clusters in $I_i$. Suppose that $H$ is a digraph on $V(G)\setminus V_0$ which contains
$s$ complete exceptional path systems (for some $s$) and whose other edges lie in~$G$. We say that $H$ \emph{agrees with $C$}
if for every edge $vv'$ of $H$ which does not lie in one of the $s$ complete exceptional path systems there is an $i$ 
with $1 \le i \le k$ so that $v\in V_i$ and $v'\in V_{i+1}$. So if $s=0$ then $H$ agrees with $C$ if and only if $H$ winds around~$C$.

Below we assume that the vertices in $V_1$ and in $V_{k/2+1}$ are ordered.
A \emph{cycle absorber} $CyA$ (with respect to $C$) in $G$ is a digraph on $V(G)\setminus V_0$ with the following properties:
\begin{itemize}
\item[(CyA0)] $CyA$ is the union of three $1$-regular digraphs $F$, $S$ and $S'$, each with vertex set $V(G)\setminus V_0$.
$F_{\rm top}$, $F_{\rm low}$, $S$ and $S'$ each contain a complete
exceptional path system (labelled $CEPS_3$, $CEPS_5$, $CEPS_4$ and $CEPS_2$ respectively) and $CEPS_i$ completely spans the interval $I_i$.
Moreover, all the edges of $F\cup S\cup S'$ which are not contained in $CEPS_2\cup\dots\cup CEPS_5$ lie in $G-V_0$
and each of $F$, $S$ and $S'$ agrees with $C$. Finally, $F^{\rm orig}$, $S^{\rm orig}$ and $(S')^{\rm orig}$ are pairwise edge-disjoint
subdigraphs of $G$.
\item[(CyA1)] For each $j=1,\dots,m$, let $P_j$ denote the path of length $k/7$ in $F_{\rm top}$ starting at the $j$th vertex of $V_1$ and
ending in $V_{k/7+1}$. (So each $P_j$ contains precisely one vertex from each cluster in~$I_1$.)
Then $P_j\cup P_{j+1}$ forms a switchable pair with $S$ for all $j=1,\dots, m$ (with indices considered modulo~$m$).
Denote the switch by $C_{4,j}$. 
\item[(CyA2)] There is a potential switch $C_{4,m+1}$ in $G-V_0$ so that $S$ contains two independent edges of $C_{4,m+1}$
but the other two edges of $C_{4,m+1}$ do not lie in $CyA^{\rm orig}$.
Moreover, $V(C_{4,j})\subseteq \bigcup_{V\in I_1} V$ for each $j=1,\dots,m+1$ and all the $C_{4,j}$ are pairwise vertex-disjoint.
\item[(CyA3)] For each $j=1,\dots,m+1$, denote the edges of $C_{4,j}$ which are contained in $S$ by $\ell_j$ and $r_j$.
Let $L$ be the (ordered) sequence of edges $\ell_1,\dots,\ell_{m+1}$ and
$R$ be the (ordered) sequence of edges $r_1,\dots,r_{m+1}$. 
$S$ is a bicycle on $V(G)\setminus V_0$ where one cycle contains all edges of $L$ in the given order
and the other cycle contains all edges of $R$ in the given order.
\end{itemize}
Moreover, $F_{\rm low}$ and $S'$ will satisfy the following conditions which are analogous to (CyA1)--(CyA3).
In (CyA2$'$), we define $I_{4,\rm low}$ to be the subinterval $V_{k/2+1} \dots V_{4k/7+1}$ of $I_4$.
\begin{itemize}
\item[(CyA1$'$)] For each $j=1,\dots,m$, let $P'_j$ denote the path of length $k/14$ in $F_{\rm low}$ starting at the $j$th vertex of $V_{k/2+1}$ and
ending in $V_{4k/7+1}$. (So each $P'_j$ contains precisely one vertex from each cluster in~$I_{4,\rm low}$.)
Then $P'_j\cup P'_{j+1}$ forms a switchable pair with $S'$ for all $j=1,\dots, m$ (with indices considered modulo~$m$).
Denote the switch by $C'_{4,j}$. 
\item[(CyA2$'$)] There is a potential switch $C'_{4,m+1}$ in $G-V_0$ so that $S'$ contains two independent edges of $C_{4,m+1}$
but the other two edges of $C'_{4,m+1}$ do not lie in $CyA^{\rm orig}$.
Moreover, $V(C'_{4,j})\subseteq \bigcup_{V\in I_{4,\rm low}} V$ for each $j=1,\dots,m+1$ and all the $C'_{4,j}$ are pairwise vertex-disjoint.
\item[(CyA3$'$)] For each $j=1,\dots,m+1$, denote the edges of $C'_{4,j}$ which are contained in $S'$ by $\ell'_j$ and $r'_j$.
Let $L'$ be the (ordered) sequence of edges $\ell'_1,\dots,\ell'_{m+1}$ and
$R'$ be the (ordered) sequence of edges $r'_1,\dots,r'_{m+1}$. 
$S'$ is a bicycle on $V(G)\setminus V_0$ where one cycle contains all edges of $L'$ in the given order
and the other cycle contains all edges of $R'$ in the given order.
\end{itemize}
The switches $C_{4,j}$ and $C'_{4,j}$ with $j <m$ will be used in the proof of Lemma~\ref{cycleabsorb} to `transform' a given $1$-regular graph $H$ whose edges wind around $C$
into a Hamilton cycle. 
We will not actually use the two switches $C_{4,m}$ and $C'_{4,m}$ defined above. However, they make our description of the construction of $F$
a little simpler. The potential switches $C_{4,m+1}$ and $C'_{4,m+1}$ will be used to `attach' the cycle absorber to the parity switcher defined in 
Section~\ref{sec:parity}. This will ensure that the `leftover' of the cycle absorber after the above transformation step also has a Hamilton decomposition.

Note that $CyA^{\rm orig}$ is a spanning subdigraph of $G$ in which the vertices in $V_0$ have in- and outdegree~4, while the others 
have in- and outdegree~3. However, $CyA$ is not actually a subdigraph of $G$, so saying that it is a cycle absorber in~$G$ is
a slight abuse of notation. 

\begin{figure}
\centering\footnotesize 
\includegraphics[scale=0.38]{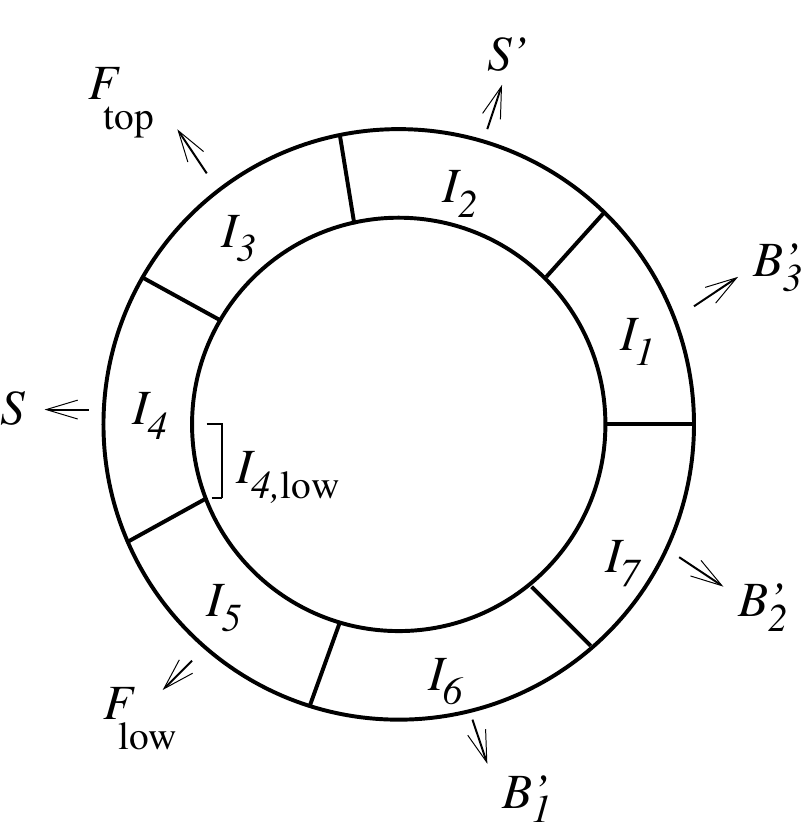}
\caption{The exceptional factor which contains the complete exceptional path systems $CEPS_i$ of the (parity extended) cycle absorber.
Each $CEPS_i$ spans the interval $I_i$ and the diagram shows how these are assigned to the $1$-factors of the cycle absorber in (CyA0) and
to the bicycles $B'_i$ of the parity switcher in the proof of Lemma~\ref{find_pa}.}
\label{figceps}
\end{figure}
The complete exceptional path systems $CEPS_i$ contained in $CyA$ will be chosen within a single exceptional factor, which has parameters $(1,7)$
(see Figure~\ref{figceps}). The assignment of the $CEPS_i$ to the different $1$-factors of the cycle absorber (and the parity switcher) is chosen in such a way that
the switches of the $1$-factor can be chosen to be vertex disjoint from the $CEPS_i$ contained in this $1$-factor.


\subsection{Using the cycle absorber}
The following lemma shows that given an arbitrary $1$-factor $H$ of $G-V_0$ which winds around $C$ and a cycle absorber $CyA$, we can `almost'
decompose $H \cup CyA$ into Hamilton cycles: we obtain a decomposition into at least two Hamilton cycles and at most two spanning bicycles in $G$.
The final step of transforming the bicycles into Hamilton cycles is done by means of a `parity switcher', defined in Section~\ref{sec:parity}.
(As discussed at the beginning of Section~\ref{sec:parity}, the difficult case is when Lemma~\ref{cycleabsorb} yields a decomposition with
exactly three Hamilton cycles and exactly one bicycle.)

\begin{lemma} \label{cycleabsorb}
Suppose that $0<1/n\ll 1/k\ll \eps\ll d\ll 1$, that $k/14\in \mathbb{N}$ and that $(G,\mathcal{P},R,C)$ is a $(k,m,\eps,d)$-scheme with $|G|=n$.
Suppose that $H$ is a $1$-factor of $G-V_0$ which winds around $C$. Let $CyA$ be a cycle absorber with respect to $C$ in $G$
such that $CyA^{\rm orig}$ and $H$ are edge-disjoint.
Then $H \cup CyA^{\rm orig}$ has a decomposition into four $1$-factors $F_1,F_2$, $F_3$ and $F_4$ of $G$
satisfying the following conditions:
\begin{itemize}
\item[{\rm (i)}]  $F_1$ contains $CEPS_3^{\rm orig}$, $F_2$ contains $CEPS_5^{\rm orig}$, $F_3$ contains $CEPS_4^{\rm orig}$
and $F_4$ contains $CEPS_2^{\rm orig}$ (where the $CEPS_i$ are as defined in~(CyA0)).
\item[{\rm (ii)}] $F_1$ and $F_2$ are Hamilton cycles in $G$.
\item[{\rm (iii)}] Each of $F_3$ and $F_4$ is either a Hamilton cycle or a spanning bicycle in $G$.
If $F_3$ is a bicycle, then one of the cycles contains $\ell_{m+1}$ and the other 
contains $r_{m+1}$. Similarly, if $F_4$ is a bicycle, then one of the cycles contains $\ell'_{m+1}$ and the other 
contains $r'_{m+1}$. 
\end{itemize}
\end{lemma}

\proof
Similarly as for $F$ (as defined in (CyA0)), we  partition $H$ into $H_{\rm top}$ and $H_{\rm low}$.
So both $F_{\rm top}$ and $H_{\rm top}$ consist of $m$ vertex-disjoint paths from $V_1$ to $V_{k/2+1}$ 
and both $F_{\rm low}$ and $H_{\rm low}$ consist of $m$ vertex-disjoint paths from $V_{k/2+1}$ to $V_1$.
Let $T:=H_{\rm low} \cup F_{\rm top}$ and let $T':=H_{\rm top} \cup F_{\rm low}$.
Then (CyA0) implies that both $T$ and $T'$ are $1$-regular digraphs on $V(G)\setminus V_0$ which agree with $C$ and correspond to
$1$-factors $T^{\rm orig}$ and $(T')^{\rm orig}$ of~$G$.

Our first aim is to perform switches between $T$ and $S$ to transform
$T$ into a Hamilton cycle on $V(G)\setminus V_0$. ($T^{\rm orig}$ will then turn out to be a Hamilton cycle of $G$.) Suppose that $T$ is not a Hamilton cycle.
For $j=1,\dots,m$, let $P_j$ be as defined in (CyA1). Recall from (CyA0) that $CEPS_3$ is the complete exceptional
path system contained in $T_{\rm top}=F_{\rm top}$. But since $CEPS_3$ completely spans $I_3$, the paths in $CEPS_3$ link
all vertices of $V_{2k/7+1}$ to those in $V_{3k/7+1}$. It follows that every cycle $D$ in $T$ visits every cluster
on $C$ except possibly $V_{2k/7+2},\dots,V_{3k/7}$. In particular, the following assertion holds
(where the $P_j$ are as defined in (CyA1)):

 \textno For each $j=1,\dots,m$, any cycle $D$ in $T$ either contains $P_j$ or
it avoids all vertices of $P_j$. Moreover, $D$ contains at least one of the $P_j$ and so it contains
the $j$th vertex $x_j$ of $V_1$ for some $j$ with $1\le j\le m$. &(\star)

We say that $i$ with $1 \le i < m$ is a \emph{switch index for $T$} if $x_i$ and $x_{i+1}$ lie on different
cycles of $T$ (i.e.~if the initial vertices of $P_i$ and $P_{i+1}$ lie on different cycles of $T$). 
Since $T$ is not a Hamilton cycle, $(\star)$ implies that there must be an $i$ with $1 \le i < m$ which is a switch index.
Our approach will be to perform a switch between the cycles $D$ and $D'$ which contain $x_i$ and $x_{i+1}$
respectively. This will reduce the number of cycles of $T$ and turn $S$ into a Hamilton cycle.
We continue in this way until $T$ is a Hamilton cycle.
The only difference in the later steps is that $S$ might already be a Hamilton cycle, in which case it
is transformed into a bicycle after the switch.

More precisely, for $i=1,\dots, m+1$, we define $L_i:=( \ell_i,\dots,\ell_{m+1})$ and $R_i:=(r_i,\dots,r_{m+1})$.
So $L_1=L$ and $R_1=R$ (where $\ell_i,r_i,L,R$ are as defined in (CyA3)). Suppose that $1 \le i < m$ and that
$T_i$ and $S_i$ are $1$-regular digraphs on $V(G) \setminus V_0$ which satisfy (a$_i$) and (b$_i$) below as well as either (CyA3$_i^-$) or (CyA3$_i^+$):
\begin{itemize}
\item[(a$_i$)] Let $D$ be any cycle of $T_i$. For each $j=i+1,\dots,m$, $D$ either contains $P_j$ or it avoids all
vertices of $P_j$. Moreover, let $e_{i-1}$ denote the edge in $C_{4,i-1}\cap P_i$. Then $D$ either contains all edges
in $E(P_i)\setminus \{e_{i-1}\}$ (possibly $D$ even contains $P_i$) or $D$ avoids all vertices of $P_i$.
\item[(b$_i$)] All of $x_1,\dots,x_i$ lie on a common cycle in $T_i$.
\item[(CyA3$_i^+$)] 
$S_i$ is a bicycle on $V(G)\setminus V_0$ where one cycle contains all edges of $L_i$ in the given order
and the other cycle contains all edges of $R_i$ in the given order. 
\item[(CyA3$_i^-$)] 
$S_i$ is a Hamilton cycle on $V(G)\setminus V_0$ which contains all edges of $L_i$ and $R_i$ in the given order
and where all edges of $L_i$ come before all edges of $R_i$. 
\end{itemize}
Thus if $i=1$, $T_1:=T$ and $S_1:=S$, then ($\star$) and (CyA3) together imply that (a$_i$), (b$_i$) and (CyA3$_i^+$) hold.
(Note that we view $\{e_0\}$ as being the empty set, so the last part of (a$_1$) says that $D$ either contains all edges of $P_1$ or $D$ avoids all vertices of $P_1$.)
So suppose first that (CyA3$_i^+$) holds for some $1 \le i < m$. We define a switch index for $T_i$ in the same way as for $T$.
If $i$ is not a switch index, let $S_{i+1}:=S_i$ and $T_{i+1}:=T_i$. Then clearly (a$_{i+1})$, (b$_{i+1})$ and (CyA3$_{i+1}^+$) hold. If $i$ is a switch index,
let $D$ be the cycle of $T_i$ which contains $x_i$ and let $D'$ be the cycle of $T_i$ which contains $x_{i+1}$. 
Then (a$_i$) implies that $D$ contains all edges in $E(P_i)\setminus \{e_{i-1}\}$ and $D'$ contains $P_{i+1}$.
But $e_{i-1}\notin C_{4,i}$ since $C_{4,i-1}$ and $C_{4,i}$ are vertex-disjoint
by (CyA2). Thus we can carry out the $C_{4,i}$-exchange to obtain $S_{i+1}$ and $T_{i+1}$. 
Then Proposition~\ref{switch} implies that the vertices of $D$ and $D'$ now lie on a common cycle $D''$ of $T_{i+1}$
In particular, $T_{i+1}$ satisfies (b$_{i+1}$). Moreover, this new cycle $D''$ will contain all edges in $E(P_{i+1})\setminus \{e_i\}$.
Together with the fact that $C_{4,i}$ avoids all the $P_j$ for $j=i+2,\dots,m$, it follows that $T_{i+1}$ satisfies (a$_{i+1})$.
Moreover, it is easy to see that $S_{i+1}$ satisfies (CyA3$_{i+1}^-$) (see Figure~\ref{figswitches}).
\begin{figure}
\centering\footnotesize
\includegraphics[scale=0.38]{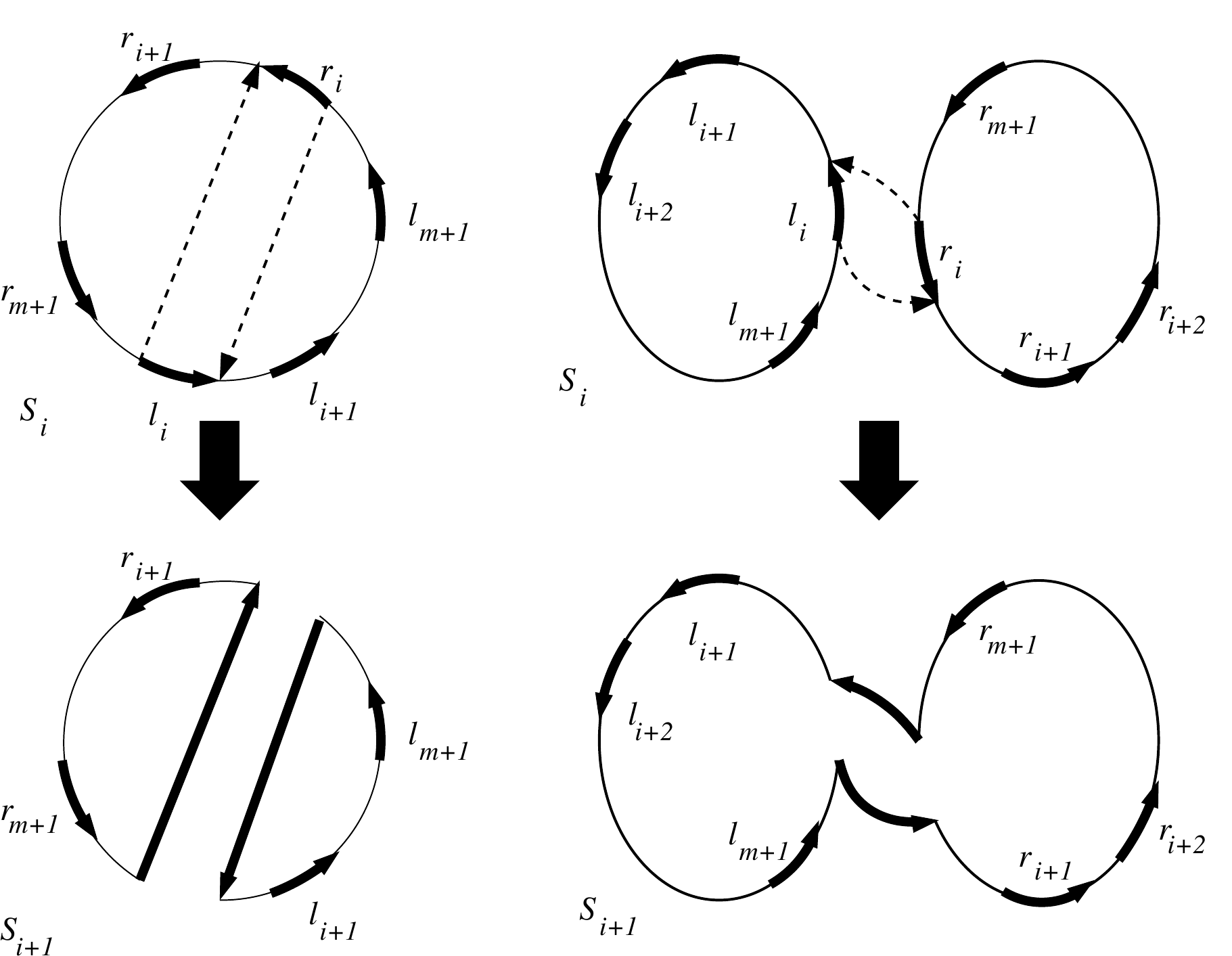}
\caption{Transforming $S_i$ into $S_{i+1}$. The left hand side illustrates the case when $S_i$ satisfies (CyA3$_i^-$) and the
right hand side illustrates the case when $S_i$ satisfies (CyA3$_i^+$).}
\label{figswitches}
\end{figure}
Suppose next that (CyA3$_i^-$) holds for some $1 \le i < m$. If $i$ is not a switch index, let $S_{i+1}:=S_i$ and $T_{i+1}:=T_i$.
Then clearly (a$_{i+1}$), (b$_{i+1}$) and (CyA3$_{i+1}^-$) hold. If $i$ is a switch index, define $D$ and $D'$ as above.
Again carry out the $C_{4,i}$-exchange to obtain $S_{i+1}$ and $T_{i+1}$. 
Again, Proposition~\ref{switch} implies that the vertices of $D$ and $D'$ now lie on a common cycle of $T_{i+1}$.
Also, it is easy to see that (a$_{i+1}$), (b$_{i+1}$) and (CyA3$_{i+1}^+$) hold (see Figure~\ref{figswitches}).

So, by induction, $T_m$ and $S_m$ satisfy (a$_m$) and (b$_m$) as well as one of (CyA3$_m^+$) and (CyA3$_m^-$).
Moreover, in both of the above cases, Proposition~\ref{switch}(ii) implies that all vertices which lie on a common cycle in
$T_i$ still lie on a common cycle in $T_{i+1}$. Together with (b$_m$) and the fact that by ($\star$) every cycle in
$T=T_1$ contains $x_j$ for some $1\le j\le m$ this means that $T_m$ is
a Hamilton cycle. Note that both $CEPS_3$ and $CEPS_4$ are edge-disjoint (actually even vertex-disjoint) from all $C_{4,j}$ (by (CyA2) and the fact that $CEPS_3$ spans $I_3$
and $CEPS_4$ spans $I_4$). Thus both $CEPS_3$ and $CEPS_4$ are unaffected by the switches we carried out and so $T_m$ and $S_m$ still contain
$CEPS_3$ and $CEPS_4$ respectively. So we can take $F_1:=T_m^{\rm orig}$ and $F_3:=S_m^{\rm orig}$.
(In particular, Observation~\ref{basicobs} implies that $F_1$ is a Hamilton cycle of $G$ and the argument for $F_3$ is similar.)%
\COMMENT{we can't apply Observation~\ref{basicobs} to $F_3$, but hopefully the reader will notice it's essentially the same thing}.

In a similar way, we obtain $F_2$ from $T'$ and $F_4$ from $S'$. This time, we let $T'_1:=T'$ and $S'_1:=S'$.
We then carry out the above procedure with $S'_i$, $T'_i$, $P'_j$ playing the roles of $S_i$, $T_i$, $P_j$ to obtain
a Hamilton cycle $T'_m$ on $V(G)\setminus V_0$. Since by (CyA2$'$) both $CEPS_2$ and $CEPS_5$ are edge-disjoint
from all $C'_{4,j}$, both $CEPS_2$ and $CEPS_5$ are unaffected by the switches we carried out and so $T'_m$ and $S'_m$ still contain
$CEPS_5$ and $CEPS_2$ respectively. We let $F_2:=(T'_m)^{\rm orig}$ and $F'_4:=(S'_m)^{\rm orig}$.
\endproof


\subsection{Finding the cycle absorber}

The following lemma guarantees the existence of a cycle absorber in a $(k,m,\eps,d)$-scheme.

\begin{lemma}\label{find_ca}
Suppose that $0<1/n\ll 1/k\ll \eps \ll d\ll 1$, that $k/14,m/2\in\mathbb{N}$ and that $(G,\mathcal{P},R,C)$ is
a $(k,m,\eps,d)$-scheme with $|G|=n$ and $C=V_1\dots V_k$. Let $I_1,\dots,I_7$ be
the canonical interval partition of $C$ into $7$ intervals of equal length. Let $EF$ be an exceptional factor with parameters $(1,7)$
with respect to $C$, $\mathcal{P}$.
For each $i=1,\dots,7$, let $CEPS_i$ be the complete exceptional path system which is contained in $EF$ and completely spans $I_i$.
Then there is a cycle absorber $CyA$ with respect to $C$ in $G$ which satisfies the following properties:
\begin{itemize}
\item[(i)] $CEPS_2,\dots,CEPS_5$ are the complete exceptional
path systems described in (CyA0).
\item[(ii)] $C_{4,m+1}\subseteq G[V_{12},V_{13}]$
and $C'_{4,m+1}\subseteq G[V_{k/2+12},V_{k/2+13}]$ (where $C_{4,m+1}$ and $C'_{4,m+1}$ are as defined in~(CyA2)
and (CyA2$'$) respectively).
\end{itemize}
\end{lemma}
\proof 
We first construct $F_{\rm top}$. Below, we assume the existence of an ordering of the vertices in any cluster $V_i$.
Let $\cB(C)$ be the union of $G[V_i,V_{i+1}]$ over all $i=1,\dots,k$. So $\cB(C)$ is a blow-up
of $C$ in which every edge of $C$ corresponds to an $[\eps,\ge d]$-superregular pair.

\smallskip

\noindent
{\bf Claim 1.} \emph{$\cB(C)$ contains a system $Q_1,\dots,Q_m$ of vertex-disjoint paths and $m$ vertex-disjoint copies
$C_{4,1},\dots,C_{4,m}$ of $C_4$ such that the following properties are satisfied:
\begin{itemize}
\item $Q_j$ joins the $j$th vertex of $V_1$ to the $j$th vertex of $V_{2k/7+1}$.
\item For each $j=1,\dots,m$, $C_{4,j}$ shares exactly one edge with $Q_j$ and one edge with $Q_{j+1}$ (where $Q_{m+1}:=Q_1$).
\item $V(C_{4,j})\subseteq \bigcup_{V\in I_1} V$ for each $j=1,\dots,m$.
\end{itemize}}

\smallskip

\noindent
To prove the claim, we first apply the Blow-up lemma (Lemma~\ref{Blow-up})%
    \COMMENT{Don't take $G[V_1,V_2]$ here since then the paths $Q_i$ won't satisfy the first property above}
to $G[V_4,V_5]$ to find $m/2$ vertex-disjoint copies $C_{4,1}, C_{4,3}, \dots, C_{4,m-1}$ of $C_4$.
Now we choose a perfect matching in each of the $C_{4,j}$.
Next apply the Blow-up lemma 
to $G[V_8,V_9]$ to find $m/2$ vertex-disjoint copies $C_{4,2}, C_{4,4}, \dots, C_{4,m}$ of $C_4$.
As before we choose a perfect matching in each of the $C_{4,j}$.
This gives a perfect matching $M_1$ in $G[V_4,V_5]$ and a perfect matching $M_2$ in $G[V_8,V_9]$.
Order the edges of each $M_i$ such that the following properties are satisfied:
\begin{itemize}
\item Edges belonging to the same $C_4$ are consecutive.
\item The edges of $C_{4,j+2}$ come directly after those of $C_{4,j}$ (modulo $m$).
\item The edges of $C_{4,1}$ are the first two edges of $M_1$ and the edges of $C_{4,2}$ are the second and third edge of $M_2$.
\end{itemize}
These matchings induce a new ordering on the vertices in the clusters $V_4,V_5,V_8,V_9$
with which we replace the orderings chosen initially.

Now we apply Corollary~\ref{linking} three times to obtain a system of $m$ vertex-disjoint paths in $\cB(C)$ which for each $j=1,\dots,m$
link the $j$th vertex of $V_{1}$ to the 
$j$th vertex of $V_{4}$, a system of $m$ vertex-disjoint paths which link the $j$th vertex of $V_{5}$ to the 
$j$th vertex of $V_{8}$ and a system of $m$ vertex-disjoint paths which link the $j$th vertex of $V_{9}$ to 
the $j$th vertex of $V_{2k/7+1}$. The union of $M_1$, $M_2$ and all these path systems forms a system $Q_1,\dots,Q_m$ of paths as required in Claim~1.

\medskip

Recall that the paths in $CEPS_3$ join $V_{2k/7+1}$ to $V_{3k/7+1}$.
Now let $Q'_1,\dots,Q'_m$ be a system of vertex-disjoint paths in $\cB(C)$ linking the vertices in $V_{3k/7+1}$ to those $V_{k/2+1}$ (in an arbitrary way).
We let $F_{\rm top}:=Q_1\cup \dots Q_m\cup CEPS_3 \cup Q'_1\cup \dots Q'_m$. 

$F_{\rm low}$ is constructed similarly, but this time we choose
$C'_{4,1}, C'_{4,3}, \dots, C'_{4,m-1}$ in $G[V_{k/2+4},V_{k/2+5}]$ and $C'_{4,2}, C'_{4,4}, \dots, C'_{4,m}$
in $G[V_{k/2+8},V_{k/2+9}]$ (and $F_{\rm low}$ will contain $CEPS_5$). 

So let us now construct $S$. Let $\cB'(C)$ be obtained from $\cB(C)$ by deleting all the edges in $F=F_{\rm top}\cup F_{\rm low}$
and define $G'$ similarly (since this decreases the in- and the outdegrees by at most $1$, the superregularity of the pairs
$G'[V_{i},V_{i+1}]$ is not affected significantly).
Choose a copy $C_{4,m+1}$ of $C_4$ in $G'[V_{12},V_{13}]$. Let $\ell_{m+1}$ and $r_{m+1}$ be a matching in $C_{4,m+1}$.
For each $j=1,\dots,m$ let $\ell_j$ and $r_j$ be the edges of $C_{4,j}$ which are not contained in $F$.

When we refer to an endvertex of an edge $e$ below, this is allowed to be either the initial or the final vertex of~$e$
(it will be clear from the context which one is meant).
For all odd $j$ with $1 \le j < m$, apply Corollary~\ref{linking} to find a system of $m$ vertex-disjoint paths in $\cB'(C)$ linking the endvertex of
$\ell_j$ in $V_5$ to the endvertex of $\ell_{j+1}$ in $V_8$ and the endvertex of $r_j$ in $V_5$ to the 
endvertex of $r_{j+1}$ in $V_8$. Apply Corollary~\ref{linking} again to find a system of $m$ vertex-disjoint paths in $\cB'(C)$ such that
one of these paths links the endvertex of $\ell_m$ in $V_9$ to the endvertex of $\ell_{m+1}$ in $V_{12}$, another path links   
the endvertex of $r_m$ in $V_9$ to the endvertex of $r_{m+1}$ in $V_{12}$ and the remaining $m-2$ paths join the remaining vertices
in $V_9$ to those in $V_{12}$. We also choose a matching $M_3$ in $G'[V_{12},V_{13}]$ which consists of $m-2$ edges and avoids
the endvertices of $\ell_{m+1}$ and $r_{m+1}$ (we find $M_3$ by applying Proposition~\ref{perfmatch} to the subgraph of $G'[V_{12},V_{13}]$
obtained by deleting the endvertices of $\ell_{m+1}$ and $r_{m+1}$).
Next we apply Corollary~\ref{linking} to find a system of $m$ vertex-disjoint paths in $\cB'(C)$ joining the
vertices in $V_{13}$ to those in $V_{3k/7+1}$. 

Let $\cQ$ denote the system of paths obtained from the union of all the paths chosen previously, of $\ell_1,\dots,\ell_{m+1}$, of
$r_1,\dots,r_{m+1}$ and of the edges in $M_3$. So each path in $\cQ$ joins some vertex in $V_4$ to some vertex in $V_{3k/7+1}$.
Consider the system $\cQ'$ of paths obtained by concatenating the paths in $\cQ$ with those in $CEPS_4$ (recall that $CEPS_4$
completely spans the interval $I_4=V_{3k/7+1}\dots V_{4k/7+1}$).
For each even $j$ with $1 \le j \le m$, let $r'_j$ denote the vertex in $V_{4k/7+1}$ which is connected to the edge
$r_j$ by a path in $\cQ'$ and define $\ell'_j$ similarly. 
Apply Corollary~\ref{linking} to find a system $\cQ''$ of $m$ vertex-disjoint paths in $\cB'(C)$ from $V_{4k/7+1}$ to $V_4$ such that
for each even $j$ with $1 \le j < m$ the vertex $r'_j$ is linked to the endvertex of $r_{j+1}$ in $V_4$
and $\ell'_j$ is linked to the endvertex of $\ell_{j+1}$ in $V_4$. Also, $\ell_m'$ is linked to $\ell_1$ and $r_m'$ is linked to $r_1$ (see Figure~\ref{figconstS}).
\begin{figure}
\centering\footnotesize
\includegraphics[scale=0.38]{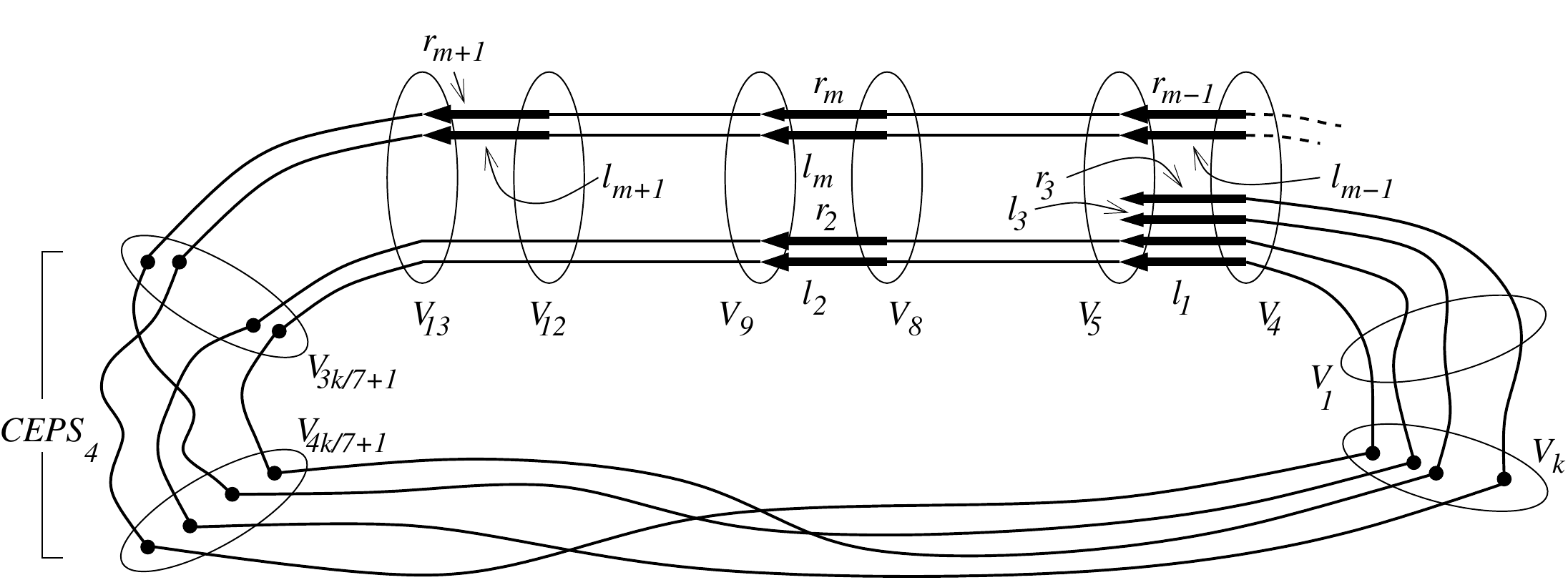}
\caption{The construction of $S$.}
\label{figconstS}
\end{figure}
We let $S$ be the union of all the paths in $\cQ'$ and $\cQ''$. Then $F$ and $S$ satisfy (CyA0)--(CyA3).

$S'$ is constructed similarly. This is possible as the switches $C'_{4,1},\dots,C'_{4,m+1}$ lie in $I_{4,{\rm low}}\subseteq I_4$
whereas the complete exceptional path system $CEPS_2$ contained in $S'$ spans $I_2$ (see Figure~\ref{figceps}).
\endproof


\subsection{The parity switcher} \label{sec:parity}
Suppose that we are given a decomposition $\cD$ of a regular digraph into $r$ edge-disjoint $1$-factors and suppose that the total number of cycles is $K$, say. 
If we carry out $C_4$-exchanges between the
cycles in these $1$-factors, then Proposition~\ref{switch} implies that the resulting total number of cycles either stays the same, increases by two 
or decreases by two after each exchange. So if e.g.~$r$ is odd and $K$ is even, then we will never be able to 
transform $\cD$ into a set of edge-disjoint Hamilton cycles if we rely only on $C_4$-exchanges between cycles in $\cD$.
The following concept of `triple switches' will allow us to change the parity of the total number of cycles in a decomposition.
In particular, the resulting parity switcher will allow us to transform the spanning bicycles which are potentially returned by 
Lemma~\ref{cycleabsorb} into Hamilton cycles.

Let $K_{2,3}$ denote the orientation of a complete bipartite graph with vertex classes of size $2$ and $3$ in which every edge
is oriented towards the vertex class of size~$3$. Note that $K_{2,3}$ is the edge-disjoint union of three matchings $M_1$, $M_2$ and $M_3$
of size~$2$. Given edge-disjoint bicycles $B_1$, $B_2$ and $B_3$ on the same vertex set, we say that they are \emph{triply-switchable} if for each $i=1,2,3$ there are
independent edges $\ell_i$ and $r_i$ lying on different cycles of $B_i$ such that the union of $\ell_i$ and $r_i$ over all $i=1,2,3$ forms a
copy $K_{2,3}^*$ of $K_{2,3}$. We say that $K_{2,3}^*$ is a \emph{$B_1B_2B_3$-switch}. A \emph{$K_{2,3}^*$-exchange} consists of
deleting the edges $\ell_i=:x x_i$ and $r_i=:y y_i$ from $B_i$ and adding the edges $x y_i$ and $y x_i$ for each $i=1,2,3$ (see Figure~\ref{K23}).
Thus for each $i=1,2,3$ one of the edges $\ell_i$, $r_i$ will be added to $B_{i+1}$ and the other edge will be added to $B_{i+2}$
(where the $B_4:=B_1$ and $B_5:=B_2$). 
\begin{figure}
\centering\footnotesize
\includegraphics[scale=0.4]{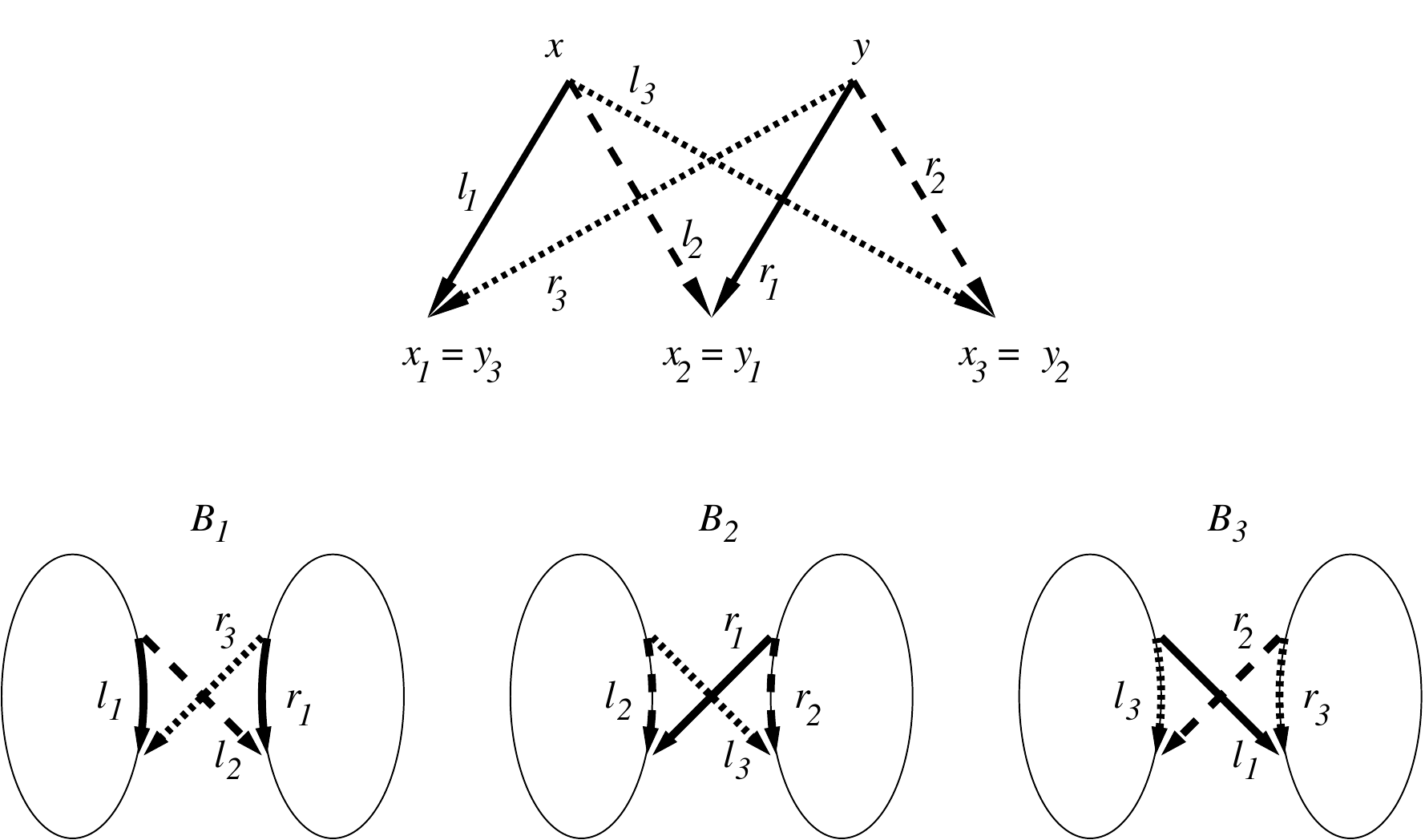}
\caption{Illustrating a $K_{2,3}^*$-exchange.}
\label{K23}
\end{figure}
Note that the digraph obtained from $B_i$ via a $K_{2,3}^*$-exchange is a Hamilton cycle
(for each $i=1,2,3$). Thus the union of three edge-disjoint triply-switchable bicycles has a decomposition into three edge-disjoint
Hamilton cycles as well as a decomposition into six cycles (two for each $B_i$). As mentioned above, this parity difference will enable us to turn the bicycle(s) 
which are potentially returned by Lemma~\ref{cycleabsorb} into Hamilton cycles (without creating any additional bicycles elsewhere).

A \emph{parity extended cycle absorber $PCA$ in $G$} with respect to $C$ is digraph on $V(G)\setminus V_0$ with the following properties:
\begin{itemize} 
\item[(PCA1)] $PCA$ is the union of two digraphs $CyA$ and $TSB$, each with vertex set $V(G)\setminus V_0$.
$CyA=F\cup S\cup S'$ is a cycle absorber in $G$. $TSB$ is the union of three $1$-regular
digraphs $B'_1$, $B'_2$ and $B'_3$ on $V(G)\setminus V_0$ such that each $B'_i$ contains a complete exceptional path system $CEPS(B'_i)$
(but no other exceptional edges) and such that $B'_i$ is a bicycle on $V(G)\setminus V_0$. Let $B_i:=(B'_i)^{\rm orig}$.
Then $CyA^{\rm orig}$, $B_1$, $B_2$ and $B_3$ are pairwise edge-disjoint subdigraphs of $G$.
\item[(PCA2)] $B_1$, $B_2$ and $B_3$ are triply-switchable. Let $K^*_{2,3}$ denote the corresponding $B_1B_2B_3$-switch
and let $\ell^*_i$ and $r^*_i$ denote the edges of $K^*_{2,3}$ contained in $B_i$.
(So $\ell^*_i$ and $r^*_i$ lie on different cycles of $B_i$.) 
\item[(PCA3)] Recall that $\ell_{m+1}$ and $r_{m+1}$ are the two edges of the switch $C_{4,m+1}$ which are contained in $S$
(where $C_{4,m+1}$ is as defined in (CyA2)). Then $B_1$ contains the other two edges $\ell_{S}$ and $r_{S}$ of $C_{4,m+1}$.
Similarly, recall that $\ell'_{m+1}$ and $r'_{m+1}$ are the two edges of the switch $C'_{4,m+1}$ which are contained in $S'$
(where $C'_{4,m+1}$ is as defined in (CyA2$'$)). Then $B_1$ also contains the other two edges $\ell_{S'}$ and $r_{S'}$ of $C'_{4,m+1}$.
\item[(PCA4)] $B_2$ and $B_3$ are pairwise switchable. Let $C_{4,B_2B_3}$-denote the corresponding switch.
\item[(PCA5)] All the switches $K^*_{2,3}$, $C_{4,m+1}$, $C'_{4,m+1}$ and $C_{4,B_2B_3}$ are pairwise
edge-disjoint. Both $C_{4,m+1}$ and $C'_{4,m+1}$ are vertex-disjoint from $CEPS(B'_1)$ while both $K^*_{2,3}$ and $C_{4,B_2B_3}$
are vertex-disjoint from $CEPS(B'_i)$ for all $i=1,2,3$.
\item[(PCA6)] One cycle of the bicycle $B_1$ contains the edges
$\ell_{S},\ell_{S'},\ell^*_1$ in that order while the other cycle of $B_1$ contains $r_{S},r_{S'},r^*_1$
in that order.  
\end{itemize}
An \emph{$s$-fold parity extended cycle absorber $PCA(s)$ with respect to $C$ in $G$} consists of $s$ parity extended cycle absorbers
whose original versions are pairwise edge-disjoint. Note that
$PCA(s)$ is a $6s$-regular spanning subdigraph
of $G^{\rm basic}$ which contains precisely $7s$ complete exceptional path systems (and no other exceptional edges).
Moreover, $PCA(s)^{\rm orig}$ is a spanning subdigraph of $G$ and
\begin{equation} \label{paritydegrees}
d^\pm(x) = 7s \ \ \forall x\in V_0 \ \ \quad \mbox{and} \quad \ \ d^\pm(y)= 6s \ \ \forall y\in V(G)\setminus V_0.
\end{equation} 

\begin{lemma} \label{cycleparity}
Suppose that $0<1/n\ll 1/k\ll \eps\ll d\ll 1$, that $k/14\in\mathbb{N}$ and that $(G,\mathcal{P},R,C)$ is a $(k,m,\eps,d)$-scheme with $|G|=n$.
Suppose that $H$ is an $s$-factor of $G-V_0$ such that $H$ is a blow-up of $C$.
Let $PCA(s)$ be an $s$-fold parity extended cycle absorber with respect to $C$ in $G$ such that $PCA(s)^{\rm orig}$ and
$H$ are edge-disjoint.
Then $H \cup PCA(s)^{\rm orig}$ has a decomposition into~$7s$ edge-disjoint Hamilton cycles of $G$.
Moreover,
each of these Hamilton cycles contains the original version of one of the $7s$ complete exceptional path systems
contained in $PCA(s)^{\rm orig}$.
\end{lemma}
\proof
Let $PCA_1,\dots,PCA_s$ denote the parity extended cycle absorbers contained in $PCA(s)$.
We apply Proposition~\ref{1factor} to find a $1$-factorization of $H$ into $H_1,\dots,H_s$. Since $H$ is a blow-up of $C$,
each $H_i$ winds around~$C$.
We claim that $H_i \cup PCA_i^{\rm orig}$ has a decomposition into Hamilton cycles $C_{7(i-1)+1}, \dots, C_{7i}$ of $G$.

Note that the claim follows if we can prove it for the case $i=1$. So let $CyA$ and $TSB=B'_1\cup B'_2\cup B'_3$ be as defined in (PCA1)--(PCA6). 
Apply Lemma~\ref{cycleabsorb} to obtain a decomposition of $H_1 \cup CyA^{\rm orig}$ into $1$-factors $F_1,\dots,F_4$ satisfying the
following conditions:
\begin{itemize}
\item[(a)] $F_1$ and $F_2$ are Hamilton cycles in $G$.
\item[(b)] Each of $F_3$ and $F_4$ is either a Hamilton cycle or a spanning bicycle in $G$.
If $F_3$ is a bicycle, then one of the cycles contains $\ell_{m+1}$ and the other 
contains $r_{m+1}$. Similarly, if $F_4$ is a bicycle, then one of the cycles contains $\ell'_{m+1}$ and the other 
contains $r'_{m+1}$. 
\item[(c)] Each $F_i$ contains (the original version of) a complete exceptional path system $CEPS_i^*$ 
(which is one of those contained in $PCA(1)^{\rm orig}$).
Moreover, $CEPS_3^*$ spans $I_4$ and $CEPS_4^*$ spans $I_2$  (where $I_i$ is the $i$th interval of the canonical interval partition 
into $7$ intervals defined earlier).
\end{itemize}
If both $F_3$ and $F_4$ are Hamilton cycles, we perform the $K^*_{2,3}$-exchange to decompose $TSB^{\rm orig}=B_1\cup B_2\cup B_3$ into three
edge-disjoint Hamilton cycles of $G$. Then these Hamilton cycles together with $F_1,\dots,F_4$ form a Hamilton
decomposition of $H_1\cup PCA_1^{\rm orig}$.

So suppose next that both $F_3$ and $F_4$ are bicycles. First perform the $C_{4,m+1}$-exchange. This turns
both $F_3$ and $B_1$ into Hamilton cycles. Let $B^1_1$ denote the Hamilton cycle obtained from $B_1$. Then (PCA6) implies that
the edges $\ell_{S'},\ell^*_1,r_{S'},r^*_1$ appear on $B^1_1$ in that order. Next we perform the $C'_{4,m+1}$-exchange. This turns
$F_4$ into a Hamilton cycle and $B^1_1$ into a bicycle $B^2_1$. Note that one of the cycles of  $B^2_1$ contains $\ell^*_1$
while the other cycle contains $r^*_1$ (this is a special case of the argument illustrated in Figure~\ref{figswitches}). 
So we can now perform the $K^*_{2,3}$-exchange. This turns each of $B^2_1$, $B_2$ and $B_3$
into a Hamilton cycle. Altogether this gives a Hamilton decomposition of $H_1\cup PCA_1^{\rm orig}$.

So suppose next that $F_3$ is a bicycle but $F_4$ is a Hamilton cycle. In this case we perform the $C_{4,m+1}$-exchange. This turns
both $F_3$ and $B_1$ into Hamilton cycles. We then perform the $C_{4,B_2B_3}$-exchange. This turns both $B_2$ and $B_3$ into
Hamilton cycles. As before, altogether this gives a Hamilton decomposition of $H_1\cup PCA_1^{\rm orig}$.

The case when $F_3$ is a Hamilton cycle but $F_4$ is a bicycle is similar to the previous case, but we perform
the $C'_{4,m+1}$-exchange instead of the $C_{4,m+1}$-exchange.

The `moreover part' follows since the switches involved in the above argument are edge-disjoint from the complete exceptional path systems
of the bicycles and Hamilton cycles involved in the corresponding exchanges.
More precisely, for the Hamilton cycles originating from the $B_i$, the `moreover part follows from 
(PCA1) and (PCA5). 
For $F_1$ and $F_2$, it follows from (a) and (c) above.
For $F_3$, it follows from (CyA2) and (c), and for $F_4$ it follows from (CyA2$'$) and (c).
\endproof

\begin{lemma} \label{find_pa}
Suppose that $0<1/n\ll 1/k\ll \eps \ll d\ll 1$, that $s/m\ll d$ and that
$k/14,m/2\in \mathbb{N}$. Let $(G,\cP,R,C)$ be a $(k,m,\eps,d)$-scheme
with $|G|=n$. Suppose that $EF_1,\dots,EF_s$ are exceptional factors 
with parameters $(1,7)$ with respect to $C$, $\cP$ whose original versions are pairwise edge-disjoint.
Then there exists an $s$-fold parity extended cycle absorber $PCA(s)$ with respect to $C$ in $G$ such that
the $7s$ complete exceptional path systems contained in $PCA(s)$ are precisely those in $EF_1\cup \dots \cup EF_s$.
\end{lemma}
\proof Choose an additional constant $\eps'$ with $\eps, s/m\ll \eps'\ll d$.
Let $G'$ be the digraph obtained from $G$ by deleting all the edges
in $EF^{\rm orig}_1,\dots,EF^{\rm orig}_s$. Note that $|N^+_G(x)\setminus N^+_{G'}(x)|\le 7s \le (\eps'/3)^2 m$
for every vertex $x$ of $G$ and the analogous condition holds for the inneighbourhoods of $x$.
Thus Lemma~\ref{deletesystem}(ii) implies that
$(G',\cP,R,C)$ is still a $(k,m,\eps',d)$-scheme.

Let us first show how to find a single parity extended cycle absorber $PCA_1$ with respect to $C$ in $G'$.
Let $I_1,\dots,I_7$ be the canonical interval partition of~$C$ into $7$ intervals $I_1,\dots,I_7$ of equal length and
let $CEPS_i$ denote the complete exceptional path system in $EF_1$ which completely spans $I_i$ (for all $i=1,\dots,7$).
Apply Lemma~\ref{find_ca} to find%
     \COMMENT{Formally we apply Lemma~\ref{find_ca} to $G'\cup CEPS^{\rm orig}_2 \cup \dots\cup CEPS^{\rm orig}_5$
and not to $G'$ since these complete exceptional path systems are not complete exceptional path systems for $G'$.
But it seems better not to say this excplicitly.}
a cycle absorber $CyA$ with respect to $C$ in $G'$ which contains $CEPS_2,\dots,CEPS_5$ (in the way described in (CyA0)).

Recall that $\ell_{S}$ and $r_{S}$ denote the edges of the (potential) switch $C_{4,m+1}$ described in (CyA2) and (PCA3) which are not contained in $CyA$.
Similarly, recall $\ell_{S'}$ and $r_{S'}$ denote the edges of the (potential) switch $C'_{4,m+1}$ described in (CyA2$'$) and (PCA3)
which are not contained in $CyA$.  
Recall from Lemma~\ref{find_ca} that $\ell_{S}$ and $r_{S}$ lie in $G'[V_{12},V_{13}]$ and 
that $\ell_{S'}$ and $r_{S'}$ lie in $G'[V_{k/2+12},V_{k/2+13}]$. 
Remove the edges of $CyA^{\rm orig}$ from $G'$ to obtain $G''$. Let $\cB''(C)$ be the union of $G''[V_i,V_{i+1}]$ over all
$i=1,\dots,k$. Proposition~\ref{superslice}(iii) implies that $\cB''(C)$ is a blow-up of $C$ in which every edge of $C$
corresponds to an $[2\sqrt{\eps'},\ge d]$-superregular pair.

When we refer to an endvertex of an edge $e$ below, this is allowed to be either the initial or the final vertex of~$e$
(it will be clear from the context which one is meant).
Our next aim is to find $TSB$. We start by finding $B_1'$. First we apply Proposition~\ref{perfmatch} to choose a
perfect matching in $G''[V_{12},V_{13}]$ which extends $\ell_{S}$ and $r_{S}$.
Now apply Corollary~\ref{linking} to find a system of $m$ vertex-disjoint paths in $\cB''(C)$ which link all vertices
of $V_{13}$ to those in $V_{k/2+12}$ such that the endvertex of $\ell_S$ in $V_{13}$ is linked to the endvertex of $\ell_{S'}$ in $V_{k/2+12}$ and
the endvertex of $r_S$ in $V_{13}$ is linked to the endvertex of $r_{S'}$ in $V_{k/2+12}$. 
Choose a perfect matching in $G''[V_{k/2+12},V_{k/2+13}]$ which extends $\ell_{S'}$ and $r_{S'}$.

Next choose a copy $K^*_{2,3}$ of $K_{2,3}$
in $G''[V_{k/2+16},V_{k/2+17}]$. For each $i=1,2,3$ let $\ell^*_i$ and $r^*_i$ be two independent edges in $K^*_{2,3}$ such that
all these six edges are distinct from each other (and so $K^*_{2,3}$ is the union of all these six edges).
Now apply Lemma~\ref{linking} to find a system of $m$ vertex-disjoint paths in $\cB''(C)$ which link the vertices in
$V_{k/2+13}$ to those in $V_{k/2+16}$ in such a way that the endvertex of $\ell_{S'}$ in $V_{k/2+13}$ is linked to the endvertex of $\ell^*_1$ in $V_{k/2+16}$
and the endvertex of $r_{S'}$ in $V_{k/2+13}$ is linked to the endvertex of $r^*_1$ in $V_{k/2+16}$.
Extend $\ell^*_1$ and $r^*_1$ into a perfect matching of $G''[V_{k/2+16},V_{k/2+17}]$.
Now apply Lemma~\ref{linking} to find a system of $m$ vertex-disjoint paths in $\cB''(C)$ which (arbitrarily) link all vertices
of $V_{k/2+17}$ to those in $V_{5k/7+1}$. Altogether this gives us a system $\mathcal{Q}$ of $m$ vertex-disjoint paths
joining all vertices in $V_{12}$ to all vertices in $V_{5k/7+1}$. 

Now extend this path system by concatenating the paths in $\mathcal{Q}$ with the ones forming $CEPS_6$.
(Recall that $CEPS_6$ spans the interval $I_6=V_{5k/7+1}\dots V_{6k/7+1}$ completely.)
Denote the resulting path system by $\mathcal{Q}'$. Order the vertices of $V_{12}$ so that the first vertex is
the endvertex of $\ell_S$ and the last vertex is the endvertex of $r_S$.
For each $j=1,\dots,m$, let $q_j\in V_{6k/7+1}$ denote the vertex
which is linked to the $j$th vertex of $V_{12}$ by a path in $\mathcal{Q}'$. Thus the path in $\mathcal{Q}'$
ending in $q_1$ contains $\ell_S$, $\ell_{S'}$, $\ell^*_1$ (in that order) while the path in $\mathcal{Q}'$
ending in $q_m$ contains $r_{S}$, $r_{S'}$, $r^*_1$ (in that order).

Now apply Lemma~\ref{linking} to find a system $\mathcal{Q}''$ of $m$ vertex-disjoint paths in $\cB''(C)$ which link each $q_j$ to the $(j+2)$nd vertex
in $V_{12}$ (where the indices are considered modulo~$m$). Let $B'_1$ be the union of all the paths in $\mathcal{Q}'$ and $\mathcal{Q}''$.
Then $B'_1$ is a bicycle on $V(G)\setminus V_0$ and $B_1:=(B_1')^{\rm orig}=\mathcal{Q}\cup CEPS_6^{\rm orig}\cup \mathcal{Q}''$
satisfies (PCA3) and (PCA6).

We next show how to find $B'_2$. Remove all edges of $B_1$ from $G''$ to obtain $G'''$. Let $\cB'''(C)$ be the union of $G'''[V_i,V_{i+1}]$ over all
$i=1,\dots,k$. Extend $\ell^*_2$ and $r^*_2$ into a perfect matching of $G'''[V_{k/2+16},V_{k/2+17}]$.
Choose a copy $C_{4,B_2B_3}$ of $C_4$ in $G'''[V_{k/2+20},V_{k/2+21}]$. Choose two independent edges $\ell^\diamond_2$ and $r^\diamond_2$
of $C_{4,B_2B_3}$ and let $\ell^\diamond_3$ and $r^\diamond_3$ denote the other two edges.
Now apply Lemma~\ref{linking} to find a system of $m$ vertex-disjoint paths in $\cB'''(C)$ which link all vertices
of $V_{k/2+17}$ to those in $V_{k/2+20}$ such that the endvertex of $\ell^*_2$ in $V_{k/2+17}$ is linked to the endvertex of $\ell^\diamond_2$ in $V_{k/2+20}$ and
the endvertex of $r^*_2$ in $V_{k/2+17}$ is linked to the endvertex of $r^\diamond_2$ in $V_{k/2+20}$.
Extend $\ell^\diamond_2$ and $r^\diamond_2$ into a perfect matching of $G'''[V_{k/2+20},V_{k/2+21}]$.
Now apply Lemma~\ref{linking} to find a system of $m$ vertex-disjoint paths in $\cB'''(C)$ which (arbitrarily) link all vertices
of $V_{k/2+21}$ to those in $V_{6k/7+1}$. Altogether this gives us a system $\mathcal{S}$ of $m$ vertex-disjoint paths
joining all vertices in $V_{k/2+16}$ to all vertices in $V_{6k/7+1}$. 

Now extend this path system by concatenating the paths in $\mathcal{S}$ with the ones forming $CEPS_7$.
(Recall that $CEPS_7$ spans the interval $I_7=V_{6k/7+1}\dots V_{1}$ completely.)
Denote the resulting path system by $\mathcal{S}'$. Order the vertices of $V_{k/2+16}$ so that the first vertex is
the endvertex of $\ell^*_2$ and the last vertex is the endvertex of $r^*_2$.
For each $j=1,\dots,m$, let $s_j\in V_{1}$ denote the vertex
which is linked to the $j$th vertex of $V_{k/2+16}$ by a path in $\mathcal{S}'$.
Now apply Lemma~\ref{linking} to find a system $\mathcal{S}''$ of $m$ vertex-disjoint paths in $\cB'''(C)$ which link each $s_j$ to the $(j+2)$nd vertex
in $V_{k/2+16}$ (where the indices are considered modulo~$m$). Let $B'_2$ be the union of all the paths in $\mathcal{S}'$ and $\mathcal{S}''$.
Then $B'_2$ is a bicycle on $V(G)\setminus V_0$.
$B'_3$ can be chosen in a similar way as $B'_2$ except that it contains $\ell^*_3,r^*_3,\ell^{\diamond}_3, r^{\diamond}_3$ and $CEPS_1$ instead of
$\ell^*_2,r^*_2,\ell^\diamond_2, r^\diamond_2$ and $CEPS_7$. Let $TSB:=B_1\cup B_2\cup B_3$ and $PCA_1:=CyA\cup TSB$.
Then $PCA_1$ satisfies (PCA1)--(PCA6)

We now remove the edges of $PCA_1^{\rm orig}$ from $G'$ and repeat the above process to find the remaining $s-1$ edge-disjoint parity
extended cycle absorbers $PCA_2,\dots,PCA_s$.
To see that this is possible, denote the digraph obtained from $G'$ by the removal of the edges of $PCA^{\rm orig}_1,\dots,PCA^{\rm orig}_i$ by $G_i$
(where $i< s$). Note that $|N^+_{G'}(x)\setminus N^+_{G_i}(x)|\le 7s \le \eps' m$
for every vertex $x$ of $G'$ and%
   \COMMENT{get $7s$ instead of $6s$ for the exceptional vertices since $F\subseteq CyA$ contains 2 complete exceptional sequences}
the analogous condition holds for the inneighbourhoods of $x$.
Thus by Lemma~\ref{deletesystem}(ii)
$(G_i,\cP,R,C)$ is still a $(k,m,3\sqrt{\eps'},d)$-scheme.
\endproof

\section{Proof of Theorem~\ref{decomp}}\label{sec:stopseln}

The following lemma shows that we can cover all edges induced by a small exceptional set using a small number of edge-disjoint 
Hamilton cycles.

\begin{lemma} \label{coverV0}
Suppose that 
$0<1/n \ll \eps  \ll \nu\le \tau\ll\alpha\le 1$.
Let  $G$ be a robust $(\nu,\tau)$-outexpander with $\delta^0(G) \ge \alpha n$ and
let $V_0$ be a set of vertices in $G$ with $|V_0| \le \eps n$.
Then there is a set of $\eps n$ edge-disjoint Hamilton cycles in $G$ which contain all edges of $G[V_0]$.
\end{lemma}
\proof
Note that when viewed as an undirected graph, $G[V_0]$ has maximum degree less than $\eps n$.
So by Vizing's theorem, we can partition the edges of $G[V_0]$ into  $t:=\eps n$ matchings
$M_1,\dots,M_t$ (some of these may be empty). 
For each matching $M_i$ in turn, we find a Hamilton cycle $C_i$ which contains all edges of $M_i$ and which is 
edge-disjoint from $C_1,\dots,C_{i-1}$. Suppose that we have found $C_1,\dots,C_{i-1}$.
Let $G_i$ be the graph obtained from $G$ by removing the edges of $C_1,\dots,C_{i-1}$.
Note that $G_i$ is still a robust $(\nu/2,\tau)$-outexpander with $\delta^0(G_i) \ge \alpha n- (i-1) \ge \alpha n/2$.
Let $G_i'$ be the graph obtained from $G_i$ by contracting the edges of $M_i$.
More precisely, we successively replace each directed edge $ab$ of $M_i$ by a vertex whose outneighbours are the current outneighbours of $b$
and whose inneighbours are the current inneighbours of $a$.
Then $G_i'$ is still a robust $(\nu/3,2\tau)$-outexpander with $\delta^0(G_i') \ge \alpha n/3$.
Thus $G_i'$ contains
a Hamilton cycle $C_i'$ by Theorem~\ref{expanderthm}. 
But $C_i'$ corresponds to a Hamilton cycle $C_i$ in $G_i$ (and thus $G$) containing $M_i$. 
So we can continue until we have found $C_1,\dots,C_t$ as required.
\endproof

Before proving Theorem~\ref{decomp}, we will combine the chord absorbing step and the cycle absorbing step
into a single `robust decomposition' lemma. Roughly speaking, this means that we can find a 
sparse `robustly decomposable digraph' $G^{\rm rob}$,
so that for an arbitrary very sparse regular digraph $H$ on the same vertex set, the digraph $H \cup G^{\rm rob}$ always has a Hamilton decomposition.
We will state a variant of this lemma in Section~\ref{sec:rdeclemma}. This variant will be used in~\cite{paper2,paper1}.
The main advantage of combining the chord absorbing step and the cycle absorbing step
into a single lemma is that in future applications one will not
need to know the definitions of the chord absorber and 
(the more involved) definition of the cycle absorber in order to apply it.

\begin{lemma} \label{hcdec}
Suppose that $0<1/n\ll 1/k\ll \eps \ll 1/q \ll 1/f \ll r_1/m\ll d\ll 1/\ell',1/g\ll 1$ and
that $rk^2\le m$. Let%
   \COMMENT{Need $rk^2\le m$ rather than $rk\le m/f^2$ to ensure that $qr_3/m\ll \eps$ which we need to apply Lemma~\ref{findchordabs}.}
$$r_2:=96\ell'g^2kr, \ \ \ r_3:=rfk/q, \ \ \ r^\diamond:=r_1+r_2+r-(q-1)r_3, \ \ \ s':=rfk+7r^\diamond
$$
and suppose that $k/14, k/f, k/g, q/f, m/4\ell', fm/q, 2fk/3g(g-1) \in \mathbb{N}$.
Suppose that $(G,\cP,\cP',R,C,U,U')$ is an $(\ell',k,m,\eps,d)$-setup with $|G|=n$ and $C=V_1\dots V_k$.
Suppose that $\cP^*$ is a $(q/f)$-refinement
of $\cP$ and that $EF_1,\dots, EF_{r_3}$ are exceptional factors with parameters $(q/f,f)$ 
with respect to $C$, $\cP^*$ whose original versions are pairwise edge-disjoint.
Let $\mathcal{EF}$ be the union of the $EF_i$ over all $i=1,\dots,r_3$.
Then there exists a spanning subdigraph $CA^\diamond(r)$ of $G-V_0$ for which the following holds:
\begin{itemize}
\item[(i)] $CA^\diamond(r)$ is
an $(r_1+r_2)$-regular spanning subdigraph of $G-V_0$ which is edge-disjoint from $\mathcal{EF}^{\rm orig}$.
\item[(ii)] Suppose that $EF'_1,\dots, EF'_{r^\diamond}$ are exceptional factors with parameters $(1,7)$
with respect to $C$, $\cP$ such that the original versions of all these exceptional factors are pairwise edge-disjoint
from each other and edge-disjoint from $CA^\diamond(r)\cup \mathcal{EF}^{\rm orig}$.
Let $\mathcal{EF}'$ be the union of the $EF'_i$ over all $i=1,\dots,r^\diamond$.
Then there exists a spanning subdigraph $PCA^\diamond(r)$ of $G-V_0$ for which the following holds:
\begin{itemize}
\item[(a)] $PCA^\diamond(r)$ is
a $5r^\diamond$-regular spanning subdigraph of $G-V_0$ which is edge-disjoint from
$CA^\diamond(r)\cup (\mathcal{EF}\cup \mathcal{EF}')^{\rm orig}$.
\item[(b)] Let $\mathcal{CEPS}$ be the set consisting of all the $s'$ complete exceptional path systems
contained in $\mathcal{EF}\cup \mathcal{EF}'$.
Whenever $H$ is an $r$-factor of $G-V_0$ which is edge-disjoint from $G^{\rm rob}:=CA^\diamond(r)\cup PCA^\diamond(r)\cup (\mathcal{EF}\cup \mathcal{EF}')^{\rm orig}$,
then $H\cup G^{\rm rob}$ has a decomposition into $s'$
edge-disjoint Hamilton cycles $C_1,\dots,C_{s'}$ of $G$.
Moreover, for each $i=1,\dots,s'$, the basic version $C^{\rm basic}_i$ of $C_i$ contains one of the complete exceptional path systems from
$\mathcal{CEPS}$.
\end{itemize}
\end{itemize}
The analogue holds for an $(\ell',k,m,\eps,d)$-bi-setup $(G,\cP,\cP',R,C,U,U')$ if we assume in addition that $H$ is
bipartite with vertex classes $\bigcup \cV_{\rm even}$ and $\bigcup \cV_{\rm odd}$ (where $\cV_{\rm even}$ is
the set of all those $V_i$ such that $i$ is even and $\cV_{\rm odd}$ is defined analogously).
\end{lemma}

Note that the definition of an exceptional factor and~(i) together imply that
in the original version $CA^\diamond(r)\cup \mathcal{EF}^{\rm orig}$ of $CA^\diamond(r)\cup \mathcal{EF}$ we have
\begin{equation}\label{CAdiamonddegrees}
d^\pm(x)=r_3q \ \ \forall x\in V_0 \ \ \ \text{and} \ \ \ d^\pm(y)=r_1+r_2+r_3 \ \ \forall y\in V(G)\setminus V_0.
\end{equation}
Similarly, in the original version $PCA^\diamond(r)\cup (\mathcal{EF}')^{\rm orig}$ of $PCA^\diamond(r)\cup \mathcal{EF}'$ we have
\begin{equation}\label{PCAdiamonddegrees}
d^\pm(x)=7r^\diamond \ \ \forall x\in V_0 \ \ \ \text{and} \ \ \ d^\pm(y)=6r^\diamond \ \ \forall y\in V(G)\setminus V_0.
\end{equation}
In order to construct $CA^\diamond(r)$ we will choose a chord absorber $CA$ with $CA^{\rm exc}=\mathcal{EF}$.
Similarly, in order to construct $PCA^\diamond(r)$ we will choose a parity extended cycle absorber $PCA(r^\diamond)$ such that the complete
exceptional path systems contained in $PCA(r^\diamond)$ are those in $\mathcal{EF}'$.

\removelastskip\penalty55\medskip\noindent{\bf Proof of Lemma~\ref{hcdec}. }
Choose new constants $\eps_1,d_1$ such that $\eps \ll \eps_1\ll 1/q \ll 1/f \ll r_1/m\ll d_1\ll d$. 
Note that
\begin{equation}\label{eq:rs}
r/m,r_2/m,qr_3/m\ll \eps \ \ \ \text{and} \ \ \ r^\diamond\le d_1m.
\end{equation}
We first apply Lemma~\ref{findchordabs} with $\eps_1,r_1,r_2,r_3$ playing the roles of $\eps',r_0,r'_0, r''_0$ to find a chord absorber $CA$ for $C$, $U'$
with parameters $(\eps_1,r_1,r_2,r_3,q,f)$ such that $CA^{\rm exc}= \mathcal{EF}$.
Let $CA^\diamond(r):=CA\setminus CA^{\rm exc}$. Then~(\ref{chorddegrees}) and the fact that 
$CA^{\rm exc}$ is $r_3$-regular imply that $CA^\diamond(r)$ satisfies~(i).

Note that by definition of a setup, $(G,\cP,R,C)$ is a $(k,m,\eps,d)$-scheme.
Let $G_1$ be obtained from $G$ by deleting all edges in $CA^{\rm orig}$. Thus~(\ref{chorddegrees}) implies that $G_1$ is obtained from
$G$ by deleting $r_3q$ outedges and $r_3q$ inedges at every vertex in $V_0$ and deleting $r_1+r_2+r_3$ outedges and
$r_1+r_2+r_3$ inedges at every vertex in $V(G)\setminus V_0$. But $r_3q\le \eps m$ and $r_1+r_2+r_3\le d_1m$ by~(\ref{eq:rs}).
So Lemma~\ref{deletesystem} implies
that $(G_1,\cP,R,C)$ is still a $(k,m,3\sqrt{d_1},d)$-scheme.
Since $r^\diamond/m \le d_1 \ll d$ by~(\ref{eq:rs}), we can apply Lemma~\ref{find_pa}
to $(G_1,\cP,R,C)$ to obtain an $r^\diamond$-fold parity extended cycle absorber $PCA(r^\diamond)$ such that the complete exceptional path systems
contained in $PCA(r^\diamond)$ are precisely those in $\mathcal{EF}'$.
Let $PCA^\diamond(r):=PCA(r^\diamond)\setminus \mathcal{EF}'$.
Then~(\ref{paritydegrees}) implies that  $PCA^\diamond(r)$ satisfies~(ii)(a).

To check~(ii)(b), suppose that $H$ is an $r$-factor of $G-V_0$ which is edge-disjoint from
$G^{\rm rob}=CA^{\rm orig}\cup PCA(r^\diamond)^{\rm orig}$.
Note that an $(\ell',k,m,\eps,d)$-setup is also an $(\ell',k,m,\eps_1,d)$-setup.
So we can apply Lemma~\ref{absorballH} to the $(\ell',k,m,\eps,d)$-setup $$(G,\cP,\cP',R,C,U,U')$$
and the chord absorber $CA=\cB(C)^*\cup \cB(U')\cup CA^{\rm exc}$ with parameters $(\eps_1,r_1,r_2,r_3,q,f)$
chosen before. This gives us a set $\cC_1$ of $rfk$ edge-disjoint
Hamilton cycles in $G$ such that the following conditions hold:
\begin{itemize}
\item Altogether the Hamilton cycles in $\cC_1$ contain all the edges of $H\cup \cB(U')\cup (CA^{\rm exc})^{\rm orig}$.
Moreover, all remaining edges of these Hamilton cycles are contained in $\cB(C)^*$.
\item The digraph $H_1$ obtained from $CA^{\rm orig}$ by deleting all the edges lying on Hamilton cycles in $\cC_1$
is a regular blow-up of $C$ of degree $(r_1+r_2+r-(q-1)rfk/q)=r^\diamond$.
\item The basic version of each cycle in $\cC_1$ contains one of the $s$ complete exceptional path systems contained
in $CA^{\rm exc}= \mathcal{EF}$.
\end{itemize}
Finally, we apply Lemma~\ref{cycleparity} to the $(k,m,\eps,d)$-scheme $(G,\cP,R,C)$
with $H_1$ playing the role of $H$ and with the $r^\diamond$-fold parity extended cycle absorber
$PCA(r^\diamond)$ chosen before to find a Hamilton decomposition $\cC_2$ of $H_1\cup PCA(r^\diamond)^{\rm orig}$
such that each Hamilton cycle in $\cC_2$ contains one of the $7r^\diamond$ complete exceptional path systems contained
in $PCA(r^\diamond)$, and thus one of the complete exceptional path systems contained in $\mathcal{EF}'$.
Then $\cC_1\cup \cC_2$ is a Hamilton decomposition of $H\cup G^{\rm rob}$ as required in~(ii)(b).

For the bipartite analogue, we apply the bipartite version of Lemmas~\ref{findchordabs} and~\ref{absorballH} in the above argument.
Since a $(\ell',k,m,\eps,d)$-bi-setup is a $(k,m,\eps,d)$-scheme, the remainder of the proof is the identical.
\endproof 

Finally, we can put together all of the previous results in order to prove Theorem~\ref{decomp}. First we apply Szemer\'edi's regularity lemma to $G$.
Based on the resulting partition, we find a consistent system and a universal walk (so that we have a corresponding setup).
Within these, we find and remove a preprocessing graph $PG$, a chord absorber $CA$ and a parity extended cycle absorber $PCA$.
(Actually, instead of choosing $CA$ and $PCA$ separately, we choose suitable exceptional factors and apply Lemma~\ref{hcdec} to find a 
robustly decomposable graph $G^{\rm rob}$ which is essentially the union of $CA$ and $PCA$.)
Then we apply Lemma~\ref{coverV0} to cover all edges of $G[V_0]$. Next we  
find an approximate Hamilton decomposition using Theorem~\ref{approxdecomp} in the remainder $G^\#$ of $G$, which leaves a
sparse leftover $H_0$. We then find a Hamilton decomposition of $H_0 \cup PG^{\rm orig} \cup G^{\rm rob}$.

\removelastskip\penalty55\medskip\noindent{\bf Proof of Theorem~\ref{decomp}. }
Let $\tau^*:=\tau(\alpha/2)$ be as defined in Theorem~\ref{approxdecomp}. Choose $\tau\le \tau^*$ 
such that $0 < \tau \ll \alpha$. 
Note that whenever $\nu'\le \nu$, every robust $(\nu,\tau)$-outexpander is also a robust $(\nu',\tau)$-outexpander.
So we may assume that $0 \ll \nu \ll \tau$.
Choose $n_0 \in \mathbb{N}$ so that $0 < 1/n_0 \ll \nu$.
($\tau$ and $n_0$ will be the constants returned by Theorem~\ref{decomp}.)
Now let $G$ be an $r$-regular robust $(\nu,\tau)$-outexpander on $n\ge n_0$ vertices with $r\ge \alpha n$.
We have to show that $G$ has a Hamilton decomposition. Choose additional positive constants so that 
\begin{align} \label{hierarchy}
0<1/n_0 & \ll \eta\ll d'_0\ll d''_0 \ll 1/k^*_2 \ll 1/k^*_1\ll \eps_0 \ll \eps_1 \ll \eps_2 \ll \eps \ll \eps'\ll \nonumber\\
 & \ll 1/q \ll 1/f \ll d_1 \ll  d \ll \nu \ll \tau \ll \alpha<1  \ \ \text{and} \ \ d\ll 1/g\ll \zeta\le 1/2,
\end{align}
and where $n_0,k_1^*,k_2^*,q,f,g \in \mathbb{N}$. Let
\begin{equation} \label{defconst}
\ell':=\frac{64\cdot 36}{\nu^6}, \ \  s:=\frac{64\cdot 10^7}{\nu^6}  \ \ \text{and} \ \ \ell^*:=f^2.
\end{equation}
Note that we can choose $q,f \in \mathbb{N}$ and $\nu$ in such a way that $q/f, 50 \ell^*/(s-1),\ell^*/7 \in \mathbb{N}$.
As before, we can replace $\nu$ with a suitable $\nu'<\nu$ if necessary (and then prove the theorem for all $(\nu',\tau)$-outexpanders),
so we can ensure that these divisibility conditions hold.

Apply Szemer\'edi's regularity lemma (Lemma~\ref{regularity_lemma}) 
with parameters $\eps_0,d, k^*_1$ to obtain a partition $\mathcal{P}_0=\{V'_0,\dots,V'_{k_0}\}$ of the vertices of $G$
into $k_0$ clusters and an exceptional set $V'_0$, where $1/k^*_2 \ll 1/k_0 \le 1/k^*_1$ and $|V'_0| \le \eps_0 n$.
Note that by adding at most $42g(g-1)f$ clusters to the exceptional set if necessary, we may assume that 
$k_0/14,k_0/f, k_0/g, 2fk_0/3g(g-1) \in\mathbb{N}$. Moreover, by moving at most $\ell^*$ vertices from each cluster $V'_i$ to the exceptional
set we may assume that the cluster size is divisible by $\ell^*$. Note that we still have that $|V'_0| \le 2\eps_0 n$.
Let $R_0$ denote the corresponding reduced digraph.
So every edge of $R_0$ corresponds to an $(2\eps_0,\ge d)$-regular pair. 
Let $$k:= \ell^* k_0.$$
So $$
1/k_2^* \le 1/k \le 1/k_1^*
$$
and $k/14,k/f, k/g,  2fk/3g(g-1), 50k/(s-1) \in\mathbb{N}$.
By Lemma~\ref{robustR}, $R_0$ is a robust $(\nu/2,2\tau)$-outexpander with minimum semidegree at least $\alpha k_0/2$. 
So by Theorem~\ref{expanderthm}, $R_0$ contains a Hamilton cycle $C_0$. 

Apply Lemma~\ref{randompartition} with $C_0$ playing the role of $C$ to obtain an $\eps_0$-uniform 
$\ell^*$-refinement $\mathcal{P}=\{V'_0,V_1\dots,V_k\}$ of $\mathcal{P}_0$.
Let $R$ be the digraph obtained from $R_0$ by replacing every $V'_i$ by the $\ell^*$ subclusters in $\cP$ which are contained
in $V'_i$ and by replacing every edge $V'_iV'_j$ of $R_0$ by a complete bipartite graph $K_{\ell^*,\ell^*}$ between
the two corresponding sets of subclusters in $\cP$, where all the edges of $K_{\ell^*,\ell^*}$ are oriented towards
those subclusters which are contained in $V'_j$. So $R$ is an $\ell^*$-fold blow-up of $R_0$ and $|R|=k$.
By Lemma~\ref{randompartition}(ii) each edge of $R$ corresponds to an $(\eps_1,\ge d)$-regular pair in $G$.
Moreover, Lemma~\ref{expanderblowup} implies that $R$ is still a robust $(\nu^3/8,4\tau)$-outexpander with minimum semidegree at least $\alpha k/2$.
Let $C$ be a Hamilton cycle in $R$ obtained from $C_0$ by winding $\ell^*$ times around $C_0$.
By relabeling the clusters in $\cP$ if necessary, we may assume that $C=V_1\dots V_k$.
Later on will use that this construction satisfies (CSys8)
with $1/2$ playing the role of $\theta$ (since $\cP$ is an $\eps_0$-uniform $\ell^*$-refinement of $\cP_0$).

Apply Lemma~\ref{buildU} to $R$ and $C$ in order to obtain a universal walk $U$ for $C$ with parameter $\ell'$.
Let $H$ be the spanning subgraph of $R$ which consists of all the edges contained in $C\cup U$.
Thus $\Delta(H)\le 2(1+\ell')\le 1/\nu^7$.
Since each edge of $R$ corresponds to an $(\eps_1,\ge d)$-regular pair in $G$, Lemma~\ref{superreg}
implies that we can move $\sqrt{\eps_1} |V_i|$ vertices from each $V_i$ to the exceptional set $V'_0$
to achieve that every edge of $H$ (and thus of $C\cup U$) corresponds to an $[\eps_2/2,\ge d]$-superregular pair%
    \COMMENT{Below we will move some vertices from the clusters to $V_0$. So cannot write $\eps_2$ instead of $\eps_2/2$.}
in $G$. We denote the modified exceptional set by $V_0$ and still denote the modified clusters by $V_1,\dots,V_k$.
From now on, we will view $C$ as a cycle on these clusters and $U$ as a universal walk on these clusters.
We also still write $\cP$ for the partition of $V(G)$ into the exceptional set $V_0$ and clusters $V_1,\dots,V_k$.
Let $$m:=|V_1|=\dots=|V_k|.$$ 
By adding at most $200\ell'q/f$ vertices from each cluster in $\cP$ to $V_0$ if necessary, we may assume that $m/50,m/4\ell',fm/q \in \mathbb{N}$
(in particular, $m$ is even).
Note that
\begin{equation} \label{boundV0}
|V_0| \le |V_0'|+\sqrt{\eps_1} n +200\ell'q/f \le  \eps_2 n
\end{equation}
and that every edge of $R$ still corresponds to an $(\eps_2,\ge d)$-regular pair in $G$ and every edge of $C\cup U$ still corresponds to an
$[\eps_2,\ge d]$-regular pair in $G$. Let
\begin{equation}\label{defconst1}
r_0:=\eta r, \ \ r'_0:=d'_0m, \ \ r''_0:=d''_0m   \ \ \text{and} \ \  r_1:=d_1m.
\end{equation}
(\ref{hierarchy}) and the fact that $\eta \alpha n\le r_0=\eta r\le \eta n$ together imply that 
\begin{equation}\label{eq:r}
1/n_0\ll r_0/m\ll  r'_0/m\ll r''_0/m\ll 1/k.
\end{equation}
Let $G_0$ denote the digraph obtained from $G$ by deleting all the edges between vertices in $V_0$.
Then $$(G_0,\cP_0,R_0,C_0,\cP,R,C)$$ is a consistent $(\ell^*,k,m,\eps_2,d,\nu^3/8,4\tau,\alpha/2,1/2)$-system.

Apply Lemma~\ref{randompartition}
to obtain an $\eps_2$-uniform $\ell'$-refinement $\cP'$ of $\cP$.
Let $U'$ be the universal subcluster walk with respect to $C$, $U$ and $\cP'$. (So $U'$ satisfies (ST2).)
Then $$(G_0,\cP,\cP',R,C,U,U')$$ is an $(\ell',k,m,\eps,d)$-setup.
Here we use that Lemma~\ref{randompartition}(i) implies that (ST3) is satisfied.

Apply Lemma~\ref{randompartition} to obtain an $\eta$-uniform $50$-refinement $\cP''$ of $\cP$.
Apply Lemma~\ref{findprepro} to the consistent system $(G_0,\cP_0,R_0,C_0,\cP,R,C)$ to find a preprocessing graph $PG$ in $G_0^{\rm basic}$
with parameters $(s,\eps',d,r'_0,r''_0,r_0,\zeta)$ with respect to $C$, $R$, $\cP''$.
(Here~(\ref{hierarchy}) and~(\ref{eq:r}) imply that the conditions of the lemma are satisfied.)

Let $G_1$ be obtained from $G_0$ by deleting all edges in $PG^{\rm orig}$. Thus~(\ref{predegrees}) implies that $G_1$ is obtained from $G_0$
by deleting $r_0(s-1)$ outedges and $r_0(s-1)$ inedges at every vertex in $V_0$ and by deleting $r''_0$ outedges and $r''_0$ inedges at every
vertex in $V(G)\setminus V_0$. But $r_0(s-1), r''_0\le \eps m$ by~(\ref{hierarchy}) and~(\ref{eq:r}). So Lemma~\ref{deletesystem}(i) implies that 
$$(G_1,\cP_0,R_0,C_0,\cP,R,C)$$ is still a consistent $(\ell^*,k,m,3\sqrt{\eps},d,\nu^3/16,4\tau,\alpha/4,1/4)$-system.
Furthermore, Lemma~\ref{deleteunivsystem} implies
that $(G_1,\cP,\cP',R,C,U,U')$ is still an $(\ell',k,m,\eps^{1/3},d)$-setup.
Let
$$
r^*:= r''_0-(s-1) r_0, \ \ r_2:=96\ell' g^2kr^* \ \ \text{and} \ \ r_3:=\frac{r^*fk}{q}.
$$
Note that
\begin{align}
r^*k^2  \le r''_0 k^2 & \le r''_0 (k^*_2)^2=d''_0m(k^*_2)^2\le d_1m= r_1 \label{eq:r*a}\\
  & \le m. \label{eq:r*b}
\end{align}
In particular, (\ref{eq:r*a}) implies that 
\begin{equation} \label{r123}
r^*,r''_0,r_2,r_3 \le r_1 \stackrel{(\ref{defconst1})}{=} d_1m.
\end{equation}
So
\begin{equation}\label{eq:r's}
\frac{r_2}{m} \stackrel{(\ref{r123})}{\le} 
d_1 \stackrel{(\ref{hierarchy})}{\ll} d \ \ \ 
\text{and} \ \ \ \frac{r_3}{m}\le  \frac{r^* k}{m} \stackrel{(\ref{eq:r*a})}{\le} \frac{r''_0k^*_2}{m}
\stackrel{(\ref{defconst1})}{=} d''_0k^*_2 \stackrel{(\ref{hierarchy})}{\ll} \eps.
\end{equation} 
Also,  note that (\ref{defconst}) implies that $f/\ell^*=1/f \ll 1$.  Moreover, $r_3q/fm=r^*k/m\ll d$ by~(\ref{eq:r's}).
Apply Lemma~\ref{randompartition} to obtain an $\eps$-uniform $q/f$-refinement $\cP'''$ of $\cP$.
Apply Lemma~\ref{exceptseq} to $(G_1,\cP_0,R_0,C_0,\cP,R,C)$ in order
to obtain exceptional factors $EF_1,\dots,EF_{r_3}$ with parameters $(q/f,f)$
with respect to $C$, $\cP'''$ such that the original versions of all these exceptional factors are pairwise edge-disjoint.
Let $\mathcal{EF}$ denote the union of the $EF_i$ over all $i=1,\dots,r_3$.
Let
$$r^\diamond:=r_1+r_2+r^*-(q-1)r_3.$$
Note that 
\begin{equation}\label{eq:rdiam}
r^\diamond/m\le (r_1+r_2+r^*)/m \stackrel{(\ref{r123})}{\le} 3d_1  \stackrel{(\ref{hierarchy})}{\ll} d.
\end{equation}
(\ref{eq:r*b}) guarantees that we can now apply Lemma~\ref{hcdec} to the $(\ell',k,m,\eps^{1/3},d)$-setup $(G_1,\cP,\cP',R,C,U,U')$
with $r^*$ playing the role of $r$ and with the exceptional factors $EF_1,\dots,EF_{r_3}$
chosen before to find a spanning subdigraph $CA^\diamond(r^*)$.

Let $G_2$ be obtained from $G_1$ by deleting all edges in $CA^\diamond(r^*)\cup \mathcal{EF}^{\rm orig}$. Thus (\ref{CAdiamonddegrees})
implies that $G_2$ is obtained from $G_1$ by deleting $r_3q$ outedges and $r_3q$ inedges at every vertex in $V_0$ and by deleting $r_1+r_2+r_3$ outedges and
$r_1+r_2+r_3$ inedges at every
vertex in $V(G)\setminus V_0$. But $r_3q=r^*fk\le d_1m$ by (\ref{hierarchy}) and~(\ref{eq:r*a}) while $r_1+r_2+r_3\le 3d_1m$ by~(\ref{r123}).
So Lemma~\ref{deletesystem} implies that 
$(G_2,\cP_0,R_0,C_0,\cP,R,C)$ is still a consistent $(\ell^*,k,m,3\sqrt{3d_1},d,\nu^3/32,4\tau,\alpha/8,1/8)$-system.
Together with~(\ref{eq:rdiam}) this shows that we can apply Lemma~\ref{exceptseq} to $(G_2,\cP_0,R_0,C_0,\cP,R,C)$ in order to obtain
exceptional factors $EF'_1,\dots,EF'_{r^\diamond}$ with parameters $(1,7)$
with respect to $C$, $\cP$ such that the original versions of all these exceptional factors are pairwise edge-disjoint.
Let $\mathcal{EF}'$ denote the union of the $EF'_i$ over all $i=1,\dots,r^\diamond$.

Let $PCA^\diamond(r^*)$ be as guaranteed by Lemma~\ref{hcdec}(ii).
Let $$
G^{\rm rob}:=CA^\diamond(r^*)\cup PCA^\diamond(r^*)\cup (\mathcal{EF}\cup \mathcal{EF}')^{\rm orig}  \ \ \ \mbox{and}  \ \ \ G^{\rm absorb}:=PG^{\rm orig}\cup G^{\rm rob}.
$$
Let $r^{\rm abs}_0$ be the outdegree of the exceptional vertices (i.e.~those in $V_0$) in $G^{\rm absorb}$. 
Then (\ref{predegrees}), (\ref{CAdiamonddegrees}) and (\ref{PCAdiamonddegrees}) imply that 
$r^{\rm abs}_0$ is also the indegree of the exceptional vertices. Moreover, they imply that
$$
r_0^{\rm abs}= r_0(s-1) + r_3 q+ 7r^\diamond.
$$
Let $r^{\rm abs}$ be the outdegree of the non-exceptional vertices in $G^{\rm absorb}$.
Again, (\ref{predegrees}), (\ref{CAdiamonddegrees}) and (\ref{PCAdiamonddegrees}) imply that $r^{\rm abs}$ is also the indegree of the non-exceptional vertices 
Moreover,
$$
r^{\rm abs}= r''_0+ (r_1+r_2+r_3)+ 6r^\diamond.
$$
But
\begin{align*}
r_0^{\rm abs}-r^{\rm abs}& =r_0(s-1)+(q-1)r_3+r^\diamond-r''_0-r_1-r_2\\
& =r_0(s-1)+(q-1)r_3+\left(r_1+r_2+r^*-(q-1)r_3\right)-r''_0-r_1-r_2\\
& =r_0(s-1)+r^*-r''_0=0.
\end{align*}
So $G^{\rm absorb}$ is $r^{\rm abs}$-regular. 
Moreover, 
\begin{equation} \label{boundrabs}
r^{\rm abs} \le r_0'' + (r_1+r_2+r_3)+6(r_1+r_2+r^*) \stackrel{(\ref{r123})}{\le} 22d_1m \stackrel{(\ref{hierarchy})}{\ll} dm \le dn.
\end{equation}
Let $G^\triangle$ denote the digraph obtained from $G$ by removing the edges of $G^{\rm absorb}$.
Let $r^\triangle :=r-r^{\rm abs}$ be the degree of $G^\triangle$.
Note that (\ref{boundrabs}) implies that $G^\triangle$ is still a robust $(\nu/2,\tau)$-outexpander with 
$\delta^0(G^\triangle)=r^\triangle \ge \alpha n/2$. 
Let $r^\#:=r^\triangle-\eps n$. Then~(\ref{boundV0}) implies that
we can apply Lemma~\ref{coverV0} to obtain a set $\cC^\triangle$ of $\eps n$ edge-disjoint Hamilton cycles in $G^\triangle$
which cover all the edges in $G^\triangle[V_0]$. Let $G^\#$ be the $r^\#$-regular digraph obtained from $G^\triangle$ by deleting all the
edges in these Hamilton cycles.

Note that $r-r^\#=r^{\rm abs} +\eps n \le dn + \eps n \le 2d n$ by (\ref{boundrabs}). Thus $G^\#$ is still a robust $(\nu/2,\tau)$-outexpander
(and thus also a robust $(\nu/2,\tau^*)$-outexpander) with 
$\delta^0(G^\#)=r^\# \ge \alpha n/2$. Together with our choice of $\tau^*, \nu, \eta$ and $n_0$ this shows that we can apply Theorem~\ref{approxdecomp}
to obtain a set $\cC^\#$ of $r^\#-\eta r=r^\#-r_0$ edge-disjoint Hamilton cycles in $G^{\#}$. Let $H_0$ be the digraph obtained
from $G^{\#}$ by deleting all the edges in these Hamilton cycles. Note that our application of Lemma~\ref{coverV0} ensures that $V_0$
forms an independent set in $H_0$ and so $H_0$ is a $r_0$-regular subdigraph of $G_0$. Together with
(\ref{hierarchy}) and~(\ref{eq:r}) this ensures that we can apply Corollary~\ref{preprocor} to $(G_0,\cP_0,R_0,C_0,\cP,R,C)$ with $H_0$, $\cP''$ playing the roles of $H$, $\cP'$
and with the preprocessing graph $PG$ with parameters $(s,\eps',d,r'_0,r''_0,r_0,\zeta)$
chosen before to obtain a set $\cC_0$ of $r_0s$ edge-disjoint Hamilton cycles such that
the following conditions hold:
\begin{itemize}
\item Altogether the Hamilton cycles in $\cC_0$ contain all edges of $H_0$ and each of these Hamilton cycles lies in $H_0\cup PG^{\rm orig}$.
\item Let $PG'$ be the digraph obtained from $PG^{\rm orig}$ by deleting all the edges lying in the Hamilton cycles in
$\cC_0$. Then every vertex $x\in V_0$ is isolated in $PG'$ and every vertex $x\in V(G)\setminus V_0$ has in- and outdegree $r''_0-(s-1)r_0=r^*$ in $PG'$.
\end{itemize}
The second condition above implies that $H_1:=PG'-V_0$ is an $r^*$-regular subdigraph of $G_0-V_0$.
Lemma~\ref{hcdec}(ii)(b) guarantees that $H_1\cup G^{\rm rob}$
has a Hamilton decomposition $\cC_1$.
Then $\cC^\triangle\cup \cC^\#\cup \cC_0\cup \cC_1$ is a Hamilton decomposition of $G$, as required.
\endproof


\section{Statement of the robust decomposition lemma for further use}\label{sec:rdeclemma}

We now present a `standalone' version (Lemma~\ref{rdeclemma}) of the robust decomposition lemma (Lemma~\ref{hcdec})
which is suitable for further use. 
Instead of exceptional edges and exceptional path systems, it involves `fictive edges' and `special path systems'.
One can use these in the same way as in the current proof to deal with edges at exceptional vertices.
A crucial additional advantage is that one can also use them to deal with a small number of edges connecting $G$ to another digraph.
In particular, in~\cite{paper1} we can use it for a digraph $G^*$ which consists of robust expanders $G$ and $G'$ which are connected by a small number of edges
(so $G^*$ is not a robust expander).
Similarly, in the bipartite version, we can apply it to an `almost bipartite' digraph and use the fictive edges to deal e.g.~with the small number of edges 
which do not respect the (approximate) bipartition. This is the case in~\cite{paper2}.

Suppose that $(G,\cP,R,C)$ is an $(k,m,\eps,d)$-scheme with $C=V_1\dots V_k$.
The next definition is a generalization of a complete exceptional path system.
Suppose that $k/L, m/K\in\mathbb{N}$ and let $\cI$ be the canonical interval partition of $C$ into $L$
intervals of equal length. A \emph{special path system $SPS$} (with respect to $C$) 
with parameters $(K,L)$ spanning an interval $I=U_jU_{j+1}\dots U_{j'}$ with $I \in \cI$ consists of $m/K$
vertex-disjoint paths $P_1,\dots,P_{m/K}$ such that the following conditions hold.
\begin{itemize}
\item[(SPS1)] Every $P_s$ has its initial vertex in $U_j$ and its final vertex in $U_{j'}$.
\item[(SPS2)] $SPS$ contains a matching ${\rm Fict}(SPS)$ such that all the edges in ${\rm Fict}(SPS)$ avoid the
endclusters $U_j$ and $U_{j'}$ of $I$ and such that $E(P_s)\setminus {\rm Fict}(SPS)\subseteq E(G)$.
\item[(SPS3)] $SPS$ contains precisely $m/K$ vertices from every cluster in $I$ and no other vertices.
\end{itemize}
The edges in ${\rm Fict}(SPS)$ are called \emph{fictive edges of} $SPS$.
Note that a complete exceptional path system $CEPS$ containing a complete exceptional sequence $CES$ is
a special path system where $CES$ plays the role of the set of fictive edges.

Suppose that $\cP^*$ is a $K$-refinement of $\mathcal{P}$.
For each cluster $U\in \mathcal{P}$, let $U(1),\dots,U(K)$ denote the subclusters of $U$ in $\mathcal{P}^*$. 
Consider a special path system $SPS$ as above. We say that $SPS$ has \emph{style $b$} if 
its vertex set is $U_{j}(b) \cup \dots \cup U_{j'}(b)$.
A \emph{special factor $SF$ with parameters $(K,L)$} (with respect to $C$, $\cP^*$) is a $1$-regular digraph on $V(G)\setminus V_0$
satisfying the following properties:
\begin{itemize}
\item[(SF1)] On each of the $L$ intervals $I\in\cI$, $SF$ induces the vertex-disjoint union of $K$ special path systems.
\item[(SF2)] Moreover, for each $I \in \cI$ and each $b=1,\dots, K$, exactly one of the special path systems in $SF$ spanning $I$ has style $b$.
\end{itemize}
We write ${\rm Fict}(SF)$ for the union of the sets ${\rm Fict}(SPS)$ over all the 
$KL$ special path systems $SPS$ contained in $SF$ and call the edges in ${\rm Fict}(SF)$ \emph{fictive edges of $SF$}.
Note that an exceptional factor $EF$ is a special factor
where the  exceptional edges in $EF$ play the role of the fictive edges.

We will always view fictive edges as being distinct from each other and from the edges in other digraphs.
So if we say that $SF_1,\dots,SF_r$ are pairwise edge-disjoint from each other and from some digraph $Q$ on $V(G)\setminus V_0$,
then this means that $Q$ and all the $SF_i\setminus {\rm Fict}(SF_i)$
are pairwise edge-disjoint, but for example there could be an edge from $x$ to $y$ in $Q$ as well as in ${\rm Fict}(SF_i)$ for five indices $i$ (say).
But these are the only instances of multiedges that we allow, i.e.~if there is more than one edge from $x$ to $y$, then all but
at most one of these edges are fictive edges.

Given two multidigraphs $M$ and $M'$ on the same vertex set, we write $M+M'$ for the multidigraph whose vertex set is $V(M)=V(M')$
and in which the multiplicity of $xy$ is the sum of the multiplicities of $xy$ in $M$ and in $M'$ (for all $x,y\in V(M)$).
So in the above example $Q+SF_1+\dots+SF_r$ contains six edges from $x$ to $y$.

We can now state the variant of the robust decomposition lemma.
The proof is the same as that of Lemma~\ref{hcdec} -- the special factors play the role of the exceptional factors
and the Hamilton cycles in Lemma~\ref{rdeclemma} correspond to the basic versions of the Hamilton cycles returned by
Lemma~\ref{hcdec}. The existence of fictive edges means that we formally consider multidigraphs (rather than digraphs) at several steps.
However, this does not affect the argument. Indeed,
fictive edges only occur within (pre-defined) special path systems and these are fixed building blocks that are never 
modified during the construction of the Hamilton cycles.
The only other difference is that in Lemma~\ref{rdeclemma}, $H$ need not be a subdigraph of $G$,
but this does not affect the proof either.
 
\begin{lemma} \label{rdeclemma}
Suppose that $0<1/n\ll 1/k\ll \eps \ll 1/q \ll 1/f \ll r_1/m\ll d\ll 1/\ell',1/g\ll 1$ and
that $rk^2\le m$. Let
$$r_2:=96\ell'g^2kr, \ \ \ r_3:=rfk/q, \ \ \ r^\diamond:=r_1+r_2+r-(q-1)r_3, \ \ \ s':=rfk+7r^\diamond
$$
and suppose that $k/14, k/f,  k/g, q/f, m/4\ell', fm/q, 2fk/3g(g-1) \in \mathbb{N}$.
Suppose that $(G,\cP,\cP',R,C,U,U')$ is an $(\ell',k,m,\eps,d)$-setup with $|G|=n$, $V_0=\emptyset$ and $C=V_1\dots V_k$.
Suppose that $\cP^*$ is a $(q/f)$-refinement of $\cP$ and that $SF_1,\dots, SF_{r_3}$ are edge-disjoint special factors with parameters $(q/f,f)$ 
with respect to $C$, $\cP^*$. Let $\mathcal{SF}:=SF_1+\dots+SF_{r_3}$.
Then there exists a spanning subdigraph $CA^\diamond(r)$ of $G$ for which the following holds:
\begin{itemize}
\item[(i)] $CA^\diamond(r)$ is an $(r_1+r_2)$-regular spanning subdigraph of $G$ which is edge-disjoint from $\mathcal{SF}$.
\item[(ii)] Suppose that $SF'_1,\dots, SF'_{r^\diamond}$ are special factors with parameters $(1,7)$
with respect to $C$, $\cP$ which are edge-disjoint from each other and from $CA^\diamond(r)+ \mathcal{SF}$.
Let $\mathcal{SF}':=SF'_1+\dots+SF'_{r^\diamond}$.
Then there exists a spanning subdigraph $PCA^\diamond(r)$ of $G$ for which the following holds:
\begin{itemize}
\item[(a)] $PCA^\diamond(r)$ is a $5r^\diamond$-regular spanning subdigraph of $G$ which
is edge-disjoint from $CA^\diamond(r)+ \mathcal{SF}+ \mathcal{SF}'$.
\item[(b)] Let $\mathcal{SPS}$ be the set consisting of all the $s'$ special path systems
contained in $\mathcal{SF}+ \mathcal{SF}'$.
Whenever $H$ is an $r$-regular digraph on $V(G)$ which is edge-disjoint from $G^{\rm rob}:=CA^\diamond(r)+ PCA^\diamond(r)+ \mathcal{SF}+ \mathcal{SF}'$,
then $H+ G^{\rm rob}$ has a decomposition into $s'$
edge-disjoint Hamilton cycles $C_1,\dots,C_{s'}$.
Moreover, $C_i$ contains one of the special path systems from $\mathcal{SPS}$, for each $i=1,\dots,s'$.
\end{itemize}
\end{itemize}
The analogue holds for an $(\ell',k,m,\eps,d)$-bi-setup $(G,\cP,\cP',R,C,U,U')$ if we assume in addition that $H$ is
bipartite with vertex classes $\bigcup \cV_{\rm even}$ and $\bigcup \cV_{\rm odd}$ (where $\cV_{\rm even}$ is
the set of all those $V_i$ such that $i$ is even and $\cV_{\rm odd}$ is defined analogously).
\end{lemma}


\section{Proofs of Theorems~\ref{orientcor} and~\ref{regdigraph}}\label{sec:proofs2}

The next result implies that a regular oriented graph with minimum semidegree at little larger than
$3n/8$ is a robust outexpander. Together with Theorem~\ref{decomp} this implies Theorem~\ref{orientcor}.

\begin{lemma}\label{38robust}
Let $0<1/n\ll \nu\ll \tau\le \eps/2\le 1$ and suppose that $G$ is an oriented graph on $n$ vertices with
$\delta^+(G)+\delta^-(G)+\delta(G)\ge 3n/2+\eps n$. Then $G$ is a robust $(\nu,\tau)$-outexpander.
\end{lemma}
\proof
Suppose not and let $X\subseteq V(G)$ be a set of vertices such that $\tau n\le |X|\le (1-\tau)n$ and $|RN^+_\nu (X)|<|X|+\nu n$.
Let $A:=X\cap RN^+_\nu (X)$, $B:=RN^+_\nu (X)\setminus X$, $D:=X\setminus RN^+_\nu (X)$ and $C:=V(G)\setminus (A\cup B\cup D)$.
Note that
\begin{equation}\label{eqBD}
|B|<|D|+\nu n.
\end{equation}

\noindent
\textbf{Claim~1.} $|A|+|B|+|D|\ge 2\delta^+(G)-2\tau n$

\smallskip

\noindent
To prove the claim, let us first assume that $|A|\ge \tau n/2$. Note that $e(A,C\cup D)\le \nu n(|C|+|D|)$ since
every vertex in $C\cup D$ has at most $\nu n$ inneighbours in $X\supseteq A$. Thus
\begin{align*}
e(A,A\cup C\cup D)& =e(A,A)+e(A,C\cup D)\le \frac{|A|^2}{2}+\nu n(|C|+|D|)\\
& \le \frac{|A|^2}{2}+\frac{2\nu}{\tau} |A|(|C|+|D|)\le \frac{|A|^2}{2}+|A|\frac{\tau n}{2}.
\end{align*}
So there exists a vertex $x\in A$ such that $|N^+(x)\cap (A\cup C\cup D)|\le |A|/2+\tau n/2$.
Hence $\delta^+(G)\le d^+(x)\le |A|/2+|B|+\tau n/2$. Together with (\ref{eqBD}) this implies Claim~1
in this case.

So let us next assume that $|A|\le \tau n/2$. Then $|D|\ge \tau n/2$ and
\begin{align*}
e(D,A\cup C\cup D)& =e(D,A)+e(D,C\cup D)\le |D|\frac{\tau n}{2}+\nu n(|C|+|D|)\\
& \le |D|\frac{\tau n}{2}+\frac{2\nu}{\tau} |D|(|C|+|D|)\le |D|\frac{3\tau n}{4}.
\end{align*}
So there exists a vertex $x\in D$ such that $|N^+(x)\cap (A\cup C\cup D)|\le 3\tau n/4$.
Hence $\delta^+(G)\le d^+(x)\le |B|+3\tau n/4$. As before, together with (\ref{eqBD}) this implies Claim~1.

\medskip

\noindent
\textbf{Claim~2.} $|B|+|C|+|D|\ge 2\delta^-(G)-3\nu n$

\smallskip

\noindent
To prove Claim~2, we first consider the case when $C\neq \emptyset$. An averaging argument shows that there is a
vertex $x\in C$ with $|N^-(x)\cap C|\le |C|/2$. But since $|N^-(x)\cap X|\le \nu n$ this means that
$\delta^-(G)\le d^-(x)\le |B|+|C|/2+\nu n$. Together with (\ref{eqBD}) this implies Claim~2.

So let us now assume that $C=\emptyset$. Together with the fact that $|A\cup B|=|RN^+_\nu(X)|<|X|+\nu n\le (1-\tau)n+\nu n<n$
this implies that $D\neq \emptyset$. But each vertex $x\in D$ satisfies $|N^-(x)\cap X|\le \nu n$ and so
$\delta^-(G)\le d^-(x)\le \nu n+|B|$. Together with (\ref{eqBD}) this implies Claim~2.

\medskip

\noindent
\textbf{Claim~3.} $|A|+|B|+|C|\ge \delta(G)-2\nu n$

\smallskip

\noindent
This clearly holds if $D=\emptyset$ (since $\delta(G)<n$). So suppose that $D\neq \emptyset$.
Then $e(D,D)\le e(X,D)\le \nu n|D|$ and so there is a vertex $x\in D$ with $|N^+(x)\cap D|\le \nu n$.
But since $D\cap RN^+_\nu(X)=\emptyset$ we also have that $|N^-(x)\cap D|\le \nu n$.
Thus $d(x)\le |A|+|B|+|C|+2\nu n$, which in turn implies Claim~3.

\medskip

Now Claims~1--3 together imply that
$$
3n\stackrel{(\ref{eqBD})}{\ge }3|A|+4|B|+3|C|+2|D|-\nu n\ge 2(\delta^+(G)+\delta^-(G)+\delta(G))-3\tau n>3n,
$$
a contradiction.
\endproof

\removelastskip\penalty55\medskip\noindent{\bf Proof of Theorem~\ref{orientcor}. }
Let $\tau^*:=\tau(3/8)$, where $\tau(3/8)$ is as defined in Theorem~\ref{decomp}.
Choose new constants $n_0\in\mathbb{N}$ and $\nu, \tau$ such that $0<1/n_0\ll \nu\ll \tau\le \eps/2,\tau^*$. 
Lemma~\ref{38robust} implies that $G$ is a robust $(\nu,\tau)$-outexpander and thus
also a robust $(\nu,\tau^*)$-outexpander. So we can apply Theorem~\ref{decomp} with $\alpha:=3/8$ to
find a Hamilton decomposition of $G$.
\endproof

The next result implies that a regular digraph with minimum semidegree at little larger than
$n/2$ is a robust outexpander. Similarly as before, together with Theorem~\ref{decomp} this implies Theorem~\ref{regdigraph}.

\begin{lemma}\label{regdiexpander}
Suppose that $0<\nu\le \tau\le \eps<1$ are such that $\eps\ge 2\nu/\tau$.
Let $G$ be a digraph on $n$ vertices with minimum semidegree $\delta^0(G)\ge (1/2+\eps)n$.
Then $G$ is a robust $(\nu,\tau)$-outexpander.
\end{lemma}
\proof
Consider any set $S\subseteq V(G)$ with $\tau n\le |S|\le  (1-\tau)n$.
Let $RN:=RN^+_{\nu,G}(S)$. We have to show that $|RN|\ge |S|+\nu n$. Suppose first that $|S|\ge n/2$.
Then every vertex of $G$ has at least $\eps n\ge \nu n$ inneighbours in $S$. So $RN=V(G)$. So we may assume that
$|S|\le n/2$. But
\begin{align*}
(1/2+\eps)n|S| \le e(S,V(G)) & =e(S,RN)+e(S,V(G)\setminus RN)\le |S||RN|+\nu n^2\\
& \le |S||RN|+\frac{\nu}{\tau}n|S|
\end{align*}
and so $|RN|\ge (1/2+\eps-\nu/\tau)n\ge (1+\eps)n/2\ge |S|+\nu n$, as required.
\endproof 

\removelastskip\penalty55\medskip\noindent{\bf Proof of Theorem~\ref{regdigraph}. }
Let $\tau^*:=\tau(1/2)$, where $\tau(1/2)$ is as defined in Theorem~\ref{decomp}.
Choose new constants $n_0\in\mathbb{N}$ and $\nu, \tau$ such that $0<1/n_0\ll \nu\ll \tau\le \eps,\tau^*$. 
Lemma~\ref{regdiexpander} implies that $G$ is a robust $(\nu,\tau)$-outexpander and thus
also a robust $(\nu,\tau^*)$-outexpander. So we can apply Theorem~\ref{decomp} with $\alpha:=1/2$ to
find a Hamilton decomposition of $G$.
\endproof

\section{Acknowledgements}

We would like to thank John Lapinskas for an idea which led to a simplification of the cycle absorbing argument.
We are extremely grateful to Allan Lo for his detailed comments on a draft of this paper.

\medskip

{\footnotesize \obeylines \parindent=0pt

Daniela K\"{u}hn, Deryk Osthus 
School of Mathematics
University of Birmingham
Edgbaston
Birmingham
B15 2TT
UK
}
\begin{flushleft}
{\it{E-mail addresses}:
\tt{\{d.kuhn,d.osthus\}@bham.ac.uk}}
\end{flushleft}

\end{document}